\documentclass[a4paper,12pt,openany]{amsbook}

\usepackage{amsmath}
\usepackage{amssymb,amsbsy,amsmath,amsfonts,amssymb,amscd}
\usepackage{enumitem,latexsym,amsthm,graphics}
\usepackage{tikz}
\usepackage{tikz-3dplot}
\usetikzlibrary{arrows}
\usepackage{tikz-cd}
\usepackage{pdfpages} 
\usepackage{array,color,comment,hyperref,longtable}
\usepackage{xy}
\input xy
\xyoption{all}

\usepackage[centering]{geometry} 

\usepackage{suffix}

\newcommand\chapterauthor[1]{\authortoc{#1}\printchapterauthor{#1}}
\WithSuffix\newcommand\chapterauthor*[1]{\printchapterauthor{#1}}

\makeatletter
\newcommand{\printchapterauthor}[1]{%
	{\parindent0pt\vspace*{-\baselineskip}%
	\center\linespread{1.1}\scshape\small#1%
	\par\nobreak\vspace*{1.5\baselineskip}}
	\@afterheading%
}
\newcommand{\authortoc}[1]{%
	\addtocontents{toc}{\vspace{-0.5\baselineskip}}%
	\addtocontents{toc}{%
	\protect\contentsline{chapter}%
	{\hfill\mdseries\scshape\protect\scriptsize#1}{}{}}
	\addtocontents{toc}{\vskip+0.25\baselineskip}%
}
\makeatother

\newcounter{AppCnt}
\newcommand*{\appendixOn}{
	\renewcommand*{\thesection}{Appendix~\Alph{AppCnt}}\addtocounter{AppCnt}{1}%
}
\newcommand*{\appendixOff}{
	\addtocounter{section}{-1}\renewcommand*{\thesection}{\arabic{section}}%
}

\usepackage{listings} 
 \lstdefinelanguage{Magma}%
  {%
   otherkeywords={:=,+:=,-:=,*:=},%
   procnamekeys={function,func,intrinsic,procedure,proc,return},%
   morekeywords={true,false},%
   morekeywords=[2]{adj,and,cat,cmpeq,cmpne,diff,div,eq,ge,gt,in,is,join,le,lt,%
          meet,mod,ne,notadj,notin,notsubset,or,sdiff,subset,xor},%
   morekeywords=[3]{assigned,break,by,case,catch,continue,declare,default,%
          delete,do,elif,else,end,eval,exists,exit,for,forall,fprintf,if,local,%
          not,print,printf,quit,random,read,readi,repeat,restore,save,select,%
          then,time,to,try,until,vprint,vprintf,vtime,when,where,while},%
   morekeywords=[4]{clear,forward,freeze,iload,import,load},%
   morekeywords=[5]{assert,assert2,assert3,error,require,requirege,requirerange},%
   morekeywords=[6]{car,comp,cop,elt,ext,frac,hom,ideal,iso,lideal,loc,map,%
          ncl,pmap,quo,rec,recformat,rep,rideal,sub},%
      sensitive,%
      morecomment=[l]//,%
      morecomment=[s]{/*}{*/},%
      morestring=[b]"%
  }[keywords,procnames,comments,strings]%
 
\lstnewenvironment{code_magma}[1][]
{\lstset{basicstyle=\scriptsize\ttfamily, columns=fullflexible,
language=Magma,
keywordstyle=\color{red}\bfseries,
commentstyle=\color{blue},tabsize=6,
numbers=left, numberstyle=\tiny,
stepnumber=1, numbersep=5pt, frame=single, #1}}{} 

\newcommand{\Sn}{{\mathfrak S}}
\newcommand{\es}{\mathsf{e}}
\newcommand{\fs}{\mathsf{f}}
\newcommand{\gs}{\mathsf{g}}
\newcommand{\hs}{\mathsf{h}}
\newcommand{\ks}{\mathsf{k}}
\newcommand{\ls}{\mathsf{l}}
\newcommand{\vs}{\mathsf{v}}
\newcommand{\ws}{\mathsf{w}}
\newcommand{\Cn}{{\mathfrak C}}
\DeclareMathOperator{\ev}{ev}
\DeclareMathOperator{\Sym}{Sym}
\DeclareMathOperator{\sign}{sign}
\DeclareMathOperator{\Spec}{Spec}
\DeclareMathOperator{\PSigmaL}{P\Sigma L}
\DeclareMathOperator{\PSL}{PSL}

\def\red{\color[rgb]{1,0,0}}

\newcommand\sC{{\mathcal C}}
\newcommand\sT{{\mathcal T}}
\newcommand\sD{{\mathcal D}}
\newcommand\sE{{\mathcal E}}
\newcommand\sA{{\mathcal A}}
\newcommand\sF{{\mathcal F}}
\newcommand\sG{{\mathcal G}}
\newcommand\sI{{\mathcal I}}
\newcommand\sJ{{\mathcal J}}
\newcommand\sL{{\mathcal L}}
\newcommand\sU{{\mathcal U}}
\newcommand\sQ{{\mathcal Q}}
\newcommand\hS{{\mathcal S}}

\newcommand\sB{{\mathcal B}}
\newcommand\sP{{\mathcal P}}
\newcommand\sN{{\mathcal N}}
\newcommand\sK{{\mathcal K}}
\newcommand\sX{{\mathcal X}}

\newcommand\sH{{\mathcal H}}

\newcommand\sM{{\mathcal M}}

\newcommand{\codeshorteff}[3]{\left[\underline{#1},#2\right]_{#3}}

\newcommand{\vek}[1]{\mathbf{#1}}

  \providecommand{\abs}[1]{\lvert#1\rvert}

\usepackage{textcomp}
               
\newcommand\om{\omega}
\newcommand\la{\lambda}
\newcommand\Lam{\Lambda}

\newcommand\be{\beta}
\newcommand\e{\epsilon}
\newcommand\s{\sigma}
\newcommand\z{\zeta}

\newcommand\Ga{\Gamma}
\newcommand\De{\Delta}
\newcommand\ga{\gamma}
\newcommand\de{\delta}

\DeclareMathOperator{\Pic}{Pic}

\DeclareMathOperator{\Div}{Div}

\DeclareMathOperator{\Gr}{Gr}

\DeclareMathOperator{\supp}{supp}

\newcommand{\CC}{\ensuremath{\mathbb{C}}}
\newcommand{\RR}{\ensuremath{\mathbb{R}}}
\newcommand{\ZZ}{\ensuremath{\mathbb{Z}}}
\newcommand{\QQ}{\ensuremath{\mathbb{Q}}}

\newcommand{\sS}{\ensuremath{\mathcal{S}}}

\newcommand{\NN}{\ensuremath{\mathbb{N}}}
\newcommand{\hol}{\ensuremath{\mathcal{O}}}

\newcommand{\PP}{\ensuremath{\mathbb{P}}}

\newcommand{\FF}{\ensuremath{\mathbb{F}}}

\newcommand{\ra}{\ensuremath{\rightarrow}}

\newcommand{\F}{\ensuremath{\mathbb{F}}}

\def\eea{\end{eqnarray*}}
\def\bea{\begin{eqnarray*}}

\DeclareMathOperator{\Aut}{Aut}
\DeclareMathOperator{\End}{End}

\DeclareMathOperator{\Sing}{Sing}

\DeclareMathOperator{\rank}{rank}
\DeclareMathOperator{\corank}{corank}

\DeclareMathOperator{\Milnor}{Milnor}

\newcommand\dual{\mathrel{\raise3pt\hbox{$\underline{\mathrm{\thinspace d
\thinspace}}$}}}

\newcommand\qe{\ifhmode\unskip\nobreak\fi\quad $\Box$}       

\def\BOX{\hfill\lower.5\baselineskip\hbox{$\Box$}}

\newtheorem{theorem}{Theorem}
\newtheorem{theo}[theorem]{Theorem}
\newtheorem{remark}[theorem]{Remark}
\newenvironment{rem}{\begin{remark}\rm}{\end{remark}}
\newtheorem{question}[theorem]{Question}
\newtheorem{prop}[theorem]{Proposition}
\newtheorem{cor}[theorem]{Corollary}
\newtheorem{lemma}[theorem]{Lemma}
\newtheorem{example}[theorem]{Example}
\newenvironment{ex}{\begin{example}\rm}{\end{example}}

\theoremstyle{definition}
\newtheorem{defin}[theorem]{Definition}
\newenvironment{definition}{\begin{defin}\rm}{\end{defin}}

\newcommand\NR{} 
\def\NR[#1,#2]{%
 node[draw, rounded corners ](#1){#2}
}

\newcommand{\sV}{\ensuremath{\mathcal{V}}}

\DeclareMathOperator{\im}{Im}

\DeclareMathOperator{\GL}{GL}

\def\red{\color[rgb]{1,0,0}}

\usepackage{hyperref}

\makeatletter
\def\tagform@#1{\maketag@@@{\ignorespaces#1\unskip\@@italiccorr}}
\makeatother

\newcolumntype{H}{@{}>{\lrbox0}l<{\endlrbox}} 

\begin{document}
\renewcommand{\thesection}{\thechapter.\arabic{section}}

\title[Nodal surfaces, coding theory and Cubic discriminants]{Varieties of Nodal surfaces, coding theory  and Discriminants of cubic hypersurfaces}

\thanks{ 
\ \\
\ \\
The authors acknowledge support of the ERC 2013 Advanced Research Grant - 340258 - TADMICAMT.\\
\ \\
\ \\
\textit{Dedicated to Ciro Ciliberto  on the occasion of his 70-th  birthday.} }

\author{by {\bf Fabrizio Catanese},\\ in  collaboration with:\\ Yonghwa Cho (parts 2,4),\\ Stephen Coughlan (part 2),\\ Davide Frapporti (part 2),\\
Michael Kiermaier (part 4),\\ Alessandro Verra (part 3),\\   and with appendices by Michael Kiermaier and Sascha Kurz}
\address{Fabrizio Catanese, Michael Kiermaier,  Sascha Kurz: 
 Mathematisches Institut der Universit\"{a}t
Bayreuth, NW II\\ Universit\"{a}tsstr. 30,
95447 Bayreuth, Germany.}
\email{Fabrizio.Catanese@uni-bayreuth.de}
\email{kiermaier@uni-bayreuth.de}
\email{sascha.kurz@uni-bayreuth.de}

\address{
Fabrizio Catanese:  Korea Institute for Advanced Study, Hoegiro 87, Seoul, 
133--722.}

\address{
Yonghwa Cho: Department of Mathematics, Gyeongsang National University, 501, Jinju-daero, Jinju-si, 52828  Gyeongsangnam-do, Republic of Korea.
}\email{yhcho@gnu.ac.kr }

\address{Stephen Coughlan:
Department of Mathematics \& Computer Studies,
Mary Immaculate College,
South Circular Road,
Limerick, Ireland,
V94 VN26
}\email{stephen.coughlan@mic.ul.ie}

\address{Davide Frapporti:  Politecnico di Milano,
Dipartimento di matematica, Via Bonardi 9,
20133 Milano, Italy.}
\email{davide.frapporti@polimi.it}

\address{Alessandro Verra: Dipartimento di Matematica, Universita' degli Studi di Roma Tre. 
}\email{sandro.verra@gmail.com}

\newgeometry{left=1.75cm,right=1.75cm, top=3cm, bottom=2cm}  
\maketitle

\newgeometry{left=4.2cm,right=4.2cm, top=3cm, bottom=3cm}  
\begin{abstract}
In this research monograph we exploit  two  binary codes associated to a projective nodal surface (the strict code $\sK$ and, for even degree $d$,  the extended code $\sK'$)
to investigate the `Nodal Severi varieties' $\sF(d, \nu)$ of nodal surfaces  in $\PP^3$ of degree $d$ and with $\nu$ nodes,
and their incidence hierarchy, relating partial smoothings to code shortenings.

Our first main result (the case of surfaces of degree $d \leq 3$ being elementary)  solves a  question which dates back 
over 100 years: it says that the irreducible components of $\sF(4, \nu)$
are in bijection with the isomorphism classes of their extended codes $\sK'$, that these  are exactly  all the  34 possible shortenings of the  extended Kummer code $\sK'_{Kum}$ ($\nu \leq 16$), and that a component is in the closure of another if and only if
the code of the latter is a shortening of the code of the former. 

 We then extend this result to the classification  of  all the other varieties of nodal K3 surfaces, of  all  possible   degrees $d = 2h$: again the connected components of the 
Nodal Severi variety $\sF_{K3}(d, \nu)$ of nodal K3's of degree $d$ and with $\nu$ nodes are in bijection with the isomorphism classes of their extended codes $\sK'$;  we classify  the   codes $\sK'$, showing that their occurrence  
 depends on the  congruence class of  $d$ modulo $16$; we also see
  that some unexpected families occur.

In 1979  Beauville  used coding theory to prove that $\mu(5)$, the maximum number of nodes that a nodal surface 
 in $\PP^3$ of degree 5 can have,
equals 31;  hence the `Togliatti quintics', the   discriminants of 
Togliatti cubics (cubic hypersurfaces in $\PP^5$
with 15 nodes) are the quintics with $\mu(5) = 31$ nodes.

For surfaces of degree $d=5$  in $\PP^3$ we determine (with one possible exception which remains an open question) all the possible codes $\sK$ of nodal quintics, and
show that we  do not get all the codes with weights 16 and 20 satisfying the  inequality pointed out by Beauville. 
Moreover, for several cases of $\sK$,
we show the irreducibility of the corresponding open set of $\sF(5, \nu)$: in particular, using the previous results on nodal K3 surfaces, we show the  irreducibility of the
respective  families  of
Togliatti quintic surfaces and of Togliatti cubic fourfolds  in  $\PP^5$.

 The fourth main  result, in the final part of  this paper,  shows that  a `Togliatti-like'  description holds for surfaces of degree 6 with the maximum number of nodes $\mu(6)=65$:
 they are discriminants of  cubic hypersurfaces in $\PP^6$
with  31  nodes (also of cubics with  32 nodes), and we have an irreducible 18-dimensional  family of them.

For degree $d=6$, our main result is based on these novel auxiliary results:

1) the study of the half-even sets of nodes on sextic surfaces,

2) the investigation of discriminants of  cubic hypersurfaces $X$, and the description of the  relation between the nodal singularities of $X$  and the nodal singularities of
the discriminant surface $Y$,

3) the computer assisted proof that, for $\nu = 65$, both codes $\sK, \sK'$ are uniquely determined,

4) the explicit description of these interesting codes, starting from the Barth sextic with $65$ nodes,
ending up with its relation to the Doro-Hall graph (in particular, we find  explicit discriminantal 
and determinantal realizations of the Barth sextic).

The codes of sextic surfaces are investigated, with a partial classification in the case of  `many' nodes.
\end{abstract}

\newgeometry{left=3cm,right=3cm, top=1cm, bottom=3cm} 
\tableofcontents 
\restoregeometry 

\newpage

\section*{Introduction}

In this paper we shall first of all consider  irreducible hypersurfaces $X$ of degree $d$  in complex projective space $\PP^N : = \PP^N_{\CC}$: these
are the zero sets of an irreducible homogeneous polynomial $ F(x_0, \dots, x_N)$ of degree $d$,
$$ X : = \{ x : = (x_0, \dots, x_N) \in \PP^N | F(x_0, \dots, x_N) = 0 \}.$$

The singular locus is the subset $\Sing (X) : = \{ x | \frac{\partial F }{ \partial x_i } = 0 \  \forall i=0, \dots, N \}$,
and a singular point $x^* \in X$  is said to be a {\bf node} if the quadratic part of the Taylor development of $F$ at $x^*$ is of maximal rank ($=N$,
in view of the Euler relation $ \sum_i x_i \frac{\partial F }{ \partial x_i } = d F$), 
that is,  the ($N + 1$) gradients $\nabla  \frac{\partial F }{ \partial x_i }$ have rank $=N$ at $x^*$.

A hypersurface $X$ is said to be {\bf nodal} if its singular points are just nodes.

\begin{defin} We propose in this paper the name of {\bf Nodal Severi variety $\sF_n(d, \nu)$} for the locally closed
set, inside the projective space parametrizing all hypersurfaces in $\PP^{n+1}$, consisting of the nodal hypersurfaces $X$ of degree $d$ and dimension $n$
in $\PP^{n+1}$ having exactly $\nu$ nodes as singularities.
\end{defin}

Severi in fact \cite{anhang} studied the case $n=1$, where one has to add the hypothesis that $X$ is irreducible,
and it turns out  that $\sF_1(d, \nu)$ (which  is nonempty if and only if $ \nu \leq \frac{ (d-1)(d-2)}{2}$), is  irreducible 
(the assertion, made in an appendix, Anhang F, of \cite{anhang}, had an incomplete proof,
see  \cite{harris}).

Several authors (see for instance  \cite{chiacil}, \cite{cilded})  extended these investigations to the case of nodal curves lying on smooth surfaces;
here, instead, we  mainly want to focus on the case of hypersurfaces and varieties of  dimension at least two.

We have then the following (mainly unanswered) questions:

\begin{question}
(1)  basic question: when is $\sF_n(d, \nu)$ non empty ?

The above question relies on the fact that, for fixed dimension $n$,  the number $\nu$ will be bounded by a function of the degree $d$, 
and this leads to another question:

(2)  defining 
$$  \mu_n(d) : =\max \{ \nu | \sF_n(d, \nu) \neq \emptyset \},$$
determine $  \mu_n(d)$ explicitly.

In turn, knowing the maximum  possible number of nodes $  \mu_n(d) $ that a hypersurface of degree $d$ can have, does not answer
the first question, since one may ask a weaker question:

(3) in particular, for which values of $(n,d)$ do we have the {\bf no-gaps phenomenon} that 
$\sF_n(d, \nu)$ is non empty for all $\nu \leq  \mu_n(d)$?

In turn, once we know that such a Severi variety $\sF_n(d, \nu)$ is non empty, one may ask

(4) when is $\sF_n(d, \nu)$ irreducible? How to determine its irreducible components?

(5) When is $\sF_n(d, \nu)$ smooth?

\end{question}

\bigskip

Concerning the last question (5), one can give a sufficient condition for smoothness, namely the condition 
of unobstructedness (which will be used in the paper), meaning that the Severi variety is the fibre of
a submersion. This criterion is used for instance by  Dimca and Kloostermann \cite{dimca} \cite{remke}\footnote{we became aware of their results after everything was written up.}  to show that 
$\sF_2(d, \nu)$ is smooth for $d=5,6,7$, and has pure codimension $\nu$ in the space of degree $d$ surfaces.
In \cite{remke}  examples are given, showing that $\sF_2(d, 4 (d-1))$ has a nonempty singular locus, for  $ d \geq 9$.

Already quite interesting, difficult and wide open is the case of surfaces: $n=2$. In this case  we shall often denote $\sF_2(d, \nu)$ simply by $\sF(d, \nu)$, and $  \mu_2(d)$ by $  \mu(d)$. 

To give a vague idea of the state of the art, let us give the following definition:

\begin{defin}
Let $X$ be a  hypersurface of degree $d$ in $\PP^{n+1}$ with isolated singularities, and let 
$$s(X): = |\Sing(X)|$$ be the number of singular points,
and $$ \Milnor(X) : = \sum_{P \in X} \Milnor (X_{ P})$$ the sum of the Milnor numbers at the singular points;
recall here that the Milnor number is the dimension of $\hol_{X,P} / (  \frac{\partial F }{ \partial x_0 }, \dots  , \frac{\partial F }{ \partial x_{n+1} })$, and that the Milnor number
 equals $1$ only for the nodes.

We can then consider the maximal number of singular points
$$  s_n(d) : = \max \{ s(X)  | X  \subset \PP^{n+1}, \deg(X) = d, \  |\Sing(X )| < \infty  \},$$
and similarly 
$\Milnor_n(d)$.

For $n=2$ we use the simpler notation $ s(d) : = s_n(d)$, $\Milnor(d) :=\Milnor_2(d)$.
\end{defin}

We clearly have a chain of inequalities:
$$  \mu_n(d) \leq  s_n(d) \leq  \Milnor_n(d).$$

 The  inequality $  \mu_n(d) \leq   \Milnor_n(d)$ is strict, since Miyaoka proved \cite{miyaoka} $\mu(d) \leq \frac{4}{9} d (d-1)^2$,
while we can show by easy examples (Example \ref{Milnor}) that  $\Milnor(d) \geq \frac{1}{2} d (d-1)^2$.

Also the inequality $ s_n(d) \leq \Milnor_n(d)$ is strict (but not by such an asymptotically large  difference as before) as we can see comparing  general inequalities by Bruce \cite{bruce}
for $s_n(d)$ with these examples (see again Example \ref{Milnor}).

Concerning the first inequality, one may ask, assuming that life is simple,  the following  

\begin{question}
Is it true that $\mu(d) =  s(d)$, or even $\mu_n(d) =  s_n(d)\  \forall n$?
\end{question}

It is in general very difficult to determine the  numbers $$  \mu_n(d),  s_n(d) ,   \Milnor_n(d).$$

Chmutov \cite{chmutov} showed the existence, for each degree $d$, of nodal surfaces $X$ of degree $d$ with
$\nu(X)$ equal to a polynomial of degree $3$ in $d$ (depending on the congruence class of $d$ modulo $6$),
and with leading coefficient $\frac{5}{12} d^3$, so that 
that asymptotically $  \mu_n(d)$ grows like at least $ \frac{5}{12} d^3$.

By Miyaoka's inequality and Chmutov's examples the ratio $ \frac{ \mu_n(d)} {d^3}$ satisfies 
$$ \frac{15}{36} - o(d)  \leq \frac{ \mu_n(d)} {d^3} \leq  \frac{16}{36} ,$$
and it is an intriguing question to ask for the determination of the two numbers in the interval $[ \frac{15}{36}, \frac{16}{36}]$,
$\liminf  ( \frac{ \mu_n(d)} {d^3})$ and $\limsup  ( \frac{ \mu_n(d)} {d^3})$.

We conclude this general discussion simply saying  here that $ s(d) = \mu(d)$ for $ d \leq 4$, and  that the exact values of $ \mu(d) = \mu_2(d)$ are known only for $d \leq 6$,
$$  \mu(d)=0,1,4,16,31,65 \ {\rm for } \ d=1,2,3,4,5,6. $$

\begin{itemize}
\item 
$d=1,2$: trivial;
\item
$d=3$: so-called Cayley cubic, first apparently found by Schl\"afli, see \cite{schlaefli}, \cite{cremona},  \cite{cayley};
\item
$d=4$: Kummer surfaces \cite{kummer} (first examples found by Fresnel, 1822);
\item
$d=5$: Togliatti quintics, see \cite{togliattiquintics} for the inequality $ \mu(5) \geq 31$, 
and Beauville \cite{angers} for the inequality $ \mu(5) \leq 31$;
\item
$d=6$: Barth \cite{barth} $ \mu(6) \geq 65$, Jaffe and Ruberman \cite{j-r} $ \mu(6) \leq 65$ (the  proof was then simplified  by Jonny Wahl,
see \cite{wahl} and \cite{p-t}).

\end{itemize}

Also, it  turns out that in this range $d \leq 6$ (as we shall explain in this article) the answer to questions $1$ and $3$ is that there are no gaps
(in view of \cite{cat-ces} and \cite{rohn86}, \cite{rohn87} the only new result here is for the case $d=5$, see corollary \ref{nogap}):

\begin{cor}\label{no-gap}
If $n=2$, $d \leq 6$, then  we have the `no gaps' phenomenon for nodal surfaces: $\sF_2(d, \nu)$ is non-empty for all $ \nu \leq \mu(d)$.
 This means: for each $ \nu$ with $ 1 \leq \nu \leq \mu (d)$ there exists a nodal surface $X$ of degree $d$ in $\PP^3$
 with precisely $\nu$ nodes.
 \end{cor}

The answer to question $4$ is classical for $n=2$, $d \leq 3$: the Nodal Severi varieties are smooth and irreducible.

\medskip

While the case $d=4$, in spite of a lot of classical work (\cite{rohn86}, \cite{rohn87}, \cite{jessop1916quartic}), was until now open
(see the remarks in the first three lines of page 349 of \cite{endrass2}). 
We are  here able to classify
the irreducible components of $\sF_2(4, \nu)$, for $\nu \leq 16$, we have indeed our main theorem for nodal quartics:

\begin{theo}\label{Nodal-Quartics}

(I)  The subset of the Nodal Severi variety  $\sF(4, \nu)$ corresponding to  the nodal quartic surfaces with  $\nu $  nodes and  fixed  binary code
$\sK''$ is a smooth connected component.

(II) The possible codes $\sK''$ appearing are  all the shortenings of the extended  Kummer code $\sK''_{Kum}$.

(III) A component of   $\sF(4, \nu)$ is in the closure of another component of $\sF(4, \nu')$
if and only if $\nu \geq \nu'$ and the second code is a shortening of the first.

\end{theo}

The main thrust of the above theorem is of being  a conceptual statement, as opposed to hardly readable and very long computer generated lists, but indeed it 
leads to a very explicit classification, and 
the possible codes $\sK''$ appearing are listed in 
 Section \ref{append_quadratic_Reed-Muller_code} (11 nontrivial codes in their minimal embedding);
 the classification\footnote{again, everything is done by hand using linear algebra.} shows  for instance that 
the Nodal Severi variety $\sF (4, \nu)$ has exactly one irreducible component for $ \nu = 0,1,2,3,4,5, 14,15,16$,
exactly two irreducible components for $ \nu = 6,7,13$, exactly three irreducible components for $ \nu = 8,9, $
exactly four  irreducible components for $ \nu = 11,12, $ exactly five  irreducible components for $ \nu = 10. $
Hence  its stratification is made of 34 strata.

A direct consequence is:

\begin{cor}\label{fungr-quartics}
Let $Y$ be a nodal quartic surface in $\PP^3$, with $\nu$ nodes, where $ \nu \leq 15$.

 Then the fundamental group $\pi_1 (Y^*)$ of its smooth locus is finite and isomorphic to its code 
 $\sK \cong  (\ZZ/2)^k$.

For $\nu = 16$ $\pi_1 (Y^*)$ is the semidirect product
$\ZZ^4 \rtimes  \ZZ/2$ corresponding to multiplication by $-1$.
\end{cor}

The main difference from the case of curves is given here by topology: given a nodal hypersurface $X$, we consider the
nonsingular part of $X$, $X^* : = X \setminus \Sing(X)$. In the case of curves we just delete $ 2 \nu$ points from a Riemann surface 
of genus $ g = \frac{(d-1)(d-2)}{2} - \nu$, so the topology is detected by the integers $d, \nu$.

For the case of surfaces $Y$, instead,  we have four important topological invariants of the connected components of $\sF(d, \nu)$.
The first two are
 the fundamental group $\pi_1 (Y^*)$,
 and its abelianization  $H_1(Y^*, \ZZ)$, which is a vector space over $\ZZ/2$ and  a quotient of $(\ZZ/2)^{\nu}$. The second two are  
similarly defined: letting  $Y'$ be the complement in $Y^*$ of a smooth  hyperplane section, we have 
the fundamental group $\pi_1 (Y')$,
and its abelianization  $H_1(Y', \ZZ)$,  a quotient of $(\ZZ/2)^{\nu} \oplus (\ZZ/d)$ (see Corollary \ref{2}
and Corollary  \ref{codehomology2}).

While the fundamental groups are intriguing groups, and difficult to compute for $ d \geq 5$, these homology groups
are somewhat easier to calculate, and through them coding theory comes into play, as first pioneered  in the seminal paper
\cite{angers}.

\begin{defin}
Let $Y$ be a nodal surface in $\PP^3$, and let $\sS$ be its singular set $$\sS : = \Sing (Y) = \{P_1, \dots, P_{\nu}\}:$$
  then the {\bf strict code} $\sK \subset (\ZZ/2)^{\sS}$ associated to $Y$ is defined as the image of  the  injective
 linear map   dual to  the surjective linear map
 $ (\ZZ/2)^{\sS} \ra H_1(Y^*, \ZZ)$. 
 
 If $d$ is even,   the {\bf extended code} $\sK' \subset (\ZZ/2)^{\sS} \oplus (\ZZ/2)H$ associated to $Y$
  is defined via the  injective  linear map dual the surjective linear map  obtained from the
  surjection $(\ZZ/2)^{\sS} \oplus (\ZZ/d)H \ra H_1(Y', \ZZ)$
  by setting $\frac{d}{2} H = 0$. It is a {\bf bicoloured code}, as we are now going to explain.
\end{defin}

In order to fully understand  the statement of  Theorem \ref{Nodal-Quartics} above, one  needs to recall some basic concepts
and results from coding theory, for this we refer to the first section of Part 1.

\bigskip

The main use of coding theory is based  on the determination of the cardinalities of strictly even sets of nodes (weights of vectors in $\sK$),
and of half-even sets of nodes (vectors in the projection $\sK''$ of $\sK' \setminus \sK$ inside  $ \FF_2^{\sS}$), which was initiated
in \cite{babbage}, \cite{angers},  and continued in several papers (\cite{cascat}, \cite{endrass1}, \cite{cat-ton}), and which has lead to
the  complete   classification of
the $\sK$-weights for surfaces of degree up to $6$ (we  did not yet achieve the complete classification for 
the $\sK'$-weights of sextic surfaces,  but we obtain here, concerning the  half-even sets, new results   sufficient for our purposes).

In degree $d=4$, the weights for $\sK$ can only be $8,16$, and those for $\sK''$ can only be  $8, 16, 6,10$.
While the classification of the codes $\sK$, even for more general K3 surfaces,  is well known since long time,
the classification of the codes $\sK'$ is a much harder and newer task,  because  of the occurrence of more than 4 possible weights
(which depend on the congruence class of $d$ modulo $8$).

The main theorem for quartics, theorem  \ref{Nodal-Quartics}, is based on several ingredients: first of all elementary coding theory,
then the relation between unobstructedness and shortenings (theorem \ref{thm_d_realized}), then Nikulin's theorems on primitive embeddings
of lattices in the K3 lattice  \cite{nikulin} and finally  the Torelli theorem for K3 surfaces (see \cite{torelli} \cite{torelliB}).

The relation with lattices occurs as follows:
for nodal surfaces $Y$, with minimal resolution $\tilde{Y}$, the lattice
$\Lambda : = H^2 (\tilde{Y}, \ZZ)  $ contains the sublattice 
$$\sL = \bigoplus_i \ZZ E_i \oplus \ZZ H.$$
Now,  the extended code $\sK'$ determines the saturation   $\sL^{sat} : = \Lambda \cap \QQ \sL$ of $\sL$,
which is primitively embedded inside $\Lambda : = H^2 (\tilde{Y}, \ZZ)  $.

In the 70's Nikulin developed a theory of primitive embeddings of lattices, generalizing results of Milnor et al. 
to the non unimodular case. Nikulin's results offer criteria for existence and unicity of such embeddings.

These results are very useful for the class of K3 surfaces, of which resolutions of nodal quartic surfaces 
are a special case.

 If we fix  a `potential' K3 code $\sK'$,
 it turns out, using the  results of Nikulin on integral quadratic forms that 
there is at most one    primitive embedding of  $\sL^{sat} : = \Lambda \cap \QQ \sL$ inside $\Lambda$.
Some delicate calculations show that the existence of such an embedding depends on the congruence class of $d$ modulo $16$, while the potential K3 codes depend only on the congruence class of $d$ modulo $8$.

\bigskip

In this way the theorem for quartics extends also to all nodal K3 surfaces as follows (see Theorem \ref{mtK3}
for fuller details and a complete classification).

\begin{theo}
The connected components of the 
Nodal Severi variety $\sF_{K3}(d, \nu)$ of nodal K3' s of degree $d$ ($d$ must be an even number) and with $\nu$ nodes are in bijection with the isomorphism classes of their extended codes $\sK'$ (equivalently, of the pair of codes $\sK \subset \sK'' \subset \FF_2^{\nu}$).

\bigskip

The weights $t$ of vectors in $\sK'' \setminus \sK$
satisfy the following congruence:
$$   2 t  - d \equiv 0 \ ( mod \ 8) .$$

Apart from sporadic cases, which are however obtained via the projection from a node of a nodal K3 surfaces, we have
several main stream cases:

\smallskip

(0) If $ d \equiv 0 \ ( mod \ 8) $ then  $ t \in \{ 4,8,12 \}$ and we get all the shortenings  $\sK''$ of the K8 code 
(generated by the strict Kummer code $\sK_{Kum}$ and by the characteristic function  of an affine plane $\pi$).

\smallskip

(2) If $ d \equiv 2 \ ( mod \ 8) $ and $ t \in \{ 5,9 \}$, we get all the shortenings  $\sK''$  of the code  
generated by the linear functions
on $\FF_2^4$ and by the quadratic function $x_1 y_1 + x_2 y_2 + 1$. 

\smallskip

(4-6) If $ d \equiv 4 \ ( mod \ 8) $ and $ t \in \{ 6,10 \}$, we get for $\sK''$ all the shortenings of the extended Kummer code
$\sK''_{Kum}$.

(4-4) If  $ d \equiv 4 \ ( mod \ 8) $ and there is some  $t=2$ we get all   the shortenings of the special 
code $\sK'$ occurring for $ \nu=14$, $k'=4$.

\smallskip

(6) If $ d \equiv 6 \ ( mod \ 8) $  we get all the shortenings $\sK'$ of the two codes
occurring  for $\nu=15$:   the first $\sK'$   is the strict Kummer code, in the second $\sK'$
 there is a weight $4$ vector.
\end{theo}

In part I we also give a complete list of all the codes $\sK'$ ($\sK''$)  which occur, making the above theorem suitable
for applications.

An interesting consequence is the following (see Theorem \ref{fund-gr-nodalK3})

\begin{theo}
Let $Y$ be a nodal K3 surface with $\nu$ nodes. Then, for $\nu \leq 15$, the fundamental group $\pi_1(Y^*)$ of the 
smooth locus of $Y$ is isomorphic to the  strict code $\sK \cong (\ZZ/2)^k$ associated to $Y$.

 For $\nu = 16$ $Y$ is a Kummer surface and  $\pi_1 (Y^*)$ is the semidirect product
$\ZZ^4 \rtimes  \ZZ/2$.
\end{theo}

\bigskip

In the  third  part of the paper we proceed to study the Nodal Severi varieties $\sF(d, \nu)$ for surfaces of degree $d=5$.
In this degree the possible weights can only be $16,20$, and this fact was one basic tool which allowed Beauville in his seminal paper \cite{angers}
to show that $\mu(5) = 31$, and that if $\nu= 31$, then $\sK$ is the biduality code of $\sK \cong \FF_2^5$.

His first crucial   observation consisted in  what we call here the {\bf B-inequalities= Basic inequalities = Beauville inequalities},  
which are the lower bounds   $ k \geq \nu -  \frac{1}{2} b_2(d)$, $ k' \geq \nu + 1 -  \frac{1}{2} b_2(d)$, see proposition \ref{ineq}).

Using  a result of \cite{babbage}, namely that every quintic surface with $\nu$ nodes and with an even set of nodes of cardinality $16$ is
 the discriminant  of a cubic hypersurface $X \subset \PP^5$ with $\nu - 16$ nodes, for  the projection onto $\PP^3$
with centre a line $L \subset X^*$ (contained in the smooth locus of $X$, as we show here),
 Beauville was able to show that every quintic with $31$ nodes is a Togliatti quintic,
that is, it is  the discriminant  of a cubic hypersurface $X \subset \PP^5$ with 15 nodes.

We prove here a new result which  is not unexpected, and actually we have believed for long time  to hold true: both Togliatti quintics (quintic surfaces with $31$ nodes)
and Togliatti hypercubics (cubic hypersurfaces $X \subset \PP^5$ with 15 nodes) form  respective  irreducible families.

Again this result is based also on similar ingredients to the previous one on quartics, 
in particular it relies on the Torelli theorem for K3 surfaces, this time applied to complete intersections of type $(2,3)$ inside $\PP^4$.

We actually prove more: 

\begin{theo}
The Nodal Severi varieties $\sF_2(5, 31)$ and  $\sF_4(3, 15)$ are irreducible. 

For $\nu \geq 29$ the weight $20$ does not appear in the code $\sK$ and the correspondence  given by taking discriminants of cubic hypersurfaces yields a bijection between the
respective  irreducible components
of  $\sF_2(5, \nu)$ and  $\sF_4(3, \nu - 16)$. 

 These are listed in Subsection \ref{degree6}.

In the Nodal Severi variety $\sF_2(5, 28)$ there is exactly one   irreducible component which  corresponds to the case where $\sK$ contains 
two  codewords
of weight $20$.

All the codes $\sK$ of nodal quintic surfaces are just:

(1) either  the shortenings of the biduality code  (simplex code), coming from smoothings  of the Togliatti quintics, or

(2) the shortenings of the unique  2-dimensional code  with two  codewords
of weight $20$,

(3) possibly the unique  2-dimensional code  with (exactly) two  codewords
of weight $16$.

 Moreover,  $\sF_2(5, 20)$ contains as  irreducible component the variety of quintic symmetroids ($ dim (\sK) = 1$, with the weight $20$),
 and at least two other components (correponding to shortenings of the simplex code, one with $ dim (\sK) = 1$ and  the weight $16$, 
 the other with  $ dim (\sK) = 0$).
\end{theo}

 Appendix B  completely classifies all the codes with weights $16,20$, and those among them which satisfy the B-inequality
(a necessary condition for the codes of nodal surfaces), 
hence the codes which could potentially be codes of some nodal quintic surface.

Our main result is that all these do  actually occur as codes of some quintic surface, with one exception (third in the list)
and  the possible exception
of one case (sixth in the list) which we cannot resolve, leaving therefore an interesting open question.

\bigskip

The fourth  part of the paper,  and most of the second part, devoted to cubic hypersurfaces and their discriminants, 
are concerned with  our initial aim: the study of the codes of nodal sextics, especially of those  with 65 nodes,
the maximum allowed. 

The main underlying idea was that the method of discriminants of cubic hypersurfaces should also work for $d=6$,
and we show that this is the case, even if  in a rather  subtle way.

Recall  that Barth constructed \cite{barth} a sextic  with 65 nodes, and admitting an icosahedral symmetry;
 we show here, by the way,  that  the Barth sextic admits as  group of automorphisms exactly  
 the Icosahedral group  $(\ZZ/2) \times \mathfrak A_5$.

Our research began with describing first the codes $\sK_B, \sK_B'$ of the Barth sextic, in an explicit way. Later on Sascha Kurz
  was able 
to show,  (see appendix C) via extensive computer calculations (using that the weights of $\sK$ lie in the set $\{24, 32, 40, 56\}$,
and that $\sK$ can be extended to $\sK'$) that every sextic surface with $65$ nodes has the same extended
code as the Barth sextic. 

Using the partial classification of half-even sets of nodes on sextics that we achieve here in subsection \ref{half-even-sets-sextic}\footnote{A similar classification of cardinalities of even sets of nodes on surfaces of degree $7$ could pave the way to determining $\mu(7)$, but the enterprise looks not easy to fulfill.}, 
the knowledge  of the code for $\nu = 65$, and the study of discriminants performed in some sections of the paper
(parts 2 and 4), 
we are able to show the following theorem:

\begin{theo}
Every sextic $Y$ with 65 nodes is the discriminant of a cubic hypersurface $ X \subset \PP^6$ with  31 nodes 
for the projection
of $X$ onto $\PP^3$ with centre a  plane $L$ intersecting $\Sing(X)$ in $3$  nodes.

  $Y$ can be also realized as the discriminant of a cubic hypersurface $ X \subset \PP^6$ with  32 nodes 
for the projection
of $X$ onto $\PP^3$ with centre a  plane $L$ intersecting $\Sing(X)$ in $2$  nodes.
\end{theo}

The subtle and quite interesting point is that, unlike  the case of quintics, where we project using Togliatti cubics, those  with the maximal number of nodes,
we cannot use here  the Segre cubic $ X_S \subset \PP^6$, which has the maximal number ($=35$) of nodes: since all planes contained in $X_S$
are subsets of a three dimensional subspace contained in $X_S$. 

We also show that the Segre cubic $X_S$ is locally rigid (up to projective equivalence)  in the variety $\sF_5(3,35)$.

We conjecture that every cubic hypersurface $ X \subset \PP^6$ with 35 nodes is projectively equivalent to $X_S$, 
and we pose  similar questions, for instance regarding  the Segre cubic hypersurfaces in projective spaces of even dimension.

 The crucial question about the irreducibility of $\sF_2(6,65)$ motivates the analogous question about  the irreducibility of 
$\sF_5(3,32)$. 

Unlike the case of nodal quartics and quintics, where the binary codes occurring are easy to describe in terms of elementary linear algebra,
for the case of the Barth sextic (the case of 65 nodal sextics),  the description of the codes $\sK_B, \sK_B'$ is rather intriguing.
As discovered by Michael Kiermaier after Yonghwa Cho's explicit description of these codes, these are associated to the Doro-Hall graph, which we briefly describe now.

\begin{theo}
The code $\sK_B$ of the Barth sextic $Y_B$ is associated to the following graph $\Ga$, whose vertices are the nodes of $Y_B$,
and where an edge joins two vertices $P, P'$ if and only if the line $\overline{P P'}$ intersects $Y_B$ in 4 distinct points.

For each vertex $P$, there are exactly $10$ vertices at distance $1$, $30$ at distance $2$, and $24$ at distance $3$.
The code $\sK_B$ is generated by the characteristic functions of the 3-spheres of $\Ga$ (sets of points at distance 3 from one vertex).
\end{theo}

The graph $\Ga$ above turns out to be  the so-called Doro-Hall graph, whose group of automorphisms
is equal to   the group of collineations 
$$ \PP \Sigma L (2, 25) : = \PP S L (2, \FF_{25}) \rtimes (\ZZ/2)$$
(the non-normal subgroup is the one generated by the Frobenius automorphism of $\FF_{25}$,
yielding a semi-projectivity).

Doro was the first to construct this graph  $\Ga$, whose  set of vertices is the conjugacy class of  the Frobenius automorphism, with an edge connecting two
elements if and only if they commute. Later Hall \cite{hall} showed that this was the largest of the class of 
the so-called locally Petersen graphs.

Since the Petersen graph is the graph associated to the vertices of the projective dodecahedron (quotient of the
dodecahedron by the antipodal map), and the dodecahedron is the dual polyhedron to the icosahedron,
it is natural to ask whether it is possible to find a direct explanation, starting from the explicit equation of the Barth surface
exhibiting its $\mathfrak A_5$-symmetry, of the fact that the code $\sK_B$ is the Doro-Hall code associated
to the graph $\Ga$ as we explained above \footnote{A picture of the Doro-Hall graph can be found in \url{https://www.wolframalpha.com/input?i=Hall+Graph},
but it is not easy to see the local dodecahedral symmetry, since the number of vertices is 65, and the number of edges is 325,
hence it is too intertwined.}.

\bigskip

 The scripts supporting the computer assisted proofs are available at the arXiv-webpage of the manuscipt, see \cite{Scripts}.

 \bigskip
 
{\bf Acknowledgements. }{\em Fabrizio  would like to thank Michael L\"onne for useful advice concerning the  proof of theorem \ref{Nodal-Quartics}, and Ciro Ciliberto for answering some queries. We would all like to heartily thank  Alfred Wassermann for his initial contribution which lead to the classification of the codes of 65- nodal sextics, achieved in \cite{kurz}, see appendix C.

 Thanks also to the referee for some advice concerning the exposition.}

\chapter{Generalities and nodal K3 surfaces}
\chapterauthor{Fabrizio Catanese}

\section{Coding theory notation and basics}

\begin{enumerate}
\item
Given a field $\FF$, an  $\FF$-linear  code is an $\FF$-vector subspace $\sK$ of the vector space $\FF^{\sS}$ of the $\FF$-valued functions on a given set $\sS$;
  {\bf the dimension of $\sK$ shall be denoted by   $k$}.  
  
  Two such codes  
$\sK_1, \sK_2$
are equivalent if there is a bijection $\sS_1 \cong \sS_2$ between the two sets carrying $\sK_1$ to $\sK_2$.

We have a {\bf binary} code if $\FF = \FF_2 : = \ZZ/2$.

We shall for simplicity restrict to this case.

It goes without saying that for a binary code any function is the characteristic function $\chi_A$ of a set $A$,
and addition corresponds to the {\bf symmetric difference $$ (A \setminus B)\cup (B \setminus A)= (A\cup B)  \setminus (A \cap B).$$}
\item
A bicoloured $\FF$-linear code is  an $\FF$-vector subspace $\sK'$ of the vector space $\FF^{\sA \cup \sB}$, where 
$\sA \cup \sB$ is a disjoint union. Two such codes $\sK'_1, \sK'_2$
are equivalent if there are bijections $\sA_1 \cong \sA_2$, $\sB_1 \cong \sB_2$  carrying $\sK'_1$ to $\sK'_2$.

 To $\sK'$ we associate the code $\sK : = \sK' \cap \FF^{\sA}$.

In this paper  $\sB$ shall consist of a single element $ H$ (which we identify to its
characteristic function), and, since we shall make  the assumption that $H \notin \sK'$,
  the bicoloured code $\sK'$ will be  determined
by a pair of codes $\sK \subset \sK'' \subset \FF^{\sA}$, where  $\sK$ is as above and  $\sK''$ is the projection of $\sK'$ into $\FF^{\sA}$
(see item iii) below).

In particular $k' : =  \dim (\sK') =  \dim (\sK'')$.
\item
Given a code $\sK \subset \FF^{\sS}$, and $\sN \subset \sS$, the {\bf shortening} $\sK_{\sN} $ of $\sK$ with respect to $\sN$
is just the intersection $\sK \cap \FF^{\sN}$ ($\sK_{\sN} $ is  just the subspace of the functions in $\sK$ with support in $\sN$).

The {\bf projection} of $\sK$ determined by $\sN$ is just the image of $\sK$ via the restriction map $p_{\sN} : \FF^{\sS} \ra \FF^{\sN}.$

Hence the shortening  with respect to $\sN$ is just the kernel of the projection determined by $\sS \setminus \sN$.

\item
The {\bf effective length} $n$ of $\sK$ is just the cardinality of the union $\sS' \subset \sS$ of the supports of all  the functions in $\sK$. 
$\sK$ is said to be a {\bf spanning code} if $ n = \nu$, where $\nu : = | \sS |$  is called the {\bf length } of the code.
\item
For a vector $v \in \sK$, the {\bf weight } $w(v)$ (also called Hamming weight) is just the cardinality of the support of the function $v$.
The {\bf minimum weight }  $w(\sK)$ is the minimum of the weights of non-zero vectors $v \in \sK$.

The {\bf weight enumerator function} of $\sK$ is the function $w_{\sK}(x,y) = \sum_i a_i x^{n-i} y^i$, where $n$ is the effective length and $a_i $ is the number of vectors of $\sK$ with weight equal to $i$.
\item
Dual to the injection $\sK \subset \FF^{\sS}$ we have a surjection $\phi : \FF^{\sS} \ra \sK^{\vee}$.
The {\bf dual} code of $\sK$ is just $\ker (\phi)$, that is,  the perpendicular $\sK^{\perp} \subset \FF^{\sS}$ to $\sK$.
Given a subset  $\sN \subset \sS$, duality exchanges the projection with the shortening: i.e.,  
$$(p_{\sN}(\sK))^{\perp} = (\sK^{\perp})_{\sN}.$$
$\sK$ is said to be a {\bf projective code} if the dual code admits no vector of weight equal to $2$.

\item
The {\bf biduality code} is just the natural embedding $\sK \ra \FF^{\sK^{\vee} \setminus \{0\}}$.
Its dual code  is called the {\bf Hamming code}, and  the biduality code is also called the {\bf simplex} code
(some texts use the historically correct but cumbersome term: `dual Hamming code').
\item
The {\bf McWilliams identity} determines the weight enumerator of the dual code:
\begin{equation}\label{mw1} w_{\sK^{\perp}} (x,y) = \frac{1}{2^k} w_{\sK}(x+y,x-y).
\end{equation}
\item
The McWilliams identity reads out as a series of identities, the first two being:
\begin{equation}\label{mw2} \sum_{i >0} a_i = 2^k -1, \ \sum_{i >0} i a_i = 2^{k -1} n.\end{equation}
\item
Assume that we have a code $\sK$ where all the weights are just $2^m$, and let $k=dim(\sK)$, $n $ be  the effective length.
In this case we get from the McWilliams identity:
$$ 2^m (2^k -1) = 2^{k -1} n \Rightarrow n = (2^k -1) 2^ {(m - k +1)}  \Rightarrow k \leq m+1.$$
 If $k = m+1$, we have the simplex code, in general a subcode of the simplex code. 
 \item
 Bonisoli \cite{bonisoli} showed that
each code with only one weight is gotten from the  simplex code via a product structure: that is, 
$\sK \subset \FF_2^{X\times Y}$ via the composition of the inclusions $\sK \subset \FF_2^{ Y} \subset \FF_2^{X\times Y}$,
in shorthand notation $v(x,y) = v(y)$.
\end{enumerate}

In the next subsection we discuss the case of spanning codes, and Radon duality.

\subsection{Codes and Radon duality on finite projective spaces}

Consider now more generally a linear code over the finite field $\F_q$, that is, $\sK $ is a vector subspace 
of $\FF_q^{\nu}$ of dimension $k$ and which is spanning (full), that is, it has  effective  length $n = \nu$. 

Let $\PP $ be the set of points of the projective  space $\PP^{k-1}_{\FF_{q}}$.

 Up to a permutation automorphism of  $\FF_q^n$, such a spanning $\F_q$-linear code $\sK$ 
 is determined by a multiset of $n$ 
 points in $\PP$, given   by the  $n$  rows   of a generator matrix  $G$ (a multiset is  a map 
 $\mathcal{C} : \PP \to \ZZ_{\geq 0}$): because no such row can be zero, else $\sK$ is not spanning.
 
  Two multisets yield the same code if and only if they are projectively
equivalent (this amounts to a change of basis for $\sK$). Moreover, since  $\sK$ has dimension $k$, the multiset, 
that is, the support of the function
$\sC$, must span $\PP$.

For $A \subseteq \PP$ we  define the counting function $\mathcal{C}(A) := \sum_{x \in A} \mathcal{C}(x)$,
which shall also be  called the multiplicity of the set $A$.

 To a codeword  $\vek{v} = G \vek{x}$ corresponds  the hyperplane $H_{\vek{x}} : = \vek{x}^\perp$, and the weight  $w(\vek{v}) $
of the codeword equals $  \mathcal{C}(\PP \setminus H_{\vek{x}})  = n - \sC(H_{\vek{x}})$.
And two codewords lead to the same hyperplane if and only if the corresponding vectors  $\vek{x}$ yield the same point of $\PP$.

Before we show that the weights of code vectors determine the isomorphism class of a spanning code,
we recall a known definition, introduced by Bolker, see for instance \cite{bolker}, \cite{kung}.

\begin{defin}
 The function on the dual projective space $\PP^{\vee}$ (whose elements are the  hyperplanes $ H \subset \PP$) 
 defined by $\sC(H) : = \sum_{x \in H} \sC(x) $ shall be called   the Radon transform $R(\sC)$
of the function $\sC(x)$.
\end{defin}
 
The Radon transform $R(\sC)$ determines the function  $\sC(x)$ by what we call  Radon duality.

\begin{theo}[\bf Radon duality]\label{radon} 

We have the following Radon duality formula:

$$ \sC(x) =    \frac{1}{q^{k-2} } [R (R(\sC)) (x) -  \frac{q^{k-2} -1}{q-1} \int_{\PP}  \sC)].$$

In particular, if the  function $\sC(x)$ has  zero average, that is, $\int_{\PP} \sC(x)=0$, then we  have  the simpler
formula:
$$ \sC(x) =    \frac{1}{q^{k-2} } R (R(\sC)) (x).$$

Finally, the Radon transform $R(\sC)$ determines the function  $\sC(x)$ as follows:

$$ \sC(x) =    \frac{1}{q^{k-2} } [R (R(\sC)) (x) -  \frac{ (q^{k-2} -1)}{(q^{k-1} -1)} \int_{\PP^{\vee}} R (\sC)] .$$
\end{theo} 

\begin{proof}
 In fact,  the double Radon transform is
$$ R (R(\sC)) (x) = : F(x)  =  \sum_{x \in H} \sC(H) = \frac{1}{q-1} [(q^{k-1} -1) \sC(x) + (q^{k-2} -1) \sum_{y \neq x} \sC(y) ],$$
hence we get
 
$$  F(x) = q^{k-2}     \sC(x) +  \frac{1}{q-1} (q^{k-2} -1)  \sum_{y } \sC(y) =   q^{k-2}  \sC(x) +  \frac{1}{q-1} (q^{k-2} -1)  \int_{\PP}  \sC ,$$

  the first assertion follows by dividing by $q^{k-2} $,
while the second assertion follows right away from the first.

For the third assertion, observe that 
$$\int_{\PP^{\vee}} R (\sC) = \sum_H R (\sC) (H) = \sum_H \sum_{x \in H} \sC(x) = $$
$$ = \frac{1}{q-1} [(q^{k-1} -1) \sum_{x} \sC(x)] = 
 \frac{1}{q-1} [(q^{k-1} -1) \int_{\PP} \sC].$$

Hence 
$$  F(x) = q^{k-2}    \sC(x) + \frac{1}{q-1} (q^{k-2} -1) \int_{\PP} \sC =  q^{k-2}    \sC(x) + \frac{ (q^{k-2} -1)}{(q^{k-1} -1)} \int_{\PP^{\vee}} R (\sC).$$

And  the third assertion follows now.
\end{proof} 

\begin{cor}\label{weights}
The isomorphism class of a spanning code $\sK$ is completely determined by the weights of its code vectors.
\end{cor}

\begin{proof}
Follows immediately from the previous Theorem \ref{radon} on Radon duality and 
from the fact that $\sC(x)$ determines $\sK$.
\end{proof}

\subsection{ Kummer codes}

To go back to our classification of the irreducible components of the variety of nodal surfaces of degree $4$,
recall that, for nodal surfaces of even degree $d$, we have a natural embedding
$$ \sK \subset \sK', \ \sK = \sK' \cap \FF_2^{\sS},$$
where of course the codimension of $\sK$ in $\sK'$ is at most $1$.

The Kummer codes are just the codes of the Kummer surfaces, the quartic surfaces with $16$ nodes, and
these codes are related to  the so-called Reed-Muller codes.
The strict Kummer code $\sK_{Kum} $ is a first order Reed Muller code, the ($5$-dimensional) space of affine functions on $\FF_2^4$,  
$$\sK_{Kum} : = Aff (\FF_2^4 , \FF_2)   \subset \FF_2^{\FF_2^4}$$
while the extended Kummer code $\sK' _{Kum}$ is the ($6$-dimensional) code  $\sK' _{Kum} \subset (\FF_2^{\FF_2^4} \oplus \FF_2 H)$ 
whose projection in $ \FF_2^{\FF_2^4}$ is generated by the space of affine functions on $\FF_2^4$, and by a 
rank two quadratic  function $\be = x_1 y_1 + x_2 y_2$.

When we speak of shortenings of $\sK'_{Kum}$, we always refer to subsets $\sN \subset \sS$.

\section[Normal  surfaces in $\PP^3$,  binary codes of nodal surfaces]{Normal surfaces in $\PP^3$, and binary codes of nodal surfaces}

Let $Y \subset  \PP^3 : = \PP^3_{\CC}$ be a normal surface, with singular set  $ \Sing (Y) = \{ P_1, \dots, P_{\nu}\}.$

Set $Y^* : = Y \setminus \Sing(Y)$: the topological structure of $Y^*$ is an important invariant of $Y$, which
is not changed by equisingular deformations; in particular the fundamental group $\pi_1 (Y^*)$
and its Abelianization, the first homology group $H_1(Y^*, \ZZ)$, are two basic such invariants.

Given the singular point $P_i$, let (see \cite{milnor}) $U_i$ be a neighbourhood  of $P_i$ obtained by  intersecting $Y$ with a small 
Euclidean ball with centre $P_i$, and set $U_i^* : = U_i \setminus \{P_i\}$; let further $F_i$ be the Milnor fibre of the singularity $(Y, P_i)$, 
and let $\Ga_i$ be  the local fundamental group
$\pi_1( U_i^*)$. 

\begin{prop}\label{1}
Let $Y \subset  \PP^3$ be a normal surface and let $Y^*$ be its smooth locus: then 
$\pi_1 (Y^*)$ is normally generated by a quotient of the free product $\Ga_1 * \Ga_2 * \dots *  \Ga_{\nu}$ of the local fundamental groups at the singular points.

In particular, $H_1(Y^*, \ZZ)$ is  a quotient of the direct product of the Abelianizations of the $\Ga_i$'s, and we have a surjection: 
$$\Ga^{ab} _1 \times  \Ga^{ab} _2 \times \dots \times  \Ga^{ab}_{\nu}  \twoheadrightarrow H_1(Y^*, \ZZ).$$
\end{prop}
\begin{proof}
We observe that $Y$ deforms to a smooth hypersurface $Y'$, which has $\pi_1 (Y') = 1$ by Lefschetz' theorem.

Define $F$ to be the union of the Milnor fibres $F_i$, together with a tree connecting the base point  to each  
Milnor fibre with a respective segment (in particular $\pi_1(F) = 1$ by Milnor's theorem \cite{milnor}
stating   that $F_i$
is homotopically equivalent to a bouquet of spheres). 

The first van Kampen 's theorem says that $ 1 = \pi_1 (Y') $ is the quotient of the free product $\pi_1 (Y^*) * \pi_1(F) = \pi_1 (Y^*)$
by the subgroup normally generated by $\pi_1 (F^*)$, where $F^* : = F \cap Y^*$. Now, by construction, $\pi_1 (F^*)$
is the free  product $\Ga_1 * \Ga_2 * \dots *  \Ga_{\nu}$, hence the desired assertion follows.

When we abelianize, we obtain a surjection of $\Ga_1 * \Ga_2 * \dots *  \Ga_{\nu}$ onto  $H_1(Y^*, \ZZ)$, which therefore factors through a surjection
$\Ga^{ab} _1 \times  \Ga^{ab} _2 \times \dots \times  \Ga^{ab}_{\nu}  \twoheadrightarrow H_1(Y^*, \ZZ).$
\end{proof}

We are especially interested in the simplest possible case, where $Y$ is a nodal surface, that is, all singular points of $Y$ are
{\bf nodes}, singularities of multiplicity two with nondegenerate Hessian (hence locally biholomorphic to the singularity
$y_1 y_2 -  y_3^2  = 0$ inside $\CC^3$).

\begin{cor}\label{2}
If $Y$ is a nodal surface in $\PP^3$, with $ \Sing (Y) = \{ P_1, \dots, P_{\nu}\},$ then $H_1(Y^*, \ZZ) \cong (\ZZ/2)^k$,
and there is a natural surjection $\oplus_1^{\nu}  (\ZZ/2) \twoheadrightarrow H_1(Y^*, \ZZ).$
\end{cor}
\begin{proof}
A node is the quotient  $\CC^2 / (\pm 1)$, since, if $u_1, u_2$ are local coordinates on $\CC^2$, the quotient is embedded by $y_i : = u_i^2,
\ i=1,2$ and by $y_3 : = u_1 u_2$. 

In particular, the local fundamental group is then $\ZZ/2$.
\end{proof} 
 
\subsection{Global orbifold abelian covering}
\begin{rem}\label{general-normal-surface}
Proposition \ref{1} and Corollary \ref{2} (and other results  below) hold more generally for: 

i) $Y$ a normal surface whose 
singularities are rational double points (these have  local homology groups $\Ga_i^{ab}$ which are finite) and

ii)  whose minimal resolution
$S$ is simply connected (since $Y$ deforms  differentiably to a smooth $Y' $, which is diffeomorphic to $S$, see for instance  \cite{cat-sympl}).

For such  surfaces then $H_1(Y^*, \ZZ) $ is finite.
\end{rem}

\begin{defin}
Let $Y$ be a normal surface such that, setting $Y^*= Y \setminus Sing (Y)$,  $|H_1(Y^*, \ZZ)| < \infty $.

Then there is a finite Galois covering $Z \ra Y$, with $Z$ normal,  associated to the surjection 
$\pi_1 (Y^*) \ra H_1(Y^*, \ZZ)$, and with group $G : = H_1(Y^*, \ZZ).$

{\bf $Z$ is called the Global (orbifold) abelian covering of $Y$.}

The Galois group being abelian, we have a splitting
$$ p_* (\hol_Z) = \bigoplus_{\chi \in G^*} \sF_{\chi},$$
where, for $\chi$ a character of $G$,  the eigensheaf $\sF_{\chi}$ is torsion free of generic rank 1 and indeed reflexive. 
\end{defin}

The Global abelian covering has irregularity $ q(Z): = h^1(\hol_Z) = 0$ if and only if  $ q(Y)=0$ and $h^1( \sF_{\chi})=0$
for all characters $\chi$. 

Moreover, if $Y$ has quotient singularities, $Z$ is smooth if and only if the local fundamental group $\Ga_i$ injects into $G$.
This criterion is particularly interesting for the case of singularities which are nodes: $Z$ is smooth if and only if 
$\Ga_i \cong \ZZ/2$ injects into $H_1(Y^*, \ZZ).$ Dually, if and only if the dual map is surjective.

\medskip 

A geometric consequence of  corollary \ref{2} is that every connected finite Galois  covering $Z \ra Y$, ramified only on the singular set of $Y$,
and with abelian  Galois group $G$, has $G= (\ZZ/2)^r$. Understanding  the Galois theory of such covers reduces to the study of a binary code,
defined as follows.

\begin{defin}
Let $Y$ be a nodal surface in $\PP^3$: then its binary code $\sK \subset (\ZZ/2)^{\nu}$ is defined as the 
image of the dual map
of the surjection $$\bigoplus_1^{\nu}  (\ZZ/2) \twoheadrightarrow H_1(Y^*, \ZZ).$$
\end{defin}

We want now to show that the above definition is the same as the one given by Beauville in \cite{angers}.
Recall for this purpose the following theorem \cite{modulispace}

\begin{theo}\label{jdg}
Let $X$ be a smooth compact complex manifold of complex dimension $n$, and let $D = D_1 \cup \dots \cup D_{\nu}$ 
be a reduced divisor,
where each $D_i$ is irreducible. 

Then the surjection  $ H_1( X \setminus D , \ZZ) \twoheadrightarrow H_1( X  , \ZZ)$ has kernel equal to the cokernel of
$ \rho : H^{2n-2} (X, \ZZ) \ra  H^{2n-2} (D, \ZZ) = \oplus_1^{\nu} \ZZ [D_i]$.

For $n=2$, $\rho(L) = \sum_i (L \cdot D_i) [D_i]$.
\end{theo}

For  $Y$   a nodal surface in $\PP^3$, let $\tilde{Y} $ the blow up at the singular points: then it is well known that $\tilde{Y} $
is smooth and the inverse image of $P_i$ is a (-2)-curve, i.e. a smooth curve $E_i \cong \PP^1$ such that $E_i^2 = -2$.

\begin{cor}\label{codehomology}
Let $Y$  be a nodal surface in $\PP^3$, and let $\tilde{Y} $ be the blow up of $Y$ at  the singular points: 
then its binary code $\sK \subset (\ZZ/2)^{\nu}$
equals to the kernel of the map 
$$  \bigoplus_{i=1}^{\nu} (\ZZ/2 ) [E_i] \ra H^2 ( \tilde{Y}  , \ZZ/2).$$
The vectors in $\sK$ correspond to subsets $\sN$ of $\Sing(Y)$ such that $\sum_{i \in \sN} E_i$ is linearly equivalent
to $2L$, for some divisor class $L$ on $\tilde{Y}$.

The cardinality $t$ of such sets $\sN$, i.e., the weights of the code vectors, are divisible by $4$, and indeed by $8$
if the degree $ d : = \deg (Y)$ is even.

Finally, the dimension $k$ of $\sK$ is at least $\nu - [ \frac{1}{2} \  b_2(\tilde{Y})]$.
\end{cor}
\begin{proof}
Since $Y$ is nodal, by the theorem of Brieskorn and Tjurina  \cite{brieskorn} \cite{tjurina}, $\tilde{Y} $ is diffeomorphic to a smooth surface of degree $d = \deg(Y)$
in $\PP^3$, which is simply connected, by Lefschetz' theorem.  In particular, $H^2 (\tilde{Y}, \ZZ)$ is a free $\ZZ$-module,
  $H^2 (\tilde{Y}, \ZZ/2 ) = H^2 (\tilde{Y}, \ZZ) \otimes_{\ZZ} \ZZ /2$,
and the second assertion follows immediately from the first.

By Theorem \ref{jdg}, applied to $ X : =  \tilde{Y} $ with $D : = E_1 \cup \dots \cup E_{\nu}$,
we get that $$H_1 (Y^* , \ZZ) \cong coker [ H^2 (\tilde{Y}, \ZZ) \ra  \bigoplus_{i=1}^{\nu} \ZZ [E_i] ].$$

Since $E_i \cdot E_j = 0, i \neq j$, and $E_i^2 = -2$, we obtain the surjection 
$$  \bigoplus_{i=1}^{\nu} (\ZZ/2 ) [E_i] \ra H_1 (Y^* , \ZZ),$$
which we have already shown in another way.

Since we have an exact sequence 
$$H^2 (\tilde{Y}, \ZZ) \ra  \bigoplus_{i=1}^{\nu} \ZZ [E_i] \ra H_1 (Y^* , \ZZ) \ra 0,$$
we tensor with $\ZZ/2$ and obtain an exact sequence
$$H^2 (\tilde{Y}, \ZZ/2) \ra  \bigoplus_{i=1}^{\nu} \ZZ /2 [E_i] \ra H_1 (Y^* , \ZZ) \ra 0,$$
and taking the dual $\ZZ/2$-vector spaces, we obtain the exact sequence  
$$ 0 \ra \sK \hookrightarrow  \bigoplus_{i=1}^{\nu} (\ZZ/2 ) [E_i] \ra H^2 ( \tilde{Y}  , \ZZ/2).$$

The statement about $| \sN|$ was proven in \cite{babbage}, proposition 2.11.
The final assertion follows since the dimension of the kernel equals $\nu$ minus the dimension of the image,
which is at most 1/2 of the dimension $b_2(\tilde{Y})$ of $H^2 ( \tilde{Y}  , \ZZ/2)$,
since the image is an isotropic subspace for the intersection form, which is nondegenerate.
\end{proof}

In the case where the degree $d : = \deg (Y)$ is even we have a bigger code containing $\sK$.

\begin{defin}
(1) Let $Y$ be a nodal surface in $\PP^3$ of even degree $d$, and let $H$ be a smooth hyperplane section of $Y$. 

Then the {\bf extended code} $\sK '$ is defined as the kernel 
$$\sK ': =  \ker [  \bigoplus_{i=1}^{\nu} (\ZZ/2 ) [E_i] \oplus (\ZZ/2 )  [H]  \ra H^2 ( \tilde{Y}  , \ZZ/2)].$$

(2) The sets $\sN$ which correspond to code vectors in $\sK$ were defined in \cite{babbage} and later in
\cite{cascat} as strongly even 
(even) sets of nodes, whereas those sets corresponding to code vectors in $\sK' \setminus \sK$
were called weakly even (later: half-even) sets of nodes. 

\end{defin}

The weakly even (half-even) sets of nodes, which are defined if $d = 2m$, have cardinality $t$
such that $ t \equiv m (mod \ 4)$.

\begin{ex}\label{cone} (The quadric cone)
For $d=2$ there is only one nodal surface, the quadric cone  $Y : = \{ y_1 y_2 - y_3^2 = 0\}$.

Since a curve $E$ with $E^2 = -2$ is not divisible by $2$, it follows that $\sK = 0$, while we claim that
$\sK ' $ is spanned by $E + H$.

 $Y$ is in fact  the quotient of $\PP^2$, with coordinates $(u_0, u_1,u_2)$, by the action of $\ZZ/2$
 sending $(u_0, u_1,u_2) \mapsto (u_0,-  u_1, - u_2)$. The quotient is embedded by 
 $$y_i : = x_i^2, \ i=0,1,2, y_3 : = u_1 u_2,$$
 hence $$\PP^2 / (\ZZ/2) \cong Y : = \{ y_1 y_2 - y_3^2 = 0\},$$
 and $Y$ has a double covering ramified only on the singular point $y_1 = y_2= y_3 = 0$
 and the hyperplane section $ H :=  Y \cap \{ y_0 = 0 \}$.
Blowing up $\Sing(Y)$ we get a double covering of  $ \tilde{Y} $ isomorphic to the blow up of $\PP^2$ at the point $u_1= u_2=0$,
and ramified exactly on $E \cup H$.

Hence the divisor $E + H$ is divisible by $2$, and yields a class in $\sK '$. 

Also, $\pi_1 ( Y^* \setminus H) \cong \ZZ/2$ and $\PP^2$ is the Global abelian covering of $Y$.
 
\end{ex}

We have again another characterization for the extended code.

\begin{cor}\label{codehomology2}
Let $Y$  be a nodal surface in $\PP^3$,  let $\tilde{Y} $ be the blow up of $Y$ at  the singular points, and let $Y'$ be the complement
in   $\tilde{Y} $ of the union of the  exceptional curves $E_i$ with  a smooth hyperplane section $H$ (not passing through any node).

(i) Then the first homology group $H_1( Y' , \ZZ)$ is a quotient 
$$ \bigoplus_{i=1}^{\nu} (\ZZ/2 ) [E_i] \oplus (\ZZ/d ) [H] \ra H_1 (Y' , \ZZ),$$ 
and

(ii) moreover,   its extended binary code $\sK' \subset (\ZZ/2)^{\nu} \oplus (\ZZ/2 ) [H]$,
by definition equal to the kernel of the map 
$$  \bigoplus_{i=1}^{\nu} (\ZZ/2 ) [E_i] \oplus (\ZZ/2 ) [H] \ra H^2 ( \tilde{Y}  , \ZZ/2),$$
determines the saturation  of the lattice $\sL$ generated by the
curves $E_i$ and by the hyperplane section $H$ inside the  lattice $
\Lambda : = H^2 (\tilde{Y}, \ZZ)  $.

 More precisely, the saturation  $\sL^{sat} : = \Lambda \cap \QQ \sL$ consists of the vectors
 $$ \sL^{sat} = \{ \frac{1}{2} ( \sum_i a_i E_i + b H )| \sum_i (a_i E_i + b H) (mod \ 2) \in \sK' \}.$$ 

 The same result holds for nodal surfaces $Y$ whose minimal resolution $S = \tilde{Y} $ is simply connected
and such that the hyperplane class $H$ yields an indivisible element in $\Pic(S)$.
\end{cor}

\begin{proof}

Arguing  as in corollary \ref{codehomology},  we see that  the code $\sK'$ is  the cokernel 
 $$H_1 (Y^* \setminus H , \ZZ) \cong coker ( H^2 (\tilde{Y}, \ZZ) \ra  \bigoplus_1^{\nu} \ZZ [E_i] \oplus \ZZ H ),$$
 of the dual map of the embedding inside the  lattice $
\Lambda : = H^2 (\tilde{Y}, \ZZ)  $ of the lattice $\sL$ generated by the
curves $E_i$ and by the hyperplane section $H$.

Assertion (i) follows then since $E_i^2 = -2, E_i \cdot E_j =0, {\rm if} \ i \neq j, H^2 = d, H \cdot E_i =0$.

In view of the exact sequence 
$$0  \ra \sL \ra \Lambda \ra F \oplus T \ra 0 ,$$
where $T$ is a torsion group and $F$ a free abelian group, dualizing we get (in view of the unimodularity of $\Lambda$)
$$0  \ra F^{\vee} \ra \Lambda \ra \sL^{\vee} \ra Ext^1( T, \ZZ) \ra 0,$$
showing that $T \cong H_1 (Y^* \setminus H , \ZZ)$.

 Consider now the saturation $\sL' : =  \sL^{sat} $ of $\sL$ inside $\Lambda$; $\sL' $  is primitively embedded.

Since $ \sL \subset \sL'$, and $ \sL' /  \sL \cong T$, it follows that $\sL'  =  \sL^{sat} $ is generated by $\sL$
and by classes $\la \in \Lambda$ of the form 
$$  \la =  \sum_i \frac{1}{2} a_i E_i + \frac{1}{d} b H , a_i, b \in \ZZ.$$

Since $2 \la \in \Lambda$, we obtain that $ \frac{2b}{d}  H \in \Lambda$.

But $H$ is indivisible (cf. for instance \cite{babbage}, page 464), hence $d$ divides $2b$; if $d$ is odd, this implies that $d$ divides $b$,
hence we may assume $b=0$, if $d = 2 m$ is even, then this means that $m$ divides $b$, so that
$  \la =  \sum_i \frac{1}{2} a_i E_i + \frac{1}{2} b' H , a_i, b' \in \ZZ,$
hence the class of $\la$ yields a divisor class $ L $ such that
$  2 L =  \sum_i  a_i E_i + b' H , a_i, b' \in \ZZ,$
and corresponds exactly to a codeword in $\sK'$.

 The same proof  works in the  more general  case of a  nodal surface $Y$ satisfying the two properties 
that its minimal $S = \tilde{Y}$ is simply connected, and that  the class of $H$ is indivisible; it suffices to  observe that the only facts that we have used
 are that $\Pic(\tilde{Y})$ is a free abelian group, because $\tilde{Y}$ is simply connected, 
and that the class of $H$ is indivisible.
\end{proof}

\begin{rem}
In particular, the  same result holds also for a nodal K3 surface $Y$, with minimal resolution $S$,
such that class of the hyperplane divisor  $H$ on $S$ is  indivisible.
\end{rem}
\subsection{Examples}
\begin{ex}(A cubic with maximal Milnor number)\label{3A_2}

A similar example to Example \ref{cone} is the quotient $\PP^2 / (\ZZ/3) $ for the action such that $(u_0, u_1,u_2) \mapsto (u_0,\e  u_1, \e^2 u_2)$,
where $\e$ is a primitive third root of unity.

The quotient is embedded by $x_i : = u_i^3, \ i=0,1,2 $ and by $y_3 := u_0 u_1u_2$, so that 
$$\PP^2 / (\ZZ/3) \cong Y : = \{ x_0 x_1 x_2 = x_3^3\}.$$

In this case the triple covering is only ramified in the three singular points $y_3=y_i=y_j = 0$
( for $0 \leq i < j \leq 2$), hence we conclude that (since $\PP^2$ minus three points is simply connected)
$$\pi_1 (Y^*) \cong \ZZ/3,$$ and again we have a Global abelian covering.

This cubic surface is remarkable, since it has 3 singular points with Milnor number $2$ (locally isomorphic to the singularity
$ x_1 x_2 = x_3^3$), and $3 \cdot 2 = 6$ is bigger than the maximum number of singular points that a normal cubic can have,
which equals to $4$.

For each normal cubic surface $Y$, which is not the cone over a smooth cubic curve, its singularities are rational double points,
hence the sum $m : = Milnor (Y)$ of the Milnor numbers of the singular points equals the number of the exceptional $(-2)$-curves 
appearing in the resolution $\tilde{Y}$; since $\tilde{Y}$ is the blow up of the plane in 6 points, the rank of
$H^2(\tilde{Y}, \ZZ)$ is at most 7, hence (these $(-2)$-curves being numerically independent, see \cite{artin}) 
$ Milnor(3) \leq 6$, and equality is attained in this example.

\end{ex}

\begin{ex}(The Cayley cubic)
A normal cubic can have at most 4 singular points, and if it does have 4 singular points,
then  it is projectively equivalent to the Cayley cubic
$Y : = \{ \s_3(x_0, x_1, x_2, x_3) = 0\}$, where $\s_3$ is the third elementary symmetric function (the four singular points are the 4 coordinate points).
\end{ex}

We give another proof of  the following  known theorem (see \cite{adt}, and Theorem 34 of \cite{cat-kummer} which follows the same idea but contains  more explicit details)

\begin{theo}
If $Y$ is the Cayley cubic, then 
 $$\pi_1 (Y^*) \cong \ZZ/2.$$
\end{theo}
\begin{proof}
In fact, $Y$ admits a double covering $Z$ ramified only in the four singular points: this can  be seen by representing
$Y$ as the determinant of a symmetric matrix of linear forms, or as a consequence of the B-inequality of Proposition \ref{ineq}
which shows that 
the code $\sK$ of the Cayley cubic has dimension $1$ and is generated by the vector with 4 coordinates equal to $1$.

 $Z$ is a smooth Del Pezzo surface of degree $6$, 
hence isomorphic to the blow up of $\PP^2$
in three points. 
Since $Z$ minus a finite number of points is simply connected,
follows that $\pi_1 (Y^*) \cong \ZZ/2,$
and $Z$ is the Global abelian covering.
\end{proof}

Instead for other nodal cubics the associated code is trivial, because (see \cite{babbage}) the number of nonzero coordinates
of a vector in $\sK$ is always divisible by $4$, and divisible by $8$ if the degree of the nodal surface is even.

We have more (see Proposition 35 of \cite{cat-kummer} for a simple proof):

\begin{prop}
If $Y$ is a nodal cubic with $\nu \leq 3 $  nodes, then $Y^*$ is simply connected.
\end{prop}

\subsection{Range of non triviality of the binary codes}

The important discovery by Beauville is that for surfaces of low degree with many nodes these codes are necessarily
nontrivial.

We use the following standard notation which we have already introduced: $\sK$ is a $\sC^n_k$ binary code
if $\sK \subset \FF_2^{\nu}$
has  dimension $k$, and effective length $n$ ($n$  is the cardinality of  the minimal subset $\sN \subset \{1, \dots, \nu\}$
such that $\sK \subset \FF_2^{\sN}$);
$\sK$ is said to be a spanning code if $ n = \nu$.

Beauville remarks that the second Betti number $b_2(d)$ of a smooth surface of degree $d$ 
equals the one of $\tilde{Y}$ if $Y$ has degree $d$, and by Noether's formula
$ b_2(d)  = e-2= 12 \chi - K^2 -2 = 10 + 12 $ $ d-1 \choose{3}$ $ - d (d-4)^2.    $
Since the span  of the curves $E_i$ (and of $H$, if $d$ is even) inside $ H^2( \tilde{Y}, \ZZ/2)$ is an
isotropic subspace,  Beauville obtained the

\begin{prop}[\bf B-inequalities = Beauville inequalities]\label{ineq}
$$ k \geq \nu - \frac{1}{2} b_2 (d) = \nu - ( 5 + (d-1)(d-2)(d-3) ) + \lceil \frac{1}{2} d(d-4)^2 \rceil ,$$  
where we denote by $\lceil  m \rceil$ the smallest integer greater or equal to $m$, and
$$ k' \geq \nu + 1 - \frac{1}{2} b_2 (d) = \nu + 1 - ( 5 + (d-1)(d-2)(d-3) ) + \lceil \frac{1}{2} d (d-4)^2 \rceil ,$$   
\end{prop} 

\begin{rem}
\label{rem_low_degree}
We get the following inequalities for low degree $d$:

\begin{itemize}
\item
$ d=3 , k \geq \nu - 3$,
\item
$ d=4 , k \geq \nu - 11$, $ k' \geq \nu -10$,
\item
$ d=5 , k \geq \nu - 26$,
\item
$ d=6 , k \geq \nu - 53$, $ k' \geq \nu -52$.

\end{itemize}

Since, by the Miyaoka inequality \cite{miyaoka}, it is known  that $\nu \leq \frac{4}{9} d (d-1)^2$, the B-inequality
 is useful only if $(4/9 ) d (d-1)^2  > [b_2(d)/2]$. 

I.e., only  if
$$    9 d (d-4)^2  - 90 \geq    (d-1) [ 10 d^2 - 82 d + 108] ,$$
equivalently
$$ ( 20 - d)d^2  \geq  46 d - 18 ,$$
which clearly does not hold for $ d \geq 20$,  fails also for $d=18, 19$,
and holds for $ 3 \leq d \leq 17$.
\end{rem}

\section{Code shortenings and unobstructed nodal  surfaces }\label{code_shortenngs}

\begin{defin}
1) Consider a linear code $\sK \subset \FF^{\hS}$, and a subset $\sN \subset  \hS$.

Then the {\bf shortening} of the code $\sK$ relative to the subset $\sN$, denoted by 
$Res(\sK,\sN)$, or more briefly by $\sK_{\sN}$, was defined as

$$\sK_{\sN} : = Res(\sK,\sN) := \{f | f \in \sK, \ \supp(f) \subset \sN \} \subset \FF^{\sN}.$$

In case of an extended code $\sK' \subset \FF^{\hS} \oplus \FF H$,
we denote by shorthand notation $$\sK'_{\sN} = Res(\sK',\sN) : = Res(\sK',\sN \cup {H}).$$

2) A binary code $\sK \subset \FF_2^{\hS}$ is said to be $d$-realized if $\sK$ is the code
of a nodal surface of degree $d$ in $\PP^3$.
\end{defin}

\begin{defin}
If $Y$ is a nodal hypersurface of degree $d$ in $\PP^n$, $Y$ is said to be {\bf unobstructed} if its singular points
impose linearly independent conditions on the vector space of homogeneous polynomials of degree $d$.
\end{defin}

The geometrical meaning of unobstructedness is that the deformations of $Y$ have a submersion onto
the space of local deformations of the singularities, so that one can obtain  independent smoothings
of all the singular points.

In fact, (see for instance \cite{wahl1}, \cite{cime}, \cite{sernesi}) for a normal surface $Y$  in $\PP : = \PP^n$, we have (by definition) the exact sequence 
$$ 0 \ra N^*_Y \ra \Omega^1_{\PP} \otimes \hol_Y \ra  \Omega^1_Y \ra 0,$$ 
where the conormal sheaf $N^*_Y \cong \hol_Y(-d)$ s generated by the differential of $F$, where $F$ is the
polynomial equation (of degree $d$) of $Y$.

Dualizing with respect to $\hol_Y$, we get an exact sequence 
$$ 0 \ra \Theta_Y : = \sH om ( \Omega^1_Y, \hol_Y) \ra  \Theta_{\PP} \otimes \hol_Y \ra N_Y \cong \hol_Y(d) \ra \sE xt^1( \Omega^1_Y, \hol_Y) \ra 0,$$ 
which can be split into two short exact sequences, the second one being:
$$ 0 \ra N'_Y \ra N_Y \cong \hol_Y(d) \ra  \sE xt^1( \Omega^1_Y, \hol_Y) \ra 0.$$
The sheaf $ \sE xt^1( \Omega^1_Y, \hol_Y)$ is, for a normal surface, concentrated on the singular points,
and there equal to the quotient Tjurina algebra 
$$\hol_{\PP} / (F, \partial F /  \partial x_i),$$ where the $x_i$ are local coordinates.

Unobstructedness means, when we pass to the long exact cohomology sequence,
in view of the surjection $ H^0(\hol_{\PP}(d)) \ra H^0(\hol_Y(d))$,  exactness of 
$$ 0 \ra H^0(N'_Y)  \ra H^0(\hol_Y(d))  \ra  H^0(\sE xt^1( \Omega^1_Y, \hol_Y)) \ra 0.$$

The main result is that then the projective space $\PP( H^0(\hol_{\PP}(d)))$, which parametrizes
surfaces of degree $d$, at the point corresponding to $Y$, has a submersion onto the manifold
$\sL oc Def (Y)$ of local deformations of the singularities of $Y$. Hence in the neighbourhood of
$Y$ we can find all possible local deformations of the singularities of $Y$.

This leads to a very effective criterion:

\begin{theo}
\label{thm_d_realized}
Let $Y$ be an unobstructed nodal surface of degree $d$: then all shortenings of the code $\sK$
(as well as those of  $\sK'$ if $d$ is even) are $d$-realized, i.e. they are realized by some nodal surface of degree $d$.

\end{theo}

\begin{proof}
Let $\hS : = \Sing (Y)$, and let $\sN \subset \hS$. Then there is a local deformation $Y'$ of $Y$ which keeps the
nodes in $\sN$ and does a smoothing for  the others.

By the Brieskorn-Tyurina theorem \cite{brieskorn} \cite{tjurina}, there is a smooth family whose fibres contain  $\tilde{Y}$ and $\tilde{Y'}$,
and $(Y')^*$ is diffeomorphic to the complement of the union of the curves $E_i$, for $ i \in \sN$, while
$Y^*$ is diffeomorphic to the complement of the union of all the curves $E_i$, for $ i \in \hS$.

Then from  the surjective map $\oplus_{i \in \hS} \FF_2 [E_i] \ra H_1 (Y^* , \ZZ)  = \sK_{Y}^{\vee}\ra 0,$
whose kernel we shall denote here by $V$,
if we divide by the subspace $\oplus_{i \in \hS \setminus \sN} \FF_2 [E_i]$,
we obtain a surjection $\oplus_{i \in \sN} \FF_2 [E_i] \ra H_1 ((Y')^* , \ZZ) = \sK_{Y'}^{\vee} \ra 0,$
and get 
$$ \sK_{Y}^{\vee} = \FF_2^{\sS} / (V ), \ \ 0 \ra (V + \FF_2^{\sS \setminus \sN} ) \ra \FF_2^{\sS} \ra \sK_{Y'}^{\vee} \ra 0.$$

Dualizing the above exact sequence  we obtain (using that $ (A/B)^{\vee} = ker (A^{\vee}  \ra B^{\vee} ) $):

$$0 \ra \sK_{Y'} \ra \FF_2^{\sS} \ra V^{\vee} \oplus  \FF_2^{\sS \setminus \sN}$$
whence 
$$  \sK_{Y'} = \ker (  \sK_Y \ra  \FF_2^{\sS \setminus \sN}),$$
 which is what we wanted to prove.
\end{proof}

The next question is: which nodal surfaces are unobstructed?

Of course, one necessary condition is that the number of nodes $\nu$ be 
strictly smaller than $ \frac{1}{6} (d+3)(d+2)(d+1) $.
This condition is false in general, as shown by Beniamino Segre \cite{segre1952sul}, since he proved that
$ \frac{\mu(2d) }{8 d^3} > \frac{\mu(d) }{d^3}$.  Hence,
starting from the case of Kummer surfaces, where $d=4$, $\mu(4)=16$, we get  $$\limsup \left(\frac{\mu(d) }{d^3}\right)) > \frac{1}{4} = 0,25$$
 and using the Barth sextic that 
$$\limsup \left(\frac{\mu(d) }{d^3}\right) > \frac{65}{216} = 0, 3009 \dots .$$

Indeed Beniamino Segre constructed also easy explicit examples:

\begin{ex}\label{bsegre}
For even degree $d = 2m$, let $\la _1 (x) , \dots, \la _d (x)$ 
be general linear forms, and let $B(x)$ be a general form of degree $m$.

Then the surface
$$ Y : = \{ x |  \la _1 (x)  \dots \la _d (x) - B(x)^2 = 0 \}$$
has $\nu(Y) = \frac{1}{4} d^2 (d-1)$, its nodes being the points 
such that there are $ i < j$ with $\la _i (x)=  \la _j (x) = B(x)=0.$
\end{ex}

The above  examples by B. Segre were used by Burns and Wahl \cite{burns-wahl} to 
produce smooth surfaces $S$ (the minimal resolutions of the nodal surfaces $X$)
whose local moduli space is singular.

It is not yet clear whether there are examples   of Nodal Severi varieties
$\sF(d, \nu)$ which are not  smooth (cf.  question 5 in the introduction).

However, in small degree, the known hypersurfaces which yield $\nu = \mu(d)$ are unobstructed.

This is obvious for $d=2,3$; for instance, for $d=3$, we  have  the Cayley cubic,  whose
 singular points are the four coordinate points.
 
 \begin{theo}\label{unobstructed}
 The following nodal hypersurfaces are unobstructed:
 
 \begin{enumerate}
 \item
 The 4-nodal Cayley cubic of equation $\s_3 (x_0, \dots, x_3) = 0$.
 \item
 The 16-nodal Duke of Cefal\'u surface of equation $$(\sum_0^3 x_i^2)^2 - 3 \sum_0^3 x_i^4 = 0 \ ;$$
 \item
 The 10-nodal Segre cubic hypersurface in $\PP^4$ of equations (in $\PP^5$)
 $$ s_1 (x) : = s_1 (x_0, \dots, x_5) = 0, \ \  s_3 (x)  = 0.$$

 \item
 The Goryunov 15-nodal cubic hypersurface in $\PP^5$ of equations (in $\PP^6$) 
 $$ s_1 (x) : = s_1 (x_0, \dots, x_5) = 0, \ \ 2 s_3 (x) - 3 z s_2(x) + 12 z^3 = 0.$$
\item
 The 35-nodal Segre cubic hypersurface in $\PP^6$ of equations (in $\PP^7$)
 $$ s_1 (x) : = s_1 (x_0, \dots, x_7) = 0, \ \  s_3 (x)  = 0.$$
 \item
 the 31-nodal Togliatti surface of equation
 \begin{gather*}
2(x^5 -5x^4w -10x^3y^2 -10x^2y^2w +20x^2w^3 +5xy^4 -5y^4w +20y^2w^3 -16w^5 )\\
+5 (x^2+y^2+bz^2+zw+dw^2)^2z=0
\end{gather*}
 \item
 the 65-nodal Barth sextic of equation 
 $$(\tau^2 x^2 - y^2) (\tau^2 y^2 - z^2)  (\tau^2 z^2 - x^2) - \frac{1}{4} ( 2 \tau -1)  w^2 (x^2 + y^2 + z^2 - w^2)^2=0,$$
 where $\tau : = \frac{1}{2} ( 1 + \sqrt{5})$ is the inverse of the golden ratio.
 \end{enumerate}
 \end{theo}
 
 \begin{proof}
 i) is obvious, since the singular points are the 4 coordinate points $e_i$, and $x_i^3 (e_j) = \de_{i,j}$.
 
 \smallskip
 
 ii) The 16 singular points are just the G- orbit of $(0,1,1,1)$, where $G$ is the semidirect product of 
 $(\ZZ/2)^3 \rtimes \mathfrak S _4$, and $(\ZZ/2)^3$ acts multiplying the coordinates by $\pm 1$.
 
 We write the homogeneous forms of degree $4$ as
 $$ \sum_1^4 a_i z_i^4  + \sum_{\{i,j\}} d_{i,j} z_i^2 z_j^2 + \sum_{i \neq h} b_{i,h} z_i^3 z_h + \sum
 _{i \neq j} e_{i,j} z_i \frac{z_1 z_2 z_3 z_4}{z_j}.$$
 
 Imposing the vanishing on the 4 points $(0,1,\e_3, \e_4)$, for $\e_3 , \e_4 \in \{ 1, -1\}$,
 leads to the equations
 $$(\sum_i a_i) - a_1 + d_{2,3} + d_{2,4}  + d_{3,4}  = 0$$
 $$ b_{2,3} + b_{3,2 } + e_{4,1} = 0$$
 $$ b_{2,4} + b_{4,2 } + e_{3,1} = 0$$
 $$ b_{3,4} + b_{4,3 } + e_{2,1} = 0.$$
 
 Acting with $\mathfrak S _4$, we get 12 equations for $\{i,h,r,s\} = \{1,2,3,4\}$:
$$ b_{i,h} + b_{h,i } + e_{r,s} = 0$$
$$ b_{i,h} + b_{h,i } + e_{s,r} = 0$$
implying that the 12 coefficients  $e_{r,s}$'s  are determined by the $ b_{i,h}$'s.

And we get 4 equations
$$(\sum_i a_i) - a_j + \sum_{r,s \neq j} d_{r,s}   = 0,$$
showing that the 4 $a_j$'s are determined by the coefficients $d_{r,s} $.

\smallskip

iii) In general the singular points of the Segre cubic are the points where the gradients of $s_1, s_3$ are proportional,
hence $x_i^2 = x_j^2$ for all pairs $i,j$. The condition $s_1(x)$ = 0, implies that half of the $x_i$ are equal to $1$,
the other half to $-1$. For $\PP^5$ we get 10 singular points, corresponding to the partitions of $\{0,1,2,3,4,5\}$
into two subsets of 3 elements.

Consider the  cubic polynomials of the form $ (x_i + x_j) (x_a + x_b)(x_c - x_d)$ where the 6 indices are all different.
Such a polynomial vanishes on a singular point unless $x_i = x_j$, $x_a = x_b$ whereas $x_c$ and $x_d$ have opposite sign.

So, it vanishes for all but 2 singular points, for instance $(1,1,-1,-1,1,-1)$ and $(1,1,-1,-1,-1,1)$
if the indices are in increasing order $1,2,3,4,5,6$. 

However, for any two such different singular points, since we are in projective space, they have
four equal coordinates, two of them $ = 1$, the other two $=-1$, and two opposite coordinates.

Hence there is a polynomial as above taking in both points the same value $\pm 8$, and vanishing in the others.
Hence we see that the image of the evaluation map contains a basis of $\CC^{10}$.

\smallskip
 
 iv)  The 15 singular points are the $\mathfrak S_6$-orbit of the point $(1,1,1,1,-2, -2, -1)$, which
 corresponds, as a permutation representation,  to the set of unordered pairs of elements 
  in the set $\{0,1,2,3,4,5\}$. The same holds for the subspace 
 of cubic polynomials with basis $ F_{i,j} : = ( x_i x_j )z$, for $ i < j$.
 
 The representation splits (\cite{james-liebeck} page 203) as the sum of the trivial representation, and two irreducible
 representations of respective dimensions 5,9. The former, added to the trivial representation,
 yields the standard permutation representation on 6 letters, and has basis $f_i : = x_i z  \sum_{j \neq i}  x_j $.
 
 The kernel is therefore the sum of some of these three irreducible representations.
 
 The invariant polynomial $ \sum_{i < j} a_{i , j} ( x_i x_j  )z$ does not vanish on 
 
 $(1,1,1,1,-2, -2, -1)$,
 since its value is $ - (6 - 2 \cdot 8 + 4 )=  6$.
 
 Also, if the kernel would contain the 5-dimensional irreducible representation, then $f_1 - f_4$ would be in the kernel,
 but its value on $(1,1,1,1,-2, -2, -1)$ equals $ - ( (3 -2-2) - (-8 + 4)) = - 3 \neq 0$.
 
 The final possibility is that the kernel is either zero, or equals  the 9-dimensional irreducible representation, but then the rank of the evaluation map would be 6, and this is a contradiction. Indeed, the condition that $\sum a_{i,j}F_{i,j}$ vanishes
 on $(1,1,1,1,-2, -2, -1)$ is
 $$ \sum_{i< j \leq 3} a_{i,j} - 2 \sum_{i \leq 3} (a_{i,4} +  a_{i,5} )+ 4 \sum a_{4,5}=0.$$
 To show that these linear equations have rank at least $7$, we reduce modulo $2$, 
 and it suffices then to verify that the  functions with coefficients in $\ZZ/2$
 associated to pairs of  elements $\{x,y\}$  in the set $\{0,1,2,3,4,5\}$,  $ \sum_{i < j , i,j \neq x,y} a_{i,j}$
 generate a subspace of dimension at least $7$, which is tedious but true (the pairs not containing $0$
 map to the full 5-dimensional space generated by the $a_{0,j}$, then the pairs containing $0$ generate at least a 2-dimensional space).
 
  v)  See Theorem \ref{local-uniqueness}.
   
   vi) and vii) See the computer calculation contained in section \ref{append_unobstructedness}.
 \end{proof}
 
 \begin{rem}\label{K3unobstructed}
 More generally, all nodal surfaces of degree $4$ and nodal K3 surfaces can be shown to be unobstructed, see theorem 3.7
 of \cite{burns-wahl}, especially example 3.9.
 
  The following surfaces: $$Y_3 : = \{ w^3 = xyz \}, \ Y_4  : = \{ w^4 = xyz (x+y+z) \}$$ are unobstructed, 
 while  $Y_d  : = \{ w^d = \Pi_1^d L_i(x,y,z,w) \}$, where the $L_i$'s are general linear forms,  is obstructed for $ d \geq 5$.
 
 Nodal surfaces of degree $d \leq 7$ have been shown to be unobstructed by \cite{dimca} and \cite{remke}.
 
 \end{rem}
 
 \begin{cor}\label{nogap}
 For degree $ d \leq 6$,  we have the `no gaps' phenomenon for nodal surfaces,
 this means, for each $ \nu$ with $ 1 \leq \nu \leq \mu (d)$ there exists a nodal surface of degree $d$ in $\PP^3$
 with precisely $\nu$ nodes.
 
 More generally, for each $\nu \leq \binom{d+1}{3}$, there exists a nodal surface of degree $d$ with exactly $\nu$ nodes.
 \end{cor}
 
  \begin{proof}
  The second statement follows from proposition 2.25 of \cite{babbage}, asserting that 
 for a linearly symmetric set of nodes, the nodes impose independent conditions on the surfaces of degree $d-1$,
 hence a fortiori on the surfaces of degree $d$. Hence  for  all $\nu \leq \binom{d+1}{3}$  there is a partial smoothing 
 yielding the desired number $\nu$ of nodes.
  
 Another proof for the case of degree $d=5$ is therefore also gotten as follows: for $ \nu \leq 20$, we are done by the unobstructedness
 of a general quintic symmetroid; while we use, for $ 16 \leq \nu \leq 31$, corollary \ref{goryunov}. 
 
 Ditto for the case of degree $d=6$ using v).
 \end{proof}
 
 \begin{rem}
 (i) The result is new for $d=5$, while for $d=6$ it follows from the explicit constructions by Catanese-Ceresa 
 \cite{cat-ces} ( $\nu \leq 64$)
 and Barth \cite{barth} $\nu=65$.
 
 (ii) It is an interesting question to ask for which degrees $d$ the `no gaps' phenomenon holds true (this is a special case,
 for the case of surfaces,  of question 
 3 mentioned in the introduction).
  \end{rem}
 
 \subsection{An instructive extended code: the case of the Beniamino  Segre surfaces} 
We want here to give a simple illustration of the application of coding theory to geometry,
giving a characterization of the Segre surfaces of Example \ref{bsegre} in low degree.

These have equation 
$$ Y : = \{ x |  \la _1 (x)  \dots \la _d (x) - B(x)^2 = 0 \},$$
where $B(x)$ is a form of degree $ m = \frac{1}{2} d$, and   the points 
 $\la _i (x)=  \la _j (x) = B(x)=0,$ varying $ 1 \leq i < j \leq d$, are nodes.
 
 The equation can be seen in many ways, for each permutation of $\{1, \dots, d\}$ and for all $ r < d$,  as the determinant of a symmetric $2 \times 2$ matrix of forms,
 $F (x) = det A(x)$, where
 \begin{equation}
A (x) : =\left(\begin{matrix}\la _{i_1} (x)  \dots \la _{i_r} (x)&B(x)\cr  B(x) &\la _{i_{r+1}} (x)  \dots \la _{i_d}(x) 
\end{matrix}\right).
\end{equation}
As we see from \cite{babbage}, such a determinantal expression produces an even set of nodes
with cardinality $ m r (d-r)$, and which is $\frac{1}{2}$-even for $r$ odd, and strictly even for $r$ even. 

For $r=1$ we obtain a codeword $w_i$ of $\sK''$ with weight $m (d-1)= \frac{1}{2} d (d-1)$, and it is easy to see that 
these vectors (codewords) generate a subcode $\sK''_S$ of dimension $d-1$, since the sum 
$  w _1 +   \dots + w _d =0 $ of these vectors is zero; we shall call this code the Segre code.

This code is derived from the 2-partition code $\sC (2, \Sigma) $ on the set $ \Sigma : = \{1, \dots, d\}$,
which is defined as follows.  

Let $\sM $ be the family of subsets $\sD$ of $\Sigma$ with 2 elements, and associate to  a 2-partition of $\Sigma$
($\Sigma$ is the disjoint union $\sA \cup \sB$), the function $\phi_{\sA , \sB}$ such  $\phi_{\sA , \sB} (\sD) = 1$
if and only if $\sD $ contains an element of $\sA$ and an element of $\sB$ (these functions shall be called 2-partition functions).

Then
$$\sC (2, \Sigma)  \subset \FF_2^{\sM} , \ \sC (2, \Sigma)  := \{ \phi_{\sA , \sB} | \Sigma = \sA \cup \sB, \sA \cap \sB = \emptyset.\} $$
$\sC (2, \Sigma) $ has $2^{d-1}$ elements, equal to $\frac{1}{2}$ the cardinality of the power set $\sP (\Sigma)$,
hence $\sC (2, \Sigma) $ has dimension $(d-1)$ \footnote{One can define in a similar way the $k$-partition functions for all $ k \leq d$: these 
do not in general form a code, they only generate a code which we call $\sC (k, \Sigma) $}.  The sum of elements $ \phi_{\sA , \sB} + \phi_{\sA' , \sB'}$ corresponds to the symmetric difference
$ \sA \Delta \sA' $, that is, to the 2-partition $ ((\sA \cap  \sA' ) \cup (\sB \cap  \sB' )) \cup ( (\sA \cap  \sB' ) \cup (\sA'  \cap  \sB ))$.

We see that $\sK''_S$ is an  $m$-times  repeated  $\sC (2, \Sigma) $ code,
i.e. obtained via the product structure $\sN : = \sM \times  \{1, \dots, m\}$.

We have then the following characterization

\begin{theo}\label{Segre}
Let $Y$ be a nodal surface of degree $d=2m$ in $\PP^3$, and assume that the code $\sK''$
admits a shortening which is (isomorphic to) the Segre code  $\sK''_S$ in degree $d$, i.e. 
$\sK''_S$ is an  $m$-times  repeated  $\sC (2, \Sigma) $ code.

Then, for $ d \leq 8$,  $Y$ admits an equation of the form 
$$ Y : = \{ x |  \la _1 (x)  \dots \la _d (x) - B(x)^2 = 0 \},$$
where the $\la_i(x)$ are linear forms.

\end{theo} 

 \begin{proof}
  Let  $\sN_i$ be the $\frac{1}{2}$-even set corresponding to the codeword $w_i$, so that
$|\sN_i| =  m (d-1)$:  we have a linear equivalence on the minimal resolution $S$ of $Y$,
$$ 2 L_i \equiv  - \sum_{P \in \sN_i} E_P + H.$$

Using Theorem 1.10, part (ii) of Endrass \cite{endrass1}, we infer that $L_i$ is effective for $ d \leq 8$,
hence there is a linear form $\la_i(x)$ such that its divisor on $S$ is of the form $ 2 C_i + \sum_{P \in \sN_i} E_P$,
where $C_i$ contains no exceptional divisor (the multiplicity of a linear form at a node being either $1$ or $0$).

Consider now the product $\Lam (x)  : = \la _1 (x)  \dots \la _d (x)$.  Its divisor  on $S$  
 is an effective divisor of the form $2D$, because the sum of the vectors $w_i$ equals zero.

Then, since there is no torsion in $\Pic(S)$,  $D$ is cut by a form $B(x)$ of degree $ m = \frac{1}{2} d$, and we have that 
$ \la _1 (x)  \dots \la _d (x) - B(x)^2 = 0$ is zero on $Y$. But it is not identically zero because the linear forms are distinct,
hence we obtained the equation of $Y$.
\end{proof}

\begin{rem}
It was conjectured by Endrass \cite{endrass1} on the basis of a result by Gallarati \cite{gallarati}  that for half-even sets of nodes $\sN$ the cardinality is always at least $ \geq \frac{1}{2} d (d-1)$,
equality holding if and only if $\sN$ is symmetric and $L$ is effective, that is, it corresponds to writing the equation
$F$ of $Y$ as $ \la(x) \phi(x) - B(x)^2$, with $\la(x)$ a linear form.

If the conjecture is true, then the above characterization of the B. Segre surfaces of Example \ref{bsegre} holds for all (even) degrees.
\end{rem}

\section[Codes and `fundamental groups' of  nodal quartics ]{Codes and `fundamental groups' of  nodal quartic surfaces }
\label{sec_quartic_surfaces}

A nodal quartic surface can have at most 16 nodes.

 This  follows from the cruder estimate that
the dual surface $Y^{\vee}$ of a nodal surface $Y$ with $\nu$ nodes has degree $\de = d (d-1)^2 - 2 \nu$, 
and the observation that  the degree $\de$ of $Y^{\vee}$ is at least $d$ (for $d\geq 4$ the argument is that,
since $Y$ and  $Y^{\vee}$ are birational, they have the same  geometric genus, while, for $d \leq 3$,
 this follows by biduality   $(Y^{\vee})^{\vee} = Y$).

Hence, for $d=3$, we get $\nu \leq 4$, as we already discussed, while for $d=4$ $\nu \leq 16$,
equality holding if and only if $Y^{\vee}$ is also a quartic surface.

We shall now discuss the well known fact that if $\nu = 16$, then $Y$ is the quotient of a
principally polarized  Abelian surface
$A$ by the group $\{\pm 1\}$, embedded by the theta functions of second order on $A$
 (see for instance \cite{cat-kummer}, which shows also  that, in suitable coordinates, $Y= Y^{\vee}$).

This can be shown in several ways, but we choose here the coding theoretic approach, 
pioneered by Nikulin \cite{nikulin-kum}: the main idea is that,  knowing that the
strict  code $\sK$ associated to a nodal quartic has weights only $8$ or $16$,
there exists a codeword of weight $16$, hence a double covering ramified exactly 
at the $16$ nodes.

We begin with a rather elementary  result in coding theory, for which we give a detailed proof as a
preparation for other precise results (a shorter proof could be obtained invoking Bonisoli's theorem \cite{bonisoli}).

\begin{prop}\label{quartic-codes}
Let $\sK \subset \FF_2^{\nu}$ be a code with $\nu \leq 16$ and weights only $8, 16$.
Then  $\sK$ is contained in  the Kummer code $ \sK_{Kum} : = Aff (\FF_2^4, \FF_2)$, a first order Reed-Muller code,
 consisting of the vector space of affine functions inside $\FF_2^{\FF_2^4}$.

Moreover, if the weights are never equal to $16$, it is a shortening of the Kummer code. 
\end{prop}

\begin{proof}
We observe that  $\sK$ is  contained in $\FF_2^{16}$, where there 
is a  unique vector $e$ of weight $16$.

We take $V : = \sK$ if $\sK$  contains $e$, otherwise we let $V$ be the span  of $\FF_2 e$ and of 
$\sK$.

Notice that $e$ corresponds to the constant function $1$.

If there is another non zero vector $v_1 \in V$, we can assume by a permutation of coordinates that 
$v_1 = \sum_1^8 e_i$.

If $V$ has dimension strictly bigger than two, take $v_2$ linearly independent of $v_1, e$.
We can write $v_2 = v_2' + v_2''$, where $v_2' = \sum_{i \in A}  e_i, \ v_2'' = \sum_{i \in B} e_i$,
and $A \subset \{1,2, \dots, 8\}, \ B \subset \{9,10, \dots, 16\}$.

The condition that the weights of $v_2$ and $v_1 + v_2$ are equal to $8$ (since they are both different from $0,e$)
amount to 
$$ |A| + |B| = 8,  \ 8 -  |A| + |B| = 8 \leftrightarrow |A| = |B| = 4.$$

By a permutation of the first $8$ and of the second $8$ coordinates we achieve that
$v_2 = \sum_1^4 e_i + \sum_9^{12} e_i$. 

If $V$ has dimension at least $4$, take $v_3$ linearly independent of $e, v_1, v_2$.

Then $v_3$ and $v_2 + v_3$ are different from $0,e$ and, writing as before $v_3 = v_3' + v_3''$,
$v_3' = \sum_{i \in A'}  e_i, \ v_3'' = \sum_{i \in B'} e_i$,
and $A' \subset \{1,2, \dots, 8\}, \ B' \subset \{9,10, \dots, 16\}$,
we find that $| A'| = 4$ and that the symmetric difference $ (A \cup A' ) \setminus (A \cap A')$
has cardinality $4$. Similarly for $B, B'$. 

Hence we can permute the coordinates, leaving the partition determined by the integer part of $\frac{i}{4}$
invariant, and assume that $v_3 = \sum_{i \equiv 1,2 (mod \ 4)} e_i$.

A similar argument shows that , if $k := dim V \geq 5$, then if $v_4$ is linearly independent of $e, v_1, v_2, v_3$
then we can permute the coordinates, leaving the partition determined by the integer part of $\frac{i}{2}$
invariant, and assume $v_4 = \sum _{i \equiv 1 (mod \ 2)} e_i$.  

That $V$ has at most dimension $5$ follows from a similar argument, or from the Mac Williams identities 
\eqref{mw2} applied to
a supplement $W$ of $\FF_2 e$ inside $V$.

In fact we have then, if $(W)$ is a $C^n_h$ code, that is, a code of effective length $n$ and dimension $h$,  and $a_8$ is the number of vectors of
weight $8$, $$a_8 = 2^h -1, \ \  8 a_8 = n 2^{k-1} \Rightarrow n = (2^h-1) 2^{4 - h},$$ 
implying that $h \leq 4$.

We construct now a bijection of $\{1, \dots, 16\}$ with $\FF_2^4$, by using the functions
$v_1, v_2, v_3, v_4$ (we identify $e_i$ with its characteristic function).

Hence it is proven that $\sK$ is contained in the Kummer code.

For the second assertion, observe that,
if $\sK$ does not contain $e$, there are the four cases, where by a permutation of 
the coordinates we may assume that a basis of $\sK$ is given by $v_1, \dots ,v_k$.

For $k=1$ we get the shortening associated to any subset  of $ \{1,2, \dots, 11 \}$ which contains the subset $ \{1,2, \dots, 8\}$,
 for $k=2$
the shortening associated to the  subsets  $ \{1,2, \dots, 12 \}$ and $ \{1,2, \dots, 13 \}$, for $k=3$
the shortening associated to the subset $ \{1,2, \dots, 14 \}$ and
for $k=4$
the shortening associated to the subset $ \{1,2, \dots, 15 \}$.

Indeed, when we take the shortening of the Kummer code by removing one point in affine space, this point may be chosen as the origin
in a vector space, and we have the 4-dimensional  space $\sK$ of linear functions. Removing one, respectively two vectors, the dimension of $\sK$ drops to 3, respectively 2. Removing a third vector, there are two cases, depending on whether the three vectors are linearly independent or not;
the dimension drops to 1, respectively 2. 

Once we remove at least four vectors plus the origin, three of them are independent
and generate a 3-dimensional linear subspace $U$, and if the dimension of the shortening $\sK$  stays 1, 
this means that the fourth vector we are removing is another one of the four remaining nonzero vectors in $U$.

Finally, the condition that $\sK$ does not contain $e$ is necessary, since $\FF_2 e$ is not a shortening
of the Kummer code.
\end{proof}
\begin{rem}
In fact, as done in \cite{symposia} for generalized codes of this type, it is convenient to have
the following combinatorial representation of the Kummer code as a chessboard code.

We represent $\{1, \dots, 16\} \cong \FF_2^4$ as $W_1 \times W_2$, where $W_i := V_i^{\vee}$
and $V_i$ is a 2-dimensional vector space over $\FF_2$. 

Then $V_1 \otimes V_2$ is the space of linear functions on $W_1 \times W_2$, and the space $\sK$
of affine functions on $W_1 \times W_2$ is represented as the set of functions $f$ such that, 
$ \forall i,j$,
\begin{enumerate}
\item
$\  f(i,j) = f(0,0) + f(0,j) + f(i,0),$ 
\item
$ \sum_j f(i,j) = 0$, 
 \item
$ \sum_i f(i,j) = 0. $
\end{enumerate}
In fact, then, if $ \{e_1, e_2\} $ is a basis for $W_1$, and $ \{\e_1, \e_2\} $ is a basis for $W_2$,
then $F := f + f(0,0)$ satisfies the three equations
\begin{equation} F(i,j) = F(0,j) + F(i,0), 
\end{equation} 
\begin{equation}   F(0, \e_1) + F(0, \e_2) + F (0, \e_1 + \e_2)=0, 
\end{equation} 
 \begin{equation} F(e_1, 0 ) + F(e_2, 0) + F (e_1 + e_2, 0)=0,
 \end{equation}  hence $F$ is linear.
 
{\bf In this representation, $v_1$ has entries equal to $1$ exactly on the first two rows, $v_2$ on the first and third row,
 $v_3$ on the first two columns, $v_4$ on the first and third column. }
\end{rem}

\begin{prop}\label{extendedKummer}
Let $\sK_{Kum}$ be the Kummer code $ Aff (W_1 \times W_2, \FF_2) \subset \FF_2^{W_1 \times W_2}$, 
where $W_1, W_2$ are 2-dimensional vector spaces
over $\FF_2$, and let $\sK_{Kum} \subset \sK' \subset \FF_2^{W_1 \times W_2} \oplus \FF_2 H$
be a code such that $\sK_{Kum} \neq \sK'$, $\sK_{Kum}$ is the $W_1 \times W_2$-shortening of $\sK'$,
and where the weights of vectors in $\sK' \setminus \sK$ are either 
$7$ or $11$.

Then there is only one isomorphism class for the bicoloured code $ \sK'$,
and $\sK'$ is  generated  by $\sK_{Kum}$ and by the pair $(\be, H)$, where $\be$ is the bilinear function $\be(x,y) : = x_1 y_1 + x_2 y_2$. We call this bicoloured code the {\bf extended Kummer code $\sK'_{Kum}$}.

\end{prop}
\begin{proof}
We use the notation introduced in  the previous remark. $\sK'$ is generated by $\sK$ and by the pair $(g, H)$,
where the function $g$ has the property 
\smallskip

\begin{equation}\label{w}  {\rm the  \ weight \ of \  }  (g +f ) \ {\rm  equals \ } 6\ {\rm  or } 10, \forall f \in \sK.
\end{equation}

\smallskip

By replacing $g$ by $ g +1$, we may assume that the weight of $g$ equals $6$.
Conversely, if property \eqref{w}  is satisfied, then necessarily $\sK'$, the code generated by $\sK$ and by the pair $(g, H)$,
fulfills the required properties.

Letting $$a_i : = \sum_j g(i,j), \ b_j : = \sum_i g(i,j)$$
be the respective sums of the rows and of the columns of $g(i,j)$, we have
$$a_1 + a_2 + a_3 + a_4 = 6 , \ b_1 + b_2 + b_3 + b_4 = 6.$$

Adding suitable  linear combinations of $v_1, v_2, 1$ we may add to $g$ two rows having 
all  entries equal to $1$.  For instance, if we add the first two rows, we obtain the condition that
$$ 2 + 2 (a_3 + a_4) = (4-a_1) + (4-a_2) + a_3 + a_4 = 6 \ {\it or } \ 10 \Leftrightarrow a_3 + a_4 = 2 \ {\it or } \ 4.$$ 

Playing the same game for the other pairs of rows, we conclude that for each partition
into two pairs, one pair gives sum equal to $2$, and the other gives sum equal to $4$.

As a consequence, we observe  that, if for instance $a_1 + a_2 = 4$, then necessarily $a_1 \leq 3$, and $a_2 \geq 1$.

It follows that, for all $i$, we have $ a_i \leq 3$, else there is a pair with $a_i + a_j \leq 1$.
 If there is a $j$ with 
$a_j =3,$ then all others are odd, and equal to $1$; otherwise, $\forall i, \ a_i \leq 2$, and since 
 there is a $j$ with 
$a_j =2,$ we conclude that all the $a_i$ 
's are even.

Hence either the four numbers $a_i$ are, after a permutation, $(2,2,2,0)$, or  they are $(3,1,1,1)$.

The same must also be true for the column sums. 

To simplify the argument, we use the fact that the group of affinities of $W_i$ equals  the full symmetric group,
permuting rows for $i=1$ (respectively: columns for $i=2$)
and that we can use the symmetries in $ Aff(W_1) \times Aff(W_2)$, and also exchange rows and columns (this amounts to exchanging $W_1 $ with $W_2$): we shall consider  two functions  equivalent if obtained one from the other via a sequence of these operations.

{\bf Case $(3,1,1,1)$ for the row sums.}

Here there is  a row with three entries equal to $1$: hence, using the above symmetries we can assume that it is the first row, and that its last entry
equals $0$. 

Assume that also the column sums are of type $(2,2,2,0)$: then the last column is equal to zero, 
and the first 
three columns contain exactly one more $1$,  according to a permutation of the three indices.
Using the above symmetries, we can assume that the minor obtained deleting first row and last column is the
identity matrix, and we have the function $G$ corresponding to the matrix
\[\renewcommand{\arraystretch}{1.0}
G:=\begin{pmatrix}
1 & 1 & 1 & 0 \\
1 & 0 & 0 & 0 \\
0 & 1 & 0 & 0\\
0 & 0 & 1 & 0
\end{pmatrix}.\]

Assume now that also the column sums are of type $(3,1,1,1)$:
then either  the last column  has all entries equal to $1$, and we have the function $g$ corresponding to the matrix
\[\renewcommand{\arraystretch}{1.0}
g:=\begin{pmatrix}
1 & 1 & 1 & 0 \\
0 & 0 & 0 & 1 \\
0 & 0 & 0 & 1\\
0 & 0 & 1 & 1
\end{pmatrix}.\]

or the last column contains just one  entry equal to $1$, which we can assume to be the one in the last row.
Then we get  the function $g'$ corresponding to the matrix
\[\renewcommand{\arraystretch}{1.0}
g':=\begin{pmatrix}
1 & 1 & 1 & 0 \\
1 & 0 & 0 & 0 \\
1 & 0 & 0 & 0\\
0 & 0 & 0 & 1
\end{pmatrix}.\]

{\bf Case $(2,2,2,0)$ for the row sums.}

The case where the  column sums are  of type $(3,1,1,1)$ is equivalent to the case of the function $G$,
so we are left with the case where also the  column sums are  of type $(2,2,2,0)$.

In this case we may assume that the first row and the first column are zero,
and then since all other rows and columns contain exactly one entry equal to zero,
we may assume that the entry zero occurs exactly for the diagonal elements, hence 
we get  the function $h$ corresponding to the matrix
\[\renewcommand{\arraystretch}{1.0}
h:=\begin{pmatrix}
0 & 0 & 0 & 0 \\
0 & 0 & 1 & 1 \\
0 & 1 & 0 & 1\\
0 & 1 & 1 & 0
\end{pmatrix}.\]

Up to the above symmetry, we have then four cases, corresponding to the functions 
$g , g', G, h $ such that 
\begin{itemize}
\item
$g$ has support contained in the first row and last column, and vanishes for $(1,4)$
\item
$g'$ has support consisting of the first three elements of  the first row and  column, and  of $(4,4)$  
\item 
$G$ has value $1$ for the first three elements of the first row, and  for the three elements $( i+1, i)$
\item
$h$  vanishes on the first row and column, and on the  elements $( i, i)$.
\end{itemize}

We show now the equivalence of these four cases by using other symmetries, given by other affinities
of $W_1 \times W_2$, such that for instance
$$ (i,j) \in W_1 \times W_2 \mapsto A(i,j) = (i, j + \phi(i)),$$
where $\phi : W_1 \ra W_2$ is an affine map (in particular, we may get any permutation!).

These above transformations preserve the rows, and we have similar ones preserving the columns.

For instance, $ j \mapsto j + e_1$ transposes the first two columns and the last two ones.
Define $e_3 = e_1 + e_2$, $\e_3 = \e_1 + \e_2$ and for each sequence of four points  $[v_0, v_1, v_2, v_3]$
of $W_2$, take   $\phi$ such that 
$$0  \mapsto v_0, \ \e_1  \mapsto v_1, \ \e_2  \mapsto v_2, \ \e_3  \mapsto v_3.$$

Then:
\begin{enumerate}
\item
for $ [0,  e_1,  e_2, e_3]$, $\phi$ 
 sends $g'$ to a matrix of type $(3,1,1,1) \times (2,2,2,0)$ (equivalent to $G$)
 \item
for $ [e_1,  e_3,  e_2, 0]$, $\phi$ 
 sends $h$ to a matrix of type $(3,1,1,1) \times (2,2,2,0)$ (equivalent to $G$)
 \item
 for $ [e_3,  e_2,  0, e_1]$, $\phi$ 
 sends  $g$ to a matrix of type $(3,1,1,1) \times (2,2,2,0)$ (equivalent to $G$) 
\end{enumerate}

Hence there is only one case, and the proof is accomplished, if we observe that
$$ h \sim  \be: = x_1 y_1  + x_2 y_2, $$
since $h$ is the matrix of the form $x_1 y_2  + x_2 y_1$.
\end{proof}

\begin{remark}
The proof could also be accomplished (as we did in the first version) by showing that in each
 of the four cases the extra function 
belongs to the extended Kummer code.

 For instance, $\sim$ denoting equivalence, 
$$1 + g \sim q: =  x_1 x_2 + x_1 + x_2 + y_1 y_2=  \be + x_1 + x_2, \  
1 + G \sim q(x_1, x_2, x_1 + y_1, x_2 + y_2). $$

We opted here  for a conceptually clearer proof.

\end{remark}

\begin{prop}\label{kummershort}
Let $\sK' \subset \FF_2^{\nu} \oplus \FF_2 H$ be a bicoloured code such that 

(i) $\nu \leq 16$,

(ii)  $\sK : = \sK' \cap  \FF_2^{\nu}$ has only weights $8,16$ while 

(iii) the projection $\sK''$ of $\sK'$
 to $\FF_2^{\nu}$ has the property that the vectors in $\sK'' \setminus \sK$
 have only weights $6$ or $10$. 
 
 Then $\sK'$ is contained in the extended Kummer code $\sK'_{Kum}$.
 
 In particular, if $\sK'$  satisfies the B-inequality $ k' : = dim (\sK ') \geq \nu - 10$, 
 $\sK'$ is a shortening of the extended Kummer code.
\end{prop}

\begin{proof}

Let us prove the  first assertion, namely the inclusion $\sK' \subset \sK'_{Kum}$.

If $\sK' = \sK$, this is the content of proposition \ref{quartic-codes}, hence we may now assume that $\sK' \neq \sK$.

We can clearly also assume that $\sK''$ contains the vector $e$, that is, the function $1$. Observe that the assertion
is  true if the dimension of $\sK$ equals $1$ or $5$ (by virtue of proposition \ref{extendedKummer}).

Let us use the same notation as in propositions  \ref{quartic-codes} and \ref{extendedKummer}, and consider a function $g \in \sK''$
of weight $6$. 

If  the dimension $k$ of $\sK$ equals $2$, we have that $g = g' + g''$, where $g' = \sum_{i \in A} e_i, g'' = \sum_{i \in B} e_i,$
and since $$ |A| + |B| = 6, \ 8 - |A| + |B| \in \{6,10\}, \  8 +  |A| - |B| \in \{6,10\}$$
we obtain $$ 2 ( |A| - |B|) \in \{ 0, -4, 4\}.$$

Since $|A| = |B|=3$ is impossible,  we may assume, possibly exchanging $A$ with $B$, that   
$$  |A| = 4, |B|=2.$$

By a suitable permutation of the set $\sS$, leaving  invariant  the partition $\supp(v_1),  \supp (v_1 + e)$,
we achieve  that the function $g$ equals the one of proposition \ref{extendedKummer}, and we are done for $k=2$.

If  the dimension $k=3$, then we may assume that $\sK$ is generated by $v_1, v_2, e$, and we have that the row sums $a_i$
satisfy $a_1 + a_2 + a_3 + a_4 =6$.  As in proposition \ref{extendedKummer} we conclude that these numbers
are either $(2,2,2,0)$ or $(3,1,1,1)$. Hence, by a suitable permutation of the set $\sS$, leaving  invariant  the partition 
given by the rows, we achieve that the function $g$ is either the function $g$ or the function $h$ of proposition \ref{extendedKummer}, hence we are done for $k=3$.

For the  final case where the dimension of $\sK$ equals $4$ the argument is similar:
the four rows split into $8$ blocks, first and second half, and we may assume without loss of generality that the sum of the first pair of columns is equal to $2$.

For row sums $(3,1,1,1)$ there are two possibilities, namely 
\[\renewcommand{\arraystretch}{1.0}
\begin{pmatrix}
2 \ {\rm or} \ 1 & 1  \ {\rm or} \ 2\\
0 \ {\rm or} \ 1 & 1 \ {\rm or} \ 0 \\
0 & 1 \\
0 & 1 
\end{pmatrix}.\]

For row sums $(2,2,2,0)$ there are also a priori two possibilities, namely 
\[\renewcommand{\arraystretch}{1.0}
\begin{pmatrix}
0  & 2\\
1 \ {\rm or} \ 0 & 1 \ {\rm or} \ 2 \\
1  \ {\rm or} \ 2 & 1 \ {\rm or} \ 0  \\
0 & 0 
\end{pmatrix},\]
but the latter option cannot occur, since, adding $v_2$ and then $v_4$, we obtain a vector
of weight $2$, which is not possible.

It is easy now to verify that the three possibilities are given by the respective functions 
$g$, $G$ with first and last columns exchanged, $h$ with first and last rows exchanged.

\smallskip

We pass now to proving the  last assertion, that if $\sK' \subset \sK'_{Kum}$
and $k' \geq \nu - 10$, then $\sK' $ is a shortening of $ \sK'_{Kum}$.

We argue on $\sK'' \subset \sK''_{Kum}$.

   If $k'=6$ we have equality, hence as well if $\nu=16$
   by the B-inequality $k' \geq \nu - 10$.
   
  Let 
  $$ \sA : = \supp( \sK''_{Kum}) \supset \sN' : = \supp (\sK'').$$
  
  Let $\sC \supset \sK'$ be the $\sN'$-shortening of the extended Kummer code.
  
  If $ | \sN'| = 15, 14, 13$ then $\sC$ has dimension $d$ equal respectively to $d=5,4,3$,
  while $ k ' $ is at least respectively (since $k' \geq \nu - 10$) $5,4,3$,
  hence $d = k'$ and $\sK'' = \sC$.
 
 If $ | \sN'| = 12$, our assertion is proven in the same way if $d=2$; and $d=3$ only if,
  looking at the list given in Section \ref{append_quadratic_Reed-Muller_code} for the shortenings  $\sC$,
 we have for $\sC$ case (1) with $\nu = 12$, and $d=3$.

Obviously, if $d=1$, we conclude since by assumption $k' \geq 1$. 

Observe also that the case where $\sK'' = \sK$, $k'\leq 2$, is evidently a shortening 
of the extended Kummer code (cases (II) and (a)).

In the critical case (1) for $\sC$,  with $| \sN'| = 12$, and $d=3$, 
 $\sK'$ is a subspace of $\sC$, which is spanned by $y_1, y_2, x_1 y_1 + x_2 y_2$; 
 since we may assume $\sK \neq \sK''$, if $k' = 2$, we find that $\sK''$ is either isomorphic to the  subspace
of the shortening (III-2), or to the one  
generated by $y_2,  x_1 y_1 + x_2 y_2 + y_1 =  x_1 y_1 + x_2 y_2 + y_1^2 =(  x_1+  y_1) y_1 + x_2 y_2 $,
which is equivalent to the shortening (III-2), and we are done.

 If $ | \sN'| = 11$, then $d \leq 2$, hence we are done unless $k'=1$, $k=0$: then  we use the normal form of quadratic forms to conclude that we have one of the two shortenings (b) or (c).
\end{proof}

\medskip

\begin{cor}\label{extended}
Let $Y$ be a nodal surface of degree $d=4$ in $\PP^3$, and with $\nu =16$ nodes. 
Then $Y$ is a Kummer surface, $Y = A / \pm 1$, where $A$ is a principally polarized Abelian surface.

In particular, $\pi_1 (Y^*)$ is the   semidirect product $\ZZ^4 \rtimes \ZZ/2$
corresponding to multiplication by $\pm 1$.

Its extended code $\sK'$ is the extended Kummer code $
\sK'_{Kum}$, that is, the direct sum $\sK_{Kum} \oplus \FF_2 \be$ of proposition \ref{extendedKummer},
 where $\be$ is the bilinear form $ \be: = x_1 y_1  + x_2 y_2$.
 
 More generally, let $Y$ be a nodal surface of degree $d$ and with $\nu =16$ nodes. Then $ d $ is
 divisible by $4$, $ d = 4m$, and $Y$ is a Kummer surface, $Y = A / \pm 1$, where $A$ is an Abelian surface
 with a polarization of type $(1,m)$.
 
 Again $\pi_1 (Y^*)$ is the   semidirect product $\ZZ^4 \rtimes \ZZ/2$
corresponding to multiplication by $\pm 1$.
 
 Its extended code  $\sK'$ is the extended Kummer code for $m$ odd, and for $m$ even the space generated
 by the affine functions and the characteristic function  $f_{\pi}$ of an affine plane $\pi$. 
 
\end{cor}

\begin{proof}
By the B-inequality the code $\sK$ of $Y$ has dimension at least $5$, hence
the previous proposition \ref{quartic-codes} tells that we have exactly the Kummer code, in particular there exists
a double covering $A \ra Y$ branched exactly on the $16$ nodes (and where $A$ is smooth).

We have that $K_A $ is a trivial divisor, in particular $A$ is minimal.

By proposition 2.11 of \cite{babbage} we have that $\chi (\hol_A) = 0$, hence, by the
classification theorem of algebraic surfaces, $A$ is an Abelian surface. Since there is an involution $\iota$
on $A$ which has fixed points, we can assume, up to a translation on $A$, that
$\iota$ is multiplication by $\pm 1$.

Finally, the pull back of the hyperplane divisor $H$ on $Y$ is a divisor $D$ with $D^2 = 8$
and with $h^0 (A, \hol(D) ) \geq 4$. Hence $D$ is a polarization of type $(d_1,d_2)$ with
$d_1 d_2 = 4$. Since however all the sections of $H^0 (A, \hol_A(D))$ are even functions,
it follows that $d_1= d_2= 2$ and $D$ is twice a principal polarization.

Writing $A = \CC^2 / \Lambda$, where $\Lambda \cong \ZZ^4$,
$Y$ is the quotient of $\CC^2$ by the group $\Ga$ of transformations $ z \mapsto \pm z + \la$, for $\la \in \Lambda$.

$\Ga$ is then isomorphic to the semidirect product $\ZZ^4 \rtimes \ZZ/2$
corresponding to multiplication by $\pm 1$.

We calculate its abelianization, set $\Lam : = \ZZ^4$ and let $g$ be the generator of $ \ZZ/2$:
since $g^2 = 1,  g \la g = - \la$, in the abelianization $\sK$ we have 
$$ 2 [g] = 0,\ \forall \la \in \Lam,  \  [\la] = - [ \la] \Leftrightarrow 2 [\la] = 0,$$
hence $\sK \cong (\ZZ/2)^5$.

Since $\CC^2 \setminus \Lambda$ is simply connected, and $Y^* = (\CC^2 \setminus \Lambda) / \Ga$,
where the action is free, we obtain that $\pi_1 (Y^*) \cong \Ga$.

The extended code $\sK'$ is as claimed by virtue of proposition \ref{extendedKummer},
but one can also use geometry as follows. $\sK'$ is obtained adding a linear equation saying that there is a divisor $L$ on $\tilde{Y}$
such that $ H \equiv 2 L  +  \sum_{i \in \sN} E_i$.

Take a symmetric theta divisor $D$, so that $ D$ is $\iota$-invariant,
and it descends to a Weil divisor $D'$ on $Y$.  Observe that $D$ passes exactly  through the six odd half periods,
i.e., those for which $\be(x_1,x_2,y_1,y_2)$ equals $1$, and with multiplicity one.

The desired divisor $L$ is the proper transform of $D'$: and since $ 2 D'$ is a hyperplane section, we have the desired
linear equivalence $ H \equiv 2 L  +  \sum_{i \in \sN} E_i$, where $\sN$ is the set of six nodes corresponding to 
set of the
odd half periods.

 In the more general case where $Y$ has $16$ nodes and degree $d$, we have to slightly change some part 
of the  argumentation as follows.

As before $ Y = A / \pm 1$.

The extended code $\sK'(Y)$ has dimension $6$, hence there is a divisor $L$ on $\tilde{Y}$
such that $ H \equiv 2 L  +  \sum_{i \in \sN} E_i$. 

Since, for some $h \in \NN$, $ L + h H$ is effective, for the pull back $D$ of $H$ we have that $ ( 2h+1) D \equiv 2 \De$,
where the divisor $\De$ is effective. Hence $D$ is $2$ divisible, $ D = 2 D'$ and $D^2$ is divisible by $8$,
hence $d = H^2$ is divisible by $4$, $ d = 4m$. And then $D'$ is of type $(a,b)$ with $ ab = m$.

Recall that $ a | b$, hence, if $ a \neq 1$, then $D' = a D''$ and the divisor $ D = 2 a D''$ is the pull-back of $H$,
and since $  2  D '' $ is the pull back of a Cartier divisor $H''$ on $Y$, we get $ H = a H''$,
contradicting that $H$ is indivisible. Hence $a=1, b = m$.

\smallskip

The cardinality of the half-even set $\sN$ (which must be $\neq 0$) is the number of half periods where
the divisor $D'$ has odd multiplicity. We can take $D'$ the pull back of a principal polarization
under an isogeny with kernel $ \ZZ/ m$.

More precisely, we consider the lattice 
$$ 
\Lam' : = (\ZZ \oplus  m \ZZ) \oplus \tau \ZZ^2 \subset \Lam : = (\ZZ \oplus   \ZZ) \oplus \tau \ZZ^2 \subset \CC^2 ,$$ 
for $\tau$ in the Siegel domain, and calculate the Riemann 
theta function $\theta (\tau, z)$ at the half periods.

The group of half periods for $\Lam'$ maps, for $m$ odd, isomorphically to the group of half periods for $\Lam$,
hence for $m$ odd we get the points where the multiplicity is odd, these  correspond to the solutions of $ \be(x_1, x_2, y_1, y_2) = 1$.

While, for $m$ even, $x_2$ maps to $2x_2$, hence we get the solutions 
of $x_1 y_1 = 1$: these are the points of   the affine plane $x_1 = y_1 =1$.
\end{proof}

We want now to give a first sharpening of a result shown in \cite{cko}, theorem B and in \cite{campana}, asserting the  finiteness of the fundamental group
of the smooth locus; we shall  extend this result to other nodal K3 surfaces in the sequel (Theorem \ref{fund-gr-nodalK3}).

\begin{theo}\label{Quartics}
Let $Y$ be a nodal quartic surface in $\PP^3$, with $\nu$ nodes, where $ \nu \leq 15$: then
the fundamental group $\pi_1 (Y^*)$ of its smooth locus is finite and is isomorphic to its binary code $\sK$,
which is in turn determined by its dimension $ k,\ 0 \leq k \leq 4$.

The subset of the Nodal Severi variety  $\sF(4, \nu)$ corresponding to  nodal quartic surfaces with fixed number $\nu \leq 16$ of nodes and  fixed  binary code
$\sK'$ is smooth and connected, hence irreducible.

The possible codes $\sK'$ appearing are in particular all the shortenings of the extended  Kummer code of
corollary \ref{extended}, and they
 are listed in Section \ref{append_quadratic_Reed-Muller_code} (cases (0), (00), (i), (1), (2) , (II),(III-1), (III-2), (a), (b), (c), (d), hence we have 11 nontrivial codes in their minimal embedding).

The Nodal Severi variety $\sF (4, \nu)$ has exactly one irreducible component for $ \nu = 0,1,2,3,4,5, 14,15,16$,
exactly two irreducible components for $ \nu = 6,7,13$, exactly three irreducible components for $ \nu = 8,9, $
exactly four  irreducible components for $ \nu = 11,12, $ exactly five  irreducible components for $ \nu = 10. $
Hence  its stratification is made of 34 strata, where each stratum is in the boundary of another if and only if the corresponding 
extended code admits the other as a shortening, as in the following Genealogy tree = Kummer nontrivial genealogy for nodal quartic surfaces (\autoref{fig:Kummer_genealogy}).

\end{theo}

\begin{figure}\caption{Kummer nontrivial genealogy for nodal quartic surfaces}
\label{fig:Kummer_genealogy}

\

\begin{tikzpicture}[scale=.5]


\node[anchor=west]  at (-0.5,1.5) {\# of Nodes};   

\node[anchor=west]  at (1,0) {16};   
\node[anchor=west]  at (1,-2) {15};
\node[anchor=west]  at (1,-4) {14};
\node[anchor=west]  at (1,-6) {13};
\node[anchor=west]  at (1,-8) {12};
\node[anchor=west]  at (1,-10) {11};
\node[anchor=west]  at (1,-12) {10};
\node[anchor=west]  at (1,-14) {9};
\node[anchor=west]  at (1,-16) {8};
\node[anchor=west]  at (1,-18) {7};
\node[anchor=west]  at (1,-20) {6};
\node[anchor=west]  at (1,-22) {5};
\node[anchor=west]  at (1,-24) {4};

\path 
(10,0)  \NR[N_0,0]
(10,-2) \NR[N_00,00] 
(10,-4) \NR[N_i,i];

\draw (N_0)--(N_00)--(N_i);

\path
(8,-6) \NR[N_1a,1] (12,-6) \NR[N_2,2];

\draw (N_i)--(N_1a); \draw (N_i)--(N_2);

\path
(6,-8) \NR[N_1b,1] (10,-8) \NR[N_III_2a,III-2] (14,-8) \NR[N_II,II] (18,-8) \NR[N_III_1,III-1];

\draw (N_1a)--(N_1b); \draw (N_1a)--(N_III_2a); 
 \draw (N_2)--(N_III_2a); \draw (N_2)--(N_III_2a); \draw (N_2)--(N_II); \draw (N_2)--(N_III_1); 

\path
(10,-10) \NR[N_III_2b,III-2] (14,-10) \NR[N_a11,a] (16,-10) \NR[N_b11,b] (20,-10) \NR[N_c11,c];

\draw (N_1b)--(N_III_2b);
\draw (N_III_2a)--(N_III_2b); \draw (N_III_2a)--(N_a11); \draw (N_III_2a)--(N_b11);
\draw (N_II)--(N_a11);
\draw (N_III_1)--(N_a11); \draw (N_III_1)--(N_b11);\draw (N_III_1)--(N_c11);

\path
(10,-12) \NR[N_III_2c,III-2] (14,-12) \NR[N_a10,a] (16,-12) \NR[N_b10,b] (20,-12) \NR[N_c10,c] (22,-12) \NR[N_d10,d];

\draw (N_III_2b)--(N_III_2c); \draw (N_III_2b)--(N_a10); \draw (N_III_2b)--(N_b10);
\draw (N_a11)--(N_a10);\draw (N_a11)--(N_d10);
\draw (N_b11)--(N_b10);\draw (N_b11)--(N_d10);
\draw (N_c11)--(N_c10);\draw (N_c11)--(N_d10);

\path
 (14,-14) \NR[N_a9,a] (16,-14) \NR[N_b9,b]  (22,-14) \NR[N_d9,d];

 \draw (N_III_2c)--(N_a9); \draw (N_III_2c)--(N_b9);
\draw (N_a10)--(N_a9);\draw (N_a10)--(N_d9);
\draw (N_b10)--(N_b9);\draw (N_b10)--(N_d9);
\draw (N_c10)--(N_d9);
\draw (N_d10)--(N_d9);

\path
 (14,-16) \NR[N_a8,a] (16,-16) \NR[N_b8,b]  (22,-16) \NR[N_d8,d];

\draw (N_a9)--(N_a8);\draw (N_a9)--(N_d8);
\draw (N_b9)--(N_b8);\draw (N_b9)--(N_d8);
\draw (N_d9)--(N_d8);

\path
 (16,-18) \NR[N_b7,b]  (22,-18) \NR[N_d7,d];

\draw (N_a8)--(N_d7);
\draw (N_b8)--(N_b7);\draw (N_b8)--(N_d7);
\draw (N_d8)--(N_d7);

\path
 (16,-20) \NR[N_b6,b]  (22,-20) \NR[N_d6,d];

\draw (N_b7)--(N_b6);\draw (N_b7)--(N_d6);
\draw (N_d7)--(N_d6);

\path  (22,-22) \NR[N_d5,d] (22,-24) \NR[N_d4,d];

\draw (N_b6)--(N_d5);
\draw (N_d6)--(N_d5) -- (N_d4);
  
\end{tikzpicture}

\end{figure}

\newpage
\begin{rem}
 The full list of the connected components of $\sF(4, \nu)$ is given in the Table of Section
\ref{append_quadratic_Reed-Muller_code} (another  description of the possible  codes  
was given in \cite{endrass2}).

Since $\sF(4, \nu)$ is smooth, the connected components are also irreducible components.
\end{rem}

\section{Potential nodal K3 lattices}
Before we begin with the proof of Theorem  \ref{Nodal-Quartics}, we  establish some general results valid for all nodal K3 surfaces.
For this purpose we introduce the following definition (condition ii) being dictated by Proposition \ref{2,3}).

\begin{definition}
Consider a lattice $$\sL: =   \bigoplus_1^{\nu} \ZZ E_i \oplus \ZZ H = : \sL^0 \oplus \ZZ H,$$
of rank $ \nu + 1 \leq 17$,
with integral symmetric bilinear  form defined by:
$$ E_i \cdot E_j = - 2 \de_{ij}, H \cdot E_i = 0, H^2 = d : = 2 d' > 0.$$
($\sL$ generated by  a basis which in the case of a nodal algebraic K3 surface corresponds to the $(-2)$-exceptional curves
 $E_i$ and  the hyperplane section $H$).

We define further   a lattice $ \sL' $ which is candidate to be   primitively embedded inside $\Lambda$,
the  $22$-dimensional K3 lattice.

$\sL' $ consists of the vectors
 $$ \sL' : = \{ \frac{1}{2} ( \sum_i a_i E_i + b H )| a_i, b \in \ZZ, \ \sum_i (a_i E_i  ) + b H (mod \ 2) \in \sK' \},$$ 
 where $\sK'$ is a bicoloured binary code 
 $$\sK' \subset  (\bigoplus_1^{\nu} ( \ZZ/2 ) E_i \oplus ( \ZZ/2 )  H ) = : V^0 \oplus ( \ZZ/2 )  H = : V .$$

We say that $\sL'$ is a {\bf potential nodal K3 lattice} if  $\sK'$ is a {\bf potential nodal K3 code},
which by definition amounts to  the requirements:

\begin{enumerate}
\item
letting $\sK : = \sK' \cap  ( \bigoplus_1^{\nu} ( \ZZ/2 ) E_i)$, then the weights $w(v)$ of the vectors in $\sK$ are 
divisible by $8$,
\item
the weights $w(v)$ of the vectors in $\sK' \setminus \sK$ satisfy, once we set $w(v) = : t +1 $,
$$ 4 | (t - d').$$ 
\item
the B-inequality $dim(\sK') \geq \nu - 10$ is satisfied,
\item
the vector $H \notin \sK'$ (hence, $\sK'$ is isomorphic to its projection $\sK''$ inside $V^0$).

\end{enumerate}

Let us finally  define $\sL^*$ by the property that $ \sL^* /  \sL \cong \sK$,
so that $\sL^* = \hat{\sL} \oplus \ZZ H \supset \sL^0 \oplus \ZZ H = \sL$
and $ \hat{\sL} /  \sL^0 \cong \sK$.

\end{definition}

\begin{rem}
Recall the definition of discriminant group of $ \sL' $, defined (see  Nikulin \cite{nikulin}) 
 by $$ Disc (\sL') : =  \sL'^{\vee}  /  \sL', \ \sL'^{\vee}: = Hom (\sL', \ZZ).$$
 
Since $ \sL \subset \sL'$, and $ \sL' /  \sL \cong \sK'$, we have the series of (finite index) inclusions:
$$ 0 \ra   \sL \ra \sL' \ra \sL'^{\vee} \ra \sL^{\vee} ,$$
where $ Disc(\sL) =  \oplus_1^{\nu} (\ZZ/2) E_i \oplus (\ZZ/2d') H,$
and the successive quotients for the series of inclusions are respectively isomorphic to $\sK'$, $ Disc (\sL') $, $V / ((\sK')^{\perp})$.

The last assertion follows since 
 $\sL'^{\vee}\subset \sL^{\vee} =    {\sL^0} ^{\vee} \oplus \ZZ ( \frac{1}{2d'} H )$ is defined by the condition 
$$\sL'^{\vee} = \{ u : =  \frac{1}{2} ( \frac{a}{d'} H + \sum_i b_i E_i ) \  |  u \cdot \sL'  \in \ZZ \} . $$

$$\Leftrightarrow   \sL'^{\vee} = \{ u |
 a  \de -  \sum_{i \in \sA} b_i \equiv 0 \ (mod \ 2) ,  \forall \ v' : = \frac{1}{2}  ( \de H + \sum_{i \in \sA} E_i ) \in \sK' \}.$$

Indeed, by the exact sequence,  $ 0 \ra \sL \ra \sL' \ra \sK' \ra 0$ (from which we deduce, by dualizing,
the other exact sequence:
$ 0 \ra  \sL'^{\vee} \ra \sL^{\vee} \ra Ext^1( \sK', \ZZ) \cong \sK'$)
it suffices to observe that $u \in  \sL'^{\vee}$ if and only if 
$$
 a \de  -  \sum_{i \in \sA} b_i \equiv 0 \ (mod \ 2) , \forall \ v' : = \frac{1}{2}  ( \de H + \sum_{i \in \sA} E_i ) \in \sK' \Leftrightarrow 
 a H + \sum_i b_i E_i \in (\sK')^{\perp}.$$

\end{rem}

\begin{prop}

1) The discriminant group of $ \sL' $,  $ Disc (\sL')  =  \sL'^{\vee}  /  \sL'$,
admits an orthogonal  decomposition $$ D (\sK') \oplus (\ZZ/2)^{\nu - n},$$ where $n$ is the effective length of the code $\sK''$, 
projection of $\sK'$ in $V^0 : = (\ZZ/2)^{\nu}$.

2) In the case where $ \sK' = \sK$ we have an orthogonal direct sum:
$$ Disc ( \sL') = \ZZ/d \oplus Disc ( \hat{\sL}).$$
3) Moreover if $ \sK' = \sK$ $Disc ( \hat{\sL}) \cong (\ZZ/2)^{\nu - 2 k'}$, $ k' = k = dim \sK$.

4) Also if $ \sK' \neq  \sK$ we have a group decomposition $$ Disc ( \sL') = \ZZ/d \oplus (\ZZ/2)^{\nu - 2 k'},$$
 unless possibly if $d'$ is odd, and $ \nu = 2 k'$.

5) The binary part of the finite Abelian group $ Disc (\sL') $ is then of the form 
$ Disc (\sL')_2 \cong \ZZ/2^r \oplus (\ZZ/2)^{\nu - 2k'}, $
where $2^r$ is the maximal power of $2$ which divides $d$.

\end{prop}

\begin{proof}

First of all, we consider the support $\sS'$ of $\sK''$, which has $n$ elements, and we view $\sL'$
as the direct sum of the lattice $\sM$ corresponding to the $\sS' \cup \{H\}$-shortening of $\sK'$,
and of the lattice generated by $\oplus_{i \notin \sS'} E_i$.
Then our first assertion follows right away, if we set $ D (\sK') : = Disc (\sM)$.

The second  assertion is clear since by assumption 
$$\sL' = \sL^* = \hat{\sL} \oplus^{\perp} \ZZ [H].$$

The third assertion  follows because of the series of (finite index) inclusions:
$$ 0 \ra   \sL^0  \ra \hat{\sL}  \ra (\hat{\sL})^{\vee} \ra (\sL^0)^{\vee} ,$$
where $ Disc(\sL^0) =  \oplus_1^{\nu} (\ZZ/2) [E_i] = : V^0,$
and the successive quotients are respectively isomorphic to $\sK$, $ Disc (\hat{\sL}) $, $V^0 / (\sK)^{\perp} 
\cong Ext^1(\sK, \ZZ)$.

For the fourth assertion,  we first claim that  we can find a vector $v \in \sK'^{\perp}$ such that the coefficient $a$
of $H$  is odd. (that is, $ |\supp (v) \cap \sA| $ is odd 
for all  $v' \in \sK'$ written in the above form). Otherwise, each vector  in $\sK'^{\perp}$ is contained in $V^0$,
hence $ \sK'$ contains the vector $H$, contradicting assumption iv).

Since $\frac{1}{d'} H \in \sL'^{\vee} $, there is an element  $ v'' \in \sL'^{\vee} $ with $a=1$, and this element
gives an element of the discriminant of order $ d = 2 d'$ if $d'$ is even:  because $d' v'' = \frac{1}{2} H \notin \sK'$.

If instead $d'$ is odd,  and $d' v'' \in \sK'$, then the vector $v \in \sK'$, and we are done if  we 
show that we may choose 
$v$ not to be in $\sK'$. Otherwise, $\sK'^{\perp} \subset  \sK' $; since $d'$ is odd, then $t$ is odd, hence
the vectors in $\sK'$ have even weight, hence $  \sK' \subset \sK'^{\perp} $, hence we have equality
$  \sK' =  \sK'^{\perp} $, and $ k' = \nu - k'$.

Since we have an element of order $d$ of the discriminant group,  and $ Disc(\sL) \cong \ZZ/d \oplus (\ZZ/2)^{nu}$,
and as we saw  $ |Disc (\sL')| = d \cdot 2^ {\nu - 2 k'} $, follows the desired assertion.
\end{proof}

\begin{cor}\label{length}
The  length $l (\sL')$ (that  is, the minimal number of generators of the discriminant group $Disc(\sL')$)
satisfies the inequality  
$$l (\sL') \leq    \rank(\Lambda) - \rank(\sL') = 21- \nu.$$

Moreover equality holds if and only if $ \nu = 10 + k'$.

\end{cor}

\begin{proof}
If $\sK = \sK'$, then $ Disc (\sL')  = \ZZ/d \oplus (\ZZ/2)^{\nu - 2 k'}$, hence we are done
by the B-inequality ($k' = dim(\sK') \geq \nu -10)$, since  the length
$$l (\sL') = 1 + \nu - 2 k' \leq 1 + 20 -\nu = 22 - (\nu + 1) =  \rank(\Lambda) - \rank(\sL').$$

If instead $\sK \neq \sK'$  and  $ Disc (\sL^*)  = \ZZ/d \oplus (\ZZ/2)^{\nu - 2 k'}$,
we get the same inequality.

Finally, if $ \nu = 2 k'$, then $l (\sL') \leq  2 + \nu - 2 k' =2 < 4 \leq 21 - \nu.$
\end{proof}

We want now to show that if a potential nodal K3 lattice $\sL'$ admits an isometric embedding into the K3 lattice $\Lam$, then this
embedding is unique (up to isometries  of $\Lam$).

To this purpose we  shall  apply Nikulin's theorem 1.14.4  of \cite{nikulin}, asserting that under three (sufficient) conditions
there exists a unique primitive embedding. This assertion must be taken with a grain of salt, since what really Nikulin means is
just that there exists at most one such embedding (it is unique, if it exists).

The conditions are: 

\begin{itemize}
\item
1) Positivity of $(\Lam) > $ Positivity of $(\sL') $, Negativity of $(\Lam) > $ Negativity of $(\sL') $,
\item
2) $\rank (\Lam) \geq  \rank (\sL') + 2 + l (Disc(\sL')_p)$ for all primes $p \neq 2$,
\item
3) If $\rank (\Lam) =   \rank (\sL') + l (Disc(\sL')_2)$, then the quadratic form $q$ of $\sL'$ splits 
a summand 
either of  the form $u_+^{(2)}(2) $ or of the form $v_+^{(2)}(2)$;
\item
4) Condition 3) is automatically verified if $ Disc(\sL') \cong A' \oplus (\ZZ/2)^3$ (see remark 1.14.5 of \cite{nikulin}).
\end{itemize}

To explain the above conditions, recall that, if $M$ is an integral  lattice, 
$ A_M : = Disc(M) = M^{\vee} / M$ inherits a bilinear form with values in $\QQ / \ZZ$, 
and an associated {\bf finite} quadratic  form $ q_M : A_M \ra \QQ / 2 \ZZ$.

Writing the discriminant group $A_M$ as the direct sum of its $p$-primary components  $({A_M})_p$,
we can write $ q_M = \oplus_p ({q_M})_p$.

The $p$-adic part of the finite quadratic form can be computed by writing the normal form
for the  $p$-adic lattice $ M_p : = M \otimes_{\ZZ} \ZZ_p$, since $ ({q_M})_p = q_{ M_p}$.

Now, see also \cite{conway-sloane}, $M_p$ splits as an orthogonal direct sum of lattices of rank $\leq 2$, which in Nikulin's notation are called
\begin{enumerate}
\item
$K_{\theta}^{(p)} (p^h)$, corresponding to the $1 \times1$-matrix $( \theta p^h)$, where $\theta$ is a unit, 
uniquely defined modulo $( \ZZ_p^*)^2$ (hence we have four possibilities $\theta= 1,3,5,7$);
the associated discriminant group is $\cong \ZZ/ p^h$, and the value of the finite quadratic form 
on the generator $e' :=p^{-h} e$ is $q_{\theta}^{(p)} (p^h) (e') = \theta \cdot p^{-h}$.
\item
$U^{(2)}(2^h) $, corresponding to  the $2 \times2$-matrix
 \begin{equation}
 \left(\begin{matrix}0&2^h\cr 2^h&0
\end{matrix}\right),
\end{equation} 
its associated quadratic form is denoted $u_+^{(2)}(2^h) $;
the associated discriminant group is  $\cong (\ZZ/ 2^h)^2$ and the value of the finite quadratic form 
on the vectors $ e'_1,   e'_2 $ is $2^ {-(h-1)}$.
\item
$V^{(2)}(2^h) $, corresponding to  the $2 \times2$-matrix
\begin{equation}
 \left(\begin{matrix}2^{h+1}&2^h\cr 2^h&2^{h+1}
\end{matrix}\right),
\end{equation} 
its associated quadratic form is denoted $v_+^{(2)}(2^h)$,
the associated discriminant group is  $\cong (\ZZ/ 2^h)^2$ and the value of the finite quadratic form 
on the vector $ e'_1 + e'_2 = 2^{-h} (e_1 + e_2)$ is $2^ {-(h-1)}$.

\end{enumerate}

Observe that always $ h \geq 1$, else one gets a trivial summand to the discriminant group.

\begin{theo}\label{unicity}
Given a potential nodal K3 lattice, if it admits an isometric embedding into the K3 lattice $\Lam $,
then this embedding is unique.

\end{theo}

\begin{proof}
Condition 1) is verified, since $\Lam$ has positivity index $3$, negativity $19$, these for $\sL'$
are respectively $1, \nu \leq 16$.

Condition 2) is verified since, for a prime $ p \neq 3$, then $ l (Disc(\sL')_p) \leq 1$, since $ Disc(\sL')_p $ is cyclic
(there is an element of order $d'$), hence we are done, since $ 22 > 20 \geq \nu + 4$.

Regarding the prime $ p=2$, equality holds if and only if $ \nu = 10 + k'$ by Corollary \ref{length}.

Condition 3) is verified, indeed we shall  now  explain Nikulin's condition 4) concerning the prime $p=2$.

First of all, recall that $ Disc (\sL')_2 \cong \ZZ/2^r \oplus (\ZZ/2)^{\nu - 2k'}, $
hence we only have one  summand of type $K_{\theta}^{(2)} (2^r)$,
and all other summands have $ h =1$.

Hence the quadratic form $q_2$ of $\sL'$ can only split summands of type 
$u_+^{(2)}(2) $, $v_+^{(2)}(2)$, $q_{\theta}^{(2)} (2)$,
plus only  one summand of the form   $q_{\theta}^{(2)} (2^r)$.

Either condition 3) is verified, or we have at least $\nu - 2k'$ copies of summands of the type $q_{\theta}^{(2)} (2)$.

We use now Prop. 1.8.2 of \cite{nikulin}, formula d). The possible $\theta$'s are distinguished by their congruence class (modulo $8$), 
namely, $1,3,5,7$.  However, by formula i) ibidem, $q_{\theta}^{(2)} (2) \cong q_{5\theta}^{(2)} (2)$, hence for the forms
$q_{\theta}^{(2)} (2)$ there are only 2 isomorphism classes.

If we have three summands, then two must give the same class, so we may suppose that 
they have the same $\theta$.  If the  other is with $\theta' \equiv - \theta (mod 4)$,
we split off $u_+^{(2)}(2) $, if instead $\theta' \equiv  \theta (mod 4)$ we split off  $v_+^{(2)}(2)$.

Condition 3) is therefore verified as soon as $\nu - 2k' \geq 3$. Condition 3)  pertains just to the case where
equality in  claim \ref{length} is attained. Hence we should have $\nu = 10 + k'$, and we are done unless
$\nu \leq 2 k' + 2$. We only need to observe that $k' \leq 6$, hence $10 + k'  \leq 2 k' + 2$ is impossible,
so we are finished.
\end{proof}

The conclusion of the above theorem is that 
  $ \sL' $ and  its  embedding in $\Lambda$
is uniquely defined by  the code $\sK'$ and by the integer $d$. 

The next question is when does a potential nodal K3 lattice admit an isometric embedding into the K3 lattice $\Lam$.

Theorem 1.12.2 of \cite{nikulin} answers this question giving necessary and sufficient conditions.

These conditions are: 

\begin{itemize}
\item
1') Positivity of $(\Lam) \geq  $ Positivity of $(\sL') $, Negativity of $(\Lam) \geq $ Negativity of $(\sL') $,
\item
2') $\rank (\Lam) \geq  \rank (\sL') +  l (Disc(\sL'))$, 
\item
3')  a certain congruence condition for the primes $p \geq 3$ such that $\rank (\Lam) =  \rank (\sL') +  l (Disc(\sL')_p)$, 
\item
4') a certain congruence condition between $2$-adic units (here $ K(q_2)$ is the 2-adic lattice having associated quadratic form $q_2$):
$$ |Disc(\sL')| \equiv \pm  discr ( K(q_2)) \ ( mod \ (\ZZ_2^*)^2) ,$$
which must be verified in the case where 
$$\rank (\Lam) =  \rank (\sL') +  l (Disc(\sL')_2),$$
 and where moreover $q_2$ does not split off a summand 
 $q_{\theta}^{(2)} (2)$.
\end{itemize}

\begin{rem}
i) As we shall see, the last condition 4') will be the only one which is not always verified, and which, fixed the code $\sK'$,
 will lead to a distinction
depending on  the congruence class of $d$ modulo $(16)$. This distinction is important, 
since  we shall  classify the potential nodal K3 codes,
and show that their classification depends on the congruence class of  $d$ modulo $(8)$.

ii) In many cases however we shall show the existence of these embeddings by purely
geometric methods (smoothings or projections).
\end{rem}

\subsection{Proof of Theorem \ref{Nodal-Quartics} = Theorem \ref{Quartics}}
\begin{proof} 
Consider a Kummer surface $Y_0$, which is unobstructed (cf. instance the Duke of Cefal\'u surface illustrated in Section \ref{code_shortenngs}):
then if $Y$ is a small deformation of $Y$, then $Y^*$ is obtained from $\tilde{Y_0}$ by removing $16-\nu$
($-2$)-curves $E_i$.

We have seen  that the fundamental group of $Y_0^*$ is the  affine group of transformations of $\RR^4$
$$ x \mapsto \pm x + \la , \la \in \ZZ^4,$$
and it is generated by simple small loops $\ga_i$ around the curves $E_i$, which correspond to
the symmetry around a  half period $\frac{1}{2} \la_i$ of $\RR^4 / \ZZ^4$,
$\ga_i (x) = - x + \la_i$. 

Without loss of generality we may assume that the half period $0$ corresponds to one of the
$16-\nu$
($-2$)-curves $E_i$ which we stick in again to get $Y^*$ from $Y_0^*$.
Hence the fundamental group of $Y^*$ is a quotient of $\Ga $ by the normal subgroup generated by
$ x \mapsto -x$, and is therefore a quotient of $(\ZZ/2)^4$, hence it is Abelian, and equal to $H_1(Y^*, \ZZ)$.
The latter is in turn isomorphic to $\sK$.

That $\sK$ is determined by its dimension was shown in the course of the proof of proposition \ref{quartic-codes}.

Clearly $\nu$, $\sK \subset \sK'$, are deformation invariant for equisingular deformations,
and we can take different shortenings of  the Kummer code yielding the same number $\nu$, but different codes:
this shows that fixing only $\nu$ we get a disconnected set.

On the other hand, the code $\sK'$, as in corollary \ref{codehomology2},  is shown to be equal to the cokernel 
 $$H_1 (Y^* \setminus H , \ZZ) \cong coker [ H^2 (\tilde{Y}, \ZZ) \ra  \oplus_1^{\nu} \ZZ [E_i] \oplus \ZZ H].$$
of the dual map of the embedding inside the $22$-dimensional K3 lattice $
\Lambda : = H^2 (\tilde{Y}, \ZZ)  $ of the lattice $\sL$ generated by the
curves $E_i$ and by the hyperplane section $H$.

We have seen (ibidem) that 
the saturation  $\sL' := \sL^{sat} : = \Lambda \cap \QQ \sL$ consists of the vectors
 $$ \sL' = \sL^{sat} = \{ \frac{1}{2} ( \sum_i a_i E_i + b H )| \sum_i (a_i E_i + b H) (mod \ 2) \in \sK' \}.$$ 

$ \sL' $ is  primitively embedded inside $\Lambda$.

\bigskip

 We have proven that for each nodal quartic surface  
 $\sK'$ is a shortening of the extended Kummer code
 $\sK'_{Kum}$ (proposition \ref{kummershort}), 
 indeed the proof shows that every potential nodal K3 code for $d=4$ occurs as  such a shortening.
 
 Hence, for $d=4$, all the potential  nodal K3 lattices $\sL'$ come from small deformations of 
 a Kummer quartic surface, hence are all the lattices associated to a nodal quartic surface.
 
  Theorem \ref{unicity} asserts  that any such lattice $\sL'$ admits a unique primitive embedding in the K3 lattice $\Lam$,
 and we observe that it possesses one such   embedding,  induced by the inclusion $\sL' \subset \sL'_{Kum}$.

Therefore we can take a marking of the second cohomology (a lattice  isomorphism of $ H^2 (\tilde{Y}, \ZZ) $ with the
K3 lattice $\Lambda^0$ of a Kummer surface, such that the subgroup $\sL' \subset \sL'_{Kum}$ maps onto a fixed subgroup $\sL^{'0 }\subset \sL'_{Kum}$.
Denote by $e_i$ the image of the class of $E_i$, for $ i = 1, \dots , 16$ and by $e_{20}$ the image of the class of $H$.

We can complete to an orthogonal basis, $e_1, \dots, e_{22}$ of $\Lambda^0 \otimes_{\ZZ} \RR$,
such that the intersection form is negative  definite on the span of the first 19 vectors,
and   positive definite on the span of the last 3   vectors.

The period domain is 
$$ \sD : = \{ [ \omega ] \in \PP ( \Lambda^0 \otimes_{\ZZ} \CC) | \om \wedge \om  = 0, \ \om \wedge \overline{\om} > 0 \},$$
and if $\om = (z_1, \dots, z_{19}, w_{20}, w_{21}, w_{22}, )$ the above equations boil down to:
$$ \sum_1^{19} z_i^2  = \sum_1^3 w_j^2, \ |w| > |z|.$$ 

$\sD$  contains the  submanifold 
$$\sM = \{ [ \omega ]  | \om \wedge e_i = 0, i=1, \dots, \nu, i= 20\} = \sD \cap \{ z_1= \dots z_{\nu} = w_{20} =0 \},$$
which  we shall   show to   have two  connected components.

Indeed, we write $Z : = (z_{\nu+1}, \dots, z_{19}), \ u := (y,v) : = (w_{21} + i  w_{22}, w_{21} - i w_{22})$.

Then we have the subset $\sM$ of projective space consisting of the vectors $(Z, u) = ( Z, ( y,v))$ such that 
$$ (**) \  \sum_j Z_j ^2 =  yv, \ 2 | Z|^2  < |y|^2 + |v|^2 .$$ 

Hence $u = (y,v) \neq (0,0)$ and by the second inequality either $ y \neq 0$, or $v \neq 0$. 
Therefore $\sM$ has a holomorphic map  to the projective line $\PP^1$ via $(y,v)$.

By the Cauchy inequality we have $ | \sum_j Z_j ^2  | \leq  | Z|^2$, hence if the inequality (**) holds
we must have, setting $ r : = | Z|^2$:

$$ 2 | y v| \leq 2r <  |y|^2 + |v|^2 \Rightarrow |y|^2 - |v|^2 \neq 0 \Leftrightarrow |y| \neq |v|.$$

Hence we have two disjoint open sets, defined by $ |y| < |v|$, respectively by $ |v| < |y|$,
exchanged by complex conjugation.

Let us show that these open sets  are connected; without loss of generality let us stick to the case $ |v| < |y|$:
then  $ y \neq 0$ and,  since we are in projective space, we set $y=1$, and we get
$y=1$,$| v | < 1$.

Since we have 
$v =  \sum_j Z_j ^2 $, we get  the  domain
$$ \sD: =  \{ Z : = (z_{\nu+1}, \dots, z_{19}) || \sum_j Z_j ^2  | < 1,  2 | Z|^2  < 1 + | \sum_j Z_j ^2  |^2 \}.$$

This  domain $ \sD$ is connected (indeed, it is known to be a symmetric bounded domain of type IV), as we now shortly prove.

In fact, the inequalitites imply $|Z| < 1$,   $ \sD$ 
contains the origin in its interior and for every complex line
$L$ through the origin (we assume now $ |Z| = 1$ for the direction of the line $\{ \la Z| \la \in \CC\}$ and set $b : = | \sum_j Z_j ^2  | $) the intersection $ \sD \cap L $ is the following  open set in $\CC$:
$$ \{ \la \in \CC | \  - 2  |\la |^2    +  |\la |^4  b^2    + 1 > 0 \ , \  |\la |^2 b < 1  \} , $$
where $ 0 \leq b \leq 1$. 
This last open set  is connected since for $b=0$ we get, setting  $ r : =   |\la|^2 $, $r  < \frac{1}{2}$; whereas, for $ 0 < b \leq  1$
we get 
$$ r b < 1, \ b^2 (b^2 r^2 - 2 r + 1) >0  \Leftrightarrow  r b < 1 , \  1 - r b^2  > \sqrt{1 - b^2 } \Leftrightarrow   r b <  \frac{1}{b} (1 -  \sqrt{1 - b^2 }) \leq 1. $$

\medskip

 We end  the proof  by invoking firstly the following tools:

\medskip

(1) the surjectivity of the period map for K3 surfaces and

(2)  the Torelli theorem for K3 surfaces (see for instance \cite{torelli}, \cite{torelliB})
implying that the given family of nodal quartic surfaces with $\nu$ nodes, and with fixed extended code $\sK'$
is a holomorphic  image of $\sM$, and has  therefore at most two connected components. 

(3) Moreover, since the deformations of a Kummer surface are unobstructed (it suffices that this holds for just one of them, even if this is true in general)
each such shortening $\sK'$ is realized,  see  Section \ref{append_quadratic_Reed-Muller_code}, especially corollary \ref{deg=4} and the list preceding it;
and

(4) each family of such nodal quartic surfaces contains small deformations of a Kummer surface, since these correspond in the period domain to the subspace $$\sM_{Kum} = \{ [ \omega ]  | \om \wedge e_i = 0, i=1, \dots, 16, i= 20\} = \sD \cap \{ z_1= \dots z_{16} = w_{20} =0 \}.$$
 
(5) we need an argument to see that there is only one connected component: for this, we observe that if we take the complex conjugate 
surface $\overline{Y}$ of $Y$,
then we have an orientation preserving diffeomorphism among them, which however sends the holomorphic 2-form $\eta$
to $\overline{\eta}$. Hence  the period point of   $\overline{Y}$ is obtained by exchanging the two coordinates $y,v$,
since $\eta$, $\overline{\eta}$ span the positivite definite  subspace $H^{2,0} \oplus H^{0,2}$, which is orthogonal to $H$. 

Hence $Y$ and $\overline{Y}$ have period points which, with the given marking, belong to the two connected components
of $\sD$. The proof is then completed if we show that there is, for each choice of $\sK'$, a real surface $Y $ ($Y = \overline{Y}$).

This can be proved via the following observation: take the Duke of Cefalu' Kummer quartic surface $Y_0$ (see theorem 
\ref{unobstructed}, ii)). $Y_0$  has a polynomial  equation with coefficients in $\ZZ$, and also its singular points have
coordinates $\in \{0,1\} \subset \ZZ$; hence all the local components of the smoothing of $Y_0$ are defined over $\RR$
and contain real surfaces. We are done in view of point (4) above.
\end{proof}

\section{Cubic 4-folds and nodal K3 surfaces of degree 6}

The present short section gives some motivation for the study of nodal K3 surfaces of degree 6,
 which shall be useful also for the study of nodal quintic surfaces.

We begin by recalling a well known elementary result, which is however crucial for our description

\begin{lemma}\label{projection}
Let $P$ be a point of multiplicity $2$ of a cubic hypersurface $ X \subset \PP^N$.
Take coordinates $(x,z) : = (x_0, \dots, x_n, z) $ where $ N = n+1$ and $P = (0,1)$;
consider the Taylor development at $P$ of the equation $F$ of $X$, $ F = Q(x) z + G(x)$.

Define $Y \subset \PP^n$ as the complete intersection $Y = \{ x| Q(x) = G(x) = 0\} \subset \PP^n$.
Then

\begin{enumerate}
\item
$\Sing(X) \subset \{ (x,z) | Q(x) = G(x) = 0\} =  (P * Y):$  the singular set of $X$ is contained in the cone over $Y$ with vertex $P$.
\item
If $x \in Y$, $x \notin \Sing(Q)$, then the line $ L = P * x$ (contained in $X$),  contains at most another singular point $(x,z)$  of $X$ ($ x \neq 0$), and this happens if and only
if $x \in \Sing(Y)$.
\item
In the previous situation, $x$ is a node of $Y$ if and only if $(x,z)$ is a node of $X$.
\item
If $ x \in \Sing(Q) \cap Y$ ($ \subset \Sing (Y)$) , then the line $ L = P * x$ contains another singular point of $X$ (different from $P$) if and only if $ L \subset \Sing(X)$;
 this happens if and only if $ x \in \Sing(G)$. 
\end{enumerate}
\end{lemma}
\begin{proof}
i): $$ \{(x,z) \in \Sing(X) | x \neq 0 \} = $$
$$ = \{(x,z)  | x \neq 0 , \ Q(x) = 0, G(x) = 0, z (\partial_i Q)  + \partial_i G = 0, \ \forall i=0, \dots, n\} .$$

ii): if there is an $i$ such that $\partial_i Q (x) \neq 0 $, then  $z$ is uniquely determined, and it exists if and only if the gradient of $G$
is proportional to the one of $Q$.

iv): if the gradient of $Q$ vanishes at $x$, there is another singular point if and only if also the  gradient of $G$ vanishes at $x$.

iii):  observe that we have a node of $X$ if and only if the ideal $(Q, z (\partial_i Q)  + \partial_i G)$
is locally the maximal ideal of the point; but, if $\partial_0 Q \neq 0$,  this ideal is the ideal 
$$(Q, - ( \partial_0 G)( \partial_i Q ) + (\partial_i G) (\partial_0 Q )).$$

 We may take $Q$ as a local coordinate $y$, and then in local coordinates $(u_j , y)$ we get the ideal $(y, \partial_j G)$,
 which is the Jacobian ideal of $Y$.
\end{proof}

\begin{cor}\label{projection_c}
$X$ is nodal if and only if

(I) $Q$ is smooth

(II) $Y$ is nodal.

Moreover, in this case  the number $\ga$ of nodes of $X$ equals the number of nodes $\nu$ of $Y$ plus 1.

Conversely, if $Y$ is nodal, then the corank of $Q$ is at most $2$, and if the corank equals $2$, then
$X$ has $\nu-3$ nodes plus the other singular point $P$, which is not a node.
\end{cor}
\begin{proof}
For the first assertion, 
(I) is a necessary condition, since $P$ is a node and $Q$ is the quadratic part of $F$ at $P$.
If $Q$ is smooth, then the equivalence of $X$ and $Y$ being nodal follows from iii) of Lemma \ref{projection}.

For the second assertion, since $Y$ has isolated singularities, $\Sing(Q)$ has projective dimension at most $1$, and if equality holds,
then $G$ intersects this singular line in exactly $3$ distinct points, which are nodes for $Y$.
\end{proof}

\section{Nodal K3 surfaces of arbitrary degree.}\label{K3}

 The following proposition motivates the definition, given earlier, of potential nodal K3 codes $\sK'$.

\begin{prop}\label{2,3}
Let $Y$ be 

\begin{itemize}
\item
a nodal K3 surface in $\PP^4$ (a complete intersection of type $(2,3)$), respectively
\item 
a nodal K3 surface
of degree $ d = 4m + 2$ (where the ample divisor $H$ is indivisible).
\end{itemize}

Then the number $\nu$ of nodes of $Y$ is at most $15$, the weights of the code $\sK$ are equal to $8$,
while 
\begin{itemize}
\item
the weights of $\sK'$ are divisible by $4$ in the first case,  respectively
\item
  the weights of $\sK'$ are divisible by $4$ for $m$ 
odd, and congruent to $2$ modulo $4$ if $m$ is even. 

\end{itemize} 

Moreover, $dim (\sK) \geq dim (\sK') -  1$, 
$dim (\sK') \geq \nu - 10$ (hence $ dim (\sK) \geq \nu - 11$).

Finally, if   $Y$ is a nodal K3 surface of degree  $d = 4m$,  the weights of the code $\sK$ 
are  either $8$ or $16$, 
 while the weights    of vectors in  $\sK' \setminus \sK$, when written as $w = t+1$, 
satisfy $ t \equiv 2m \ mod (4)$
(that is,  for $m$ odd $  t \equiv 2 \ mod(4)$,
for $m$ even  $ t \equiv 0 \ mod(4)$).

\end{prop}
\begin{proof}

The determination of the weights follows  from the theorem of Riemann Roch: if $ 2 L = \sum_1^t E_i + \de H$ on the resolution of singularities of $Y$,
with $\de \in \{0,1\}$,
then $\chi(-L) = \chi (\hol) + \frac{1}{2} L^2$, hence $L^2$ is an even number.

But $ 4 L^2 = - 2 t + d  \de $ hence $4$ divides $ \frac{d}{2}   \de - t$. For $\de = 1$ we get that $ t \equiv \frac{d}{2}  \ mod (4)$;
for $\de =0$ we use the usual fact that the first cohomology group carries a non-degenerate alternating form, hence  $L^2$
is divisible by $4$ and $t$ is divisible by $8$.

For the other assertion, we know in any case that $\nu \leq 16$: if equality holds, then by the B-inequalities the dimension of $\sK$ is equal to $ 16 - 11=5$,
hence $\sK$ is the strict Kummer code and $\sK' \neq \sK$.  Therefore, for $m$ odd,   there is a vector in $\sK'$ with  $ t \equiv 3 \ ( mod \ 4)$, and adding the vector in $\sK$ with all coordinates equal to $1$, we get
a vector in $\sK$ with   $t' = 16 - t \equiv 1 \ ( mod \ 4)$, a contradiction;  similarly for the case $m$ even.
\end{proof}

\begin{rem}\label{R_d=2}
If the degree $d$ is  $2$, then $ Y$ is a finite double cover of the plane branched 
on a plane curve $B$ with only nodes as singularities, which correspond to the nodes of $Y$. 

The codewords in $\sK' \setminus \sK$ correspond to components $\De$ of $B$ of odd degree,
with $t$ equal to the number of nodes where $\De$ has odd multiplicity ($t=5$ if $\De$ is a line,
or has degree $5$,
$t=9$ if $\De$ is a cubic). 

The codewords in $\sK$ correspond to divisors  $\De \leq B$ of even degree,
with $t$ equal to the number of nodes where $\De$ has odd multiplicity.

Then we have an easy picture: the maximum number $\nu = 15$ is obtained when $B$ consists of $6$
lines intersecting in $15$ points (and we get the projection of a Kummer surface in $\PP^3$ if there is a smooth conic tangent to all the
six lines in points different from the $15$ nodes).

In this case each line in the branch curve $B$ produces a codeword with weight $6$ in $\sK'$, and  ($t=5$) a codeword of weight $5$ in the projected code $\sK''$, namely $$ v_i : =  \sum_{i\neq j} E_{i,j}.$$ 
The sum of these $6$ vectors $ v_i $ is the zero vector, and these vectors span $\sK'$, which has dimension $5$. 

By our above remarks on the weights of vectors in $\sK'$, $\sK'$ is contained in the projection of the extended Kummer code
(contained in the second Reed -Muller code for the affine space $\FF_2^4$). If we take  the hyperplane point $H$ as corresponding
 to the origin of  the affine space, we get the following description (see Section \ref{append_quadratic_Reed-Muller_code}). 

{\bf (d=2)  $\sK'$ consists of the subspace generated by the linear forms and
the quadratic function $q'(x,y) : = x_1 y_1 + x_2 y_2 + 1$.}

All the other configurations are obtained by smoothings of the curve which is the union of six lines, and the corresponding codes
are shortenings of the above code.

In view of Proposition \ref{kummershort} we record the conclusion of the above discussion.

\begin{prop}\label{d=2}
Let $\sK' \subset \FF_2^{\nu} \oplus \FF_2 H$ be a bicoloured code such that $\nu \leq 15$, such that the weight
of vectors of $\sK$ is $8$, and the other weights of $\sK'$ are $6$ or $10$.

Then $\sK'$ is a shortening of the subspace of the second order  Reed-Muller code which is   generated by the linear forms and
the quadratic function $q'(x,y) : = x_1 y_1 + x_2 y_2 + 1$.

\end{prop}

It is a fun exercise to classify geometrically all the possible cases (we proceed by smoothing one node at each step).  Recall that $ k' : = dim (\sK')
 \geq k := dim (\sK)$,
and observe that  in the following list $\sK' = 0$ if $B$ is irreducible.

\begin{enumerate}
\item
$\nu=14$: conic plus $4$ lines ($k=3, k'=4$)
\item
$\nu=13$: two conics plus $2$ lines ($k=2, k'= 3$) or nodal cubic plus $3$ lines ($k=2, k'= 3$);
\item
$\nu=12$: three conics ($k=k'=2$) or smooth cubic plus $3$ lines ($k=2, k'= 3$) or $3$-nodal quartic plus two lines ($k=1, k'= 2$)  or
nodal cubic plus conic plus line ($k=1, k'=2$)
\item
$\nu=11$: two nodal cubics ($k=0, k'=1$) or smooth cubic plus conic plus line ($k=1, k'=2$) or 
$2$-nodal quartic plus two lines ($k=1, k'= 2$)  or $6$-nodal quintic plus line ($k=0,  k'=1$)
or conic plus three-nodal quartic ($k=k'=1$) 
\item
$\nu=10$: nodal cubic plus smooth cubic  ($k=0, k'=1$) or conic plus two-nodal quartic ($k=k'=1$) or 
$1$-nodal quartic plus two lines ($k=1, k'= 2$) or irreducible $5$-nodal quintic plus line ($k=0,  k'=1$) or  irreducible $10$-nodal sextic
\item
$\nu=9$: two smooth cubics ($k=0, k'=1$) or conic plus one-nodal quartic ($k=k'=1$) or 
smooth quartic plus two lines ($k=1, k'= 2$) or irreducible $4$-nodal quintic plus line ($k=0,  k'=1$) or
 irreducible $9$-nodal sextic
\item
$\nu=8$: conic plus smooth quartic ($k=k'=1$) or irreducible $3$-nodal quintic plus line ($k=0,  k'=1$) or  irreducible $8$-nodal sextic
\item
$5 \leq  \nu \leq 7$: irreducible $(\nu -5)$-nodal quintic plus line ($k=0,  k'=1$);
\item
$0 \leq  \nu \leq 7$: irreducible $\nu$-nodal sextic.
\end{enumerate}

All these configurations can be realized over the reals, and we get in this way, as we shall further see,
the classification of the nodal Severi varieties of  K3 surfaces of degree $2$.

\end{rem}

We mention here a result of \cite{kalker}, theorem 5.6, where,   following ideas of Togliatti, he    constructed  
 degree six K3 surfaces with $15$ nodes (we do not reproduce the full proof since we reobtain this result
 as a consequence of Nikulin's theorems on quadratic forms).

\begin{prop}\label{15-6}
There exists a configuration in the plane, of  a nodal curve $B$ of degree six consisting of 
 two smooth conics $\{Q_1=0\}, \{ Q_2 =0\}$,   plus $2$ lines $\{Q_3=0\}$, and of two smooth conics $C_1 = \{ q_1 = 0\}, \ 
 C_2 = \{ q_1 = 0\}$, both not passing through the nodes of $B$, and   everywhere tangent to $B$ in distinct points,
 indeed such that
 $$Q_1 Q_2 Q_3 = r^2 + 4 q_0 q_1 q_2 ,$$
 where  $r$ is a cubic form, and $q_0$ is  the square $x_0^2$ of a linear  form.

Let  $p : Y \ra \PP^2$ be the double cover   branched over $B$, yielding a  nodal degree two K3 surface $Y$,
which is birational to the degree six K3 surface $Y' \subset \PP^4$,
$$  y_0 y_1 + q_0 (x_0, x_1, x_2) = 0, \ \ y_0 q_1 (x_0, x_1, x_2) + y_1 q_2 (x_0, x_1, x_2) + r (x_0, x_1, x_2) =0.$$

 $Y'$ has $15$ nodes.

\end{prop}
\begin{proof}{\em (Idea)}
The points $e_0, e_1$ are  nodes for $Y'$, and projection from the line $\{ x_0 = x_1 = x_2=0 \}$ joining them yields a double cover of
the plane branched on $B$: actually projection from $e_1$ yields the nodal quartic surface of equation
$$ y_0^2 q_1 (x) - q_0(x) q_2(x) + r(x) y_0 =0,$$

Since $B$ has $13$ nodes, we see that  $Y'$ has $13 + 2 = 15$ nodes.
\end{proof}

\medskip

For our application, we need however  a stronger result, namely a K3 surface which is defined over $\RR$, and with all the nodes being real points.

This was achieved in \cite{cat-kummer}, where the following theorem was proven:

 \begin{theo}\label{real}
For each degree $ d = 4m$ there is a nodal K3 surface $X'$ of degree $d$ with $16$ nodes which is defined over $\RR$, and such that
all its singular points are defined over $\RR$.

 Similarly,  for each degree $ d = 4m - 2$ there is a nodal K3 surface $X''$ of degree $d$ with $15$ nodes which is defined over $\RR$, and such that
all its singular points are defined over $\RR$.

\end{theo}

\subsection{Codes for K3 surfaces of degree $d \equiv 6 \ (mod \ 8)$}\label{degree6}
Next, we shall now classify all the pairs of codes as in proposition \ref{2,3} with $m$ odd;  later  we shall state  analogues of
the theorem concerning nodal quartic surfaces in $\PP^3$.

The following theorem is due to collaboration with Michael Kiermaier.

\begin{theo}
\label{codes(2,3)}
Let $\sS$ be a set of cardinality $\nu \leq 15$, and consider a non-zero bicoloured code $\sK' \subset \FF_2^{\sS} \oplus \FF_2 H$,
such that

(i) the code $\sK : = \sK' \cap \FF_2^{\sS}$ admits only the weight $8$,

(ii) the weights of vectors in $\sK'$ are divisible by 4,

(iii) $dim (\sK) \geq \nu - 11$,
$dim (\sK') \geq \nu - 10$, $dim (\sK') \leq dim (\sK) + 1$.

Then  $\sK$ is a shortening of the simplex code, in particular its dimension $k \leq 4$, and we can  view  $\sS \subset \PP : = \PP(\FF_2^4)$,
and  the non-zero elements of $\sK$ as the characteristic functions of complements of hyperplanes $ H \subset \PP$
 with the property that $ \sT:  = (\PP \setminus \sS ) \subset H$.
 
 We have several cases :
 
 (1) $\sK = \sK'$: $k= k' =1,2, $ accordingly $\nu \in \{ 8,9,10, 11\}$, resp. $\nu = 12$;
 
 (2) $\sK \neq  \sK'$,  the weight $16$ occurs in $ \sK'$: then $\nu = 15$ and $\sK'$ is the strict Kummer code;
 
  (3) $\sK \neq  \sK'$, all the weights are $8$: then $ k' =1,2,3,4 $, $\nu \in \{ 7, 8,9,10, 11\}$, resp.$\nu = 11, 12$,
  resp. $\nu = 13$, resp. $\nu = 14$.
  
  (4)   $\sK \neq  \sK'$, there exists a weight $4$ vector  (if $ k >0$ this condition  is implied by the condition that there  exists  
  a weight $12$ vector): then $k=0, 1,2,3,4$, and respectively 
  $\nu \in \{4,5,6,7, 8,9,10, 11\}$, resp. $\nu \in \{ 9,10, 11,12 \}$, resp. $\nu \in \{ 12, 13 \}$, resp. $\nu =14$, resp.  $\nu =15$.
  The bicoloured code $\sK'$ in this case (4) is unique for all pairs $k, \nu$ except 
  for $k=1, \nu \geq 11,$ when we have two cases,     and also  for $k=2, \nu = 13$, where we  have two cases.

(5) $\sK \neq  \sK'$, $k'=1$, $\nu=11$,  there is only the weight $12$ in $\sK'$.

\end{theo} 
\begin{proof}
In the first case, as we already saw,  $\sK = \sK' $ is determined by its dimension $k=k' \in \{1,2,3,4\}$, and then its length $n$ equals, respectively,
$8,12,14,15$. Since we have $ n \leq \nu \leq 10 + k$, only the cases $k=1,2$ are possible.

In the second case obviously $\nu = 15$, hence $k' \geq 5$, so we have the the strict Kummer code.

In the third case again  $\sK \subset  \sK' $ are determined by their  dimensions $ k = k' -1, \ k' \in \{1,2,3,4\}$ and again we have that 
$ n \leq \nu + 1 \leq 11 + k'$, hence all the cases $k'=1,2,3,4$ are possible.

The remaining  cases are more interesting, with  exception of  the subcases  $k=0$, which are straightforward  to discuss:
there is only a non zero weight, and this can be either $4$ or $12$, but in the latter case we have   
$\nu \leq 11, \nu \geq 11 \Rightarrow \nu=11$, and we are in case (5).

Assume that $ k > 0$: then let $u$ be a vector with $w(u)=4$. We denote as usual by $v_1$ the vector corresponding to the first two columns (here $H$ corresponds to the entry in the last row and last column).

Let $a$ be the number of non-zero coordinates of $u$ in the first two columns, and $b$ the number of non-zero coordinates
in the last two columns, and not the coordinate of $H$ (that is, we are looking at the code $\sK''$,   projection from $H$ of the code $\sK'$);
 then $ a+ b= 3$, and adding $v_1$ we must have $ (8-a) + b = 5 + 2b  \equiv 3 \ (\mod \ 4) $, which means that
$b = 1,3$. We have two cases: $(a,b)= (2,1)$ or $(a,b)= (0,3)$, distinguished by the non occurrence - occurrence of the weight $12$ in $\sK'$.

Conversely, if the weight $12$ occurs, and $k=1$,  then $ a+ b= 11$ and $ 19- 2a  \equiv 3 \ (\mod \ 4) $, hence $a$ is even. 
However, by the B-inequality, $\nu \leq 12$, hence $ b \leq 4$, hence  $ b = 3$ (and we have the previous case).

For the respective dimensions $k=4,3,2 $ we must  have    $\nu =15$, resp. $\nu =14$, resp. $\nu \in \{ 12, 13 \}$.

These are the respective cases where we remove from $\PP$ a set $\sT$ of cardinality $0,1$, or just two or three points 
contained in two hyperplanes, hence two or three collinear points.

$\sK''$ is spanned by the characteristic function of a subset $\sA \subset \sS = \PP \setminus \sT$ which has $3$ or $11$
elements. Moreover, for each hyperplane $ H$ containing $\sT$, if $\sA$ contains $b$ points in $H$, and $a$ points not
in $H$, then we must have that $ b + 8  - a 
\equiv 3 \ ( mod \ 4) $. For $ a+ b = 3$ this means as before  that  $ b = 1,3$.  For $ a+ b = 11$ this means that  $b$ is odd, 
and since $a \leq (\nu - 8)  \leq 7$, $ b= 5,7$ (then respectively $ a= 6, 4$).

Assume that there is a vector of $\sK'$ corresponding to such a set $\sA$ of cardinality $3$. Each plane $H$ must contain 
either all the 3 elements of $\sA$, or just one. Said in other words, every plane containing $\sT$ intersects $\sA$ in only one point,
or contains $\sA$. 

 If $\sT$ has $0$ elements, this means that $\sA$ consists of three collinear points.
 
  If $\sT$ has $1$ element $T$, let  $\{ A_1,A_2, A_3\}  =  \sA$: then two points in $\sA$ cannot be collinear with $T$,
  else there are planes $H$ containing two point of $\sA$. There is a plane $H'$ containing $\sA$  and $\sT$, 
  and this case is always combinatorially equivalent to the one where the three points in $\sA$ are collinear, and the
  line $\Lam$ they span does not contain $T$.
  
  If $\sT$ has $3$ elements, which form a line $L$, there are a priori two possibilities: projection from $L$ to a skew line $\Lam$ 
  maps the three points $A_i$
  to three different points (this case  is  combinatorially equivalent to the one where the three points in $\sA$ are collinear, and the
  line $\Lam$ they span is skew to $L$), or there exists a plane $H'$ containing $L$ and $\sA$ (combinatorially, there is only one case).
  
  However, the second case is not possible: since adding to the characteristic function of $\sA$ the characteristic function of
  the complement of a plane $H$ containing $L$ and different from $H'$ we would get the characteristic function of a set
  $\sB$ of cardinality $8-3 = 5$, which is a contradiction.

Let us look at the case where $\sT$ consists of two points $T_1, T_2 \in L$, and let us denote by $T_3$ the third point of $L$.

If  $T_3 \in \sA$, then each plane $H$ containing $L$ contains $T_3$, and if $T_3 = A_3$,
there is a plane containing $A_1$ and $A_2$, which is combinatorially equivalent to the case  where the points of $\sA$ are collinear, 
but $L$ and  $\sA$ are lines  incident in $T_3= A_3$. 

If $\sA \subset \PP \setminus L$, again we have either three collinear points, spanning a line $\sA$ skew to $L$, 
or there is a plane $H'$ containing $L$ and $\sA$, but this gives a contradiction as before.

These are exactly three different cases.

Finally, if there is no such set $\sA$ of cardinality $3$, we have that elements of $\sK'' \setminus \sK$ correspond to subsets $\sB$
of cardinality $11$. Then each plane $H$ (containing $\sT$) must contain exactly $5$ or $7$ points of $\sB$.

The second alternative is not possible if $\sT$ is non-empty, since $\sT \subset H$, $\sB \cap \sT = \emptyset$. 

If  $\sT$ is empty, then $\sB$ consists of a plane $H_1$ and $4$ other points. Hence there are two planes $H_1, H_2$
such that either $\sB = H_1\cup  H_2$, or there are  points $P, P'$ such that   $\sB = ((H_1\cup  H_2 ) \setminus \{P\}) \cup \{P'\}$.

Take a plane $H$ containing the line $H_1\cap  H_2$ and different from these two planes: then adding the characteristic function 
of the complement of $H$ we get the characteristic function of a set $\sA$ with either 3 or 4 points: in both cases
we reach a contradiction.

We can therefore assume that $\sB$ is such that every hyperplane $H$ (containing $\sT$) must contain exactly $5$ 
 points of $\sB$. Then $\sT$ has at most two elements, and if it has two elements, for each plane $H$ containing $\sT$
 we have $$ (H \setminus \sT ) \subset \sB \Rightarrow \sB = \PP \setminus \sT,$$ a contradiction.
 
 If $\sT$ is one point, take two points $P, P'$ not in $\sB$, and a plane  $H$ containing $\sT, P, P'$: then 
 $H$ contains less than $4$ points of $\sB$, a contradiction. Similarly if $\sT$ is empty, take a plane $H$ containing three points in the complement of $\sB$.
\end{proof}
\begin{cor}
(I) We have two codes as above for $\nu=15$, two also for $\nu=14$,  three  for $\nu=13$, five  for $\nu=12$, 
 seven  for $\nu=11$, four for $\nu = 9,10$,
three for $\nu=8$, two for $\nu = 7$, and just one for $\nu \in \{ 4,5,6\}$. For a total of 35 nontrivial bicoloured codes $\sK'$.

(II) All the above bicoloured codes are shortenings of those two with $\nu= 15$.

(III) Case (3) with $\nu = 14$ is a shortening of case (4) with $\nu = 15$.

\end{cor}

\begin{proof}
Assertion (I) is a direct consequence of the previous theorem, assertion (II) is a consequence  of its proof.

Indeed, cases (1) and (3) are clearly shortenings of case (2), while cases (4) contemplate always the case where $\sA$ is a line,
hence they shortenings of the case (4) with $\nu=15$. If we take the shortening of  case (4) with $\nu=15$ corresponding to a projective basis of
$\PP^3$, then we get case (5).

Concerning assertion (III), recall that  $\sS$ is here the projective space $\PP^3_{\FF_2}$, the code $\sK'$ is generated
by the characteristic functions of complements of planes $\pi$, and by the characteristic function $\chi_{\sA}$, where
$\sA$ is a line.

The vectors in $\sK'' \setminus \sK$ correspond to $\chi_{\sA}$ and $\chi_{\sA} + 1 + \chi_{\pi}$, and if $\pi \supset \sA$ we get the characteristic
function of a set of cardinality $11$ containing $\sA$.  If $\sA \cap \pi = \{ P\}$, then we get the characteristic  function
of a set of cardinality $7$  containing exactly the point $P$ of $\sA$.

Hence, if we take the shortening corresponding to the complement of a point of $\sA$, we get a code whose weights
are just $7,8$.
\end{proof}

\bigskip

We borrow the two following results from \cite{kalker}, and we provide proofs for them since   the Thesis \cite{kalker} is not so easily accessible  (even if for the first result we shall give another proof in Theorem \ref{t=15}). 

\begin{prop}\label{no-Kummer}
Let  $Y$ be  a nodal complete intersection of type $(2,3)$ in $\PP^4$ 
 and let $\sN$ be an half-even set of nodes on $S$. Then  we have

a)  its cardinality  $t : = | \sN |$ satisfies $ t \leq 11$, 

b)  its associated sheaf is Cohen-Macaulay, 

c)  the associated double cover is a regular surface.

In particular the extended code $\sK'$ of $Y$ cannot be the strict Kummer code (case (2) of the previous classification in Theorem \ref{codes(2,3)}).
 
\end{prop}

\begin{proof}
We use a similar notation to the one used in the section on half-even sets of nodes on sextics: $Y$ is  a nodal complete intersection of type $(2,3)$ in $\PP^4$,  $S : = \tilde{Y}$ is its minimal resolution,
 $\sN$ is an half-even set of nodes on $S$, and $t : = | \sN |$.

We denote by $E_1, \dots, E_{\nu}$ the exceptional $(-2)$-curves,
therefore there exists a divisor $L$ on $S$ such that
$$(*)  \sum_{i \in \sN} E_i + H \equiv 2 L. $$

We denote by  $\De: =  \sum_{i \in \sN} E_i = \sum_1^t E_i$.

In this situation we attach to $(*)$ a double covering $ f : Z \ra S$ ramified on $\De + H$, where $H$ is 
here any smooth hyperplane section, and such that
$$ f_* (\hol_Z) = \hol_S \oplus \hol_S (-L). $$

The inverse images of the  $(-2)$-curves $E_1, \dots, E_t $ are  $(-1)$ curves, and can be contracted 
to smooth points, yielding $p : Z \ra Z'$ and a finite double covering $ f' : Z' \ra Y$.

We have $$K_Z = f^* (K_{\tilde{Y}} + L) = f^* ( L)  $$
and, since by Kodaira-Ramanujam  vanishing $H^1(Z, \hol_Z( K_Z + n f^* H))= 0$ for $n \geq 1$,
we get
$$0 = H^1(Z, \hol_Z(   f^* (n H + L)) =   H^1(S, \hol_S(    n  H + L)) \oplus   H^1(S, \hol_S(    n  H )).$$

We obtain $H^1(S, \hol_S(   m  H + L) ) = 0$ for $ m \geq 1$, and using again

(i) Serre duality on $S$: 
$$h^1(S, \hol_S(   m  H + L) ) =  h^1(S, \hol_S(  (- m)  H - L) )  =  h^1(S, \hol_S(  - (m +1) H + L - \De))  $$

(ii) the exact cohomology sequence associated to the exact sequence
$$ 0 \ra  \hol_S(  a H + L - \De) \ra  \hol_S(  a H + L )  \ra \hol_{\De} (L) \ra 0,$$
(we have vanishing of all cohomology groups of $\hol_{\De} (L) $ since $h^j(\hol_{\PP^1} (-1) = 0$),
yielding
$$(**) \ h^i( \hol_S(  (a +1)  H - L ) ) = h^i( \hol_S(  a H + L - \De) ) = h^i ( \hol_S(  a H + L ) ) \ \ \forall i,$$ 

we conclude that 

$$h^1(S, \hol_S(   m  H + L) ) =  h^1(S, \hol_S(  (- m)  H - L) )  =  h^1(S, \hol_S(  (- m-1)  H + L )  ,$$
hence we have  that the first cohomology groups $h^1(S, \hol_S(   m  H + L) ) $ and $h^1(S, \hol_S(   m  H - L) ) $
are zero for all $m$ if and only if $h^1(S, \hol_S(   m  H + L) ) = 0$ for $ m=0 $ or for $m =  -1$,
or, equivalently, $Z$ is a regular surface.

We use  now, following \cite{kalker},  that $Z'$ is a canonical surface of general type, hence $H^1(\hol_{Z'} (2 K_{Z'}))=0$.

Write $f^*(H) = 2 H'$ on $Z$, then, since $H'$ is disjoint from the $(-1)$ curves contracted by $p$, 
the Leray spectral sequence  implies
$$H^1(\hol_{Z} (f^*(H)))= H^1(\hol_{Z} (2 H'))=  H^1(p_* \hol_{Z} (2 H')) = H^1(\hol_{Z'} (2 K_{Z'}))=0.$$
Hence $h^1(S, \hol_S(     H - L) ) = 0$, and $h^1(S, \hol_S(   L) ) = 0$, so that all cohomology groups vanish.

In particular, we have that 
$$0 \leq \chi (S, \hol_S(   L) ) = \chi (\hol_S) + \frac{1}{2} L^2=  2 + \frac{1}{4} (-t + 3) \Leftrightarrow t \leq 11.$$
\end{proof}

\smallskip

The following is lemma 3.13 of \cite{kalker}

\begin{prop}\label{corank2}
If $Y$ is a nodal complete intersection of type $(2,3)$ in $\PP^4$,
 $\sK'$ contains a vector of weight $w=4$ if and only if the unique quadric $Q$ containing $Y$ has corank $=2$.
\end{prop}

\begin{proof}
One direction is easy, because a quadric of rank $3$ is the join of its vertex, a line $L$, with a smooth plane conic $C$:
and the join of $L$ with the tangent lines of $C$ yield a hyperplane $H$ which is everywhere tangent to
$Q$, hence to $Y$, and passes in general (simply) through the three nodes of $Y$, intersection of $L$
with the cubic $G$ such that $Y = \{ Q=G=0\}$.
\smallskip

In the other direction, as in the previous proposition, since $t=3$, we get $$\chi (S, \hol_S(   L) ) = \chi (S, \hol_S(  - L) ) =  2 .$$

Since $2L \equiv H + \De =  H+ E_1 + E_2 + E_3$, then $L$ is effective, $-L$ is not, and $|L|$ has dimension $1$.
Since $ L \cdot E_i = -1$, we have that $ |L| = |D| + \sum_1^3 E_i$, where $D^2 = 0$, since 
$L^2 = 0$. Since $D$ is nef,   we get a base point free pencil $|D|$ on our K3 surface, consisting of elliptic curves $D$ with $ D \cdot H = 3$,
corresponding to plane cubic curves on $Y$. Each of these planes is therefore contained in $Q$, moreover, these planes contain the three nodes 
of $\sN$. Hence the three nodes are collinear (since we have infinitely many planes), they belong to a line $L$ contained
in these planes and in $Q$. If  $Q$ had rank =5, then it would not contain any plane; if instead  $Q$ had rank =4,
then it would be the join of its vertex point $V$ with a rank $4$ quadric $Q'$  in $\PP^3$, hence a family of planes contained in
$Q$ would be the join of $V$ with a family of lines in $Q'$. But then these planes would intersect just in the vertex $V$.

We already know that the corank of $Q$ is at most $2$ (since $Y$ has isolated singularities),
hence $Q$  has corank $=2$.
\end{proof}

\begin{rem}\label{disjoint}
 The above result shows that the condition that $Y = Q \cap G$, where $\corank (Q) = 2$, is a topological property of the
nodal surface $Y$. Hence, the  nodal Severi varieties are split into two disjoint open sets by the condition $\corank (Q) = 2$,
or $\corank (Q) \leq 1$.
\end{rem}

\bigskip

\begin{theo}\label{(2,3)K3}
The irreducible  components of the Nodal Severi variety of nodal K3 surfaces $Y$ of degree $6$ in $\PP^4$ 
are in bijection with those isomorphism classes of the extended codes $\sK'$ of theorem  \ref{codes(2,3)}, 
which are different from case (2), the strict Kummer code.

There is only one irreducible component for $\nu = 15$, and for it the unique quadric $Q$ containing 
$Y$ satisfies $\corank (Q) = 2$; while there are exactly 
two irreducible components for $\nu=14$, distinguished by the condition $\corank (Q) = 2$ or  $\corank (Q) < 2$.
All other components are gotten from small deformations of these two components, and belong to two disjoint open sets
(cf. (i) of remark \ref{disjoint}).

In particular, the Nodal Severi variety $\sF_4(3,15)$ of Togliatti cubics (the cubic fourfolds with $15$ nodes) is irreducible.

\end{theo}
\begin{proof}

The last assertion follows easily from the previous ones, since by Corollary \ref{projection_c} we have that giving the pair of a Togliatti cubic 
$X \subset \PP^5$
and a node $P \in X$  is equivalent, up to projectivities,  to giving a nodal K3  surface $Y  \subset \PP^4$ of degree $6$ with $14$ nodes,
and such that the unique quadric $Q$ containing it has $ \corank (Q) \leq 1$. 

\bigskip

We are going to prove the first assertion  exactly in the same way as in theorem \ref{Nodal-Quartics}, using Nikulin's primitive embedding theorems,
and the Torelli theorem for K3 surfaces. 

 The first step is to observe that, by Theorem \ref{unicity},  if the the extended codes $\sK'$ of theorem  \ref{codes(2,3)}, 
different from case (2), admit  a primitive embedding in the K3 lattice $\Lam$, then this embedding is unique.

\smallskip

Later on, to prove existence, we shall use for each of them the embedding induced by being a  shortening of 
the two codes $\sK'$ with $\nu = 14$. 

Then the Torelli theorem shall prove the second assertion: showing, as  in theorem \ref{Nodal-Quartics}, that these components are irreducible, and that there are exactly 
two irreducible components for $\nu=14$, distinguished by $\corank (Q) = 2$, $\corank (Q) < 2$,
and all other components are gotten from small deformations of these two components.

The third  assertion follows right away from the fact that for $\nu = 14$, $\corank (Q) < 2$,
there is only one irreducible component.

In order  to prove existence,  because all these codes are shortenings of two fixed ones,
and since case (2) cannot occur by proposition \ref{no-Kummer},  it suffices to show existence for
these two:

(A) case (4) with $\nu = 15$,  and

(B) case (3) with $ \nu = 14$.

We can use geometry in case (B), since we can take the Goryunov cubic $X$, and a node $P \in X$ to obtain,
projecting with centre $P$, a complete intersection of type $(2,3)$ possessing $14$ nodes by virtue of Corollary \ref{projection_c}.
And then use the unobstructedness of $X$ (theorem \ref{unobstructed}).

Using the Goryunov cubic is important, since again we obtain  K3 surfaces defined over the reals, and with all the singular points 
defined over the reals, so that all the local smoothings are defined over the reals and have real points.

\smallskip

The same can be done   for case (A), where existence and  the reality conditions are guaranteed by theorem \ref{real}.

\medskip

 More generally, we can also apply Nikulin's existence  theorems.
 
  In the case that $ \nu + 1 \leq 11$ we are indeed done by theorem 1.12.4 of \cite{nikulin},
  else we shall use  theorem  1.12.2 of \cite{nikulin}, part d).
 
 In fact $\Lam$ satisfies that the signature is divisible by $8$, and we have already shown 
 that the inequalities 2) ibidem hold true. Property 3) is satisfied since for every odd prime $p$ (here, just $p=3$) the
 length of the $p$-primary part of the discriminant group is $1 < 22- (1 + \nu)$.
 
 We remain with property 4), concerning the prime $p=2$, where  we have  to verify that, if $\nu = 10 + k'$, then
 
  I) either the quadratic form $q_2$ splits off a rank $1$ term with diagonal element $\theta \cdot 2$, ($\theta$ a 2-adic unit)
 or
 
  II) the following congruence is valid:

 $$ | A_q | \equiv \pm \ disc (K(q_2)) \ ( mod (\ZZ_2^*)^2) ,$$ 
 
 where $ A_q$ is the discriminant group, $K(q_2)$ is the unique 2-adic lattice of rank $l (A(q_2))$ with discriminant form $q_2$.

Once more, we associate to a code vector $v = \sum_i a_i E_i + a H$, $ a = 0,1$ the lattice element
 $ \frac{1}{2} ( \sum_i a_i E_i + a H)$, which we still denote $v$.

 We have a basis of $\sL'$ provided by $$v_1, \dots , v_k , v' , E_1 , \dots, E_{10}, $$
 where $v_1, \dots , v_k$ correspond to  a basis of $\sK$, $v'$ to an element in $\sK' \setminus \sK$.
 Set for commodity $ v_{k+1} = v'$. 
 
 The elements $v_1, \dots , v_k , v_{k+1} $ generate an isotropic subspace for the quadratic form,
 while $v_i E_j = 0$ or $=-1$ according to the property that $j$ is not or is on the support of $v_i$,
 while as usual $E_i^2 = -2, \ E_i E_j = 0, \ i\neq j, $ $ i, j = 1, \dots , 10$.

Since $ k ' = k+1 \leq 5$, we can pair each $v_j$ with some $ E_j$, without loss of generality
 we assume that  $v_j E_j =-1$, for $ j = 1, \dots, k+1$. 
 
 We can reduce the intersection matrix modulo $4$, since  the quadratic form takes vales in $ \frac{1}{2} \ZZ / 2 \ZZ$,
and then with  column operations followed by   the same operations for the rows,  
 we end up with the  direct sum of $-2$ times the identity  $( 6 \times 6 )$- matrix  
 with  $ k+1= 5  $ copies  of the matrix  $A$, equivalent to the matrix $U$:
 
 \begin{equation}
A : = \left(\begin{matrix}0&-1\cr -1&-2
\end{matrix}\right)
\cong \ U : = \left(\begin{matrix}0&-1\cr -1&0
\end{matrix}\right).
\end{equation} 
 and we are  done concerning the existence of a primitive
 isometric embedding, since we have verified that property I) holds.  \end{proof}

\subsection{Nodal K3 surfaces of all degrees.}

 Our next task shall  be  to extend our previous results concerning nodal K3 surfaces of degree $d=2, 4,6$ to the case of all
degrees $ d = 2 d'$.

First of all one has  to determine the  a priori possible bicoloured codes $\sK'$ occurring, 
equivalently, the pair of codes $\sK \subset  \sK'' \subset \FF_2^{\nu}$, keeping in mind the inequality $k' \geq \nu - 10$,
$ k' -1 \leq k \leq k' $ (and later  see whether the corresponding lattices admit a primitive embedding in the K3 lattice).

We shall  do this according to the congruence class of $ d$ (mod $8$).

\smallskip

In all the cases $ d \equiv 2 \ mod (8)$, $ d \equiv 4 \ mod (8)$, $ d \equiv 6 \ mod (8)$ 
we have a first description of the possible codes under the assumption   that  the weights $t$ of
vectors in $\sK'' \setminus \sK$, equivalently the weights $t+1$ of
vectors in $\sK' \setminus \sK$, are the same as in the respective cases $ d=2,4,6.$

\bigskip
If $ d \equiv 4 \ mod (8)$ we have seen in Proposition \ref{2,3} that $ t \in \{ 2,6,10,14\}$,
while for $d=4$ only $t=6,10$ occur.

Not only the case $t=2$, but also the case $t=14$ cannot be excluded code-theoretically.  

\begin{ex}
For instance,  we have 
the case $t=\nu = 14$, $k'= 4$,
where, for each $v \in \sK$, setting $ u := \sum_1^{14} e_i \in \sK''$,  $ u +v$ has weight $6$; in this case $t$ can only be $6,14$. 
\end{ex}

\begin{lemma}\label{tequal14}
If the weight $t=14$ occurs for $\sK''$, then  $\nu = 14$ and the code $\sK''$ is unique, as in the previous example.
\end{lemma} 
\begin{proof}
We show now that,  if $t=14$ occurs, then necessarily $\nu = 14$: $\nu \geq 14$ is obvious, and if $\nu=16$, then $\sK$ is the strict Kummer code,
hence, if $\sA$ has cardinality $14$, then there is an affine function $x_W$ vanishing in exactly one point of the complement of $\sA$;
hence $\sK''$ would admit a vector with   $t= 8$, a contradiction. Similarly if $\nu=15$, we can find a linear function $x_W$ not vanishing in the complement of $\sA$, so that we would again get $ t =8 $, a contradiction.

Then there is  $ u := \sum_1^{14} e_i \in \sK''$ and $\sK$ is generated, up to a permutation, by $ \sum_1^{8} e_i$,
$ \sum_5^{12} e_i$, $\sum_0^3 (e_{1+ 4j} + e_{2+ 4j})$.
\end{proof}

\medskip

 We shall discuss later the case where $t=2$ occurs.

\bigskip
If $ d \equiv 8 \ mod (8)$ we are done: we have seen in proposition \ref{2,3} that $ t \in \{ 4,8,12,16\}$, but indeed we must have $ t \in \{ 4,8,12\}$.

In fact, the case $t=16$ can be here  excluded right away, since then $\nu=16$, dim $\sK \geq 5$, hence $\sK$ is the Kummer code,
therefore follows that $H$ is $2$-divisible, a contradiction. 

For $d=8 m, \nu = 16$, indeed all $ t \in \{ 4,8,12\}$ do in fact occur, as  can be immediately seen:
if $\nu = 16$, $\sK$ is the Kummer code, and  $t =4$ occurs if and only  $t =12$ does, since $\sK$ contains the codeword
with weight $16$. Hence  $t=4$ and  $t=12$ must occur because of a pure code-theoretical argument: since   the case where only $t=8$ occurs would lead to $\sK'' = \sK$,
a contradiction to $ dim (\sK'') \geq 5$.

We shall classify the  corresponding codes  in Theorem \ref{K8}, which shows in particular that they are all the shortenings of a single code, which we call
the K8 code, see corollary \ref{extended}.

\bigskip

If $ d \equiv 2 \ mod (8)$ we have seen in proposition \ref{2,3} that $ t \in \{ 1, 5, 9, 13 \}$, and, for $d=2$, only  $t=5,9$ occur.

\begin{lemma}\label{tequal13}
The case  $t=13$ can occur  only if $\nu=13, k=2$. The code $\sK''$ is then unique.
\end{lemma}
\begin{proof}
If $ t =13$ occurs, then $\nu \geq 13$, and then recall  that $ k' \geq \nu - 10$.

 For $ \nu=15$ we have that $\sK$ corresponds to hyperplane complements in $\PP^3_{\FF_2}$,  and we would have the characteristic function of a set $\sB$ of cardinality $13$,
the complement of two points $P_1, P_2$. We have  planes $\pi$ containing the two points, respectively one of them, respectively none.
Then we get vectors $u \in \sK''$ with respective weights $5, 7, 9$, a contradiction. 

For $\nu = 14$, we are taking the shortening corresponding to $P_1$, hence complements of planes $\pi$ containing $P_1$.
Again, take a plane not containing $P_2$ but containing $P_1$, and again we get $t=7$, a contradiction.

For $\nu=13$, we have no code theoretical restriction, the other half-even sets have then cardinality $5$.

We conclude that $t=13$ can occur  only if $\nu=13, k=2$ and the code $\sK''$ is unique.
\end{proof}

We shall later discuss the case $t=1$.

\bigskip

We are now going to settle the existence question for the case $t=13$.

\begin{theo}\label{t=13}
The case $t=13$ (hence also $\nu=13, k=2$) occurs exactly for $d \equiv 10 \ (mod \ 16)$.

Moreover, the corresponding primitive embedding is unique.

\end{theo}

\begin{proof}
We have seen that $t=13$ can only occur for $ d = H^2 = 2 + 8h$. Set $d' := d/2 = 1 + 4h$.
We know that here $\nu = 13$, hence we have the lattices 
$$ \sL = \bigoplus_1^{13} \ZZ E_i \oplus \ZZ H, \ $$
$$ \sL'  = \langle \sL, v_1 : = \frac{1}{2} \sum_1^8 E_i, 
 v_2 : = \frac{1}{2} \sum_5^{12} E_i,  v_3 : = \frac{1}{2} (H + \sum_9^{13}  E_i \rangle.$$ 
 In particular, a basis of $\sL'$ is given by 
 $$v_1, v_2, v_3, E_2, \dots E_{11}, E_{13}.$$
 The dual lattice $(\sL')^{\vee}$ can be explicitly described (requiring integrality of the values on $v_1, v_2, v_3$) as 
 $$ (\sL')^{\vee} = \{  \frac{1}{2} (\sum_i b_i E_i + \frac{a}{d'} H)| \sum_1^8 b_i \equiv 0 (2),  \sum_5^{12} b_i \equiv 0 (2), a +  \sum_9^{13} b_i \equiv 0 (2)\}.$$
 
 Set $u_3 : =  \frac{1}{2} ( E_{13} + \frac{1}{d'} H)$. Then $u_3$ has order $d$ in the Discriminant group
 $$ \sA_{\sL'} : =  (\sL')^{\vee}  / \sL' ,$$
  since $ d' u_3 \equiv   \frac{1}{2} ( E_{13} +  H) \neq  0 \ (mod \sL' ).$
  
  Hence, defining $b' :=  \sum_1^4 b_i , b'' : = \sum_5^8 b_i , b''' : = \sum_9^{12} b_i $,
  and 
  $$ \hat{\sL} : = \{  \frac{1}{2} (\sum_i b_i E_i )| 
  b' \equiv b'' \equiv b''' \equiv b_{13} \ (mod 2)\},$$
  $$ \sA_{\sL'} \cong (\ZZ/d) u_3 \oplus \hat{\sL}/ \sL' \cong     (\ZZ/d)  \oplus (\ZZ/2)^7.$$
  
  Hence the length (minimal number of generators)  $\ell ( \sA_{\sL'}) = 8$, and also the 
  length of the 2-primary part $=8$.
  
  Since $$ 8 = \ell (( \sA_{\sL'})_2)\leq 22 - rk(\sL') = 22 - 14,$$
  by  Theorem 1.12.2 of \cite{nikulin} there is an embedding of $\sL'$ in the K3 lattice if and only
  if condition 4) holds. This requires to calculate the 2-adic discriminant finite quadratic form.
  
  For this, we define $ w : =  \frac{1}{2} (E_1 + E_5 + E_9 + E_{13})$, and we complete $u_3, w$
  to a basis of  $(\sL')^{\vee}$ choosing a basis of 
    $$ \tilde{\sL} : = \{  \frac{1}{2} (\sum_1^{12}b_i E_i )| 
  b' \equiv b'' \equiv b''' \equiv 0  \ (mod 2)\}.$$
  We note that $ \tilde{\sL}  $ is isomorphic to the direct sum of three copies of 
 $ \{  \frac{1}{2} (\sum_1^{4}b_i E_i )| 
 \sum_i b_i  \equiv 0  \ (mod 2)\},$
 with basis 
 $$  \frac{1}{2} (E_1 + E_2),  \frac{1}{2} (E_3 + E_4),  \frac{1}{2} (E_2 + E_3), E_1.$$
 It suffices, to find the whole basis of $ \tilde{\sL}$, to take the previous $4$ vectors
 and add $4$, respectively $8$, to the lower indices.

 We observe  that  the three summands described above are orthogonal to each other, and are orthogonal to $u_3$.
 While $w$ has the property that the scalar product of $w$ with each summand has values in $\ZZ$.

Hence,  if we take as fourth summand  
  $\ZZ u_3 \oplus \ZZ w$, we have given an orthogonal direct sum for the quadradic form $q_2$ of the discriminant group.
  
  The form of the summand  $\ZZ u_3 \oplus \ZZ w$ is 
  $$- q (a u_3 + b w) =  a^2 \frac{2h}{4h+1} + a b + 2 b^2 ,$$
  hence its binary (2-primary)  part has matrix
  $$ U : = 
  \frac{1}{2} 
  \begin{pmatrix}
    0 & 1  \\ 
    1 & 0 \\  
  \end{pmatrix}.
$$
  Instead, for the summand with basis 
 $  \frac{1}{2} (E_1 + E_2),  \frac{1}{2} (E_3 + E_4),  \frac{1}{2} (E_2 + E_3), E_1,$
 we get the matrix
 $$ 
  \frac{1}{2} 
  \begin{pmatrix}
    2 & 0  & 1 & 2\\ 
    0 & 2  & 1 & 0\\
     1 & 1  & 2 & 0\\
      2 & 0  & 0 & 4\\
  \end{pmatrix}
$$
which, modulo $2 \ZZ$, boils down to the same matrix where last row and column are deleted.

Subtracting second row to the first, and doing the same for the columns, we get   modulo $2 \ZZ$
the matrix 
$$ V: = 
  \frac{1}{2} 
  \begin{pmatrix}
    2 & 1  \\ 
    1 & 2 \\  
  \end{pmatrix}.
$$
The corresponding 2-adic  quadratic forms, in Nikulin's notation, have respective matrices
 $$ u : = 
  \begin{pmatrix}
    0 & 2  \\ 
    2 & 0 \\  
  \end{pmatrix}
   v: =  
  \begin{pmatrix}
    4 & 2  \\ 
    2 & 4 \\  
  \end{pmatrix}.
$$
By condition 4) of Nikulin's theorem 1.12.2 of \cite{nikulin} (the others are trivially satisfied) there exists a primitive embedding of $\sL'$ 
into the K3 lattice if and only if 
$$ \pm  disc(q_2) =  \pm 3^3 2^8 \equiv  |  \sA_{\sL'}|= 2^8 (4h+1)  mod (\ZZ_2^*)^2,$$
$$ \Leftrightarrow \pm 3^3 \equiv 4h+1 (mod 8) \Leftrightarrow \pm 3 \equiv 4h+1 (mod 8) \Leftrightarrow h \equiv 1 (mod 2).$$

Unicity of the primitive embedding follows by the general result of Theorem \ref{unicity}, in this case it follows also directly
 from the above computations and applying theorem 1.14.4 of \cite{nikulin}:  condition 3) is satisfied, since $q_2$ is a direct sum where both 
 $ u = u_+^{(2)}(2) , v = v_+^{(2)}(2)$ occur by our calculation (the other conditions are also trivially satisfied).
\end{proof}

\bigskip

If $ d \equiv 6 \ mod (8)$ we have seen in proposition \ref{2,3} that $ t \in \{ 3, 7, 11, 15 \}$
and the question (in view of the case $d=6$, that we have seen)  is when does the case $t=15$  occur.
This occurence is  equivalent to $\nu=15, k' =5$, and $\sK'$ is the Kummer code.

In this case we shall show 

\begin{theo}\label{t=15}
The case $t=15$ (hence also $\nu=15, k'=5$, $ d \equiv 6 \ (mod \ 8)$ and $\sK'$ is the Kummer code) 
occurs exactly for $d \equiv 14 \ (mod \ 16)$.

Moreover, the corresponding primitive embedding is unique.

\end{theo}

\begin{proof}
The proof is similar in spirit  to the one of Theorem \ref{t=13}, but the calculations quite different.

A  quick proof is the following: $d = H^2 \equiv -2 \ (mod 8)$, hence $H^2 = 2 ( -1 + 4h)$,
and the finite quadratic form is the same as the code $\sK$  of a  Kummer surface.

What varies in condition 4) of Theorem 1.12.2 of \cite{nikulin} is the size of the discriminant group,
which has order $ |4h-1| 2^6$. We claim that the discriminant for $\sK$ is $2^6$, 
hence there exists a primitive embedding in the K3 lattice if and only if $4h-1 \equiv \pm 1 \ (mod 8)$,
this means that $h$ is even, hence $d \equiv -2 \ (mod 16)$.

A slick argument is that we have embedding for the Kummer code $\sK$. A more precise 
and interesting calculation is the following, where $ M : = \{ 0,1,2,3\}$, and $\sL = (\ZZ)^{M \times M}$:

$$\sK = \{ (k_{ij}) \in (\ZZ/2)^{M \times M} | k_{00} + k_{0j}  + k_{i0}  + k_{ij}  = \sum_j k_{0j} = \sum_i k_{i0} =0 \in \ZZ/2\}.$$
Just by the above description 
$$ \sK^{\perp} = \langle  e_{00} + e_{0j}  + e_{i0}  + e_{ij}, \forall 1 \leq i,j \leq 3, \sum_j e_{0j} , \sum_i e_{i0} \rangle.$$
$ dim (\sK) = 5, dim (\sK^{\perp}) = 11$, and 
$$ \sL' = \langle \sL , \{ \frac{1}{2} w \rangle _ {w \in \sK}, \ (\sL' )^{\vee}= \langle \sL , \{ \frac{1}{2} w \rangle _ {w \in \sK^{\perp} }.$$

A basis of $(\sL' )^{\vee}$ is given by 
$$ U_{ij} : = \frac{1}{2} ( e_{00} + e_{0j}  + e_{i0}  + e_{ij}), U_{0*} : = \frac{1}{2} \sum_j e_{0j}, U_{*0} : = \frac{1}{2} \sum_i e_{i0},$$
$$ e_{00} , e_{01} , e_{02} ,e_{10} ,e_{20}  .$$

$q_2$ is defined on $$\sA_{\sL'} = (\sL' )^{\vee} / \sL' = \sK^{\perp} / \sK \cong (\ZZ/2)^6.$$

To calculate the associate bilinear form with values in $\QQ / \ZZ$, we can forget about the basis elements in $\sL$.

We see that the negatives of the intersection numbers are:
$$ - U_{ij}^2 = - U_{0*}^2 = - U_{*0}^2 = 2,$$
$$  - (U_{ij} , U_{hk}) =   \frac{1}{2} \  {\rm for } \   i\neq h, j \neq k, {\rm else} \ =1, \ -(U_{0*},U_{*0})=
\frac{1}{2} .$$

An easy calculation shows, using twice that a (block) matrix (with $A = ^t A$)
$$ 
  \begin{pmatrix}
    0 & A&A  \\ 
   A & 0&A  \\
   A & A&0  \\
  \end{pmatrix}
$$
is equivalent modulo $2 \ZZ$  to 
$$ 
  \begin{pmatrix}
    0 & A  \\ 
   A & 0  \\
  \end{pmatrix}
$$

 that the binary quadratic form is given by three copies of 
 $$ U : = 
  \frac{1}{2} 
  \begin{pmatrix}
    0 & 1  \\ 
    1 & 0 \\  
  \end{pmatrix},
$$
one occurring for the span of $U_{0*},U_{*0}$;
hence $q_2$ is a direct sum of three copies of 
 $ u = u_+^{(2)}(2) $. Follows that its discriminant equals $2^6$ as claimed.
 
  Note that the occurrence of a summand of the type $ u = u_+^{(2)}(2) $ confirms the 
  proven  (Theorem \ref{unicity})  unicity of the primitive embedding. 
\end{proof}

\bigskip

\subsection{One-node projection}

We shall   now resort to  geometry, namely,  projecting the surface $Y$ from a node of $Y$, to obtain more information,
and in this way we shall introduce a new operation on these bicoloured codes $\sK'$, that we shall call {\bf one-node projection}.

In the next proposition we use an old result by Saint Donat \cite{sd}, to obtain that for a general surface in a nodal Severi variety 
of $\nu$-nodal K3 surfaces of degree $d$ projection from a node gives an embedding of the blow up of the surface at the nodal point.

\begin{prop}{\bf (Saint Donat)}\label{proj-node}
Let $Y$ be a nodal K3 surface of degree $d = 2g-2 \geq 10$ embedded in $\PP^{g}$,  let $P_1$ be a node of $Y$,
and let $S : = \tilde{Y}$ be the minimal resolution of $Y$.

Then projection from the node $P_1$ is a morphism of the blow up $Y'$ of $Y$ at $P_1$,
and induces a morphism $\psi : S  = \tilde{Y} \ra \PP^{g-1}$.

The image $Z : = \psi (S)$ is either 

(0) Isomorphic to $Y'$, or

(I) a K3 surface with Rational Double points, obtained from $Y'$ contracting
the strict transforms $C$ of  lines $L \subset Y$ passing through $P_1$ ($C \cdot D = 0$,
$ D : = H - E_1$), or

(II) $\psi$ is of degree $2$ onto a rational surface $\Sigma$ of degree $g-2$,
and this happens precisely when either 

(II a) there is an irreducible curve $\sE$ of arithmetic genus $1$ ($\sE^2 = 0$) with $ D \cdot \sE = 2$, or 

(II b) $D = 2 B$, $B$ of arithmetic genus $2$, in particular $B^2 = 2, \ H^2= 10$, and $\Sigma$
is the Veronese surface in $\PP^5$.

In all the cases  different from (0)  we have that $Y$ is not  a general surface in the Nodal Severi variety
of $\nu$-nodal K3 surfaces of degree $d$: in cases (I), (II a) because the Picard number $\rho (S) > \nu + 1$.

In the case where $Y'$ is embedded, it is primitively embedded unless possibly in the case $d \equiv \ 2 \ mod(8)$
if $ H - E_1$ is 2-divisible in $\Pic(S)$.
\end{prop}

\begin{proof}
In case (I), if there is such a line, then  $S$ contains at least $\nu + 2$ curves, and we  claim that these are numerically independent.

 Otherwise, a numerical  equivalence
$ C = a H - \sum_i b_i E_i$ with rational coefficients implies, in view of $ C \cdot E_1 = 1 \Leftrightarrow b_1 = \frac{1}{2}$, and of the fact that 
$2C$ is in the span of the curves $H, E_1, \dots, E_{\nu}$, that 
$ 2 C \equiv  a H - \sum_i b_i E_i$, with $ a, b_1, \dots, b_{\nu} $ integers.

 But then  $ C H =1 \Rightarrow  2 = 2C \cdot H =  a d \geq d $ is a contradiction always, since if $Y$ is embedded we have $ d \geq 4$.
 
 In case (II) since $ D \cdot E_1 = 2$, and $E_1^2 = -2$, $E_1$ maps to a conic.
 
 In case (II a), we claim that the curves $H, E_1, \dots, E_{\nu}, \sE$ are numerically independent.
 Consider  in fact the associated $(\nu + 2 ) \times  (\nu + 2 )$  intersection matrix, where we observe that $a_i : = E_i \cdot \sE$
 is a non-negative integer. Linear dependence of the columns (since the first  $(\nu + 1 )$ ones  are  linearly independent)
 boils down to 
 $$ 
 \frac{1}{2} \sum_1^{\nu + 1} a_j^2 - \frac{4}{d} = 0, \Leftrightarrow   \sum_1^{\nu + 1} a_j^2 = \frac{8}{d} .$$
 But the left hand side is a positive integer, the right hand side is strictly smaller than $1$, a contradiction.

 In case (II b) $ 2 B = H - E_1$, and $E_1$ maps to a line in the plane $\PP^2$. This line must be everywhere tangent to the
 branch curve $\De$. This condition determines in general a family of K3 surfaces with $17$ moduli: but since the blow up
 of $Y$ in $P_1$ has $\nu - 1$ nodes, the same holds for $\De$, and we get a family of dimension $ 19 - ( 2 + \nu -1)$
 hence of codimension $\nu+1$ and not $\nu$.
 
 For the last assertion, the question is whether $ H' : = H - E_1$ can be divisible in the Picard group (in homology).
 
 Since $ H' \cdot E_1 = 2$, the question is whether $  H - E_1 = 2 L$, that is, $P_1$ is a half-even set nodes.
 But $ t=1 $ implies that $d \equiv \ 2 \ mod(8)$.
\end{proof}

\begin{prop}\label{node-proj}
Assume that $Y$ is a nodal K3  of degree $d$, and let $P_1, \dots, P_{\nu}$ be the nodes of $Y$.

Let $\sN$ be a half-even set of nodes on $Y$, and 
assume that there is   a node of $\sN$, say $P_1$, such that   $ H^* : = H - E_1$ yields an embedding of the blow up
$Y^*$ of $Y$ at $P_1$. Then the code $\sK(Y^*)$ is the shortening of $\sK(Y)$ with respect to the complement of $P_1$.

While the coset $\sK'(Y^*) \setminus \sK(Y^*)$ is the projection of the subset of the coset $\sK'(Y) \setminus \sK(Y)$ 
consisting of the words where $P_1$ has coefficient $1$.

In particular, all half-even sets of nodes on $Y$ have cardinality $t+1$,
where $t$ is the cardinality of some half-even set of nodes on a K3 surface $Y^*$ of degree $d-2$.

\end{prop}

\begin{proof}
Since $Y$ and $Y^*$ have the same minimal resolution,  we have to see when
$$ \sum_2^{\nu} a_i E_i + b H'  \in 2 \Pic(\tilde{Y}).$$
But this means 
$$   b E_1 +  \sum_2^{\nu}  a_i E_i + b H  \in 2 \Pic(\tilde{Y}),$$
and the cases $b=0$, respectively $b=1$ yield the desired assertion.

In particular, if   for all half-even sets of nodes $\sN$ on $Y$ 
 there is   a node of $\sN$, say $P_1$, such that   $ H^* : = H - E_1$ yields an embedding of the blow up
$Y^*$ of $Y$ at $P_1$, $\sN$ has 
 cardinality $t+1$,
where $t$ is the cardinality of some half-even set of nodes on a K3 surface $Y^*$ of degree $d-2$.

Thanks to proposition \ref{proj-node}, this holds for a general surface in the component of the nodal Severi variety 
of $\nu$-nodal K3 surfaces of degree $d$ containing $Y$; and we use the fact that $\sK'$ is an invariant of this component. 
\end{proof}

\begin{cor}\label{reduction-t=13}

If the weight $t=15$ occurs for $ d \equiv 6 \ mod (8)$, $ d \geq 14$, then also the weight $t=14$ occurs for $ d -2  \equiv 4 \ mod (8)$,
$d \geq 12$, 
and the weight $t=13$ occurs for $ d -4  \equiv 2 \ mod (8)$, $d \geq 10$.

\end{cor}

We give here a different, code theoretic proof, of a result which was already shown earlier in corollary \ref{extended}
(see also theorem \ref{K8}).

\begin{cor}
Let $Y$ be a $16$-nodal K3 surface of degree $d$ divisible by $8$: then all the numbers $4,8,12$ occur as weights $t$  for
vectors in $\sK'' \setminus \sK$.

The code $\sK''$ is generated by the Kummer code $\sK$ and by the characteristic function $f_{\pi}$ of an affine plane $\pi$.
\end{cor}
\begin{proof}

Projecting from a node yields, if $Y$ is general, a sextic K3 surface with $15$ nodes: hence with all the numbers $t=3,7,11$
as weights $t$  for
vectors in $\sK'' \setminus \sK$, in view of our explicit classification in the proof of  theorem \ref{codes(2,3)}.

Apply now the previous proposition \ref{node-proj}.

For the second assertion, observe that if $\phi$ is affine and not constant, the weight of $f_{\pi} + \phi$ equals
$ 12 - 2a$, where $ a : = | \pi \cap \{ \phi = 1\}| $, and this last cardinality $ a$ can be $4, 2, 0$.

Unicity of the code $\sK''$ follows for instance  from the recipe given in proposition \ref{node-proj}
and the classification for $d=6$.
\end{proof} 

The next result is contained in  Section \ref{append_K8_code}, theorem \ref{K8}:

\begin{cor}
Let $Y$ be a nodal K3 surface of degree $d$ divisible by $8$: then its code $\sK'$ is a shortening of the code of a $16$-nodal
 K3 surface $Z$ of degree $d=8$ (and all these do occur).
\end{cor}

We make now the useful remark, the analog of the observation that deformation and projection 
are operations which commute with each other:

\begin{lemma}\label{commuting}
 The two operations of one-node-shortening and node-projection, both
reducing $\nu$ by $1$, commute with each other. 

\end{lemma}

\begin{proof}

We can assume that the node-shortening corresponds to the first
element of $\sS$, the node-projection to the second element of $\sS$.

For the strict code $\sK$ we get, independently of the order chosen for the composition of the two operations,
the shortening with respect to the complement of the set $\{1,2 \}$.

For $\sK''$ the assertion is clear if there is no vector $ u \in \sK'' \setminus \sK$ with $u_2 = 1$,
or if there is such a vector with $u_2 = 1$ and with $u_1 = 0$. There remains the case where each 
such a vector with $u_2 = 1$ has $u_1 = 1$. This implies that the shortening corresponding to $1$ does not change $\sK$,
and  makes $\sK''$ become equal to $\sK$, so that then the node-projection just affects the shortening of
$\sK$ corresponding to $2$. In the other direction, we get the shortening of
$\sK$ corresponding to $2$, and we get the projection $u'$ of $u$, but $u'_1=1$  holds for each vector in $\sK'' \setminus \sK$,
hence the next shortening corresponding to $1$ reduces $\sK''$  to $\sK$.
\end{proof}

We observe  now, in particular,  that one-node projection shows the existence of the cases $t=1$ and $t=2$.

\medskip

\begin{cor}
For $ d \equiv 0 \ mod (8)$ we get all the possible shortenings of the K8 code,
and for  $ d \equiv 6 \ mod (8)$ we get all the shortenings of the one-node projection of the
K8 code; the latter  are all the codes $\sK'$ satisfying the B-inequality and with
cardinalities $t \in \{ 3,7,11\}$ of the half-even sets of nodes, which are all the possible codes
of K3 surfaces of degree $d=6$.

By further projection we get codes with $t =2$ for  $ d \equiv 4 \ mod (8)$, with $t =1$ for  $ d \equiv 2 \ mod (8)$.

\end{cor}

\begin{ex}
Let $Z$ be a quartic K3 surface with $\nu$ nodes, with  a line $L$ contained in the smooth locus of $Z$.  

Let $D$ be the hyperplane section of $Z$, and consider the linear system $ H : = 2 D + L$ on the resolution $S$ of $Z$,
which contains the  exceptional $(-2)$-curves $E_1, \dots, E_{\nu}$.

Then $H^2= 18$ and $ H \cdot L = 0$, hence, since $|H|$ yields a birational morphism
in view of the exact sequence
$$ 0 \ra \hol_S(2D) \ra \hol_S (H) \ra \hol_L(H) \cong \hol_L \ra 0,$$
 the image of $S$ is a K3 surface $Y$ with $\nu + 1 $ nodes in $\PP^{10}$, 
and with $t=1$.
Observe that we may take here $\nu = 0$,  so $Y$ has only one node.

\end{ex}

\bigskip

Assume now that we have a nodal K3 surface $Y$ with a half-even set $\{ P_1 \}$ of cardinality $t=1$: then, since every other
 half-even set  has cardinality
$t' \in \{1,5,9,13\}$, every (strictly) even set $\sN$ must be disjoint from $P_1$, in particular $\sS \setminus \{ P_1 \}$ has at most
$14$ elements, but we have indeed more:

\begin{prop}\label{t=1}
If $Y$ is a nodal K3 surface $Y$ with a half-even set $\{ P_1 \}$ of cardinality $t=1$, then $\nu \leq 13$, $ k' \leq 3$, 
equality holding if and only if $\nu = 13$.

Moreover,  its code $\sK''$ is the orthogonal direct sum $$\sK \oplus \FF_2 E_1 \subset \FF_2^{\sS \setminus \{ P_1 \}} \oplus \FF_2 E_1.$$

These codes do in fact occur, as double projections from a nodal K3 of degree $d \equiv 6 \ mod(8)$, and with
a codeword in $\sK'$ of weight $4$.

They are shortenings of the one with $\nu = 13, k'=3$.
\end{prop} 

\begin{proof}
$\sS \setminus \{ P_1 \}$ has at most
$14$ elements, and contains the support of $\sK$, hence  $k \leq 3$, hence $k' \leq 4 \Rightarrow \nu \leq 14$.

But if $\nu = 14$, then $\sS \setminus \{ P_1 \}$ has $13$ elements, hence $ k \leq 2$, a contradiction.

Hence $\nu \leq 13$, and if equality holds, then $k=2$, $k'=3$. Otherwise, if $\nu \leq 12$, $k\leq 1$.

The rest is obvious by our previous results.
\end{proof}

\begin{prop}\label{t=2}
If $Y$ is a nodal K3 surface $Y$ with a half-even set $\sB : = \{ P_1, P_2  \}$ of cardinality $t=2$, then $\nu \leq 14$, $ k' \leq 4$, 
equality holding if and only if $\nu = 14$.

These codes do in fact occur, as  projections from a nodal K3 of degree $d \equiv 6 \ mod(8)$, and with
a codeword in $\sK'$ of weight $4$. They are shortenings of the one with $\nu = 14, k'=4$.
\end{prop} 

\begin{proof}
Every even set $\sN$ of cardinality $8$ either contains $\sB$ or is disjoint from it. Since these sets   correspond 
to complements of affine subspaces containing an affine subspace $\sA$, it follows  that $k \leq 3$,
hence $ k ' \leq 4$ and $\nu \leq 14$. 

If $\nu =14, k'=4$, then the union of $\sB$ and $\sA$ is an affine plane.

The rest follows easily as before.
\end{proof}

We have already seen  that  the possible codes of nodal K3 surfaces of degree $d$ are 
not exactly those that we have for $ d=2=4=6=8$,
and to classify them all  it is now possible knowing that  $t=15$ occurs precisely
for $ d \equiv 14 \ (mod \ 16)$,  $t=14$ occurs precisely
for $ d \equiv 12 \ (mod \ 16)$,  and  $t=13$   occurs precisely
for $ d \equiv 10 \ (mod \ 16)$.
 
 \bigskip
 
\subsection{Main theorem on varieties of nodal K3 surfaces}

\begin{theo}\label{mtK3}
The connected components of the 
Nodal Severi variety $\sF_{K3}(d, \nu)$ of nodal K3' s of degree $d$ ($d$ is an even number) and with $\nu$ nodes are in bijection with the isomorphism classes of their extended codes $\sK'$ (equivalently, of the pair of codes $\sK \subset \sK'' \subset \FF_2^{\nu}$).

The weights of the code $\sK$  are either $8$ or $16$, and, if $d$ is not divisible by $4$, the code $\sK$ admits  only the weight $8$.

The weights $t$ of vectors in $\sK'' \setminus \sK$
satisfy the following relation:
$$   2 t  - d \equiv 0 \ ( mod \ 8) .$$

If $ d \equiv 4 \ ( mod \ 8) $ and $ t \in \{ 6,10 \}$, we get for $\sK'$ all the shortenings of the extended Kummer code.

The case where $ d \equiv 4 \ ( mod \ 8) $ and there is some  $t=14$ occurs exactly when $ d \equiv 12 \ (mod \ 16)$, and then we have 
$\nu=13$, $k=2$ and a unique code $\sK'$.

The case where $ d \equiv 4 \ ( mod \ 8) $ and there is some  $t=2$ occurs only for $\nu \leq 14$, and we get all the shortenings of the
code $\sK'$ which occurs for $ \nu=14$, $k'=4$; the latter is obtained  via the projection from a node of a nodal K3 
with $15$ nodes, with degree $d \equiv 6 \ (mod \ 8)$,
and with a codeword in $\sK'$ of weight $4$.

If $ d \equiv 6 \ ( mod \ 8) $,  then  $ t +1$ is divisible by $4$, and we get all the shortenings $\sK'$ of the two codes
that we get for $\nu=15$, namely, the case where $\sK'$ is the strict Kummer code $\sK_{Kum}$,
and the case where there is a weight $4$ vector in $\sK'$: in the latter case  $\sK$ is generated by the characteristic functions
of complements of hyperplanes $H \subset \PP^3_{\FF_2}$, and $\sK''$ is generated by $\sK$ 
and by the characteristic function of a line (a subset $\sA$ consisting of 3 collinear points).

The special case where $ d \equiv 6 \ ( mod \ 8) $ and there is a $t=15$ occurs exactly for  $ d \equiv 14 \ ( mod \ 16)$, and then $\nu=15$, $k'=5$,
 and $\sK'$ is the strict Kummer code.

If $ d \equiv 0 \ ( mod \ 8) $ then  $ t \in \{ 4,8,12 \}$ and we get all the shortenings  $\sK'$ of the K8 code, 
the code for which $\sK''$ is generated by the strict Kummer code and by the characteristic function $f_{\pi}$ of an affine plane $\pi$.

If $ d \equiv 2 \ ( mod \ 8) $ and $ t \in \{ 5,9 \}$, we get all the shortenings  $\sK'$  of the code  
generated by the linear functions
on $\FF_2^4$ and by the quadratic function $x_1 y_1 + x_2 y_2 + 1$. 

The case where $ d \equiv 2 \ ( mod \ 8) $ and there is some  $t=13$ occurs exactly when $ d \equiv 10 \ (mod \ 16)$, and then we have 
$\nu=13$, $k=2$ and a unique code $\sK'$.

The case where $ d \equiv 2 \ ( mod \ 8) $ and there is some  $t=1$ occurs only for $\nu \leq 13$, and we get all the shortenings $\sK'$ of the
one for $ \nu=13$, $k'=3$, which occurs via the projection from a node of a nodal K3  
with $14$ nodes, of  degree $d \equiv 4 \ (mod \ 8)$ and which has some  $t=2$.

\end{theo}

\begin{proof}
We have just classified all the codes $\sK'$ which can occur, and proven the existence of a primitive embedding
of the associated lattices $\sL'$ inside the K3 lattice $\Lam$: either establishing  geometrically
the existence of a nodal K3 surface with the given $\sK'$,  or as an application of Nikulin's Theorem
1.12.2 guaranteeing the existence of such a primitive embedding.

The proof follows now the identical pattern of the proof of Theorem \ref{Nodal-Quartics}, first  using Theorem \ref{unicity}
to show  the 
unicity of the primitive embedding of $\sL'$, and then applying verbatim as in Theorem \ref{Nodal-Quartics}
 the Torelli Theorem for K3 surfaces,
especially using that for each $\sK'$ we may find a nodal K3 surface defined over $\RR$ and with all the nodes real points
(see Theorem \ref{real} and observe that the class of codes $\sK'$ of nodal K3 surfaces defined over $\RR$ and with all the nodes real points
is closed under the operation of smoothing and of projection from one node). 
\end{proof}

An important topological consequence is the following result

\begin{theo}\label{fund-gr-nodalK3}
Let $Y$ be a nodal K3 surface with $\nu$ nodes: then the fundamental group of the smooth part satisfies:
$$\pi_1(Y^*) \cong \ZZ^4 \rtimes \ZZ/2, \ \nu = 16,$$
$$\pi_1(Y^*) \cong \sK \cong ( \ZZ/2)^k , \ \nu \leq 15,$$
in particular $\pi_1(Y^*) = 0, \ \nu \leq 7.$
\end{theo}

\begin{proof}
In the case $\nu = 16$ we have a Kummer surface, hence the stated result is shown as in Corollary \ref{extended}.

The proof of the previous Theorem \ref{mtK3} shows that every nodal K3 surface is obtained from a Kummer surface
via a sequence of operations:

1) smoothing of one node

2) projection from one node

3) equisingular deformation.

3) does not change the fundamental group $\pi_1(Y^*)$, whereas 1) and 2) have the same effect on $\pi_1(Y^*)$,
which we now describe.

Namely, projection from one node $P_i$ replaces a neighbourhood of the node $P_i$ by a neighbourhood of the exceptional 
(-2)-curve $E_i$ lying over it; whereas smoothing of $P_i$ replaces a neighbourhood of the node $P_i$ by
the Milnor fibre of the singularity. By the cited theorem of Brieskorn both have 
the effect of dividing by the subgroup normally generated by the generator $\ga_i$ of the local fundamental group at $P_i$,
which has order at most $2$.

If we start from a Kummer surface, as shown in Corollary \ref{extended}, we divide $\ZZ^4 \rtimes ( \ZZ/2) $ by an element of order
$2$, hence we get as quotient a finite group $ ( \ZZ/2)^k $, hence the fundamental group equals the first homology group.

Proceeding further, since a quotient of an Abelian group is Abelian, we get always $\pi_1(Y^*) \cong \sK \cong H_1(Y^*, \ZZ)$.

Since every code vector has weight $8$ for $\nu \leq 15$, we conclude that $\pi_1(Y^*) $ is trivial for $\nu \leq 7$.
\end{proof}

\begin{rem}
For a more general normal K3 surface $Y$ it was shown in \cite{campana} that $\pi_1(Y^*) $ is either finite or
a finite extension of $\ZZ^4$, extending previous results (as   \cite{cko}).
\end{rem}

\section[The extended Kummer code and its shortenings]{Quadratic functions,  the extended Kummer code  and its shortenings}
\label{append_quadratic_Reed-Muller_code}

Assume that $\FF$ is a perfect field of characteristic $2$, for instance $\FF_2$.

Then every quadratic function $q(x)$ on a vector space $\F^n$ can be written as the sum 
$$ q(x) = T(x) + L(x)^2,$$
where $L(x)$ is a linear function, and $T(x)$ is a triangular quadratic form
$$ T(x) = \sum_{i\leq j} a_{i,j} x_i x_j.$$

If $T(x)$ is non identically zero then there is a coefficient $a_{i,j}  \neq 0$; without loss of generality we may assume
that $a: = a_{1,2}  \neq 0$.

Then we may write 
$$ \frac{1}{a} T(x) = x_1 (x_2 + \sum_{j \geq 3} b_{1,j}  x_j) + T' (x_2, \dots, x_n) ,$$
and setting $z_2 : = (x_2 + \sum_{j \geq 3} b_{1,j}  x_j)$, we have:
$$ \frac{1}{a} T(x) = x_1 z_2 + q' (z_2, \dots, x_n)  = x_1 z_2 + T' (z_2, \dots, x_n)   + L' (z_2, \dots, x_n)^2.  $$
Taking as new variable $w_2: =  a z_2$, we reach the form:
$$ q(x) = x_1 w_2   +T'' (w_2, x_3, \dots, x_n)   + L'' (x)^2.  $$

We can take as new variable $y_1$ the sum of $x_1$ with $\sum_j a''_{2,j} x_j$, so that we reach the normal form
$$ q(x) = x_1 y_1   +T'' (x_3, \dots, x_n)   + L'' (x)^2.  $$

By induction we find new coordinates such that the quadratic form $q(x)$ takes the quasi-normal form
 $$ q(x) = x_1 y_1  + x_2 y_2  + \dots +  x_k y_k   + L ^2,
  $$
  where $L$ is a linear form.
  
  If the linear form $L$ is linearly independent of $x_1,  y_1, x_2 , y_2  ,  \dots , x_k, y_k,$
  then we reach the normal form
  $$ 0) \   q(x) = x_1 y_1  + x_2 y_2  + \dots +  x_k y_k   + x_{k+1} ^2.
  $$
  Otherwise we have the semi-normal form 
  $$   q(x) = x_1 y_1  + x_2 y_2  + \dots +  x_k y_k   + L_1(x_1, \dots, x_k) ^2 + L_2(y_1, \dots, y_k)^2 .
  $$
  Since we may view the spaces with coordinates $x : = (x_1, \dots, x_k) $ and $y : = (y_1, \dots, y_k)$ as dual spaces,
  we can change basis in the $x$ space  and then take the dual basis in the $y$-space, so that either 
  \begin{enumerate}
  \item
  $$   q(x) = x_1 y_1  + x_2 y_2  + \dots +  x_k y_k   ,  $$ or 
\item
  $$   q(x) = x_1 y_1  + x_2 y_2  + \dots +  x_k y_k   +  x_1^2 + y_1^2 ,
  $$
  or
    \item
 $$   q(x) = x_1 y_1  + x_2 y_2  + \dots +  x_k y_k   +  x_1^2 ,   $$
or
  \item
$$   q(x) = x_1 y_1  + x_2 y_2  + \dots +  x_k y_k   +  x_1^2 + y_2^2 .
  $$
  \end{enumerate}
  One sees however that cases iii) and iv) can be brought to the normal form i).
  
  Hence we have only the normal forms $0), i), ii) $, which are respectively called {\bf parabolic, hyperbolic, elliptic}.
  
  \subsection{The case $\FF = \FF_2$} Then every polynomial function of degree $2$ is the sum of a constant
  and of a quadratic function (a linear function is also equal to its square).
  
  Then each  polynomial function of degree $2$  has either the affine normal form

 $$ 0) \   q(x) = x_1 y_1  + x_2 y_2  + \dots +  x_k y_k   + x_{k+1} ^2,
  $$
  
  or 
$$ i) \   q(x) = x_1 y_1  + x_2 y_2  + \dots +  x_k y_k   ,
  $$
  or 
  $$ 
  ii) \   q(x) = x_1 y_1  + x_2 y_2  + \dots +  x_k y_k   + 1,
  $$
  since $   x_1 y_1  +  x_1^2 + y_1^2  = (x_1 + 1) + (y_1 + 1) + 1.
  $ 
  
 If $ n = 2k$, the respective functions $q(x)$ in cases $i),  \ ii)$ differ since in the even case i) the number of zeros equals $2^{k-1} (2^k + 1)$,
 in the second case it equals $2^{k-1} (2^k - 1)$.
 
 If we consider  the vector space generated by the space $\sK$ of affine functions and by a quadratic function $q \notin \sK$,
 we see therefore  that we may take as normal form for $q$ the form i), here $k$ is the rank of $q$, this is the minimal number of
 products of linear forms whose sum is $q$.
 
 For example, if $n=4 $ the support of quadratic form $x_1 y_1$ has cardinality $4$, whereas the support of the 
  quadratic form $x_1 y_1 + x_2 y_2 $ has cardinality $6$.
 
 \begin{defin}
 The first order Reed-Muller code $\sK_1$ over $\FF_2$  is the vector space of affine functions on an affine space $\FF_2^n$.
 
 The second order Reed-Muller code $\sK_2$ over $\FF_2$  is the vector space generated by the space of affine functions on an affine space $\FF_2^n$, and by the quadratic forms. The extended Kummer code is the code generated by the first order  Reed-Muller code
 and by a nondegenerate quadratic form $q(x)$ (that is, of type i) for $ n = 2k$, and of type 0) for $ n = 2k+1$).
 
 \end{defin}
 
 We see immediately from the above discussion that, for $n=2k$,  the weights of $\sK_1$ are just $2^{2k}$, attained once, 
 and the others are just $2^{2k-1}$. Whereas the weights of vectors in  $\sK_2 \setminus \sK_1$ are $2^{k-1} (2^k + 1)$,
 $2^{k-1} (2^k - 1)$.

\bigskip

\subsection{Shortenings of the extended Kummer code}

We want now to list the equivalence classes of shortenings of the (projection of the) 
extended Kummer code. In a subsequent subsection we shall determine  their incidence hierarchy, obtained via the procedure of shortening.

Before we do this, it is good to consider that, for $n=4$, we have two 2-dimensional vector spaces $W_1, W_2$ and 
$\sK''_{Kum}$ is generated by $\sK$, the vector space of affine functions on $ W_1 \times W_2 $, and by a non-degenerate bilinear function 
$\be : W_1 \times W_2 \ra \FF_2$.

In particular, the group $Aff(W_1) \times Aff(W_2)$ acts on $ W =  W_1 \times W_2$ sending $\sK''_{Kum}$ to an equivalent code.
If we want the code $\sK''_{Kum}$ to be left invariant, the linear parts must be of the form $(A^t , A^{-1})$, and we shall speak of special orthogonal transformations: this implies  for instance that if we exchange $e_1, e_2$ in $W_1$, the same must be done in $W_2$, while if we cyclically permute $e_1 \mapsto e_2  \mapsto e_3$ in $W_1$, in $W_2$ we must have the inverse permutation. 

 Moreover, we can also 
exchange the roles of $W_1, W_2$, and we can also use the symmetries
$$ (x,y) \mapsto (x+y, y), (x,y) \mapsto (x, x+  y), $$
$$ (x_1, x_2,y_1,y_2) \mapsto (x_1 + y_1,x_2,y_1,y_2),
 (x, y) \mapsto (x_1,x_2,x_1 + y_1,y_2),$$
 which leave the subspace $\sK''$  invariant.

\medskip

1) {\bf  One point shortening}: by the use of translations  we may assume that we are taking $ W \setminus \{0\}$
as subset of $W$:
then we get the space generated by $\be$ and the linear functions, of dimension $5$, and with support the whole  $ W \setminus \{0\}$. This case will be labelled (00).

2) {\bf Two points shortening.}

The first case occurs  if they have one  coordinate equal, via translations, we may assume
(possibly exchanging the roles of $W_1, W_2$)  that this coordinate
is the second, and it is equal to $0$.

Using then a translation on $W_1$ and a special orthogonal transformations, we achieve that the two points are   
$0 = (0,0)$ and $(e_1, 0)$.

Otherwise, we may assume that the first point is $0 = (0,0)$ and that the first coordinate of the second equals $e_1$,
while its second coordinate is non zero.
Applying  the symmetry
 $ (x,y) \mapsto (x,y) \mapsto (x, x+  y),$ we exclude the case of the point $(e_1, e_1)$, and reduce the case of 
 $(e_1, e_3)$ to  $(e_1, e_2)$.
 
In turn, we can reduce  the second case where the two points are $0 = (0,0)$ and $(e_1, e_2)$
to the first one, via a transformation
$$ \Phi : (x_1, x_2, y_1 , y_2 ) \mapsto  (x_1, x_2, y_1 +x_2, y_2+x_1 ).$$
In fact, $$\Phi (0) = 0 , \Phi ( e_1, e_2) = (e_1, 0).$$

This transformation leaves $\be = x_1 y_1 + x_2 y_2$ invariant, because 
$$ x_1 (y_1 +x_2) + x_2 (y_2+x_1)= x_1 y_1 + x_2 y_2.$$

 The condition of vanishing on the points $0 = (0,0)$ and $(e_1, 0)$ yields the  4-dimensional subspace:

$$ \sK' = \langle x_2, y_1, y_2, x_1 y_1 + x_2 y_2\rangle.$$
 This case will be labelled (i).
 
 3) {\bf Three points shortenings}: first of all we may then assume 
 that we delete both points $0 = (0,0)$ and $(e_1, 0)$, and a third point, which, using the 
 allowed symmetries,   can be 
one of the following points:
 
 (1)  $ (e_2,0)$ in the case where all the three points have one coordinate equal: this case  yields the  subspace generated by $  y_1, y_2, x_1 y_1 + x_2 y_2$.
 
 (2)  $ (0,e_1)$ in the case where  the coordinates of the three points are just two (in the set $ \{0,e_1\}$):
 this case  yields the  subspace generated by $  x_2,  y_2, x_1 y_1 + x_2 y_2$.
 
 While in the case where there are three coordinates for the three points, we may assume that these
 are in the set $ \{0,e_1, e_2 \}$, and we have the further options:
 
 (3)  $ (0,e_2)$,  yielding the  subspace generated by $  x_2, y_1,  x_1 y_1 + x_2 y_2$,
 
 (4)  $(e_1, e_2)$, yielding the  subspace generated by $  x_2, y_1,  x_1 y_1 + x_2 y_2$,

 (5) $(e_2, e_1)$, yielding the  subspace generated by $  x_1, y_2,  x_1 y_1 + x_2 y_2$.
 
 Indeed the case where there are four coordinates leads to $(e_2, e_3)$ which is equivalent to $(e_2, e_1)$,
 and $(e_3, e_2)$ which is equivalent to $(e_1, e_2)$.
 
 Now, cases (1), (3), (4), (5) are obviously equivalent, since in all these cases 
 we take $\be$ and  two elements of  the standard basis, 
  occurring  in different monomials of $\be$. 
 
 Cases (1) and  (2) are distinct, since in case (1) 
 the vector subspace $U$ of $W$, of dimension 2,  where the linear forms in $\sK$ vanish,
 is contained in the zero set of $\sK''$, hence the support of $\sK''$ is the complement of the subspace $U$
 and   the effective length of $\sK''$  is $n= 12$; while in case (2) 
 the support of $\sK''$ is the complement of the set of three points, i.e.  $U \setminus \{(e_1, e_1)\}$ and the effective length of $\sK''$  is $n=13$.
 
 Hence we have two non equivalent shortenings for three points, cases (1) and (2).

 4) {\bf Four points shortenings}: we can assume that we are deleting either three points of case (1) and a fourth
 point, or three points of case (2) and a fourth point.
 
 (I) We delete the subspace $U : = \{ y=0\}$ of case (1): then we get the same code $\sK''$, of dimension $3$, 
 appearing in (1).
 
 In all other cases we get a subspace $\sK''$ of dimension $2$. 
 
 (II) We delete the subspace $U : = \{ x_2 = y_2 = 0 \}$ of case (2): then $\sK'' = \sK$, since both subspaces
 have dimension $2$.
 
 (III) We are left  with the following restrictions: $ dim (\sK) = 1$, $ dim (\sK'') = 2$, 
 which we have shown to give two combinatorial
possibilities $a=2, b=4$, or $a=4, b=2$.
The corresponding codes are not isomorphic, since for $a=2$ the weight $10$ occurs, 
but  it does not occur for $a=4$.

 The two codes correspond to the  subspaces respectively
 generated by 
 
 (III-1) $x_2 + y_2, x_1 y_1 + x_2 y_2$, 
 
 (III-2) $y_2, x_1 y_1 + x_2 y_2$.
 
 In the first case (III-1) the support of $\sK''$ is the complement of $ \{ x_2 + y_2= x_1 y_1 + x_2 y_2=0 \}$
 and has cardinality $12$,
 in the second  case (III-2) the support of $\sK''$ is the complement of $ \{ y_2=  x_1 y_1= 0\}$
 and has cardinality $10$.
 
 {\bf Shortenings by at least five points.}
 
 Clearly case (III-2) leads to five and six point shortenings which do not change $\sK''$.
 
 We need therefore only to consider shortenings which reduce the dimension of $\sK''$ to $1$.
 
 There are some obvious cases, 
 
 a) $\sK = \sK''$, of dimension $1$, effective length $8$, hence $\nu = 8,9,10,11$;
 
 b) $\sK =0$, $ dim (\sK'') = 1$, generated by $x_1 y_1 + x_2 y_2$, effective length $6$, hence $\nu = 6,7,8,9,10,11$;
 
 c)  $\sK =0$, $ dim (\sK'') = 1$, generated by $x_1 y_1 + x_2 y_2 + 1$, effective length $10$, hence $\nu = 10,11$.
 
 Altogether we  have 23 nontrivial cases, for  11 distinct nontrivial codes (again here $k' : = dim ( \sK' ), k : = dim ( \sK )$,
 while $n'$ is the effective length of $\sK''$), plus $11$ trivial cases. Observe that the B-inequality $ n \leq \nu \leq 10 + k'  ( \leq 11 + k)$
 is always verified, since we have shortenings of the extended Kummer code, which satisfies the inequality.
 
 \subsection{List of the shortenings of the extended Kummer code.}
 
\[\renewcommand{\arraystretch}{1.1}
\begin{array}{c|c|cc|cc}
& \sK'' & k' & k & \nu & n \\
\hline
(0) & \langle   1,  x_1, x_2,  y_1, y_2, x_1 y_1 + x_2 y_2 \rangle  & \ 6 & 5 & 16 & 16 \\
(00) & \langle    x_1, x_2,  y_1, y_2, x_1 y_1 + x_2 y_2 \rangle  & 5 & 4 & 15 & 15 \\
(i) & \langle    x_2,  y_1, y_2, x_1 y_1 + x_2 y_2 \rangle & 4 & 3 & 14 & 14 \\
(1) & \langle     y_1, y_2, x_1 y_1 + x_2 y_2 \rangle & 3 & 2 & 12,13 & 12 \\
(2) & \langle    x_2,  y_2, x_1 y_1 + x_2 y_2 \rangle  & 3 & 2 & 13 & 13 \\
(II) & \langle    x_2,  y_2 \rangle & 2 & 2 & 12 & 12 \\
(III-1) & \langle    x_2 +  y_2, x_1 y_1 + x_2 y_2 \rangle & 2 & 1& 12 & 12 \\
(III-2) & \langle     y_2, x_1 y_1 + x_2 y_2 \rangle & 2 & 1 & 10,11,12 & 10 \\
(a) & \langle     y_2 \rangle  & 1 & 1 & 8,9,10,11 & 8 \\ 
(b) & \langle     x_1 y_1 + x_2 y_2 \rangle & 1 & 0 & 6,7,8,9,10,11 & 6 \\
(c) & \langle     x_1 y_1 + x_2 y_2 + 1 \rangle & 1 & 0 & 10,11 & 10 \\
(d) & \langle    0 \rangle & 0 & 0 & 0, 1, \dots, 10 & 0
\end{array}\]

\begin{cor}\label{deg=4}
The Nodal Severi variety $\sF (4, \nu)$ has exactly one irreducible component for $ \nu = 0,1,2,3,4,5, 14,15,16$,
exactly two irreducible components for $ \nu = 6,7,13$, exactly three irreducible components for $ \nu = 8,9, $
exactly four  irreducible components for $ \nu = 11,12, $ exactly five  irreducible components for $ \nu = 10. $
Hence  its stratification is made of 34 strata, where each stratum is in the boundary of another if and only if the corresponding 
extended code admits the other as a shortening.
\end{cor}

\begin{remark}\label{18}
 Following up the work of Rohn \cite{rohn86,rohn87}, see also \cite{jessop1916quartic}, {Endra\ss} 
  determined all the nontrivial extended codes of 
  nodal quartic surfaces in $\PP^3$ with $\nu$ nodes in \cite[Theorem 3.3]{endrass2} by applying 
  elementary coding theory.  He also showed that these   codes were attained by some nodal quartic surfaces. 
  
  Our classification above confirms the classification by Endra\ss.
    
  Theorems \ref{Nodal-Quartics} and corollary \ref{deg=4} are a sharpening of the results of Rohn and Endra\ss, because they show that $\sK'$
  determines an irreducible component of $\sF(4, \nu)$; moreover, our statement is more conceptual  since it  shows that these codes
  are exactly all the shortenings of the extended Kummer code, and puts an incidence hierarchy among them.

Observe that the list also contemplates the 5 cases where  $\sK=\sK'$. 

 (*)  Removing those cases, there are          
  exactly $18$ cases  where $$ 0 \neq \sK \neq \sK'.$$
    
\end{remark}

\subsection{The incidence relation for the shortenings of the extended Kummer code.}

We shall now determine when the codes in the above list are a shortening of each other:
for this purpose it will be sufficient to see when one is obtained form the other by a
one point shortening, which need not diminish the effective length $n$ of $\sK''$,
but certainly lowers $\nu$ by $1$.

In the course of producing the list we have shown that certain one point shortenings occur,
now we have also to determine when they cannot occur.

For a one point shortening to occur, we need $\nu$ to drop to $\nu - 1$; moreover, this certainly happens
in the trivial cases where  $n < \nu$ and the code remains the same.
Moreover, a one point shortening occurs when we take a subspace of codimension $1$ in $\sK''$.

Since $ (0) \mapsto (00)$, $ (00) \mapsto (i)$, $(i) \mapsto (1), (2)$ we are left with inspecting the
cases where $\nu \leq 13$.

Let us start w         ith $\nu=13$.

\begin{lemma}
The code (1) admits only a one point shortening to (1) (with $\nu = 12$) or to (III-2).
\end{lemma} 
\begin{proof}
The first assertion is clear, since (1) has $\nu=13$, $n=12$, and this one point shortening 
corresponds to take the point of $U \setminus \supp (\sK'')$.

Else, we take a point $(\xi, \eta)$ outside of $U$, hence with $\eta \neq 0$.

Using a special orthogonal transformation we may assume that $\eta = e_1$.

We may further assume that $\xi=0$ since translations in the $\xi$-plane leave the space $U$ invariant.

Now,  the shortening corresponding to $(0,0, 1,0)$ is exactly (III-2).
 \end{proof} 

\begin{lemma}
The code (2) does not admit a one point shortening to (1), but it does for   (III-2), (II) (III-1).
\end{lemma}
\begin{proof}
The first assertion is clear, since both codes have dimension $3$, and $n = \nu = 13$ for (2),
which implies that any shortening has strictly smaller dimension.

The other three assertions follow since, by inspection of the list, we find that
 in all three cases we are taking  a codimension $1$ subspace of $\sK''$.
 
 More precisely, for (II) we take the  point of $U \setminus \supp (\sK'')$,
 for $(0,1,0,0) $ we get  (III-2), for $(1,1,1,1) $ we get  (III-1).
\end{proof} 

Let us pass now to $\nu=12$.

We make the obvious remark that (1) can only be 1-shortened to (III-2), since it has $k'=3$, and (a), (b), (c) have $k'=1$
(and inspection of the respective bases shows that this 1-shortening ).

Whereas (II) has $n= \nu$, hence it can only be 1-shortened to $k=k'=1$, that is, only to case (a),
and one sees immediately that this occurs.

(III-2) has $\de: = \nu - n=2$, hence it can only be 1-shortened lowering $k'$ by $1$ and keeping $\de$ not smaller,
hence it cannot be 1-shortened to (c); we see that it is 1-shortened to (a), (b).

(III-1) cannot be 1-shortened to (III-2) because $k'$ cannot increase. We see 
now that it is 1-shortened to (a), (b), (c), taking the respective points
$(1,0,1,0)$, $(0,1,0,0)$, $(1,1,1,0)$: for the last we are left 
with  the quadratic form 
$$x_1 y_1 + x_2 y_2 + x_2 + y_2 = x_1 y_1  + (x_2 + 1) (y_2 +1) + 1.$$

After this, there are only trivial 1-shortenings (that is, preserving the code $\sK''$), or shortenings to the zero code.

\section[K3's of degree $d \equiv 0 \ (8)$: the K8 code and its shortenings]{The K8 code, Kummer code of degree $8$,  its shortenings and codes of K3 surfaces of degree divisible by $8$}
\label{append_K8_code}

We shall first describe  here  the  possible codes of K3 surfaces of degree divisible by $8$.

In this situation we have a pair of codes $\sK \subset \sK'' \subset \FF_2^{\nu}$,
with $\nu \leq 16$, such that
\begin{enumerate}
\item
 the weights of vectors $v \in \sK$ are divisible by $8$
 \item
 the weights of vectors $u \in \sK'' \setminus \sK$ belong to to the set $\{4,8,12\}$
 \item
 setting $ k ' : = dim (\sK''), \ k : = dim (\sK)$, we have $ k' \leq k + 1$, and
 \item
 $ \nu \leq 10 + k'$.
\end{enumerate}

Let us  recall (Propositions \ref{quartic-codes} and \ref{kummershort}) that property i) implies that $ k \leq 5 \Rightarrow k' \leq 6$, and $\sK$ is a shortening of the
Kummer code $\sK_{Kum}$. 

We have seen (Corollary \ref{extended}) that the code of a Kummer surface with a polarization of degree divisible by $8$
leads to a well determined code, that we shall now call  the K8 code $\sK''_8$. 

\begin{defin}
Let $V$ be an affine space of dimension $4$ over $\FF_2$, and let $\pi$ be a $2$-dimensional affine subspace of $V$.

Then we define the {\bf K8 code } as the code $\sK''_8$ generated by the Kummer code $\sK_{Kum}$, consisting of the affine functions
$\phi : V \ra \FF_2$, and by the characteristic function $f_{\pi}$ of the plane $\pi$.

Hence a basis of $\sK''_8$ is $ 1, x_1, x_2, x_3, x_4, f_{\pi}$.
\end{defin}

\begin{rem}
Consider a function $u \in \sK''_8 \setminus \sK$, where $\sK''_8$ is the K8-code.

Then, setting  $ u : = f_{\pi} + \phi$,  the weight of $u$ is $4$ for $ \phi = 0$, $12$ for $\phi = 1$,
else it is equal to
$$ w(u) : = 4 + 8 - 2  \cdot | \pi \cap \{ x | \phi (x) = 1\} | = 4 +  2 (  |  \{ x \in \pi | \phi (x) = 0\} |).$$
Hence,   defining $W : = \{ x | \phi(x) = 0\}$, we have:
\begin{itemize}
\item
$ w(u) = 12$ if $ \pi \subset W $,
\item
$w(u) = 4$ if $\pi \cap W = \emptyset$, else 
\item
$ w(u) = 8$ if $ \emptyset \neq( \pi \cap W) \neq \pi.$
\end{itemize}

according to the cardinalities $4,0,2$ for $ \pi \cap W$.

Hence all the shortenings of the K8 code satisfy properties i) - iv).
\end{rem}

We show now the converse,  that the codes satisfying properties i) - iv) are precisely the shortenings of 
the K8 code.

To simplify our discussion, let us split  the   result into three parts.

\begin{theo}\label{K8}
Let $\sK \subset \sK''$ be a pair of codes satisfying properties i)-iv) above.

Then

I)  for $ k' = 6$, we have that $\sK''$ is the K8 code,

II) for $k = k'$ we have $k' \leq 2$ and we have a shortening of the K8 code,

III)  also for  $ k' = k + 1 \leq 5$ we have a shortening of the K8 code.

\end{theo}

\begin{proof}
I)  we claim that there must be $u \in \sK'' \setminus \sK$ with weight $4$. 

In fact, $k= 5$, $k'=6$, hence  there   $\exists u \in \sK'' \setminus \sK$. If its weight is 12, replace $u$ by $u + 1$. If  instead the weight of all vectors in $\sK''$ is divisible by  $8$,
 then by Proposition \ref{quartic-codes} $\sK'' = \sK$,  contradicting  $k'=6$.

Such a function $u$ is the characteristic function of a set $\sA$ of cardinality $4$ such that each affine subspace intersects
$\sA$ in $0,2,4$ points. If the four points of $\sA$ were affinely independent, then there would be  an affine function vanishing on three of the points, and not on the fourth, a contradiction.
Hence the points of $\sA$  are affinely dependent, hence $\sA$ is a $2$-dimensional affine subspace $\pi$.

\smallskip

II) We have $k \leq 5$ and $ k' = k$, hence $k' \leq 5$ and, by iv),  we have $\nu \leq 15$. 
We take then the set $\sS$ of cardinality $\nu$ as a subset of $V$ minus one point, which we take as  the origin
$0$ of a vector space structure. 
Then $\sK''$ is   a subspace of the space of linear forms, hence $k' \leq 4$, and then $\nu \leq 10 + k' \leq 14$.
Hence the complement of $\sS$ contains at least  two points, say $0, v$; but then since  the functions in $\sK$ vanish in $0,v$
we have $ k \leq 3$, hence $k' \leq 3$, and then $\nu \leq 10 + k' \leq 13$.
Hence the complement of $\sS$ contains three points, hence $k \leq 2$, so that  $k' \leq 2$, and then $\nu \leq 10 + k' \leq 12$.

For $k' = 2$, the complement of $\sS$ contains $4$ points, so that these points form a $2$-dimensional affine subspace $U$.
Choose an affine plane $\pi$ such that  $\pi$ and $U$ intersect precisely at one point, say $0$; hence $U$ is the vector subspace spanned by $e_1, e_2$,
$\pi$  is the vector subspace spanned by $e_3, e_4$. $ \sK $ is generated by the linear forms $x_3, x_4$,
and, if we define $\sK''$ to be the shortening of the K8 code associated to $V \setminus U$, then the elements of $\sK''  \setminus \sK$ are functions of the form 
$$u = f_{\pi}  + \phi =  f_{\pi} + a + \sum_i b_i x_i,$$ 
where the condition of vanishing on $U$ implies $a=1$ (vanishing on $0$), 
and moreover we must have $ 1 + b_1 = 1 + b_2 = 1 + b_1 + b_2 = 0$,
a contradiction. Hence $\sK'' = \sK$, and we have the shortening of the K8 code
associated to the complement of the set $U$.

If we take the further shortening associated to the complement of the set $ U \cup \{e_3\}$ we get $k=k'=1$,
and  $\sK''$ is generated by the linear form $ x_4$.

{\bf The respective codes belonging to this class shall be respectively denoted $II_8-2$ and $II_8-1$.}

\smallskip

{\bf III)} Consider now the case $\sK \neq \sK''$, with $k' \leq 5$, hence $\nu \leq 10 + k' \leq 15$.

Hence  $\sK''$ is generated by $\sK$ and by a function $f$, 
the characteristic function of a set $\sA$. 

We have $ k \in \{ 0, \dots, 4\}$, and $\sK$ 
is isomorphic to  the space of linear forms vanishing on a linear subspace $U$ of dimension
equal to $ 4 - k$. In particular the support of $\sK$ consists of $ V \setminus U$, of cardinality $ 16 - 2^{4-k}$.

{\bf Subcase III-i)} We assume here  that $\sK''$ does not admit the weights $4,12$, that is, $a_4= a_{12} = 0$.

Then $\sK''$ is a shortening of the strict Kummer code, and if
$k'=5$ we get the whole strict Kummer code, hence $\nu = 16$, a contradiction.

 Hence $k' \leq 4$,   and  $\sK''$
is a space of linear forms vanishing on a subspace $U'$ of dimension $ 4 - k'$. We  have on one hand
$\nu  \leq 10 + k' $, on the other $\nu \geq  16 - 2^{4-k'}$: hence $$6  \leq  k' + 2^{4-k'} \Rightarrow k' = 1,2 . $$

 We get exactly two codes;

 for $k'=1$, $\sK''$ is generated by a function $f_{\pi} + x_1$, with $x_1$ attaining both values $0,1$ on $\pi$,
 
 for  $k'=2$, $\sK''$  is generated by a linear function $x_4$ and by $f_{\pi} + x_1$
 as above, and such that  also $x_1 + x_4$ attains both values $0,1$ on $\pi$.
 
 {\bf The respective codes shall be denoted $c_8$) and $III_8-0$).}
 
 {\bf Subcase III-ii-0)} We assume here  that  $a_4 \neq 0,  a_{12} = 0$ and the weight of elements of
 $\sK'' \setminus \sK$ is always $4$.
 
 If $k >0$, take a nonzero vector $v \in \sK$, and $u \in \sK'' \setminus \sK$: then the support of $u$ is contained
 in the one of $v$; hence $ k \leq 1$, and we have exactly two codes, for $k=0$ we get the span of $f_{\pi}$,
 for $k=1$ the span of $f_{\pi}, x_4$ where $\pi \cap U \neq \emptyset$.
 
  {\bf The respective codes shall be denoted $b_8$) and $III_8-2$).}
  
  \smallskip

  Observe now that  we cannot have $a_{12} \geq 2$; because, if $u, u'$ have support of 
   cardinality $12$, then the intersection of the supports  has cardinality $8$, hence $\nu \geq 16$, a contradiction.

   {\bf Subcase III-iii)} we assume here  that  $ a_{12} \neq  0 \Leftrightarrow a_{12}=1$.
   
Since $\nu \geq 12$, we have $k' \geq 2$, hence we may choose vectors $v_1, v_2$ with respective weights
$w(v_1) =8, w(v_2) = 12$. Letting $ a = | \supp(v_1) \cap \supp(v_2) |$, looking at the  weight of $v_1 + v_2$,
$$ w (v_1 + v_2) = (8-a) + (12 - a) = 2 (10 - a) \in \{4,8\} \Leftrightarrow a \in \{8, 6\},$$ 
we see that it equals $4$ if and only if the support of $v_1$ is contained in the one of $v_2$,
and that it equals $8$ if $a=6$.

{\bf The case  $  a_{12}=1, a_4 =0$ is not possible.}

In fact, since $a=6$, $ n \geq 14$, hence $k' \geq 4$, $ k \geq 3$.

We have an injective  projection of $\sK$ into $\FF_2^ { \sS \setminus \sA} = \FF_2^ { \nu - 12} $,
hence $\nu -12 \geq k \Rightarrow \nu \geq 15  \Rightarrow k+1 \geq 5 \Rightarrow \nu \geq 16$
and we have reached a contradiction.

\smallskip

{\bf Subcase III-iii-4)} we assume here  that  $  a_{12}=1, a_4 \neq 0$.

Since  $a_4 \neq 0$, then  we may assume $ w (v_1 + v_2) = 4$, hence the support of $v_1 + v_2$
and those of all other weight $4$ vectors are contained in $\sA$, the support of $v_2$.

Hence the subcode $\sK'''$ generated by $v_2$ and by the weight $4$ vectors is supported in $\sA$,
hence it has dimension at most $3$ ($\sK''' \cap \sK$ has dimension at most $2$).

If $\sK''' = \sK'' $  we have exactly two codes, according to 
  $k=1,2$: in fact, for $k=2$ we have a partition of 
the set $\sA$ into three disjoint sets of $4$ elements, each one being the support of a weight $4$ vector.

{\bf These correspond to the cases (4-3-i) and $III_8-1)$ in the list }(which contains the description showing
that they are shortenings of the K8 code).

There remains to be treated the case where $a_{12}=1$, $a_4 \neq 0$, and there is a vector in $\sK'' \setminus \sK$
with weight $8$, in particular $ k' \geq 3$.

For $ k=2,3,4,$ we have $| \sS| = :  \nu \leq 11 + k$.

Take our  weight $12$ vector, its support is a subset $\sA$ with $12$ elements
of $\sS \supset V \setminus U$.

For $k=4$, then $ U = \{0\}$ and $\sS = V^* : = V  \setminus \{0\}$. The subset $\sA$ has the property that, for all linear forms
$x_W$ with zero set a hyperplane $W$,  $f_{\sA} + x_W$ has weight $\in \{4,8\}$. 

Setting $\sB : = \sS \setminus \sA$,
$\sB$ consists of three non zero vectors, and the weight of $f_{\sA} + x_W$ equals $10 -2b$,
if we set $ b : = | \sB \cap W|$. Hence it is only possible $ b = 1,3$, but not $b=2$.

Hence every linear form vanishing on two points of $\sB$ vanishes on the third, hence $\sB$
equals $\pi \setminus \{0\}$, where $\pi$ is a plane containing the origin $0$.

Similarly, for $ k = 3$, then $U$ has $2$ elements, and $\sS$ at most $14$,
hence $\sS = V \setminus U$. 

Again the  complement $\sB$ of $\sA$  has $2$ elements, and, for all linear forms vanishing on $U$,
$f_{\sA} + x_W$ has weight $ 8 -2b$, hence we can only have $b=0, b=2$. Again $\sB$
equals $\pi \setminus U$, where $\pi$ is a plane containing $U$.

For $k=2$, the two sets $\sS \supset ( V \setminus U)$ have respectively $ \leq 13, 12$ elements, hence $\sB$ has cardinality
 $h \leq 1$, and, for all linear forms vanishing on $U$,
$f_{\sA} + x_W$ has weight $ h + 4 -2b \geq 4$, so the only possibility is $h=0$, which means that $\sS =  V \setminus U$
and $\nu = 12$.

For the three cases above, we have that $\sK''$ is spanned by $ 1 + f_{\pi}$ and by the linear forms vanishing on $U$,
a linear subspace of dimension $ 4 - k$, such that $U \subset \pi$. 
In all these cases, we have that $\sK''$ is the shortening associated to the complement of the set $\{0, e_1, \dots, e_{4-k}\}$
formed by the origin and a basis of $U$.

{\bf These cases are respectively labelled (12-1), (12-2), (12-3) in our list.}

\bigskip

   {\bf Subcase III-ii-1)} we assume here  that  $a_4 \neq 0,  a_{12} = 0$ and the weights of elements of
 $\sK'' \setminus \sK$ are not always $4$, in particular $k' \geq 2$.
 
  Observe that
 if the weight of an element $v' \in \sK'' \setminus \sK$ equals $8$, then for each nonzero vector $v \in \sK$
 the weight of $v' + v$  equals $4$, respectively $8$ if the supports of $v,v'$ have intersections
 of respective cardinalities $6,4$.

In other words,   we consider
  the case where $\sA$ has $4$ elements but there is no vector of weight $12$, and we let $W$ be the hyperplanes which are zero sets of nonzero vectors in $\sK$: these are the hyperplanes containing the zero set $U$ of $\sK$.

$\sA$ is a set of cardinality $4$, $\sA \subset \sS$, and,  for all linear forms $x_W$ vanishing on $U$, 
$f_{\sA} + x_W$ has weight (setting  $ a : = | \sA \cap W|$) $ a + (8 - (4-a)) = 4 + 2a$,
hence $a$ can only equal $0,2$.

$k = 4$: then   there is certainly a linear form vanishing on at least $3$ of the points of $\sA$,
a contradiction.

$k=3$: then projection with centre $U$ maps the points of $\sA$ to a a projective basis $\sA'$ in
the plane $\PP^2 _{\FF_2}$ with coordinates $x_2, x_3 ,x_4$.
We observe that, even if for each point $P' \in \sA'$ there are two different choices for a point in $ U * P' \setminus U$,
these choices lead to isomorphic pairs of codes.

We can therefore assume that $\sA =  \{ e_2, e_3, e_4, e_2 + e_3 + e_4\}$, which is an affine plane $\pi$.
Hence we have the code spanned by  $ f_{\pi}$ and the three linear forms $x_2, x_3 ,x_4$, which is the shortening
relative to the complement of the set $\{ 0, e_1\}$.

For $k=2$, there are two cases, according to $\sS = ( V \setminus U)$ or $\sS \neq ( V \setminus U)$,
where we may assume wlog  that $\sS = ( V \setminus U) \cup \{0\}$. 

The hyperplanes $W \supset U$ are a pencil, so they are three ($\{x_3 =0 \}$, $\{x_4 =0 \}$, $\{x_3 + x_4=0 \}$) 
and if  $\sS = ( V \setminus U)$ each of them contains none or two of the points of $\sA$: hence in this case we may assume
that $\sA \subset \{x_3 =0 \} \cup \{x_3 + x_4 =0 \}$. Take $\pi$ equal to the plane $\pi : = \{0, e_1, e_3, e_1 + e_3\}$
and consider the function $ f_{\pi} + 1 + x_2$, vanishing on $U$ and on $\pi$: its support is then the set 
$ \sA : = \{x_2 =0\} \setminus  \pi $ which has cardinality $4$ since $x_2 =0$ on $\pi$.
We get the shortening associated to the complement of the set $\{ 0, e_1, e_2\}$.

While, if $\sS = ( V \setminus U) \cup \{0\}$ there is also the possibility that $0 \in \sA$ and each hyperplane in the pencil contains exactly one point of $\sA$. This case is realized by setting $\sA$ equal to the plane $\pi : = \{0, e_3, e_4, e_3 + e_4\}$,
and we have the shortening associated to the complement of $ U^* =  \{ e_1, e_2, e_1 + e_2\}$.

For $k=1$, since we shall have $a_4=1, a_8 =2$, it follows that $a=2, n=10$ (but $\nu \leq 12$).

{\bf These cases are just respectively labelled (4-2),  (4-3-ii), $ III_8-3$ in the list.}
\end{proof}

 \subsection{The list of  the shortenings of the K8 } 
We shall  now give the list of all the shortenings of the K8 code, following the proof of the previous theorem.
Altogether we have $15$ codes with $\nu=n > 0$, and if we also allow $\nu \geq n$, we have $36$ nontrivial codes
and $47$ strata.

\smallskip

To ease the notation, we define as usual $U$ to be the zero set of the functions in $\sK$.

Recall moreover that $\pi$ is an affine plane in $\FF_2^4$, while $a_{m}$ denotes
the number of codewords of weight $m$ (and observe that this number can only decrease if we take a shortening).

\

\begin{center}
\renewcommand{\arraystretch}{1.2}
\rotatebox{-90}{$\begin{array}{c|c|cc|cc|cc}
&\sK''&k'&k&\nu&n\\
\hline
(K8) & \langle   1,  x_1, x_2,  x_3, x_4, f_{\pi} \rangle  & 6 & 5 & 16 & 16 \\
(12-1) &  \langle   1 + f_{\pi}, x_1, x_2,  x_3, x_4 \rangle & 5 & 4 & 15 & 15 & 0 \in \pi & a_{12} =1 \\
(12-2) & \langle   1 + f_{\pi}, x_2,  x_3, x_4 \rangle  & 4 & 3 & 14 & 14 &  \pi \supset U & a_{12} =1\\
(12-3) & \langle   1 + f_{\pi}, x_3, x_4 \rangle  & 3 & 2 & 12,13 & 12 &  \pi  = U & a_{12} =1\\
(4-2) & \langle    f_{\pi}, x_2,  x_3, x_4 \rangle & 4 & 3 & 14 & 14 & \pi \cap U = \emptyset & a_{12} = 0\\
(4-3-i) & \langle    f_{\pi}+ 1 +  x_2,  x_3, x_4 \rangle & 3 & 2 & 12, 13 & 12 & \pi \cap U = \{0, e_1\} & a_{12} = 0\\
(4-3-ii) & \langle    f_{\pi}+ 1 +  x_2,  x_3, x_4 \rangle & 3 & 2 & 13 & 13 & \pi \cap U= \{0\} & a_{12} = 0\\ 
(II_8-2) & \langle    x_3,  x_4 \rangle  & 2 & 2 & 12 & 12 & \pi \cap  \langle    e_1,  e_2 \rangle= \{0\}\\
(II_8-1) & \langle     x_4 \rangle &  1 & 1 & 8,9,10,11 & 8 & \pi \cap \langle    e_1,  e_2 \rangle= \{0\}\\ 
(III_8-0) & \langle     f_{\pi} +x_1, x_4 \rangle  & 2 & 1 & 12 & 12 & \pi = \{ x_2=x_3=0\} & a_4= a_{12} =0\\
(III_8-1) & \langle    1 + f_{\pi}, x_4 \rangle  & 2 & 1 & 12 & 12 & \pi \subset U & a_{12} =1\\
(III_8-2) & \langle    f_{\pi}, x_4 \rangle  & 2 & 1 & 8,\dots,12 & 8 & \pi \cap U = \emptyset & a_4=2, a_{12} = 0\\
(III_8-3) & \langle    f_{\pi}, x_4 \rangle  & 2 & 1 & 10,11,12 & 10 & |\pi \cap U | = 2 & a_4=1,  a_{12} = 0 \\
(b_8) & \langle     f_{\pi} \rangle  & 1 & 0 & 4,\dots,11 & 4 \\
(c_8) & \langle     f_{\pi} + 1 + x_4 \rangle & 1 & 0 & 8,9,10,11 & 8 \\
(d_8) & \langle    0 \rangle & 0 & 0 & 0,1, \dots, 10 & 0
\end{array}$}
\end{center}

\bigskip

\appendixOn\section[Codes of 4-ic type]{Codes of quartic type without the B-inequality}\appendixOff\label{append_coding_theoretic_results}
\vskip+10pt\chapterauthor{Sascha Kurz and Michael Kiermaier}

 This appendix complements the results of Section \ref{append_quadratic_Reed-Muller_code}, and contains 
 in particular a computer verification of the partial assertion (*) made in Remark \ref{18},
 with exception of the cases where the weight $16$ occurs. 
 
 It classifies all extended codes with the same weights as for quartic surfaces, but without assuming the
 B-inequality.
 
\medskip

Codes with exactly one nonzero weight have been classified in \cite{bonisoli}. For two weights $w$, $2w$ a classification can be found in \cite{jungnickel2018bonisoli}.

\begin{lemma}
  \label{lemma_simplex_extensions}
  Let $\sK = \sK_k$ be the $k$-dimensional binary code whose nonzero codewords have weight $8$, i.e., $\sK_k=2^{4-k} S_k$ for $0\le k \le 4$, where $\sK_0$ is the empty code and $S_k$ is the simplex code. For an 
  \textit{extended code} $\sK''$ with $k' : = \dim(\sK'')=\dim(\sK)+1$, $\sK \subset  \sK''$, $\nu(\sK'')\le 16$, and the codewords in $\sK'' \setminus \sK$ have weights in $\{6,10\}$ 
  we have exactly the following possibilities (here $\codeshorteff{n}{k'}{2}$ denotes  a binary code of effective length $n$ and of dimension $k'$):
  \begin{itemize}
    \item (b) : for $k=0$ a $\codeshorteff{6}{1}{2}$ code with weight enumerator $W(z)=1z^{0}+1z^{6}$:
          $$
            \begin{pmatrix}
              111111
            \end{pmatrix}
          $$
    \item (c) : for $k=0$ a $\codeshorteff{10}{1}{2}$ code with weight enumerator $W(z)=1z^{0}+1z^{10}$:
          $$
            \begin{pmatrix}
              1111111111
            \end{pmatrix}
          $$      
    \item (III-2) : for $k=1$ a $\codeshorteff{10}{2}{2}$ code with weight enumerator $W(z)=1z^{0}+2z^{6}+1z^{8}$:
          $$
            \begin{pmatrix}
              1111001111\\
              0000111111
            \end{pmatrix}
          $$
    \item (III-1) : for $k=1$ a $\codeshorteff{12}{2}{2}$ code with weight enumerator $W(z)=1z^{0}+1z^{6}+1z^{8}+1z^{10}$:
          $$          
          \begin{pmatrix}
            111111000011\\
            000000111111        
          \end{pmatrix}
          $$ 
    \item for $k=1$ a $\codeshorteff{14}{2}{2}$ code with weight enumerator $W(z)=1z^{0}+1z^{8}+2z^{10}$:
          $$
          \begin{pmatrix}
            11110011110000\\
            00001111111111        
          \end{pmatrix}
          $$
    \item (1) : for $k=2$ a $\codeshorteff{12}{3}{2}$ code with weight enumerator $W(z)=1z^{0}+4z^{6}+3z^{8}$:
          $$
          \begin{pmatrix}
            110001011010\\
            001101010101\\
            000011111111
          \end{pmatrix}
          $$         
    \item (2) : for $k=2$ a $\codeshorteff{13}{3}{2}$ code with weight enumerator $W(z)=1z^{0}+3z^{6}+3z^{8}+1z^{10}$:
          $$
          \begin{pmatrix}
            1100111001011\\
            0011110100111\\
            1111110011101          
          \end{pmatrix}
          $$
    \item for $k=2$ a $\codeshorteff{14}{3}{2}$ code with weight enumerator $W(z)=1z^{0}+2z^{6}+3z^{8}+2z^{10}$:
          $$
          \begin{pmatrix}
            11100011111000\\
            00011111110100\\
            00001100110011        
          \end{pmatrix}
          $$
    \item for $k=2$ a $\codeshorteff{15}{3}{2}$ code with weight enumerator $W(z)=1z^{0}+1z^{6}+3z^{8}+3z^{10}$:
          $$
          \begin{pmatrix}
            111000111110000\\
            000111111101000\\
            001111011100111        
          \end{pmatrix}
          $$    
    \item for $k=2$ a $\codeshorteff{16}{3}{2}$ code with weight enumerator $W(z)=1z^{0}+3z^{8}+4z^{10}$:
          $$
          \begin{pmatrix}
            1110001111100000\\
            0001111111010000\\
            0110110011001111
          \end{pmatrix}
          $$
    \item (i) : for $k=3$ a $\codeshorteff{14}{4}{2}$ code with weight enumerator $W(z)=1z^{0}+6z^{6}+7z^{8}+2z^{10}$:
          $$
          \begin{pmatrix}      
            10001001000111\\
            01000101110100\\
            00111100001111\\
            00000011111111          
          \end{pmatrix}
          $$  
    \item for $k=3$ a $\codeshorteff{15}{4}{2}$ code with weight enumerator $W(z)=1z^{0}+4z^{6}+7z^{8}+4z^{10}$:
          $$
          \begin{pmatrix}
            100110011111000\\
            010001111110100\\
            001111100110010\\
            110001101000001        
          \end{pmatrix}
          $$
    \item for $k=3$ a $\codeshorteff{16}{4}{2}$ code with weight enumerator $W(z)=1z^{0}+2z^{6}+7z^{8}+6z^{10}$:
          $$
          \begin{pmatrix}
            1001100111110000\\
            0100011111101000\\
            0011111001100100\\
            1111111000100011
          \end{pmatrix}
          $$
    \item (00) : for $k=4$ a $\codeshorteff{15}{5}{2}$ code with weight enumerator $W(z)=1z^{0}+6z^{10}+15z^{8}+10z^{6}$:
          $$
          \begin{pmatrix}
            100001000101011\\
            010010001001101\\
            001011101111110\\
            000111100001111\\
            000000011111111       
          \end{pmatrix}
          $$      
    \item for $k=4$ a $\codeshorteff{16}{5}{2}$ code with weight enumerator $W(z)=1z^{0}+6z^{6}+15z^{8}+10z^{10}$:
          $$
          \begin{pmatrix}
            1001100111110000\\
            0100011111101000\\
            0011111001100100\\
            1110101010100010\\
            0011100010100001        
          \end{pmatrix}
          $$         
  \end{itemize}    
\end{lemma}
\begin{proof}
  We have used the software package \texttt{LinCode} \cite{kurz2019lincode}. 
  \end{proof}

\begin{remark}
  If we only consider those $\codeshorteff{\nu}{k'}{2}$ codes from Lemma~\ref{lemma_simplex_extensions} that satisfy the \textit{B-inequality} 
  $\nu \le k' +10$, then only $8$ codes remain. 
  
  If we add the Kummer code $\sK''_{Kum}$, see proposition~\ref{quartic-codes}, with its unique 
  extension, see proposition~\ref{extendedKummer}, then we end up with $9$ codes. 
  
  Taking zero columns into account, i.e., adding 
  $x\ge 0$ zero columns to every $\codeshorteff{n}{k'}{2}$ code (of effective length $n$), such that $n+x\le k'+10$, gives $18$~possibilities.   
\end{remark}

\renewcommand{\thesection}{\thechapter.\arabic{section}}
\chapter[Cubic discriminants]{Cubic hypersurfaces, associated discriminants and low degree nodal surfaces}
\chapterauthor{Fabrizio Catanese,\ Yonghwa Cho,\ Stephen  Coughlan,\ Davide  Frapporti}

\section{Linear subspaces $L$ contained in cubic hypersurfaces }

In this section we consider a cubic hypersurface $X$  in $\PP^n$ containing a linear subspace $L$ of dimension $k$.

\begin{lemma}\label{Sing-L}
Let  $X \subset \PP^n$ be a cubic hypersurface, and let $L \subset X$ be a linear subspace  of dimension $k$.
If $L$ is contained in the smooth locus of $X$, then we have $ 2 k 
\leq n-1$, while $ 2 k \leq n$ if $X$ has isolated singularities along $L$.

In the case of equality $ n = 2k$ and if $X$ has isolated singularities along $L$, then
the schematic  intersection  $ L \cap \Sing (X)$ consists of a 0-dimensional subscheme of length $2^k$.
\end{lemma}

\begin{proof}
Set $b + 1  : = n-k$, and choose coordinates in $\PP^n$
$$( z_0, z_1, \dots, z_k, x_0, x_1, \dots, x_b)$$
such that $L = \{ x_0 = x_1=  \dots =  x_b = 0\}$.
Let $F$ be an equation of $X$: since $L \subset X$, we may write
$$ F = \sum_{i,j =0}^k a_{i,j} (x) z_i z_j + \sum_{i=0}^k q_i(x) z_i + G(x),$$
where the $a_{i,j} (x), q_i(x) , G(x)$ are homogenous polynomials of respective degrees $1,2,3$.

We can now describe the schematic intersection:
\begin{eqnarray*}
L \cap \Sing(X) &=& \{ (z,0) \mid \frac {\partial }{\partial x_h} ( \sum_{i,j} a_{i,j} (x) z_i z_j ) =0 , \ \forall h=0, \dots , b\}\\
&=& \{ (z,0) \mid \sum_{i,j} a_{i,j;h} z_i z_j  =0 , \ \forall h=0, \dots , b\} ,
\end{eqnarray*}
where $ a_{i,j} (x)= \sum_h a_{i,j;h} x_h.$
In other words, if we write $$ F = \sum_h x_h Q_h(z) +  \sum_i q_i(x) z_i + G(x),$$
$ L \cap \Sing(X) $ is defined in $L$ by the $b+1$ quadratic equations $Q_h(z)=0$, in particular it is not empty if $k \geq b +1 \Leftrightarrow 2k \geq n$,
and has dimension at least $ k-b -1 = 2k - n$.

In case of equality, and if $ L \cap \Sing(X) $ has dimension zero, it consists of the complete intersection of the $k$ quadrics
$Q_h(z)=0$, $h=0, \dots,  b=k-1$, hence the scheme has length $2^k$ by Bezout's theorem.
\end{proof}
\begin{rem}\label{3-dim-lost}
The last assertion shall be particularly useful in the special case $ k=3$, $n=6$: it says that if $X$ is nodal, a deformation of $X$
preserving $L$ must also preserve the $8$ nodes of $X$ along $L$. 
\end{rem}

\begin{cor}
For each pair $(k,n)$ such that $ 2 k +1
\leq n$ there exists a smooth cubic hypersurface $X \subset \PP^n$ which contains a linear subspace $L$ of dimension $k$.
\end{cor}
\begin{proof}
In the above notation, take $X = \{ F = 0\}$, where $ F = \sum_h x_h Q_h(z) +  \sum_i q_i(x) z_i + G(x),$
and the $n-k$ quadrics $Q_h(z)$ have no common zeroes in $L \cong \PP^k$.

By Bertini's theorem $X$ is smooth outside the base locus $L$ of this linear system, while $X$ has no singular points
in $L$ by the previous lemma \ref{Sing-L}.
\end{proof}

Consider now the Grassmann variety
$$ \sG : = Gr(k,n) : = \{ L \subset \PP^n \mid L \cong \PP^k\},$$
  the space $\sH \cong  \PP^N$ of cubic hypersurfaces in $\PP^n$,
and the incidence correspondence of subspaces contained in a cubic hypersurface
$$\sI \subset \sG \times \sH , \ \sI : = \{ (L, X)\mid L \subset X\}.$$
Here $N + 1 =\binom{n+3}3$ and the projection $p : \sI \ra \sG$ is a projective bundle with fibre isomorphic to 
$\PP^{N-c}$,
where $ c : = \binom{k+3}3$. In particular, the projection on the space $\sH$
cannot be surjective unless $g := \dim (\sG) = (k+1)(n-k) \geq c$.
This inequality amounts to $ 6 (n-k) \geq (k+3)(k+2) \Leftrightarrow 6 n \geq k^2 + 11 k + 6$,
yielding $n \geq 3$ for $k=1$, $n \geq 6$ for $k=2$, $n \geq 8$ for $k=3$, $n \geq 11$ for $k=4$.

Even if our main interest here is to  concentrate  on the special cases $k=1, n = 5$ and $k=2, n=6$, 
we proceed with  some results which hold  more generally, and also recalling the cases with smaller dimensions, 
for instance $k=1, n \leq 4$.

\begin{defin}
The Fano scheme $F_k(X)$  of $X$ is the fibre of $\sI$ over $X$, namely $ F_k(X) := \{ L \in \sG \mid L \subset X\}$.

\end{defin}
On the Grassmann variety we have the exact sequence 
$$ 0 \ra \sU \ra V \otimes \hol_{\sG} \ra \sQ \ra 0  ,$$
where $\sU$ is the universal subbundle (it has rank $= k+1$), and $\sQ$ is the universal quotient bundle
(it has rank $n-k$), and its dual sequence.

Taking $X=\{ F = 0\}$, $ F \in \Sym^3 (V^{\vee})$ induces a section $s_F \in H^0 (\sG, \Sym^3 (\sU^{\vee}))$
such that $F_k (X) = \{ L \in \sG \mid s_F (L) = 0 \}$.

In the special case $ k=1$, $s_F$ is the section of a rank $4$
bundle, while for $k=2$ $s_F$ is the section of a rank $10$
bundle, hence the expected codimension of $F_k(X)$ is equal to $4$ for $k=1$, $10$ for $k=2$. 

{\bf In general the local dimension of $F_k(X)$ at each point  is always at least $ \dim (\sG) - \binom{k+3}3$.}

To see whether equality is attained, we study the local structure of $F_k(X)$, extending an argument of  \cite{collino}
who proved: 

\begin{lemma}[{\textbf{Collino}}]
Assume that $n=6$, $k=2$, and that $L$ is contained in the smooth locus of $X$: then 
$L$ is a smooth point of $F_2(X)$ of dimension $2$ if and only if the $4$
conics $Q_0(z), \dots, Q_3(z)$ are linearly independent. 

Moreover, for all such cubic hypersurface $X$,  $F_2(X)$ is connected, and, for general (smooth)
$X$,  $F(X)$ is a smooth irreducible surface.
\end{lemma} 

In the following lemma, statements i) and ii) are not new (cf. \cite{altklei}), and statement iv) is a special case of the more general Theorem 2.1 of \cite{deb-man}.

\begin{lemma}\label{FanoTg}
Assume that $L \in F_k(X)$, and take coordinates   in $\PP^n$
$$( z_0, z_1, \dots, z_k, x_0, x_1, \dots, x_b)$$
such that $L = \{ x_0 = x_1=  \dots =  x_b = 0\}$ ($b + 1  : = n-k$).

Write the   equation $F$  of $X$ as
 $$ F = \sum_h x_h Q_h(z) +  \sum_i q_i(x) z_i + G(x).$$
 
i)  Then in the local chart of the Grassmannian $\sG := Gr (k,n)$, 
 the affine space of matrices $B$ such that $L_B = \{ x = B z\}$,
 the Zariski tangent space of $F_k(X)$ is the linear subspace of linear syzygies among the quadrics $Q_h$,
 $$ T(F_k(X))_L = \{ B | \sum_h Q_h(z) (\sum_j B_{h,j} z_j)\equiv 0 \}.$$
 
 ii) In particular, we get a smooth point of $F_k(X)$ of the expected codimension $k+3 \choose {3}$
 if and only if the quadrics $Q_h$ generate the space of cubic forms in $L$.
 
 iii) If $L$ is contained in the smooth locus of $X$, then we get a smooth point of $F_k(X)$ of the expected codimension $k+3 \choose {3}$
 for $k=1$, and for $k=2$ unless  the quadrics $Q_h$ generate a vector space of quadratic forms of vector dimension $3$.
 In this case the local codimension is at least $9$.
 
 iv) If the expected dimension $ (k+1) (b+1) - {{k+3} \choose {3}}$ of $F_k(X)$ is non-negative, then $F_k(X)$ is 
 non-empty; if it is strictly positive, then $F_k(X)$ is connected.

\end{lemma}

\begin{proof}

The local equations of the Fano scheme are written (according to the Taylor development in $B$) as:
 $$ F = \sum_h  Q_h(z) (\sum_j B_{h,j} z_j) +  \sum_i q_i(B z ) z_i + G(Bz) = 0.$$

Recall that $ L \cap \Sing(X) $ is defined in $L$ by the $b+1$ quadratic equations $Q_h(z)=0$, 
and  if $L$ is contained in the smooth locus of $X$ then $ b \geq k$ and the zero set of the linear system
of quadrics spanned by the $Q_h$'s is empty.

By rearranging the order, we may assume that $Q_0, \dots, Q_k$ have no common zeroes, hence they form a regular sequence.

We set $$ \be_h (z) : = (\sum_j B_{h,j} z_j).$$
By exactness of the Koszul complex, it follows that the subspace generated by $\sum_0^k \be_h Q_h$
has dimension equal to $(k+1)^2$.

For $k=1$ we get $(k+1)^2 = 4 ={ k+3 \choose {3}}$ and we are done. 

For $k=2$ we get $(k+1)^2 = 9 <   10 = { k+3 \choose {3}}$ and, for $ b \geq 3 > 2$, we get the full space of cubic forms
unless all the other quadrics are linear combinations of  $Q_0, Q_1, Q_2$.

In fact,
let $R$ be the polynomial ring $R := \CC [z_0, z_1, z_2]$, and let $I$ be the ideal $I : = (Q_0, Q_1, Q_2)$,
so that $\dim (I_3) = 9$.

 By a theorem of Macaulay \cite{macaulay},  
there is a perfect pairing 
$$ R/ I  \times R/I  \ra R_3/ I_3.$$
If  now $ Q_3 \notin I_2$, then there is a degree $1$ polynomial $f$ such that $Q_3 f \notin I_3$, hence
$(Q_0, Q_1, Q_2, Q_3) = R_3$ as claimed.
\end{proof}

In the special case $n=6$, $k=2$, $\sG$ has dimension $12$, and  the expected dimension of $F_k(X)$ is equal to $2$;
for  $n=5,4,3, \  k=1$, the expected dimension of $F_k(X)$ is respectively equal to $4,2,0$.

As known classically, for the case of cubic surfaces, if $X$ is smooth, then it contains $27$ lines, and a finite number
of lines  if $X$ is normal, and not    a cone.

For $X^2_3\subset \PP^3$ irreducible, $F_1(X)$ has dimension equal to $1$ if and only if $X$ is a cone, or $X$ has a singular line.
And $F_1(X)$  has dimension equal to $2$ exactly if $X$ is reducible (then it contains a plane).

Going up,  if $X^3_3\subset \PP^4$ has isolated singularities, then, as we saw,  it cannot contain a linear subspace of dimension $\geq 3$,
but it can contain some plane, a  linear subspace $\Lam$ of dimension $2$. 

Consider the incidence variety
$$\{ (\Lam, H) \mid X \supset \Lam \cong \PP^2, \ \Lam \subset H  \cong \PP^3 \} ,$$
which is a $\PP^1$-bundle over $F_2(X)$, and which has a finite map to $(\PP^4)^{\vee}$.
Hence, if $F_2(X)$ has dimension at least $1$, then there is a surface $\Sigma \subset (\PP^4)^{\vee}$
such that $ X \cap H$ is reducible for $H \in \Sigma$, hence $H$ is tangent to $X$ on a curve,
which is then a fibre of the Gauss map. Hence the dual variety is degenerate, and $X$ is a cone.

We have shown: if $X^3_3\subset \PP^4$ has isolated singularities and is not a cone, then it contains only a finite number of planes.

Consider now  the incidence variety
$$\{ (L, H) \mid X \supset L \cong \PP^1, \  L  \subset H  \cong \PP^3 \} ,$$
which is a $\PP^2$-bundle over $F_1(X)$, and therefore has dimension at least $4$ at each point.

Continuing to assume that $X^3_3\subset \PP^4$ has isolated singularities and is not a cone, we see that each plane $\Lam \subset X$
yields a 4-dimensional subvariety of the incidence variety above.

We show now: {if $X^3_3\subset \PP^4$ has isolated singularities and is not a cone, then $F_1(X)$ has dimension $2$.}

In fact, the fibre of the incidence variety over $(\PP^4)^{\vee}$ always has dimension $\geq 0$, and the general fibre has dimension $0$,
while the fibre has dimension at most $2$, and equal to $2$ only if we have a plane $\Lam \subset X$.
Moreover, the fibre has dimension $1$ if and only if $H$ is a point of the dual variety where the Gauss map has fibre of dimension $1$,
hence the fibre has dimension $1$ over at most a curve. 

The preceding discussion proves:
\begin{lemma} If $X^3_3\subset \PP^4$ has isolated singularities and is not a cone, then it contains only a finite number of planes. Moreover,
$F_1(X)$ has dimension $2$.
\end{lemma}

The next case is important for our applications:

\begin{theo}\label{dim=4}
Let $X^4_3\subset \PP^5$ be a nodal cubic hypersurface. Then the Fano scheme $F_1(X)$ has dimension exactly $4$ at each point and the Fano scheme $F_2(X)$ contains only a finite number of planes $\Lam$ passing through some node of $X$.

\end{theo}
\begin{proof}
The incidence variety
$$\{ (L, H) \mid X \supset L \cong \PP^1, \  L  \subset H  \cong \PP^4 \} ,$$
 is a $\PP^3$-bundle over $F_1(X)$, and therefore has dimension at least $7$ at each point.
 
 Projecting to  $(\PP^5)^{\vee}$, we see that the fibre dimension is at most $2$ unless $ H \cap X$ is a cone
 or it has nonisolated singularities. The first possibility is excluded since at the vertex of a cone $X$ would not have
 a nondegenerate quadratic singularity. The second possibility leads again to a positive dimensional fibre of the Gauss map.
 However, for a nodal hypersurface the Gauss map is nondegenerate, hence we have at most a surface where the fibre has positive dimension.
 
 Hence it suffices to show that the fibre can never have dimension $\geq 6$, which is obvious since otherwise the fibre 
 would be the whole Grassmannian,
 hence $ H \cap X = H$, a contradiction since $X$ contains no linear subspace of dimension $3$.
 
 For the second assertion, assume that a plane $\Lam \subset X$ passes through some node $P$ of $X$.
 Projection from $P$ sends $\Lam$ to a line $R$ contained in a K3 surface $S$, complete intersection of type $(2,3)$ in $\PP^4$,
$ S = \{ Q=G=0\}$, where $ F= z Q(x) + G(x)$ is  the Taylor development of the equation of $X$ at $P$.
 
 Then $ R^2= -2$, and $R$ does not move in $S$, hence the same holds for $\Lam$.
\end{proof}

\begin{theo}\label{cubic-5}
Let $X^5_3\subset \PP^6$ be a nodal cubic hypersurface. Then  the Fano scheme $F_3(X)$ consists of at most a finite
number of points and the Fano scheme $F_2(X)$ has dimension exactly $2$ at each point $L $
 such that  there is no $3$-dimensional linear subspace $\Lam \subset X$ with $ L \subset \Lam$.

\end{theo}
\begin{proof}
For the first assertion, remark that each point $\Lam \in F_3(X)$ contributes a $3$-dimensional subvariety of
$F_2(X)$, where the general point is a subspace $L$ contained in the smooth locus. 

By applying iii) of Lemma \ref{FanoTg} we see that the local dimension at such points is at most $3$, hence we
get only a finite number of such $3$-dimensional subspaces $\Lam$.

For the second assertion, assume that the dimension of $F_2(X)$ at $L$ equals $3$. Then, by Lemma \ref{FanoTg},
the linear system of quadrics on $L$ generated by the $Q_h$'s has dimension $3$,
and the Zariski tangent space has dimension $3$. Hence $F_2(X)$ is smooth of dimension $3$ at $L$.

We write then the equation of $X$ along $L$ as usual, where $Q_0, Q_1, Q_2$ have no common zeros on $L$,
and we may assume that $Q_3 \equiv 0$:
 $$ F = \sum_0^2 x_h Q_h(z) +  \sum_i q_i(x) z_i + G(x).$$
 
 Then local coordinates for $F_2(X)$ are given by the coefficients of $\be_3(z)$, and since the tangent space is given by the 
 subspace $\be_0=  \be_1= \be_2 = 0$, we assume that locally the coefficients of these three forms are functions of 
 the coefficients of $\be_3(z)$, which have a Taylor development beginning with order $\geq 2$.
 For shorthand notation we write, for $  j=0,1,2$, $\be_j (z) = f_j (\be_3, z)$.
 
 The condition that these subspaces are contained in $X$ amounts to:
   $$  \sum_0^2 f_h (\be_3, z) Q_h(z) +  \sum_i q_i(B (\be_3) z )) z_i + G(B (\be_3) z ) \equiv 0, $$
 where $B (\be_3) z : = (f_0 (\be_3, z), f_1 (\be_3, z),f_2 (\be_3, z),\be_3(z)).$
  
The important  point here is however the following.
Consider the universal family $\Ga_2(X)$ of planes over $F_2(X)$, 
$$\Ga_2(X) \subset F_2(X) \times \PP^6$$
and let $Z$ be the image  inside  $\PP^6$ of the projection $p_2$ on the second factor.

Note that $ \be_3, (z_0,z_1,z_2)$ are (mixed, affine resp.~homogeneous)  coordinates around $L$. 

 We claim that  $Z$ is a variety   of  dimension $3$. To show this, it suffices to observe that,
 along each such subspace $L$ the derivative of $p_2$ is, as we have shown, of rank $3$, since the image of the derivative is the subspace
 $x_0=x_1=x_2=0$.

Since $\Ga : = \Ga_2(X)$ has dimension $5$, take a smooth point $z \in Z$ such that the fibre $W : = p_2^{-1} (z)$
has dimension $2$;
then, since  $W$ is transversal to the fibres of $p_1 : \Ga \ra F_2(X)$, we have a $2$-dimensional family of
planes $L_w$ contained in $Z$ and passing through $z$. But each $L_w$ is contained in the (projective) tangent space 
$T : = TZ_z \cong \PP^3$. These are then all the planes in $T$ passing through $z$, hence their union is just $T$,
and we conclude that $ T \subset Z$. Hence by equality of dimension $ T = Z$, and since $L \subset T$ we have a contradiction.
\end{proof}

\bigskip

\section[Discriminants $Y$ of cubic hypersurfaces $X$]{Discriminants $Y$ of cubic hypersurfaces $X$ for the projection with centre $L \subset X$.}

In the above situation of a $k$-dimensional linear subspace contained in a cubic hypersurface, $ L \subset X$, consider the projection
$$ \pi_L : \PP^n \setminus L \ra \PP^b = \PP^{n-k-1},$$
such that 
$$ \pi_L (z,x) = x.$$

This rational map extends to a morphism on the blow up $\tilde{\PP}_L \subset \PP^n \times  \PP^b$ with centre $L$,
which contains the proper transform $\sC$ of $X$.
The fibre over $x \in \PP^b$ is the join $\PP_x$ of $L$ with $ (0, x)$, i.e., the $\PP^{k+1}$
with coordinates $(z_0, z_1, \dots, z_k, z_{k+1})$ parametrizing the points
$$(z_0, z_1, \dots, z_k, z_{k+1}x_0, \dots, z_{k+1}x_b).$$

The intersection of this subspace with $X$ yields the divisor $L+\sC_x$, where $L$ is defined by $z_{k+1}=0$ and $\sC_x$ is the
quadric 
$$ \sC_x : = \{ (z_0,\dots, z_{k+1}) \mid  F_x (z)  = \sum_{i,j =0}^k a_{i,j} (x) z_i z_j + \sum_{i=0}^k 2 q_i(x) z_iz_{k+1} + G(x) z_{k+1}^2= 0  \}.$$

\begin{defin}
The discriminant hypersurface of the pair $L \subset X$ is 
$$ Y : =  \{ x \in \PP^b \mid  \Sing (\sC_x ) \neq \emptyset\}.$$
This is the determinantal hypersurface 
$$ Y = \{ x \mid  \det A(x) = 0\}$$ where we define 
$A(x)$, respectively  $A'(x)$, as the symmetrical $ (k+2) \times (k+2)$, respectively $(k+1)\times (k+1)$ matrices: 

\begin{equation}\label{E.FormaMatriceX}
A(x): =\left(\begin{matrix}a_{00}(x)&\dots&a_{0k}(x)&q_0(x)\cr \dots&\dots&\dots&\dots\cr a_{0k}(x)&\dots&a_{kk}(x)&q_k(x)\cr q_0(x)&\dots&q_k(x)&
G(x)
\end{matrix}\right),
\quad 
A'(x) : =\left(\begin{matrix}a_{00}(x)&\dots&a_{0k}(x)\cr \dots&\dots&\dots\cr a_{0k}(x)&\dots&a_{kk}(x)
\end{matrix}\right).
\end{equation}

The matrix $A'$ corresponds to the linear system of quadrics in $\PP^k$, of projective dimension $\leq b$,  generated by $Q_0, \dots, Q_b$, and whose base locus is $\Sing(X) \cap L$.
\end{defin}

\begin{rem}\label{rem}
(1) Since $\sC_x$ is the fibre of $\sC \ra \PP^b$ over $x$, it follows that $Y$ is the set of  critical values of $\pi_L |_{\sC}$.

(2) The intersection of $X$ with the subspace $\PP^b : = \{ (z,x) \mid z = 0\}$ equals $\{(0,x)\mid G(x) =0\}$.

(3) Given a point $ w:= (z,x) \in X \setminus L$, there is a unique $x$ such that $ (w,x) \in \sC_x$.

(4)  A point $w = (z,0) : =  (z_0, \dots, z_k, 0 , \dots, 0)$ in $X \cap L$ yields a point 
$(w, x)$ in $\sC_x$
if and only if $z \in Q_{x} : = L \cap \sC_x = \{ z\mid \sum_h x_h Q_h (z) = 0\}$.
We have  $\sC_x \subset \PP_x$ and, identifying $\PP_x$ with $\PP^{k+1}$, the point $ (z_0, \dots, z_k,  0)$ corresponds to $(w,x)$.  

The point $(w,x)$ is singular in $\sC_x$ if and only if
\[
  A'(x) z = 0,\ \sum_i q_i(x) z_i = 0 \Longleftrightarrow  (z_0,\dots,z_k,0) \in \ker A(x),
\]
and, since $G(x) z_{k+1}^2$ vanishes to order $2$ at our point $(w,x)$, it is singular in $\sC$ if and only if
\[
	(w,x) \in \Sing \sC_x\ \ \text{and}\ \ w \in \Sing X\ (\text{that is,}\ Q_h(z)=0,\ \forall h).
\]
\end{rem}

\begin{prop}
Given a point $ w = (z,x) \in X \setminus L$, choose the unique  $x \in \PP^b$ such that $ (w,x) \in \sC_x$.
Then $w \in \Sing(X)$ if and only if $(w,x) \in \Sing ( \sC )$. Moreover,  if $\sC_x$ is smooth at $(w,x)$ and $\sC_x\neq \PP_x$, then $ w \notin \Sing (X)$.
\end{prop}
\begin{proof}
Since $w \notin L$, $X$ and $\sC$ are locally isomorphic, hence the first assertion.
If $\sC$ were singular at $(w,x)$, then the  codimension $b$ complete intersection
$\sC_x$ would also be singular at $(w,x)$.
\end{proof}

It follows that if $w \in \Sing(X) \setminus L$, then its projection point $x $ belongs to the discriminant hypersurface $Y$,
i.e., $\rank (A(x)) \leq k + 1$.

\begin{prop}\label{suspension}
Assume that $ x' \in Y$ satisfies $\rank (A(x')) = k + 1$, that is, the matrix $A(x')$ has corank $c=1$, and let $f(x)=0$ be a local holomorphic equation for $Y$ at $x'$.

(i) There is at most one possible singular point $w$ of $\sC$ lying in $\sC_{x'}$, and it has local equation
$$ \z_1^2 +   \dots \z_{k+1}^2 + f(x) = 0,$$
a so-called suspension of the singularity of $Y$ at $x'$.

Moreover, $w$ corresponds to the point $z : =  \ker (A(x'))$,
and it is a singular point (resp.~a node) of $\sC$  if and only if $x'$ is a singular point (resp.~a node) of $Y$. 

(ii) Moreover, $w$ comes from a point  $z \in L$ if and only if $ \ker A(x') \in L \Leftrightarrow$ there is a solution $ z \in L$
to $$ A'(x') z = 0, \ \sum_0^k q_i(x') z_i = 0.$$

(iii) Assume that  the point $w$ comes from a point  $z \in L$.  Then:

1)   If  $w\in \Sing(\sC)$, then $z \in \Sing(X)$.

 2)  If $k+1 \leq b-1$ and $Y$ has isolated singularities then $w$ is not a singular point of $\sC$ (and $x'$ is not in $\Sing(Y)$). 

3) If instead $k+1 = b $ and $x' \in \Sing(Y)$, then $z \in \Sing(X)$.
 If  the singularities of $Y$ are nodes,
then  $ z \in \ker A(x')$ for exactly two nodes  of $Y$.

4) Finally, if $ k \geq b$, and the singularities of $Y$ are nodes, then  $ z \in \ker A(x')$ for at most two nodes  of $Y$.

\end{prop}
\begin{proof}
Under our hypothesis we have a quadratic form in $(k+2)$-variables over the local ring $\hol$ of the point $ x' \in \PP^b$ and
its image in the residue field $\CC$ has rank equal to $k+1$.
Therefore it can be brought in diagonal form $$ \z_0^2 f(x) + \sum_1^{k+1}  \z_i^2,$$
where $f(x)$ belongs to the maximal ideal.

It follows that $f(x)=0$ is the local equation of $Y = \{ \det (A(x))=0\}$, and that we have at most a unique singular point 
$w$ of $\sC$ lying in $\sC_{x'}$, corresponding to $\z_1= \dots = \z_{k+1}=0$. There we may set $\z_0=1$, and obtain the
asserted local equation for $\sC$.

Moreover, this singular point corresponds to the point $ \ker (A(x'))$;  the other assertions of (ii) follow right away from the preceding Remark \ref{rem}.

Let us prove (iii).

Without loss of generality we may assume that $z$ is the vector $e_0$. Hence $ e_0 \in \ker A(x')$
is equivalent to requiring $a_{0,j}(x') = 0$ for all $j =0, \dots, k$, and $q_0(x')=0$.
Hence the first row and column of $A(x')$ vanish. Write  $B(x')$ for the submatrix of $A(x')$ where the first row and first column are erased. Since $A(x')$ has corank $c=1$, it follows that the determinant of $B(x')$ is invertible. 

Taking the Taylor development of   $ f (x) =  a_{0,0}(x) \det (B(x')) + \phi (x)$ at the point $x'$,  we see that $ \phi (x)$
belongs to the square of the ideal generated by $a_{0,1}(x), \dots, a_{0,k}(x), q_0(x)$,
hence it vanishes to order at least $2$ at $x'$, while $\det (B(x')) \neq 0$. We conclude that $Y$ is smooth at $x'$
if and only if  $a_{0,0}(x)$ is not identically zero.

If $a_{0,0}(x)$ is identically zero,  we see that $z \in \Sing (X)$, since, for $e_0$, $z_0=1$ and $z_1, \dots, z_k, x_0, \dots, x_b$
are local coordinates at the point $e_0$.

Moreover, in this situation, then $e_0 \in \ker A(x')$ for all the solutions of
$a_{0,j}(x') = 0$, $1\leq j\leq k$, $q_0(x')=0$. These are $k$ linear and one quadratic equations, hence for $b \geq k + 2$ we get that $Y$ does not have isolated singularities.

Assume now that $ b=k+1$ and that the singularities of $Y$ are isolated. Then the locus $\Sigma$ of points $x'$
such that $e_0 \in \ker A(x')$ is a complete intersection of the above $k+1$ equations,
hence either two points, or one with multiplicity $2$. If the singularities of $Y$ are nodes,
since $\det A(x)$ is in the square of the ideal generated by these equations, these must be local parameters,
hence we get two singular points, nodes by our assumption.

Finally, if $ k \geq b$,  we get that the finite set $\Sing (Y)$ contains  the locus $\Sigma$, which 
is therefore finite; hence, being defined by one quadratic equation and several linear equations,  it has  cardinality at
most $2$.
\end{proof}

\begin{rem}
 (I) In the case $b=3$, we see that the second assertion of (iii) above applies for $ d =\deg (Y)  \leq 5$, while for $d=6$
there is the possibility that there are some nodes of $X$ in the plane $L$.

(II) For our purposes we may restrict ourselves to the case where $r : = \rank (A(x))$ is always $\geq k$. Indeed, since the matrix $A(x)$
over the local ring $\hol$ can be brought to block diagonal form with upper entry the $r \times r$ identity matrix 
$I_r$, $ r := \rank (A(x'))$, and with lower entry  a $c \times c := (k+2 -r) \times (k+2-r)$ matrix with entries in the maximal ideal
of $\hol$.  If the corank $c \geq 3$, then $x'$ is a point of multiplicity $\geq 3$ for $Y$. 
\end{rem}

\begin{prop}\label{c=2}
Assume that $ x' \in Y$ satisfies $ \rank (A(x')) = k$, that is, the matrix $A(x')$ has corank $c=2$. 

Then there are  functions $\xi_1, \xi_2, \xi_3$  in the maximal ideal of $\hol$
such that $\xi_1 \xi_2 - \xi_3^2$ is a local holomorphic equation for $Y$ at $x'$.

Over a neighbourhood of $x'$, $\mathcal C$ is defined by an equation of the form
\[
	\xi_1 \zeta_0^2  + 2 \xi_3  \zeta_0 \zeta_{k+1} +  \xi_2  \zeta_{k+1}^2  + \sum_1^k  \zeta_i^2 \in \mathcal O[\zeta_0,\ldots,\zeta_{k+1}]
\]
as a projective variety over $\mathcal O$, where $\mathcal O$ is the local ring of $x'\in \PP^b$.

The singular points $w$ of $\sC$ lying in $\sC_{x'}$, are contained in a projective
line  $\z_1= \dots = \z_k=0$, and they are defined by the further equation $\xi_1 \z_0^2  + 2 \xi_3  \z_0 \z_{k+1} +  \xi_2  \z_{k+1}^2  =0$.

In particular, there are no such singular points $w \in \sC_{x'}$ (that is, $(w,x')\in\sC_{x'}$) if $\xi_1, \xi_2, \xi_3$ are part of a system of local coordinates,
i.e., $ \Sing(Y)$ at  $x'$ is a smooth subvariety of codimension $3$ in $\PP^b$.

There is at most one such singular point if the multiplicity of $Y$ at $x'$ is either  three or 
  it is two and   
the tangent cone of $Y$ at $x'$ is a
 quadric of  rank $\geq 2$.
\end{prop}
\begin{proof}
Under our hypothesis we have a quadratic form in $(k+2)$-variables over the local ring $\hol$ of the point $ x' \in \PP^b$ such that
its image in the residue field $\CC$ has rank equal to $k$.

Therefore it can be brought in block diagonal form as explained in the preceding Remark
$$ \xi_1 \z_0^2  + 2 \xi_3  \z_0 \z_{k+1} +  \xi_2  \z_{k+1}^2  + \sum_1^k  \z_i^2,$$
where 
$\xi_1, \xi_2, \xi_3$ are  in the maximal ideal of $\hol$.

It follows that $\xi_1 \xi_2 - \xi_3^2$ is the local equation of $Y = \{ \det (A(x))=0\}$, and the
 singular points 
$w$ of $\sC$ lying in $\sC_x$, are contained in the locus   $\z_1= \dots = \z_k=0$. 

If such a point $w$ has $\z_0 \neq 0$, then the hypothesis that
$\xi_1, \xi_2, \xi_3$ are part of a system of local coordinates
implies that we can write $\xi_1$ as a function of other coordinates, hence we have a smooth point.
Similarly for the other case $\z_{k+1} \neq 0$.

More generally, if we have a singular point $w$, then its coordinates satisfy $\z_1= \dots = \z_k=0$
and $\z_0,  \z_{k+1}$ must yield a nontrivial  relation of linear dependence 
among the gradients $\nabla \xi_i$,
$$ \z_0^2  \nabla \xi_1   + 2  \z_0 \z_{k+1}  \nabla \xi_3 + \z_{k+1}^2  \nabla  \xi_2    = 0. $$

We may assume that the singular point $w$ corresponds to $\z_{k+1} =0$, hence we have $\nabla \xi_1=0$.
If there is another singular point, we may assume that it corresponds to $\z_{0} =0$, 
hence we have $\nabla \xi_2=0$.

Therefore $\xi_1, \xi_2$ belong to the square of the maximal ideal, and we have either 
a singular point of multiplicity
at least $4$, or  a uniplanar  double point (the tangent cone is induced by $\xi_3^2$, 
and is a quadric of rank $1$).

We have a third singular point  if and only if $\nabla \xi_i = 0, \ \forall i=1,2,3,$
hence each  $\xi_i$ belongs to the square of the maximal ideal, and  $Y$ has a singular point of multiplicity
at least $4$.
\end{proof}

Therefore, under the assumptions that
$$\corank (A(x')) \leq 2 , \ \forall x' ,$$
  (in particular this holds  if $Y$ has no triple points),
and that the singular points of $Y$ have a quadratic part of rank equal to $3$,
we have found that the set of singular points of $X \setminus L$ maps injectively  into the set of singular points of $Y$ with $\corank (A(x')) =1$.

Moreover,  we have singular points of $\sC$ 
only over singular points $x' $ of $Y$ where $\corank (A(x')) =1$. 

And by Proposition \ref{suspension}(i)
there is exactly  one such point in $\sC_{x'}$. This point  comes from a point of $X \setminus L$ if $k+2 \leq b$,
while if $b= k+1$ this point can also be of the form $(w, x')$ where $w \in L \cap \Sing (X)$. 
Then by Proposition \ref{suspension}(iii)-3) there are exactly two such points $(w, x')$, $(w, x'')$ with $x', x'' \in \Sing (Y)$ 
if we further assume
$b=3$ and $Y$ with isolated singularities. 
 
 The easier  consequence of the previous discussion is  assertion (3) of the following:

\begin{theo}\label{smooth centre}

Let $Y \subset \PP^3$ be the discriminant of the projection $\pi_L$ of a cubic hypersurface $X \subset \PP^{k+4}$
from a linear subspace   $L \subset X$ of dimension $k$ (hence $\deg(Y) = k+4$).

Assume that the singular points of $Y$ are nodes, that is, they have quadratic part of rank $3$.

(1) If $L \cap \Sing X = \emptyset$, then there is a bijection between the sets $ \Sing(X)$ and $$ \Sing(Y) \cap \{x \mid \corank ( A(x)) =1\},$$ induced by the projection $\pi_L$.

(2) If $k \leq 1$, then $L \cap \Sing X = \emptyset$, so (1) holds automatically.

(3) If $k=2$, then there is a surjection 
$$  \Sing(Y) \cap \{x \mid \corank ( A(x)) =1\} \ra \Sing(X)$$ such the inverse image of each point in $ L \cap \Sing(X)$ consists of two such nodes. In this case, the singularities of $X$ are all nodes.

\end{theo}
\begin{proof}

By Proposition \ref{c=2}, there are no singularities of $\sC$ over the singular points  $x \in Y$ with $ \corank(A(x) )= 2$.

Over the singular points  $ x \in Y$ with $ \corank(A(x) )= 1$ there is exactly one singular point of $\sC$
by Proposition \ref{suspension}(i). Now by Proposition \ref{suspension}(iii)-1), if this point comes
from a point $z \in L$, then $z \in \Sing(X)$, and $k+1 > b-1 = 2$ by Proposition \ref{suspension}(iii)-2). This proves (1) and (2).

We are now going to   see that a point $ w \in L \cap \Sing(X)$
 gives rise  to  singular points of $\sC$ lying  in some  fibre $\sC_{x}$.

Recall that $w = (z,0) = (z_0, \dots, z_k, 0, \dots, 0) \in L \cap \Sing(X)$ means that $z$ is in the base locus of the
system of quadrics in $\PP^k$, $Q_0 (z) = \dots = Q_b(z) = 0$.

We look for $x\in \PP^b$ such that $ (z,0) = (z_0, \dots, z_k, 0) \in \ker A(x)$.  The condition
$ A'(x) z = 0$ then determines $k$ linear equations, whose solution set is a subspace of dimension $ \geq b - k$
of $\PP^b$.  The equation $\sum_i q_i(x) z_i = 0$ determines a quadric in this subspace,
so that we have certainly a solution $x$ as soon as $ b \geq k+1$. 

If such $x$ exists, $\mathcal C$ is singular at $(w,x)$ by Remark~\ref{rem}(4). We know that $\corank A(x)\neq 2$ by  Proposition \ref{c=2}, while the corank of $A$ is never $\geq 3$ by our assumption.

If  $\corank A(x) =1$, then by Proposition \ref{suspension}(iii)-3),
we obtain two singular points $(w,x'), (w, x'')$ of $\sC$ associated to $w \in L  \cap \Sing (X)$, where $x'$, $x''$ are nodes with $\corank A =1$.

Finally, for $k=2$, we want to show that $X$ has only nodes as singularities.
This is true for the singularities in $\Sing(X) \setminus L$, and the singularities of $\sC$.

Let us show that the singularities in $ X \cap L$ (which produce two nodes after the blow up yielding $\sC$) are also
nodes.

We use the notation of Proposition \ref{suspension}(iii),  hence $w = e_0$, and we may assume that
$a_{00} =0$, $a_{01} =x_1$, $a_{02} = x_2$, $q_0 = x_0 x_3$; then at $w$, the quadratic part of the equation of $X$
equals
$$   x_1 z_1 + x_2 z_2 + x_0 x_ 3.     $$
Since $   x_1 , z_1 , x_2 , z_2 , x_0 , x_ 3$ are local coordinates at $w$, we get a node of $X$.
\end{proof}

\begin{rem}\label{node?}
Therefore, if $ Y \subset \PP^3$ is nodal of degree $d$ and is the discriminant of the projection of $X$ with centre $ L \subset X$, then for $ d\leq 6$, $X$ has only nodes as singularities.

\end{rem}

\section{Nodal maximizing surfaces in low degree}

If $Y$ is a nodal surface of degree $d$  in $\PP^3$, then it is known that the number $\nu(Y)$ of nodes is smaller than 
the maximum possible $\mu(d)$, and that  $\mu(1) =0, \mu(2) = 1, \mu(3) =4, \mu(4)=16, \mu(5) =31, \mu(6) =65.$

A major  concern in  this paper is to show that if $3\leq d \leq 6$ and the number of nodes is the maximum possible, then $Y$ is a cubic discriminant.

\begin{ex}\label{Milnor}
It is worthwhile to observe that, already for degree $d=3$, $\mu(d)$ can be smaller, for  surfaces with isolated singularities,
than the total sum of the Milnor
numbers of the singularities. In fact, if a cubic surface has exactly $4$ singular points,
it is projectively equivalent to the Cayley cubic with equation $\s_3(x) = 0$, where $\s_3$ is the third elementary
symmetric function,
and its singularities are just four nodes.

But there is the cubic surface $ x y z = w^3$ which possesses $3$ singular points of type $A_2$,
hence realizing a total Milnor number (sum of the Milnor numbers of the singularities)  equal to $ 6$.

 This leads to the following question: for a normal surface of degree $d$ in $\PP^3$, what is the maximum that
 the total Milnor number can achieve?
 
 Generalizing the above example, we can take the surface 
 $$ X : = \{ (x) |  x_0^d = \Pi_1^d L_j (x_1, x_2, x_3) \}, L_1, \dots, L_d \ {\rm general \  linear \ forms}\} ,$$
 which has $\frac{d(d-1)}{2}$ $A_{d-1}$ singularities, hence $\Milnor(X) = \frac{1}{2} d(d-1)^2$.
  
 Bruce \cite{bruce} gave general upper bounds for $s_n(d)$, in particular he showed that  $s_2(d) = : s(d) \leq \frac{1}{2} (d-1)^3$
  if the degree $d$ is odd, else $s(d) \leq \frac{1}{2d} ((d-1)^3 (d+1) + 1)$ for $d$ even.
  
  Hence  the above  examples 
 show that 
 $$ s(d) < \Milnor (d),\quad \forall d.$$
 
   For $d=4$, this gives a quartic surface with six $A_3$ singularities and total Milnor number $18$, 
 and it is natural to expect an example with Milnor number $19$, since the Picard number of a K3 surface is at most $20$
 (see \cite{vinberg,burns-wahl}).
We are not aware of general results on the function $\Milnor_n(d)$.

If $X$ has Rational Double Points, then an easy bound is given by $ h^{1,1}(S) -1$, where $S$ is the minimal resolution of $X$.
This yields the inequality 
$$ \Milnor (X) \leq h^{1,1}(S)  -1 = b_2 (S) - 2 p_g (S) -1 = e(S) - 2 p_g (S) -3 = $$
$$ = 10 \chi(\hol_S) - K^2_S - 1 = 9 + 10 p_g(S) - K^2_S = $$
$$= 9 + \frac{5}{3} (d-1)(d-2)(d-3) - d (d-4)^2 = \frac{2}{3} d^3 -6 d^2 + \frac{7}{3} d  -1 < \frac{2}{3} d^3.$$ 
\end{ex}

\bigskip

In degree $d=4$ the surfaces with $16$ singularities have nodes as singularities and are the Kummer surfaces;  they can be obtained projecting
a $10$-nodal  cubic hypersurface $ X \subset \PP^4$ from a smooth point.

Going back to $d=3$,  (see \cite{goryunov}) the maximal number of nodes that a cubic hypersurface $X \subset \PP^m$ 
can have is $\ga(m)$, where $\ga(3) = 4, \ga(4) = 10, \ga(5)= 15, \ga(6) = 35$
(we have in general $\ga(m) = \binom{m+1}{[m/2]}$).

For $m=2h$, equality is attained by the Segre cubic
\begin{align*} \Sigma (m-1)  :& = \{ x \in \PP^{m+1} \mid \s_1 (x) =  \s_3(x) = 0\} \\
&=  \{ x \in \PP^{m+1} \mid s_1 (x) =  s_3(x) = 0\} ,\end{align*}
where $\s_i$ is the $i$-th elementary symmetric function, and $s_i = \sum_j x_j^i$ is the $i$-th Newton function.

The singular locus of $\Sigma(m-1)$ consists of $\frac12\binom{m+2}{h+1}=\binom{m+1}{h}$ nodes which form an $\Sn_{2h+2}$-orbit, generated by $(1^{h+1}:-1^{h+1})$ where $\alpha^s$ means $\alpha$ is repeated $s$ times.

Whereas, for $ m = 2h +1$ odd, Goryunov produces
\begin{align*}
T &(m-1) :=\\
 &= \{ (x, z)\mid \s_1(x)= \s_3(x) +  z \s_2(x) +   \frac{1}{12} h (h+1)(h+2) z^3 = 0 \}\\
&=  \{ (x, z)\mid s_1(x)=2 s_3(x) - 3 z \s_2(x) + \frac{1}{2} h (h+1)(h+2) z^3 = 0 \}
\end{align*}
where $(x,z)$ denotes $(x_0,\dots,x_m,z)$.
 This is projectively equivalent to Kalker's cubic hypersurface defined by
\[T(m-1):=(x_0+\dots+x_m+2x_{m+1}=x_0^3+\dots+x_{m}^3+2x_{m+1}^3=0)\subset\PP^{m+1}.\]
The symmetric group $\Sn_{m+1}$ acts on $T(m-1)$ by permuting the first $m+1$ coordinates and fixing $z$,
respectively $x_{m+1}$.
The singular locus of $T(m-1)$ consists of $\binom{m+1}{h}$ nodes, forming the $\Sn_{m+1}$-orbit of the point $(1^{h+1}:-1^{h+1}:1)$.

We have the following result, the first part of which is due to Segre \cite{segre}, another proof is to be found in \cite{kalker},
and we give an argument here based on a result which shall be established in a 
later  section, where  we shall  also show that cubic 4-folds with 15 nodes form an irreducible family.

\begin{theo}\label{local-uniqueness}
Any nodal maximizing cubic hypersurface $X$  in $\PP^4$ is projectively equivalent to the Segre cubic.

Any equisingular deformation of the Segre cubic $\Sigma(m-1)$ in $\PP^{2h}$ is projectively equivalent to the Segre cubic.

For $h\geq 3$, any equisingular deformation of the Goryunov  cubic $T(m-1)$ in $\PP^{2h+1}$ is projectively equivalent to the Goryunov cubic.

The Segre cubics $\Sigma(3)$, $\Sigma(5)$ and the Goryunov cubic $T(4)$ are unobstructed, while all other Segre and Goryunov cubics are obstructed.
\end{theo}
\begin{proof}

Here we prove the first assertion. The other assertions are proven in Section \ref{append_Segre_cubic}. 
 
Let $P=(0, 0, 0, 0, 1)$ be a node of $X \subset \PP^4$, and consider the Taylor development of its equation at the point $P $:
$ F(x,z) = z Q(x) + G(x)$.

Then  by Corollary \ref{projection_c}
the curve in $\PP^3$ given by $Q(x) = G(x) = 0$ is a nodal curve contained in a smooth quadric and has $\ga -1$ nodes,
where $\ga$ is the number of nodes of $X$. But a curve of bidegree $(3,3)$ on $Q = \PP^1 \times \PP^1$
has at most $9$ nodes, equality holding if and only if it consists of three vertical and three horizontal lines.
 \end{proof}

 \begin{remark}
 The unobstructedness of the Segre cubic in $\PP^6$
(see Theorem \ref{unobstructed}) is equivalent to the  second assertion, since cubics in $\PP^6$ are a projective space of dimension $83$,
 the group of projectivities of $\PP^6$ has dimension $48$, and the group of projective automorphisms of the $35$-nodal
 Segre cubic is  finite.
 \end{remark}
 
\begin{remark}
It is known \cite{kalker}, \cite{goryunov} that, for $m =  5$,  nodal maximizing cubics are not projectively unique.

The previous result leads to the following natural questions: is every   nodal maximizing cubic hypersurface in $\PP^m$ projectively equivalent to the Segre cubic, for $m$ even? And for $m \geq 7$ odd, to the Goruynov cubic?

\end{remark}

Given a point in a cubic hypersurface $X \subset \PP^4$, we write its equation  as usual  as
$$ \{ (z, x_0,x_1,x_2,x_3) = (z,x) \mid a(x) z^2 + 2 q(x) z + G(x) = 0 \} ,$$
so that its discriminant is the quartic surface $\{ x \mid a(x) G(x) - q(x)^2 = 0\}.$

The point $x=0$ of $X$ is smooth if $a(x) $ is not identically zero, equivalently if $Y$ is not a quadric counted with multiplicity two.

If we make this assumption, the points where the corank of the matrix $A(x)$ drops by $2$ are the points $ a(x) = q(x) = G(x) = 0$,
in general a set of $6$ distinct points (nodes of $Y$) if hypothesis (1) of Theorem \ref{smooth centre} is verified.

Hence under  assumption (1) we obtain that $\nu(Y) = 6 + \nu(X)$, and, since a Kummer surface contains a half-even set of nodes
of cardinality $6$,  we get the whole  three dimensional family of Kummer surfaces
by projecting the Segre cubic from a smooth point (see \cite{vanderGeer}, \cite[Proposition 26]{cat-kummer}).

\begin{remark}
We obtain an explicit example of a Kummer surface as discriminant of a cubic hypersurface
by choosing a point on the Segre cubic which does not lie in anyone of the linear spaces
$L' \cong \PP^2$, $L' \subset \Sigma$ which are  in the $\mathfrak S_6$-orbit of the subspace $x_i + x_ {i+3} = 0, i=0,1,2$.

Setting $ x : = (x_0, x_1, x_2, x_3)$, we take a point 
$$(x_0, x_1, x_2, x_3, y, - y - \sum_0^3 x_i)$$
with  
$$ y^2  s_1(x) + y s_1(x)^2 = \frac{1}{3} ( s_3(x)  - s_1(x)^3 ),$$
for instance $$( a, a,  a,  5a, y , - y  - 8 a )$$
with $ y^2 + 8 a y + 16 a^2 =0$, so that we may take $ a=1, \ y = -4$.

\end{remark}

In the case of Togliatti quintics $Y$ (quintics with $31$ nodes), it was proven by Beauville \cite{angers}  that
$Y$ is the discriminant of a cubic hypersurface $X \subset \PP^5$ for the projection from a line $L \subset X$,
since $Y$ possesses an even set of nodes of cardinality  16. More generally, we have:

\begin{prop}\label{16even}
Let $Y$ be a nodal quintic surface with an even set of nodes of cardinality  16. Then $Y$ is the discriminant 
of a cubic hypersurface $X \subset \PP^5$ for the projection from a line $L \subset X$, $L$ contained in the smooth locus of $X$, and $X$ has exactly
$ \ga (X) = \nu (Y) - 16$ nodes as singularities, where $\nu(Y)$ is the number of nodes of $Y$.
\end{prop}
\begin{proof}
It follows from \cite{babbage} that $Y$ is a discriminant of a cubic hypersurface $X \subset \PP^5$ for the projection from a line $L \subset X$.

And the even set of 16 nodes is exactly the set of points of $Y$ where $ \corank (A(x))= 2$.

Theorem   \ref{smooth centre} implies  that   there is a bijection 
between the set $ \Sing(X)$ and the set of nodes of $Y$ which are not in the given even set. 

Finally, the singularities of $X$ are nodes by Proposition \ref{suspension}.
\end{proof}

\begin{ex}\label{Togliatti-quintic}

Consider a  discriminant quintic surface, corresponding to a  matrix:

\begin{equation}
A: =\left(\begin{matrix}a_{00}(x)&a_{01}(x)&q_0(x)\cr  a_{01}(x)&a_{11}(x)&q_1(x)\cr q_0(x)&q_1(x)&G(x)
\end{matrix}\right).
\end{equation}

 In this  example,  let us restrict to the special case where 
  $$ a_{01}(x) \equiv  0, $$ 
in order to give explicit equations for the 16 nodes.
  
  If $a_{01}(x) \equiv  0$, we may assume 
without loss of generality that 
 $$a_{00}(x) = x_0 , a_{11}(x) = x_1,$$
 $$ q_0 (x) = q_0 (x_1,x_2, x_3), 
  q_1(x) =  q_1 (x_0,x_2, x_3) .$$
  
  The equation of the quintic takes the simple form 
\[\begin{split}  a_{00}(x)  a_{11}(x) G(x) - a_{00}(x) q_1(x)^2 - a_{11}(x) q_0(x) ^2 =\\
  x_0 x_1 G(x) - x_0 q_1(x)^2 - x_1 q_0(x)^2= 0.\end{split}\]
 
  The  even set of nodes corresponding to the matrix $A$ (points where the rank drops by two) 
consists of the 16 points, union of the two sets of 8 points, of respective equations:
$$  \{ x_0 = q_0(x) = x_1   G(x) - q_1(x)^2 = 0 \} \cup   \{ x_1 = q_1(x) = x_0   G(x) - q_0(x)^2 = 0 \} .$$

We need to assume that there are no common solutions of the equations:
$$ x_0 = x_1 = q_0(x) = q_1 (x)=0,$$
and that in each plane the conic and quartic intersect in 8 distinct points, for instance:
$$ q_0 (x_1,x_2, x_3) = x_1   G(0, x_1, x_2 , x_3) - q_1(0,0, x_2,x_3)^2 = 0$$
 yields eight points. 
\end{ex}

The following result, the converse to Proposition  \ref{16even},  is essentially  the main result of Togliatti \cite{togliattiquintics}, for which we give a slightly different 
 proof (some details were not verified in \cite{togliattiquintics}).

\begin{prop}\label{5discr}
Let $X$ be a cubic hypersurface $X \subset \PP^5$ which has exactly
$ \ga (X) $ nodes as singularities. 
Then there exist lines $L$ contained in the smooth locus of $X$, and such that the different nodes $P_i$ of $X$ 
determine distinct planes $P_i*L$ containing $L$. 

Let $Y$ be  the discriminant 
of the  cubic hypersurface $X \subset \PP^5$ for the projection from a general   line $L \subset X$.
Then  $Y$ is a  quintic surface with singular scheme consisting of $\ga(X)$ nodes plus a 
local complete intersection scheme of length $16$. For general $X$,
$Y$ is a nodal quintic surface with exactly $ 16 + \ga(X)$ nodes as singularities.

\end{prop}
\begin{proof}
Let $P_i$ be a singular point of $X$, and let $F (x) = q(x) + g(x)$ be the affine Taylor development at $P_i$.

The lines $L \subset X$ passing through $P_i$ form a cone with vertex $P_i$ over a complete intersection in $\PP^4$
of dimension $2$, given by the vanishing of the two terms $q(x), g(x)$ of the Taylor development at $P_i$,
which have no common factor. 
Hence these lines  form a 2-dimensional family.

 On the other hand, by  Theorem \ref{dim=4}, the Fano scheme of lines $L \subset X$ has dimension $4$ at each point
 and  the lines $L$ meeting one of the lines $P_i * P_j$ form a family
of dimension
at most $3$: hence the first assertion is proven.

The nodes of $X$ project to singular points of $Y$, and these are nodes of $Y$ by Proposition \ref{suspension}
unless the plane $P_i * L$ intersects $X$ in $L$ plus a double line, or is contained in $X$; 
 the second possibility is excluded since by Theorem \ref{dim=4} there are only finitely many planes in $X$ containing a node, so $L$ is not general.

For the first possibility, we use a 
dimension count. If $\pi \cong \PP^2$, and $P_i \in \pi$, the condition that $\pi \cap X$
contains a line $L' $ through $P_i$ with multiplicity $2$ means that this line $L'$ gives a point $x \in \PP^4$  where (again  $F= q + g$ is the Taylor development of $F$ at $P_i$) both $q$ and $g$ vanish, and $\pi$ gives a line tangent to both $\{ q(x) = 0\}$ and $\{ g(x) = 0\}$. Since $\{ q(x) = g(x) = 0\}$ is a surface, we get a 3-dimensional family of  lines in $\PP^4$
unless the hypersurfaces  $\{ q(x)  = 0\}$ and  $\{  g(x) = 0\}$ are everywhere tangent.
Since the singularities of $X$ are nodes, the quadric $\{ q(x)  = 0\}$ is smooth, and by the theorem of Lefschetz all divisors on it are complete intersections. Hence the hypersurfaces  $\{ q(x)  = 0\}$ and  $\{  g(x) = 0\}$ cannot be everywhere tangent.

Since the family of such lines is 3-dimensional, the family of corresponding planes $\pi$ through $P_i$ 
is also 3-dimensional, hence the family of the residual lines $L$ ($\pi \cap X = 2 L' + L$)
is also 3-dimensional and then $L$ is not general.
 
Hence the nodes $P_i$ project to points where the matrix $A(x)$ has corank $= 1$.

Observe now that, if there is a point $x'$ where the matrix $A(x)$ has corank $3$, i.e., the matrix is zero at this point,
then there is a $2$-dimensional subspace $\Lam$ in $X$ and containing
$L$. But we can observe that then the point $z=0, x = x'$ is a triple point of $X$. 

Now the other assertions will follow  from Theorem \ref{smooth centre} provided we show that the points where the rank of the
matrix $A(x)$ drops by $2$ are $16$ determinantal nodes; equivalently, the quadratic part has rank $3$ at each of them.

Now, the Fitting ideal of the $2 \times 2$-minors of $A(x)$ defines a scheme $\Sigma$ which, since the rank at the points 
of $\Sigma$ equals $1$, is locally defined by three equations (see Proposition \ref{c=2}).
Hence, if $\Sigma$ has dimension $0$, it is a local complete intersection scheme of length $16$ (s Example \ref{Togliatti-quintic}).

To show that generally the scheme $\Sigma$ consists of 16 different points, we
observe first of all that this holds for a general $X$ with $\ga = 15$ nodes
using the irreducibility of the Togliatti quintics,
and since we are exhibiting one such surface with $31$ distinct nodes,   see Theorem  \ref{Togliatti5ics} and Proposition \ref{explicitTogliatti}.

Now, each $X$ with $\ga\le 15$ nodes is a small deformation of a cubic with $15$ nodes, since    the corresponding K3 surfaces are partial smoothings of the one with 15 nodes, in view of Theorem \ref{(2,3)K3}.
\end{proof}

\begin{remark}
i)  The condition that the line $L$ is in the smooth locus of the cubic hypersurface $X \subset \PP^5$
is implied by  (but, as we just saw in the example \ref{Togliatti-quintic}, does not imply) the condition that the  three entries of the matrix $A'(x)$ are linearly independent,
in which case  we get a normal form
$$X : = \{ (z,x) |  x_0 z_0^2 + x_1 z_1^2 + 2 x_2 z_0 z_1 + 2q_0(x) z_0 + 2q_1(x) z_1 + G(x)= 0  \}$$
where we can assume $q_0 (x) = q_0(x_1, x_2, x_3),$  $q_1 (x) = q_1(x_0, x_2, x_3).$

Then the  rank of $A$ never drops  by $3$ provided the  monomial $x_3^2$ appears in $q_0,$ or in $q_1$,
or at least the monomial $x_3^3$ appears in $G$.  If all the three conditions are satisfied,
then the  points where the rank drops by $2$ are  points
of the quadric $x_0 x_1 - x_2^2$ different from  the vertex  $x_0 = x_1 = x_2=0$, 
hence for them we can   assume that either $x_0 =1$,
or $x_1 =1$. 

In general all solutions have $x_0 , x_1, x_2 \neq 0$. Because, for instance,  if $x_0 =0$ then 
$x_2=0$ and then we must have $q_0 (1,0,x_3)=0, q_1 (0,0,x_3)^2 - G(0,1,0, x_3)=0,$
which does not happen for general choice of $q_0, G$. 

Bordering the minor $x_0$ we get the  two equations
$$ x_0 q_1 - q_0 x_2 = x_0 G - q_0^2 $$
and setting $x_2=1$ and multiplying by $x_1$ the first equation, and by $x_1^2$ the second
we obtain the equivalent equations 
$$ x_2 q_1 - q_0 x_1 = x_1 G - q_1^2 ,$$
obtained bordering the minor $x_1$.

Hence all the solutions are gotten, setting  $x_0 = 1$, hence $x_1 = x_2^2$, as solutions of the  two equations
$$ q_1(1, x_2,x_3)  - x_2 q_0(x_2^2, x_2,x_3) = 0, \  G(1,x_2^2, x_2, x_3) -   q_0(x_2^2, x_2,x_3)^2 = 0. $$
 
  In this way  the  points where the rank drops by 2 lie on  the curve of contact $C$,
  obtained as the residual curve to $ \Lam : = \{ x_0 = x_2 = 0\}$ in the sextic curve 
  $$ \Lam +  C : = \{ x_0 x_1 - x_2^2 = x_0 q_1(x)  - x_2 q_0(x) = 0\}.$$
If $C$ is a smooth curve (see \cite{babbage}, Lemma 2.3) we get 16 distinct points 
as soon as the solutions to the above equations occur with multiplicity 1.

This condition pertains  to the Taylor development of the equation of $X$ along the line $L$.

ii) The quintics of Proposition \ref{5discr} are unobstructed.
 
We  use for this the smoothness result of theorem \ref{Togliatti5ics} for general $\ga$:
 $\sF_4(3,\ga)$ is smooth, due to the fact that K3 surfaces are unobstructed.
Then the variety $\sP_{\ga}$ consisting of pairs $(X,L)$, where  $L$ fulfills the properties of  not passing
 through the nodes of $X$,
 not being contained in a 2-dimensional subspace $\Lam \subset X$,  not meeting  lines connecting two nodes,
is smooth. 
  
Fix now coordinates $ (z,x)$ such that $ L = \{ x=0\}$: then we get a variety $\sP'_{\ga}$
such that $\sP_{\ga}$ is a fibre bundle over $\sP'_{\ga}$ with fibre a linear group: hence 
$\sP'_{\ga}$ is also smooth.

The discriminant map, defined on $\sP'_{\ga}$,  yields  a family of quintic surfaces $$Y \in \sF'(5,16 + \ga)$$
which is smooth, because  $\sP_{\ga}$ is a fibre bundle over it (since, for a fixed quintic surface $Y$,
the discriminant matrix $A(x)$  is determined by a $\frac{1}{2}$-even set of nodes plus the choice of a minimal set of generators 
for the associated graded module, see \cite{babbage}). Hence the singularities of $Y$ are unobstructed.
\end{remark} 

Here is an alternative to the last argument  of Proposition \ref{5discr}.

\begin{cor}\label{goryunov}
One obtains, for all $\ga = 0, \dots, 15$ nodal quintic surfaces with $16 + \ga$ nodes by taking
discriminants of cubic hypersurfaces with $\ga$ nodes, projecting from a general line $L \subset X$.
\end{cor}

\begin{proof}
It suffices to consider a particular cubic with 15 double points, for instance 
the Goryunov cubic $X$, described by the equations $\s_1 (x) = 0, \s_3 (x) + z \s_2(x) +  2  z^3 = 0$,
and show that the space of cubics in $\PP^5$ surjects onto $\oplus_1^{15} \CC_{P_i}$,
where $P_1, \dots, P_{15}$ are the nodes of $X$. This was done  in  statement  iv) of Theorem \ref{unobstructed}.

Hence the Goryunov cubic $X$ is unobstructed and one can find deformations of $X$ having any number $\ga \leq 15$ of nodes.
Then we apply the previous proposition.
\end{proof}

The maximum number of nodes, $\nu(Y) = 31$,  is obtained in the case
$\ga=15$ (see \cite{angers}).

Observe that cubics with $15$ nodes in $\PP^5$ have been studied by Togliatti in \cite{togliatticub-1}, \cite{togliatticub-2},
and by Kalker in his Leiden Thesis \cite{kalker}. Togliatti showed that $15$  is the maximum number of nodes, 
and that such hypersurfaces do in fact exist (that $15$  is the maximum is also explained in the next section).

In the following, the first assertion was proven in Theorem \ref{(2,3)K3}:

\begin{theo}\label{cubic4folds}
The Nodal Severi variety  $\sF_4(3,15)$ of cubic hypersurfaces   in $\PP^5$ with $15$ nodes is irreducible. It has  dimension $40$.

\end{theo}
\begin{proof}
The group of projectivities of $\PP^5$ has dimension $35$, while K3 surfaces with $14$ nodes depend on $ 5$ moduli.
\end{proof}

From the above we derive the following

\begin{theo}\label{Togliatti5ics}
The Nodal Severi variety $\sF(5,31)$ of Togliatti quintics, i.e. quintic  surfaces   in $\PP^3$ with $31$ nodes, is
 smooth  irreducible of dimension $24$.
\end{theo}

\begin{proof}
Every Togliatti quintic is (in many ways) the discriminant of a nodal cubic hypersurface $ X \subset \PP^5$
for the projection with centre a line $L \subset X$.

Hence $\sF(5,31)$ is the image of the variety $\sP$ of such pairs $(L,X)$ which fibres onto $\sF_4(3,15)$,
with fibre  an open subset of $F_1(X)$ contained in the  open subset  of lines not intersecting the singular locus of $X$,
not contained in one of  the finite number of planes $\Lam \subset X$, and such that $L$ does not intersect
any of the lines joining two nodes.

By Theorem \ref{cubic4folds}, to show irreducibility of  $\sF(5,31)$, it suffices to show that the  Fano scheme of lines $L$ contained in 
the smooth locus of a nodal cubic hypersurface $ X \subset \PP^5$ with $15$ nodes  is irreducible of dimension $4$.

We have seen in Theorem \ref{dim=4}  that the open set
consisting of lines not intersecting the singular locus of $X$
 is smooth and  that $F_1(X)$ is of dimension exactly $4$ at each point;
its complement consists of the lines passing through the nodes, and these lines correspond to the K3 surfaces
of degree $6$ associated to the Taylor development of the equation of $X$ at the nodes.

Therefore $F_1(X)$ is nonsingular in codimension $1$. Moreover, since it is the zero scheme of the section of
a rank $4$ bundle, which has exactly codimension $4$ in the Grassmannian, it is a local complete intersection.
Hence it is normal and, being also connected, it is irreducible.

Concerning the dimension, $\sF(5,31)$ has dimension at least $55 - 31= 15 + 9$, hence at least $9$ moduli;
and, up to projectivities,  it is dominated by a family of dimension $9$ ($5$ moduli for the cubic hypersurfaces,
plus $4$ parameters for the dimension of $F_1(X)$). Hence $\sF(5,31)$ has dimension exactly $24$.

 Moreover, since  the morphism $\sP \ra \sF(5,31)$  is a fibre bundle, we obtain that $\sF(5,31)$
is smooth if and only if $\sP$ is smooth, equivalently, if and only if $ \sF_4(3,15)$ is smooth,
which follows from the fact that K3 surfaces are unobstructed.
\end{proof}

\section{An explicit Togliatti quintic}
In this section we enumerate the 2-planes on the Goryunov--Kalker cubic hypersurface $T(4)\subset\PP^5$ and exhibit a discriminant quintic surface with $31$ nodes by projecting from an appropriate line contained in $T(4)$.

 \subsection{Linear 2-spaces on the GK-cubic 4-fold}
\begin{prop}\label{prop!2-planes} The Goryunov--Kalker cubic hypersurface $T(4)\subset\PP^5 \subset \PP^6$ contains $60$ linear subspaces of maximal dimension $2$.
These split into two $\Sn_6$-orbits represented by
\begin{gather*}
(x_1+x_2=x_3+x_4=x_5+x_7= x_6+x_7=0) \text{ and } \\ 
(x_1+x_2=x_3+x_4=x_5+x_6=x_7=0).
\end{gather*}
The first orbit consists of 45 planes. Each such plane passes through four ordinary double points and there are 12 such planes passing through each ordinary double point. The second orbit consists of 15 planes, which are inherited from the Segre cubic 3-fold $\Sigma(3)\cong T(4)\cap(x_7=0)$ and which are
disjoint from $\Sing T(4)$.
\end{prop}
\begin{proof}
We consider the $2$-planes in $\PP^6$ as points in the Grassmannian $G:=\Gr(3,7)$. 

It is immediate to verify that the above $60$ planes are contained in $T(4)$. However, the top Chern class 
of the vector bundle $\sE$  such that  the Fano scheme $F_2(T(4))\subset G$ is the zero set of a section of $\sE$
is bigger than $60$.

It turns out,  from the computations below,  that the Fano scheme $F_2(T(4))\subset G$ of $2$-planes contained in $T(4)$ is zero-dimensional and everywhere non-reduced. 

Let us  find the points of the set $F_2(T(4))$. 

Let $U_{ijk}$ denote the affine chart of $G$ whose points are represented by $3\times 7$ matrices of rank 3 in row echelon form where columns $i,j,k$ form the $3\times 3$ identity matrix. For example, a point in $U_{123}\cong \mathbb{A}^{12}$ is represented by a matrix
\[\renewcommand{\arraystretch}{1.0}
M:=\begin{pmatrix}
1 & 0 & 0 & a_{14} & a_{15} & a_{16} & a_{17} \\
0 & 1 & 0 & a_{24} & a_{25} & a_{26} & a_{27} \\
0 & 0 & 1 & a_{34} & a_{35} & a_{36} & a_{37}
\end{pmatrix}.\]

Given a point in $U_{ijk}$ represented by the matrix $M$, the corresponding $2$-plane is parametrized by the map (here the points of $\PP^6$ are seen as row vectors)
\[\PP^2\to\PP^6,\ (u,v,w)\mapsto(u,v,w)M.\]
Thus the plane $M$ lies in $T(4)$ if and only if the equations of $T(4)$ vanish identically when evaluated at $(u,v,w)M$.
This gives conditions on the coefficients $a_{mn}$ defining the scheme $F_2(T(4))\cap U_{ijk}$ in $\mathbb{A}^{12}$. We use the computer algebra system MAGMA \cite{MAGMA}  to find all points of this scheme\footnote{see the script \cite[MAGMA 24]{Scripts}.}.

In principle, we should need to perform this process for each of the affine charts $U_{ijk}$ and then exclude 2-planes that appear more than once, but since $\Sn_6$ acts on $T(4)$ by permuting the coordinates $x_1,\dots,x_6$, the affine charts of $G$ split into two orbits. Thus it suffices to consider the charts $U_{123}$ and $U_{127}$ and then apply the $\Sn_6$-action to obtain all $2$-planes on $T(4)$.
We obtain the two orbits of $2$-planes described in the statement of the Proposition.

The remaining part of the Proposition may be checked directly from the equations.
\end{proof}

\subsubsection{Lines on the GK-cubic 4-fold}

To construct a quintic surface in $\PP^3$ with 31 nodes, we search for lines $l$ contained in $T(4)$ satisfying the following two conditions:
\begin{enumerate}
\item $l$ is disjoint from every line $L$ through a pair of nodes on $T(4)$;
\item $l$ is not contained in a 2-plane of $F_2(T(4))$.
\end{enumerate}
These are open conditions on the Fano scheme of lines on $T(4)$.
\begin{lemma}\label{lem!line}
The line \[x_1-9x_5-2x_6 = x_2-7x_5-2x_6 = x_3-5x_5 = x_4+8x_5+x_6 = 0\]
satisfies the above conditions.
\end{lemma}
\begin{proof} The proof of the statement, which we verified with a computer calculation,  can also be obtained by hand, checking the various intersections and showing that the equation of $T(4)$ is contained in the above ideal. 
\end{proof}

How did we find this line? Consider the Grassmannian of lines in $\PP^6$ and parametrise affine charts $U_{ij}\cong \mathbb{A}^{10}$ in a similar way to the above Proposition \ref{prop!2-planes}. A typical point of $U_{12}$ corresponds to a matrix
\[\begin{pmatrix}
1&0&b_{13}&b_{14}&b_{15}&b_{16}&b_{17}\\
0&1&b_{23}&b_{24}&b_{25}&b_{26}&b_{27}
\end{pmatrix}\]
and the equations defining $T(4)$ induce conditions on $b_{13},\dots,b_{26}$ defining the affine chart of the Fano scheme of lines. We use the computer to search for integral points of height $\le10$ in this affine chart and check whether the required properties are fulfilled. The given line  is the point  corresponding to
$\left(\begin{smallmatrix}9&7&5&-8&1&0&-7\\
2&2&0&-1&0&1&-2\end{smallmatrix}\right)$
in the Grassmannian.

\begin{prop}\label{explicitTogliatti}
The quintic hypersurface $V\subset\PP^3$ defined by the vanishing of the determinant of the symmetric matrix
\[
\begin{pmatrix}
32y_1-24y_3+15y_4 & 2y_1-7y_3-3y_4 & m_{13} \\
&-4y_3-3y_4 & m_{23} \\
\text{sym}&&\frac13(s_3-\frac14s_1^3)
\end{pmatrix}
\]
has $31$ nodes and no other singularities. Here the notation is $s_1=\sum y_i$, $s_3=\sum y_i^3$
and
\[m_{13}=9y_1^2+7y_2^2+5y_3^2-8y_4^2-\frac72s_1^2,\quad m_{23}=2y_1^2+2y_2^2-y_4^2-s_1^2.\]
\end{prop}
\begin{proof}
We project $T(4)$ from the line obtained in Lemma \ref{lem!line}. The coordinates on the projected $\PP^3$ are
\[y_1:=x_1-9x_5-2x_6,\  y_2:=x_2-7x_5-2x_6,\ y_3:=x_3-5x_5,\ y_4:=x_4+8x_5+x_6.\]
Thus we may eliminate $x_1,\dots,x_4$ using $y_1,\dots,y_4,x_5,x_6$ and write the equation of $T(4)$ as a quadratic
form in $x_5$ and $x_6$. The matrix of this quadratic form is the one given above.

The nodes were calculated using a computer, but it should be possible to prove the assertion  without a computer. The 16 
`determinantal' nodes arising as the points where the corank of the symmetric matrix is 2 
are defined over a complicated field extension of $\mathbb{Q}$. 
\end{proof}

\section{Unobstructed surfaces: a Togliatti quintic and  Barth's sextic.}
\label{append_unobstructedness}

Let $X$ be a nodal hypersurface of degree $d$ in $\mathbb P^{n}$.
The short exact sequence
\[0\to \mathcal J(d)\to \mathcal O_{\PP^n}(d)\to\mathcal O/\mathcal J=:\mathcal{T}\to 0\]
induces the evaluation map 
\[\mathrm{ev}\colon H^0(\mathbb P^n,\mathcal O(d)) \longrightarrow H^0(\mathcal{T}).\]
Here $\mathcal J$ is the Jacobian ideal of $X$ and $\mathcal{T}$ is supported on the singular locus of $X$. 
Indeed, since $X$ has ordinary double points, $H^0(\mathcal{T})=\bigoplus_{P\in 	\Sing(X)}\mathbb C_P$.
According to Burns--Wahl  the deformations are \textit{unobstructed} if the evaluation map is surjective, 
i.e. the nodes impose linearly independent conditions.

\subsection{A special Togliatti quintic}
A Togliatti  quintic $X_T\subset \mathbb P^3$ is defined by
\begin{gather*}
2(x^5 -5x^4w -10x^3y^2 -10x^2y^2w +20x^2w^3 +5xy^4 -5y^4w +20y^2w^3 - 
\\16w^5 ) +5 (x^2+y^2+bz^2+zw+dw^2)^2z=0
\end{gather*}
where   $b=- \frac{5-\sqrt 5}{20}$ and $d=-(1+\sqrt{5})$.
The following script shows that the 31 nodes are defined over $\QQ(\sqrt{5+2\sqrt 5})$ and that $X_T$ is unobstructed.

\begin{code_magma}
Q:=Rationals();
RR<u1>:=PolynomialRing(Q);
K1<t>:=SplittingField(u1^2-5); 
RR<u2>:=PolynomialRing(K1);
K<t2>:=SplittingField(u2^2 - (2*t+ 5));
P3<x,y,z,w>:=ProjectiveSpace(K,3);

b:=-(5-t)/20;
d:=-(1+t); 

P:= (x^5 -5*x^4*w -10*x^3*y^2 -10*x^2*y^2*w +20*x^2*w^3 +
5*x*y^4-5*y^4*w +20*y^2*w^3 -16*w^5 );
Q:= (x^2+y^2+b*z^2+z*w+d*w^2);
F:= 2*P+5*z*Q^2;
TQ:=Scheme(P3,F);
Nodes:=Set(SingularPoints(TQ)); 

//We check that the singularities are the rights ones
#Nodes; HasSingularPointsOverExtension(TQ); 
for pt in Nodes do if not IsNode(pt) then "wrong!"; end if; end for;

L:=LinearSystem(P3,5); // The space of quintic in P^3
M:=[];
for section in Sections(L) do // 56 sections
	line:=[]; 
	for point in Nodes do
		Append(~line, Evaluate(section,[point[j]: j in [1..4]])); 
	end for; 
	Append(~M, line);
end for; 
Rank(Matrix(M)); // The Rank is 31, so the map is surjective
\end{code_magma}

\subsection{Barth's sextic}\label{barthunobstructed}
The Barth sextic $X_B\subset \mathbb P^3$ is defined by
\[(\tau^2x^2-y^2)(\tau^2y^2-z^2)(\tau^2z^2-x^2)-\frac 14(2\tau+1)w^2(x^2+y^2+z^2-w^2)^2\]
where   $\tau: = \frac{1}{2}(1+\sqrt 5)$ is the  golden ratio. The following script verifies that $X_B$ is unobstructed.

\begin{code_magma}
Q:=Rationals();
RR<z>:=PolynomialRing(Q);
K<t>:=SplittingField(z^2-z-1);
P3<x,y,z,w>:=ProjectiveSpace(K,3);

F:=(t^2*x^2-y^2)*(t^2*y^2-z^2)*(t^2*z^2-x^2)-
     ((2*t+1)/4)*w^2*(x^2+y^2+z^2-w^2)^2;
BS:=Scheme(P3,F);
Nodes:=Set(SingularPoints(BS));

//We check that the singularities are the rights ones
#Nodes; HasSingularPointsOverExtension(BS); 
for pt in Nodes do if not IsNode(pt) then "wrong!";end if; end for;

L:=LinearSystem(P3,6); // The space of sextic in P^3
M:=[];
for section in Sections(L) do // 84 sections
	line:=[]; 
	for point in Nodes do
		Append(~line, Evaluate(section,[point[j]: j in [1..4]])); 
	end for; 
	Append(~M, line);
end for; 
Rank(Matrix(M)); // The Rank is 65, so the map is surjective
\end{code_magma}

\section[Geometry and deformations of  nodal Segre cubics]{Geometry and deformations of  nodal Segre cubic hypersurfaces}
\label{append_Segre_cubic}

Let $n=2k$. Recall that the Segre cubic hypersurfaces are defined by
\[X_{S}\colon(x_1+\dots+x_n=x_1^3+\dots+x_n^3=0)\subset\PP^{n-1}\]
and their nodes form the $\Sn_n$-orbit of $(1^k:-1^k)$ where $\alpha^s$ means $\alpha$ is repeated $s$ times.

The Goryunov--Kalker (GK) cubic hypersurfaces are defined by
\[X_{GK}\colon(x_1+\dots+x_n+2x_{n+1}=x_1^3+\dots+x_n^3+2x_{n+1}^3=0)\subset\PP^{n}\]
and their nodes are the $\Sn_n$-orbit of the point $(1^{k-1}:-1^{k+1}:1)$.

If $n=4$, then the Segre hypersurface is a plane cubic curve with three nodes i.e.~a triangle!  Whereas, the GK surface is the Cayley cubic surface. So we will restrict ourselves to the case  $n=2k\ge6$.

In this section we prove Theorem \ref{local-uniqueness} namely we show
\begin{enumerate}
\item The Segre cubic 3-fold ($n=6$) is unique and is unobstructed.
\item The Goryunov--Kalker cubic 4-fold ($n=6$)  has non-trivial equisingular deformations and is unobstructed.
\item The Segre cubic 5-fold and the Goryunov--Kalker 6-fold are unobstructed and locally rigid.
\item All other Segre cubics and Goryunov--Kalker cubics are locally rigid and have obstructed deformations.
\end{enumerate}

Let $H$ be the hyperplane $(x_1+\dots+x_n=0)\subset \PP^{n-1}$ respectively $(x_1+\dots+x_n+2x_{n+1}=0)\subset\PP^{n}$.
The short exact sequence
\begin{equation}\label{eq!jacobian}
0\to \sJ(3)\to\mathcal O_H(3)\to\mathcal O_H/\sJ=:\mathcal{T}\to0
\end{equation}
induces the evaluation map 
\[
	\ev\colon H^0(H,\mathcal O_{H}(3))\to H^0(\mathcal{T}).
\]
Here $\sJ$ is the Jacobian ideal of $X\subset H$ and $\sT$ is supported on the singular locus of $X$. 
Indeed, since $X$ has ordinary double points, $H^0(\mathcal{T})=\bigoplus_{P\in\Sing X}\CC_P$.
According to Burns--Wahl (\cite{burns-wahl}) the deformations are unobstructed if the evaluation map is surjective, 
i.e. the nodes impose linearly independent conditions on the space of cubic hypersurfaces.

We say that $X$ is \textit{locally rigid} if any equisingular deformation of $X$ is projectively equivalent to $X$ itself. 
This condition is fulfilled if the kernel of the evaluation map is spanned by the image of $X$ under the action of the space of infinitesimal automorphisms. This occurs in every case except the GK-cubic 4-fold, where $H^0(\sJ(3))$ contains an extra 5-dimensional space of cubics.

The evaluation map is $\Sn_n$-equivariant and we prove the theorem by decomposing  the various vector spaces
into irreducible $\Sn_n$-representations and applying Schur's lemma. 

The proof is organized as follows. In Section \ref{section!rep} we recall some well-known facts about the irreducible representations of 
the symmetric group $\Sn_n$.
In Section \ref{section!Segre} and \ref{section!GK} we describe the evaluation map of the Segre cubics and the 
Goryunov--Kalker cubics, and we prove the theorem.
In Section \ref{section!points} we give the decomposition of $H^0(\mathcal T)$ into irreducible representations.

\subsection{Representation theory of the symmetric group} \label{section!rep}
We briefly recall some background on Young tableaux, tabloids and Specht modules. 
We refer to James' book \cite{JamesBook} for proofs and further details.

There is a correspondence between representations of the symmetric group $\Sn_n$ and partitions 
$\lambda:=(\lambda_1,\dots,\lambda_m)$ of $n$,  where $\lambda_1\geq \dots \geq\lambda_m$.
To a partition, one associates a Young diagram, consisting of $m$ rows each containing $\la_i$ boxes.
A Young tableau is a Young diagram with shape $\lambda$ where the boxes are labelled by $1,\dots,n$. A Young tableau is called standard if it has  strictly increasing rows and columns.

\begin{defin} 
A \textit{Young tabloid} of shape $\lambda$ is an equivalence class of Young tableaux of shape $\lambda$, with respect to the relation where two tableaux are equivalent if they differ by a permutation of the labelling of the rows. 
\end{defin}

The symmetric group $\Sn_n$ acts on the set of Young tableaux of shape $\lambda$ by permuting the labelling.
The \textit{column stabiliser} $\Cn_T$ of a Young tableau $T$ is the subgroup of $\Sn_n$ which permutes the labelling of $T$  leaving the columns invariant.

\begin{example} 
$T:=\begin{smallmatrix}
1 & 5 & 6 \\
2 & 4 \\
3 
\end{smallmatrix}$
is a standard Young tableau of shape $(3,2,1)$. As a tabloid, $T$ is equivalent to 
$\begin{smallmatrix}
6 & 1 & 5 \\
4 & 2 \\
3 
\end{smallmatrix}$. The column stabiliser of $T$ is $\Sn(\{1,2,3\})\times \Sn(\{4,5\})\cong \Sn_3\times \Sn_2$.
\end{example}

Fix a partition $\lambda$ and consider the $\CC$-vector space $W$ generated by Young tabloids of shape $\lambda$. 
That is, the elements are formal $\CC$-linear combinations of Young tabloids. 

\begin{defin} 
For a fixed Young tableau $T$,
 we define the \textit{Specht vector} $\es_T$ by
\[\es_T=\sum_{\sigma\in \Cn_T} \sign(\sigma)\cdot [\sigma(T)]\]
where $[X]$ denotes the tabloid corresponding to the tableau $X$.

The \textit{Specht module}  $S_\lambda$ is the linear subspace  of $W$ spanned by the $\Sn_n$-orbit of $\es_T$ where $T$ has shape $\la$.
This is a cyclic $\Sn_n$-module generated by all Specht vectors of shape $\lambda$.
\end{defin}

\begin{example} The column stabiliser of the tableau $\begin{smallmatrix}1 & 2 & 3 & 4 \\ 5 & 6 \end{smallmatrix}$ is isomorphic 
to $\ZZ/2\times\ZZ/2$ generated by permutations $(15)$ and $(26)$. The corresponding Specht vector is
\[\es_{\begin{smallmatrix}
1 & 2 & 3 & 4 \\ 5 & 6 \end{smallmatrix}}
=\left[\begin{smallmatrix}
1 & 2 & 3 & 4 \\ 5 & 6 \end{smallmatrix}\right]-
\left[\begin{smallmatrix}
5 & 2 & 3 & 4 \\ 1 & 6 \end{smallmatrix}\right]-
\left[\begin{smallmatrix}
1 & 6 & 3 & 4 \\ 5 & 2 \end{smallmatrix}\right]+
\left[\begin{smallmatrix}
5 & 6 & 3 & 4 \\ 1 & 2 \end{smallmatrix}\right]
\]
As a sum of tabloids, the right-hand side is independent of row-ordering. 
The left-hand side depends on the tableau, and therefore the row-ordering too e.g. 
\[
\es_{\begin{smallmatrix}
2 & 1 & 3 & 4 \\ 5 & 6 \end{smallmatrix}}
=\left[\begin{smallmatrix}
2 & 1 & 3 & 4 \\ 5 & 6 \end{smallmatrix}\right]-
\left[\begin{smallmatrix}
5 & 1 & 3 & 4 \\ 2 & 6 \end{smallmatrix}\right]-
\left[\begin{smallmatrix}
2 & 6 & 3 & 4 \\ 5 & 1 \end{smallmatrix}\right]+
\left[\begin{smallmatrix}
5 & 6 & 3 & 4 \\ 2 & 1 \end{smallmatrix}\right]
\] \end{example}

\begin{theorem}\cite[\S\S4--7]{JamesBook}
In characteristic 0, the Specht modules $S_{\lambda}$ give all the irreducible representations of $\Sn_n$. A basis of $S_\lambda$ is given by the Specht vectors corresponding to standard tableaux of shape $\lambda$.
\end{theorem}

\begin{remark} 
The partition $\lambda=(n)$ corresponds to the trivial representation $V_{(n)}$ of $\Sn_n$.
The symmetric group $\Sn_n$ acts in a natural way on  the subspace $(\sum_{i=1}^n z_i =0) \triangleleft \CC^n$ by permuting the coordinates. This  subspace is an irreducible representation called the \textit{standard representation} $V_{(n-1,1)}$ of $\Sn_n$. The partition $n=(n-k)+1+\cdots +1$ corresponds to the $k$-th exterior product of the standard representation
(cf. \cite[\S 4.1]{FultonHarris}).
\end{remark}

\subsection{Representation theory  for the Segre cubics}\label{section!Segre}
Recall that $H$ is the hyperplane $(\sum_{i=1}^nx_i=0)\subset\PP^{n-1}$. 
The symmetric group $\Sn_n$ acts in a natural way on $\PP^{n-1}$, so the vector space $V:=H^0(\mathcal O_H(1))$
is  isomorphic to the standard representation $V\cong V_{(n-1,1)}$ of $\Sn_n$.

\subsubsection{Representations of cubic polynomials}
We decompose $\Sym^3 V$ into irreducible representations and present the generators in Table \ref{table!segre-cubics} below.

First we decompose $\Sym^3V$ into symmetric and antisymmetric representations, with basis $f_{ijk}:=x_ix_jx_k$ for $1\le i<j<k\le n$ and $g_{ij}:=x_ix_j(x_i-x_j)$ for $1\le i<j\le n$ respectively. Indeed, since $\Sym^3V$ is $\Sym^3\left<x_1,\dots,x_n\right>$ modulo $(\sum x_i=0)$, we have
\[2x_a^3=x_a(\sum_{j\ne a}x_j)^2-x_a^2(\sum_{j\ne a}x_j)=\sum_j (x_ax_j^2-x_a^2x_j) \mod \left<f_{klm}\right>\]
and
\[x_a^2x_b=-x_ax_b\sum_{j\ne a}x_j=-x_ax_b^2\mod \left<f_{klm}\right>.\]
Thus $\Sym^3 V=\left<f_{ijk}\right>\oplus\left<g_{ij}\right>$. We further decompose these two summands into irreducible representations.

\begin{prop}\label{RepCubicsS}
Let $n\ge6$.
 Then 
\[\Sym^3 V\cong V_{(n)}\oplus V_{(n-1,1)}\oplus V_{(n-2,2)}\oplus V_{(n-3,3)}\oplus V_{(n-1,1)}\oplus V_{(n-2,1,1)}\,.\]
 The generators of each summand are presented in Table \ref{table!segre-cubics}.
\end{prop}

\renewcommand{\arraystretch}{1.5}

{
\begin{table}[ht]\caption{Cubics on the Segre hypersurface}\label{table!segre-cubics}

\hspace*{-1cm}\begin{tabular}{clcc}
Partition $\lambda$ & Tableau $T$ & Generator & Dimension \\
\hline
$(n) $& $\begin{smallmatrix}c_1 &\dots & c_n\end{smallmatrix}$ & $\mathsf{f}_T=\sigma_3(x_C)$ & $1$ \\
$(n-1,1)$ &$ \begin{smallmatrix}a & c_1 & \dots & c_{n-2} \\ b\end{smallmatrix}$ &
$\mathsf{f}_T=(x_a-x_b)\sigma_2(x_C) $& $n-1$ \\
$(n-2,2)$ & $\begin{smallmatrix}a_1 & a_2 & c_1 & \dots & c_{n-4} \\ b_1 & b_2 \end{smallmatrix} $&
$\mathsf{f}_T=(x_{a_1}-x_{b_1})(x_{a_2}-x_{b_2})\sigma_1(x_C) $& $\frac{n(n-3)}2$ \\
$(n-3,3) $& $\begin{smallmatrix}a_1 & a_2 & a_3 & c_1 & \dots & c_{n-6} \\ b_1 & b_2 & b_3\end{smallmatrix}$ &
$\mathsf{f}_T=\prod_{i=1}^3(x_{a_i}-x_{b_i}) $& $\frac{n(n-1)(n-5)}6$ \\
\hline
$(n-1,1)$ & $\begin{smallmatrix}a & c_1 & \dots & c_{n-2} \\ b\end{smallmatrix} $&
$ \mathsf{g}_T=
(x_a-x_b)(2x_ax_b+(x_a+x_b)\sigma_1(x_C)-s_2(x_C)) $& $n-1$ \\
$(n-2,1,1) $& $\begin{smallmatrix}a_1 & c_1 & \dots & c_{n-3} \\ a_2 \\ a_3 \end{smallmatrix} $&
$\mathsf{g}_T=(x_{a_1}-x_{a_2})(x_{a_2}-x_{a_3})(x_{a_3}-x_{a_1})$ & $\binom {n-1}2$
\end{tabular}
\end{table}}

Each row of Table \ref{table!segre-cubics} corresponds to a summand $V_\lambda$ of the decomposition: the first four rows are spanned by $f_{ijk}$ and the bottom two by $g_{ij}$. The generator column refers to a cyclic generator of $V_\lambda$ as an $\Sn_n$-module, i.e.~the $\Sn_n$-orbit of such spans $V_\lambda$ as a vector space. 
 We use $\sigma_d(x_C)$ (resp.~$s_d(x_C)$) to denote the elementary symmetric polynomial of degree $d$ 
 (resp.~the $d$-th power sum) in the variables $x_c$ indexed by $c$ in $C$. Here $C$ is the set $\{c_i\}$.

\begin{proof}
In Lemma \ref{lemma!inv} we show that the set of generators corresponding to standard Young tableaux of shape $\lambda$ form a basis for the irreducible representation $V_\lambda$. Since each row of Table \ref{table!segre-cubics} gives a distinct irreducible representation, and the sum of their dimensions is $\binom{n+1}{3}= \dim \Sym^3 V$, we are done.
\end{proof}

\begin{remark}\label{determinant} The generators for the first four rows  of Table \ref{table!segre-cubics} were obtained by analogy with the construction of the Specht module $S_{(n-p,p)}$. Note that $\sum_{\tau\in\Cn_T}\prod_{i=1}^px_{\tau(a_i)}$ is equal to $\prod_{i=1}^p(x_{a_i}-x_{b_i})$. For $V_{(n-2,1,1)}$, the generator is the Vandermonde determinant
\[\gs_T=\det\left(\begin{smallmatrix}1&1&1\\x_{a_1} & x_{a_2} & x_{a_3} \\x_{a_1}^2 & x_{a_2}^2 & x_{a_3}^2\end{smallmatrix}\right)=g_{a_1a_2}+g_{a_2a_3}+g_{a_3a_1},\]
a well-known skew polynomial of degree 3. For the skew summand $V_{(n-1,1)}$ we computed the orthogonal complement to $V_{(n-2,1,1)}$ inside $\left<g_{ij}\right>$:
\[\gs_T=g_{a_1a_2}-\sum_{i<j<k}((g_{ij}+g_{jk}+g_{ki})\cdot g_{a_1a_2})(g_{ij}+g_{jk}+g_{ki})\,.\]

\end{remark}

\begin{lemma}\label{lemma!inv}
Let $\lambda$ be a partition of $n$ as in Proposition \ref{RepCubicsS}, let
$T$ be a Young tableau of shape $\lambda$, and let $\es_T$ be the corresponding generator.

\begin{enumerate}
\item[i)] If $\tau\in \mathfrak C_T \subset \Sn_n$ preserves the columns of $T$, then $\tau(\es_T)=\es_{\tau(T)}=\sign(\tau)\es_T$.
\item[ii)] If $\gamma$ in $\Sn_C\subset\Sn_n$ permutes only the $C$-coordinates
 then $\gamma(\es_T)=\es_{\gamma(T)}=\es_T$.
\item[iii)] For any permutation $\rho$ in $\Sn_n$, we have $\rho(\es_T)=\es_{\rho(T)}$. 
\item[iv)] The set of all $\es_T$'s, with $T$ standard gives a basis of the representation $V_\lambda$.
\end{enumerate}
\end{lemma}

\begin{proof}
We give an elementary proof for the third row $\lambda=(n-2,2)$ of Table \ref{table!segre-cubics}. The 
proof in the other cases  works in exactly the same way, but to write everything out is repetitive. 
Here we consider Young tableaux $T$ of shape $(n-2,2)$:
\renewcommand{\arraystretch}{1}
\[T=\  \begin{matrix}a_1 & a_2 & c_1 & \dots & c_{n-4} \\ b_1 & b_2 \end{matrix}\]

\noindent $i)$ if $\tau$ preserves the columns then $\tau(\es_T)=(x_{\tau(a_1)}-x_{\tau(b_1)})(x_{\tau(a_2)}-x_{\tau(b_2)})\sigma_1(x_C)=\es_{\tau(T)}=\sign(\tau)\es_T$.

\noindent  $ii$--$iii)$ Using $i)$, it suffices to prove the statements for transpositions of the first row of $T$. There are three cases. If $\rho$ is a permutation of the $C$-coordinates, then clearly $\rho(\es_T)=\es_{\rho(T)}$ because $\sigma_1(x_C)$ is invariant under permutations of $C$. 
If $\rho$ swaps say $a_1$ with $c_1$, then $\rho(\es_T)=(x_{c_1}-x_{b_1})(x_{a_2}-x_{b_2})
\sigma_1(x_{\rho(C)})=\es_{\rho(T)}$.
If $\rho$ swaps two elements of $A$, then $\rho(\es_T)=(x_{a_2}-x_{b_1})(x_{a_1}-x_{b_2})\sigma_1(x_C)=\es_{\rho(T)}$. 

\noindent $iv)$ Suppose now $T$ is standard.
Then the last two monomials of $\es_T$ in lexicographic order 
 are $x_{b_1}x_{b_2}x_{c_{n-3}}+x_{b_1}x_{b_2}x_{c_{n-4}}$. Since the entries of $T$ increase along rows and columns, these trailing terms are pairwise distinct. 
Thus the standard $\es_T$ are linearly independent. 

There is a general algorithm applying the so-called Garnir relations recursively to write any nonstandard generator $\es_T$ as a linear combination of standard generators (cf.~\cite[\S7]{JamesBook}).
 In our case the algorithm terminates particularly quickly because $T$ has only two rows.
\begin{description}[labelindent=0cm, leftmargin=.5cm]
\item[Step 1] Suppose $\es_T$ is a nonstandard generator. Using  $i)$ we may assume the columns are ordered and using $ii)$ we may assume that the $C$-part is ordered.
\item[Step 2] If $a_2>c_1$ then we use the Garnir relation
$\es_T-(a_2,c_1)\es_T+(a_2,c_1,b_2)\es_T=0$ and possibly column permutations to write $\es_T$ as a linear combination of two generators $\es_{T_1}$ and $\es_{T_2}$ with $a_2(T_i)<c_1(T_i)$ for $i=1,2$. We then apply Step 3 to each of these.
\item[Step 3]If $a_1>a_2$ then we use the Garnir relation
$\es_T-(a_1,a_2)\es_T+(a_1,a_2,b_1)\es_T=0$ and  possibly column/$C$-part permutations to write $\es_T$ as a linear combination 
of  generators having $a_1<a_2$. 

We then restart from  Step 1  with each of the new generators.
\end{description}

We prove directly that the Garnir relations hold for the third row  of Table \ref{table!segre-cubics}. For step 2, we first write $C':=C\smallsetminus\{c_1\}$ and observe that
\begin{multline*}
(x_{a_2}-x_{b_2})\sigma_1(x_C)-(x_{c_1}-x_{b_2})\sigma_1(x_{C'\cup\{a_2\}})
+(x_{c_1}-x_{a_2})\sigma_1(x_{C'\cup\{b_2\}})=\\
=\left((x_{a_2}-x_{b_2})-(x_{c_1}-x_{b_2})+(x_{c_1}-x_{a_2})\right)\sigma_1(x_{C'})\\
+(x_{a_2}-x_{b_2})x_{c_1}-(x_{c_1}-x_{b_2})x_{a_2}+(x_{c_1}-x_{a_2})x_{b_2} 
=0+0
\end{multline*}
Multiplying the first line by $x_{a_1}-x_{b_1}$, gives $\es_T-(a_2,c_1)\es_T+(a_2,c_1,b_2)\es_T=0$.
In Step 3, the $C$-part of $T$ is unchanged so we get a common factor of $\sigma_1(x_C)$:
\begin{multline*}
\es_T-(a_1,a_2)\es_T+(a_1,a_2,b_1)\es_T=\\
=((x_{a_1}-x_{b_1})(x_{a_2}-x_{b_2})-(x_{a_2}-x_{b_1})(x_{a_1}-x_{b_2})+\\
(x_{a_2}-x_{a_1})(x_{b_1}-x_{b_2}))\sigma_1(x_C)  =0
\end{multline*}
\end{proof}

\subsubsection{The evaluation map and its kernel}\label{section!ev-segre}

Let $X_S\subset\mathbb P^{n-1}$ be the Segre cubic and let 
 $H^0(\mathcal{T})=\bigoplus_{P\in\Sing X_S} \CC_P$ be the vector space of global sections of the skyscraper sheaf $\sT$.  
A vector in $H^0(\sT)$ is a formal $\CC$-linear combination of points. We denote the basis vector $1_P$ in $\CC_P$ by a $2\times k$ labelled array
\[\renewcommand{\arraystretch}{1.0}
\begin{array}{cccccc}
p_1  & \dots & p_k \\
\hline
p_{k+1} & \dots & p_n
\end{array}
\begin{array}{cc}
\leftarrow\text{positive} \\
\leftarrow\text{negative} 
\end{array}
\]
where $\{p_1,\dots,p_n\}=\{1,\dots,n\}$. The first row  corresponds to the coordinates equal to $1$, and the second row  to coordinates equal to $-1$. This implicitly involves a choice of positive and negative row of the array. Swapping the rows multiplies by $-1$ in $\CC_P$. Clearly, any permutation of the $x_i$ which preserves the rows corresponds to the same point
in $\PP^{n-1}$, and so does any permutation which swaps the two rows (after applying the $\CC^*$-action on $\PP^{n-1}$). 
For example, for $n=6$ we have
\begin{multline*}
\renewcommand{\arraystretch}{0.9}\arraycolsep=3pt
\begin{array}{ccc}
1 & 3 & 4 \\
\hline
2 & 5 & 6
\end{array} = 
\begin{array}{ccc}
3 & 4 & 1 \\
\hline
2 & 6 & 5
\end{array}
= (1:-1:1:1:-1:-1)=\\
\renewcommand{\arraystretch}{0.9}\arraycolsep=3pt
-1\cdot(-1:1:-1:-1:1:1)=-1\cdot\begin{array}{ccc}
2 & 5 & 6 \\
\hline
4 & 1 & 3
\end{array}
\end{multline*}

\begin{remark} Each permutation $\sigma\in \Sn_n$ is a bijection of the set $\Sing X_S$, but \textit{not}
of the set of vectors $\{1_P \mid P\in \Sing X_S\}$, e.g.  let $\sigma=(16)(24)(35)$ 
and $P=(1:-1:1:1:-1:-1)$, then $\sigma (1_P)= -1_P$.
\end{remark}

\begin{lemma}\label{lemma!equiv} The evaluation map $\ev\colon\Sym^3 V\to H^0(\sT)$
is $\Sn_n$-equivariant.
\end{lemma}
\begin{proof}

Clearly it is enough to prove the statement for a generator $\es_T$ in $\Sym^3 V $.
We note that if  $\sigma(1_P)= - 1_Q$ then $ \es_T(P) \cdot \sigma(1_P)=
 \es_T(-\sigma^{-1}(Q)) \cdot (-1_Q) = \es_T(\sigma^{-1}(Q)) \cdot 1_Q $, since $\es_T$ has degree 3.
Analogously, if  $\sigma(1_P)=  1_Q$ then $ \es_T(P) \cdot \sigma(1_P)=
 \es_T(\sigma^{-1}(Q)) \cdot 1_Q $. Therefore, 
\begin{multline*}
 \sigma(\ev(\es_T))=  \sum_P  \es_T(P) \cdot \sigma(1_P)= \sum_P  \es_T(\sigma^{-1}P) \cdot 1_P
=\\
 \sum_P  \es_{\sigma(T)}(P) \cdot 1_P= \ev(\sigma(\es_T))\,.\end{multline*}
\end{proof}

\begin{prop}\label{prop!segre-kernel-image}
The evaluation map $\ev\colon\Sym^3 V\to H^0(\mathcal{T})$  has 
kernel isomorphic to $V_{(n)}\oplus V_{(n-1,1)}\oplus V_{(n-2,2)}\oplus V_{(n-2,1,1)}$ and 
image isomorphic to $V_{(n-1,1)}\oplus V_{(n-3,3)}$.
\end{prop}

\begin{proof}
By Schur's Lemma and Lemma \ref{lemma!equiv} it suffices to consider the evaluation map on generators of each irreducible representation.
It is easy to see that the evaluation map is identically zero on $V_{(n)}\oplus V_{(n-2,1,1)}$ by considering the generators listed in Table \ref{table!segre-cubics}.

Fix now a generator $\fs_T$ of $V_{(n-2,2)}$. 
 For each $i=1,2$ we may assume that the $a_i$- and $b_i$-coordinates of $P$ have opposite sign, otherwise $x_{a_i}-x_{b_i}$ would vanish. Thus exactly one half of the remaining $c_1,\dots,c_{n-4}$-coordinates of $P$ are positive, and one half negative.
  Hence $\sigma_1(x_C)$ vanishes at $P$ which implies $\es_T(P)$ also vanishes. 

Clearly, a generator $\fs_T$ of shape $(n-3,3)$ is nonzero on $P$ with $a_i$- and $b_i$-coordinates of opposite sign, for $i=1,2,3$. Note that $n\geq 6$ and so $\fs_T$ is not in the kernel.

Both copies of $V_{(n-1,1)}$ in $\Sym^3V$ map isomorphically to the \textit{same} subspace of $H^0(\mathcal{T})$. Indeed, let $P$ be any point whose $a$- and $b$-coordinate have opposite sign, and let 
$\fs_T$, $\gs_T$ generators for the symmetric and the antisymmetric summand, then $\fs_T(P)$ and $\gs_T(P)$ are both nonzero.
On the other hand $\sigma_2(x_C)=1-k$: the coefficient of $z^{2k-2}$  in the polynomial $(z^2-1)^{k-1}$. 
Whence
\begin{eqnarray*}
(n\fs_T+(1-k)\gs_T)(P)&=&\pm 2[ n\sigma_2(x_C)+ (1-k)(-2-s_2(x_C)) ]\\
&=&\pm 2[ n(1-k)+ (1-k)(-2-(n-2)) ]=0
\end{eqnarray*}
 Thus the subspace generated by $n\fs_T+(1-k)\gs_T$ for $T$ of shape $(n-1,1)$ is in the kernel. That is a copy of the irreducible representation $V_{(n-1,1)}$. 
\end{proof}

\begin{cor}
The evaluation map is not surjective for $n\ge 10$. Hence the Segre cubic in dimension $\ge 7$ has obstructed equisingular deformations.
\end{cor}
\begin{proof}
The image of the evaluation map has dimension $n-1+\frac{n(n-1)(n-5)}6= \binom {n-1}3$. For $n\ge10$, this is smaller than the number of nodes $\binom{n-1}{k-1}$, which is the dimension of $H^0(\mathcal{T})$.
\end{proof} 

\subsubsection{Representation of Endomorphisms}\
If $X\colon(f(x)=0)$ is a hypersurface in $H\cong \PP(V)$ then any linear endomorphism $\varphi\colon V\to V$ induces an infinitesimal deformation $\tilde X$ of $X$ defined by $(f(x+\varepsilon\varphi(x))=0)$ in $H\times\Spec \CC [\varepsilon]/(\varepsilon^2)$. Since $f(x+\varepsilon\varphi(x))=f(x)+ \varepsilon \sum_i\varphi_i(x)f_{x_i}$, where $f_{x_i}$ are the partial derivatives of $f$, we see that the generic fibre of $\tilde X$ also contains the singular locus of $X$. 
Thus we have a map from $\End V$ to the kernel of the evaluation map. Recall that $X$ is locally rigid if the kernel of the evaluation map is equal to the image of $\End V$.

\begin{lemma}\label{RepEndS} The $\mathfrak{S}_n$-representation  $\End V$ decomposes as
\[\End V\cong V_{(n)}\oplus V_{(n-1,1)}\oplus V_{(n-2,2)}\oplus V_{(n-2,1,1)}\]
\end{lemma}
\begin{proof}

Every representation of $\mathfrak{S}_n$ has real character, thus we get $\End V \cong V^*\otimes V$ is isomorphic to $V\otimes V$, and so
the decomposition of $\End V$ is
\[\End V\cong \Sym^2 V\otimes \bigwedge^2 V\cong V_{(n)}\oplus V_{(n-1,1)}\oplus V_{(n-2,2)}\oplus V_{(n-2,1,1)}\,. \]
\end{proof}

\begin{theo}
The Segre cubic $X_S\subset \PP^{n-1}$ is locally rigid for all $n\geq 6$.
\end{theo}

\begin{proof}
By Schur's Lemma,  Proposition \ref{prop!segre-kernel-image} and Lemma \ref{RepEndS}, we have an isomorphism between the image of $\End V$ and the kernel of the evaluation map for all $n\ge6$, whence $X$ is locally rigid.
\end{proof}

\subsection{Representation Theory for the Goryunov--Kalker cubics} \label{section!GK}

Recall that now $H$ is the hyperplane $(\sum_{i=1}^nx_i +2x_{n+1}=0)\subset\PP^{n}$ and define $V:=H^0(\mathcal O_H(1))$. This time, $\Sn_n$ acts by permuting the first $n$ coordinates and fixing the last one, so $V$ decomposes 
 as the standard representation plus the trivial one: $V\cong V_{(n-1,1)}\oplus V_{(n)}$.

\begin{prop}\label{RepCubicsGK}
Let $n\ge6$. Then $\Sym^3 V$ decomposes into irreducible representations as:
\[\Sym^3   V\cong V_{(n)}^{\oplus 3}\oplus V_{(n-1,1)}^{\oplus 4}\oplus V_{(n-2,2)}^{\oplus 2}\oplus V_{(n-3,3)}\oplus V_{(n-2,1,1)}\]
Generators for each summand are found in Tables \ref{table!segre-cubics} and \ref{table!GK-cubics}.
\end{prop}

\renewcommand{\arraystretch}{1.5} 
{\small
\begin{table}[ht]\caption{Extra cubics on the GK hypersurface}\label{table!GK-cubics}
\begin{tabular}{clcc}
\text{Partition } $\lambda$ & \text{Tableau} & \text{Generator} & \text{Dimension} \\
\hline
$(n)$ & $\begin{smallmatrix}c_1 &\dots & c_n\end{smallmatrix}$ & $\hs_T=\sigma_2(x_C)x_{n+1}$ & $1$ \\
$(n-1,1)$ & $\begin{smallmatrix}a & c_1 & \dots & c_{n-2} \\ b\end{smallmatrix} $&
$\hs_T=(x_a-x_b)\sigma_1(x_C)x_{n+1}$ & $n-1$ \\
$(n-2,2)$ & $\begin{smallmatrix}a_1 & a_2 & c_1 & \dots & c_{n-4} \\ b_1 & b_2 \end{smallmatrix} $&
$\hs_T=(x_{a_1}-x_{b_1})(x_{a_2}-x_{b_2})x_{n+1} $& $\frac{n(n-3)}2$ \\
\hline
$(n-1,1)$ &$ \begin{smallmatrix}a & c_1 & \dots & c_{n-2} \\ b\end{smallmatrix} $&
$\ks_T=(x_a-x_b)x_{n+1}^2$ & $n-1$ \\
\hline
$(n)$ & $\begin{smallmatrix} c_1 & \dots & c_{n} \end{smallmatrix} $&
$\ls_T=x_{n+1}^3$ & $1$
\end{tabular}
\end{table}
}

 \begin{proof} Using the identity $\Sym^m(A\oplus B)\cong\bigoplus_{i+j=m} \Sym^iA\otimes \Sym^j B$, which is well-known, together with $A\cong V_{(n-1,1)}$ and $B=\left<x_{n+1}\right>$ and  Proposition \ref{RepCubicsS} we obtain:
\begin{eqnarray*}
\Sym^3  V&\cong&\Sym^3 V_{(n-1,1)} \oplus \left<x_{n+1}\right>\otimes\Sym^2V_{(n-1,1)}\oplus \\
&&\left<x_{n+1}^2\right>\otimes V_{(n-1,1)} \oplus \left< x_{n+1}^3\right>\\
&\cong&\Sym^3 V_{(n-1,1)}\oplus 
(V_{(n)}\oplus V_{(n-1,1)}\oplus V_{(n-2,2)})
 \oplus V_{(n-1,1)} \oplus V_{(n)}\\
&\cong& V_{(n)}^{\oplus 3}\oplus V_{(n-1,1)}^{\oplus 4}\oplus V_{(n-2,2)}^{\oplus 2}\oplus V_{(n-3,3)}\oplus V_{(n-2,1,1)}
\end{eqnarray*} 
Looking at the decomposition in the first line, we can easily find cyclic generators for the summands involving $x_{n+1}$. These are reported in Table \ref{table!GK-cubics}. The summands not involving $x_{n+1}$ appear in Table \ref{table!segre-cubics}.
  \end{proof}

The cubic $X$ is locally rigid if the kernel of the evaluation map is equal to the orbit of the defining equation of $X$ under projectivities. Thus, we are going to  compare the kernel with $\End V$.

 \begin{lemma}\label{RepEndGK} The $\mathfrak{S}_n$-representation  $\End V$ decomposes as
\[\End V\cong V_{(n)}^{\oplus 2}\oplus V_{(n-1,1)}^{\oplus 3}\oplus V_{(n-2,2)}\oplus V_{(n-2,1,1)}\,.\]
\end{lemma}

\begin{proof} As in Proposition \ref{RepEndS}, we have $\End V\cong V\otimes V$ and hence 
\begin{eqnarray*}
\End (V_{(n-1,1)}\oplus V_{(n)})&\cong &  (V_{(n-1,1)}\oplus V_{(n)})\otimes (V_{(n-1,1)}\oplus V_{(n)})\\
&\cong &   V_{(n)}^{\oplus 2}\oplus V_{(n-1,1)}^{\oplus 3}\oplus V_{(n-2,2)}\oplus V_{(n-2,1,1)}.\hfill \qedhere
\end{eqnarray*}
\end{proof}

Let $X_{GK}\subset\mathbb P^{n}$ be the Goryunov--Kalker cubic and let 
 $H^0(\mathcal{T})=\bigoplus_{P\in\Sing X_{GK}} \CC_P$ be the vector space of global sections of the skyscraper sheaf $\sT$.

\begin{remark}\label{rem!GK-points}
i) The symmetric group $\Sn_n$ acts trivially on the last coordinate, we have a natural choice for a
representative of  the  basis vector $1_P$ in $\CC_P$, namely we
 denote the  basis vector $1_P$ in $\CC_P$ by the representative of $P$ having $x_{n+1}=1$.
We remark that each $P$ has $k+1$ negative coordinates. 

ii) Since the symmetric group $\Sn_n$ acts trivially on the last coordinate,  
each permutation $\sigma\in \Sn_n$ is a not only a bijection of the set $\Sing X_{GK}$, but \textit{also} 
of the set of vectors $\{1_P \mid P\in \Sing X_{GK}\}$.
\end{remark}

\begin{lemma}The evaluation map $\ev\colon\Sym^3 V\to H^0(\sT)$
is $\Sn_n$-equivariant.
\end{lemma}
\begin{proof}
Since a permutation $\sigma$ induces a bijection on the set $\{1_P\}$, the proof is immediate:
\[
 \sigma(\ev(\es_T))=  \sum_P  \es_T(P) \cdot \sigma(1_P)= \sum_P  \es_T(\sigma^{-1}P) \cdot 1_P
=
 \sum_P  \es_{\sigma(T)}(P) \cdot 1_P= \ev(\sigma(\es_T))\,.\]
\end{proof}

\begin{prop}\label{prop!GK-kernel-image}
Let $n\geq 8$. 
The evaluation map $\ev\colon\Sym^3 V\to H^0(\mathcal{T})$  has 
kernel isomorphic to $V_{(n)}^{\oplus 2}\oplus V_{(n-1,1)}^{\oplus 3}\oplus V_{(n-2,2)}\oplus V_{(n-2,1,1)}$ and 
image isomorphic to $V_{(n)}\oplus V_{(n-1,1)}\oplus V_{(n-2,2)}\oplus V_{(n-3,3)}$.
\end{prop}

\begin{proof}
Using the determinantal description of a generator $\es_T$  of $V_{(n-2,1,1)}$ given in Remark \ref{determinant}  we see immediately that $\es_T(P)$ vanishes for any singular point $P$ and thus $\es_T$ is in the kernel.

A generator $\fs_T$ of shape $(n-3,3)$ is nonzero on a point $P$ with $a_i$- and $b_i$-coordinates of opposite sign, for $i=1,2,3$. Clearly, such a singular point exists on $X$ if and only if $k-1\geq 3$, i.e.~$n\geq 8$.
If $n=6$ then $V_{(n-3,3)}$ is contained in the kernel of the evaluation map.

Both copies of $V_{(n-2,2)}$ in $\Sym^3V$ map isomorphically to the \textit{same} subspace of $H^0(\mathcal{T})$. For a given tableau $T$ consider the generators $\fs_T$ and $\hs_T$.
 For each $i=1,2$ we may assume that the $a_i$- and $b_i$-coordinates of $P$ have opposite sign, otherwise $x_{a_i}-x_{b_i}$ would vanish. 
 Thus exactly  $k-3$ of the remaining $c_1,\dots,c_{n-4}$-coordinates of $P$ are positive, and  $k-1$ are negative. Hence $\sigma_1(x_C)=-2$  and $x_{n+1}=1$, which implies that $\fs_T$ and $\hs_T$ do not vanish at $P$, but
 $\fs_T+2 \hs_T$ does.
 
 The remaining two cases are similar:

Two copies of $V_{(n)}$ are in the kernel and one copy is in the image, because $\ls_T$ evaluates to $1$ but $\fs_T+2(k-1) \ls_T$ and $\hs_T+(k-2) \ls_T$ evaluate to zero for any singular point $P$.
Three copies of $V_{(n-1,1)}$ are in the kernel and one copy is in the image, because $\fs_T +(k-4 )\ks_T$, $\gs_T+k\ks_T$, $\hs_T+2\ks_T$ evaluate to $0$ for any $P$, although $\ks_T$ evaluates to $\pm2$ for certain $P$.
\end{proof}

As consequence we get

\begin{theo}
a) The Goryunov--Kalker cubic $X_{GK}\subset \PP^{n}$ is locally rigid for all $n\geq 8$.

b) The evaluation map is not surjective for $n\ge 10$. Hence the Goryunov--Kalker cubic in dimension $\ge 8$ has obstructed equisingular deformations.
\end{theo}
\begin{proof}
a) By Schur's Lemma and Propositions \ref{RepCubicsGK}, \ref{prop!GK-kernel-image}, we have an isomorphism between the image of $\End V$ and the kernel of the evaluation map.

b) The image of the evaluation map has dimension $n+\frac{n(n-3)}2+\frac{n(n-1)(n-5)}6=\binom n3$. For $n\ge10$, this is smaller than the number of nodes $\binom{n}{k-1}$, which is the dimension of $H^0(\mathcal{T})$.
\end{proof} 

\begin{remark}
By the proof of Proposition \ref{prop!GK-kernel-image}, when $n=6$ the kernel of the evaluation map is bigger than $\End V$:
\[\ker(\ev)\cong\End V \oplus V_{(3,3)}\]
whence
$X$ is not locally rigid. The  extra summand  $ V_{(3,3)} $ has basis
\[\begin{array}{l}
(x_1-x_4)(x_2-x_5)(x_3-x_6)\,,
(x_1-x_3)(x_2-x_5)(x_4-x_6)\,, \\
(x_1-x_3)(x_2-x_4)(x_5-x_6)\,, 
(x_1-x_2)(x_3-x_5)(x_4-x_6)\,, \\
(x_1-x_2)(x_3-x_4)(x_5-x_6)\,.
\end{array}
\]
We predicted the existence of this extra summand by computing Euler--Poincar\'e characteristics of the short exact sequence \eqref{eq!jacobian} in the case $n=6$. Here, the saturation of the Jacobian ideal $\sJ$ requires five cubic generators. In general, the extra generators of the saturation have degree $n-3$.

We have seen in Corollary \ref{projection_c} that the projection from an ordinary double point $P$ on a cubic 4-fold $X$ has an associated K3 surface $S$ of degree 6, and vice versa. The K3 surface is the intersection of the cubic with the projectivised tangent cone at $P$. Alternatively, $S\subset\PP^4$ is defined by equations $A_2(x_1,\dots,x_5)=B_3(x_1,\dots,x_5)=0$ where the Taylor expansion of $X$ at the node $P_{x_6}:=(0,0,0,0,0,1)$ is $Ax_6+B$. If $X$ has 15 ordinary double points, then $S$ has 14 $A_1$ singularities. Such K3 surfaces with 14 $A_1$ singularities have $19-14=5$ moduli, which should account for the extra summand in the kernel.
\end{remark}

\subsection{Representations of points}\label{section!points}\
In this section we decompose the vector space $H^0(\mathcal{T})=\bigoplus_{P\in\Sing X_S} \CC_P$ into irreducible $\Sn_n$-representations. Even though we were able to prove Theorem \ref{local-uniqueness} without the results of this section, we think that it may be of independent interest. The proofs are in the same spirit as those for cubic polynomials, so we just sketch them here.

\subsubsection{Representations of Segre points}\

A vector in $H^0(\mathcal{T})=\bigoplus_{P\in\Sing X_S} \CC_P$ is a formal $\CC$-linear combination of points. We denote the basis vector $1_P$ in $\CC_P$ by a $2\times k$ labelled array as in Section \ref{section!ev-segre}.

For the remainder of this section we consider Young diagrams of shape $\lambda=(n-p,p)$ with $p$ odd.
These always have two rows, and to fix notation, we label a general tableau of shape $(n-p,p)$ as follows: 
\[\renewcommand{\arraystretch}{0.9}\arraycolsep=3pt
\begin{matrix}
a_1  & \dots & a_{p} & c_1 & \dots & c_{2l}\\
b_1 & \dots & b_{p}
\end{matrix}
\]
We write $A$ for the tuple $(a_1,\dots,a_p)$ and similarly $B$, $C$. Since $n$ is even, we may write $2l:=n-2p$ for the length of $C$.

To a fixed Young tableau $T$, we associate a decomposition of the Segre cubic $X_S$ as follows:
\[T\rightsquigarrow X_S:\left(\sum_{i=1}^p (x_{a_i}+x_{b_i})+\sum_{j=1}^{2l}x_{c_j}=\sum_{i=1}^p (x_{a_i}+x_{b_i})q_{a_i,b_i}+\sum_{j=1}^{2l}x_{c_j}^3=0\right)\]
where $q_{a_i,b_i}:=x_{a_i}^2-x_{a_i}x_{b_i}+x_{b_i}^2$. Note that $X_C:=X_S\cap\PP^{2l-1}_{c_1,\dots,c_{2l}}$ is the Segre cubic of dimension $2l-3$. Thus $X_C$ has an $\Sn_{2l}$-orbit of nodes generated by $(1^l:-1^l)$.
\begin{defin}
We define the \textit{Specht generator} $\vs_T\in H^0(\mathcal{T})$ to be
\[\vs_T:=
\sum_{\tau\in \Cn_T}\sum_{P\in \Sing(X_C)}\sign(\tau)
{\frac{\tau(A) \cup P^+}{\tau(B) \cup P^-}} 
\]
where $P^+$ denotes the set of positive coordinates of $P$, and $P^-$ the negative. For this to be well-defined, we use the convention that $P$ is in the
 $c_{2l}$-\textit{distinguished affine chart} of $X_C$, this means that $c_{2l}$ is  in $P^+$ (i.e.~$x_{c_{2l}}=1$).
\end{defin}
Lemma \ref{lemma!points-inv} implies that $\vs_T$ is independent of this choice of distinguished affine chart.

\begin{example}
For example, let $T=\begin{smallmatrix}
1  & 2 & 3 & 4 & 5 & 6 & 7 \\
8 & 9 & 10 
\end{smallmatrix}$. Then $X_C$ is the plane cubic with three nodes represented by $\frac{47}{56}$, $\frac{57}{46}$, $\frac{67}{45}$ in $\PP^3_{x_4,x_5,x_6,x_7}$, so that
\begin{eqnarray*}
\renewcommand{\arraystretch}{0.9}\arraycolsep=3pt
\vs_{\begin{smallmatrix}
1  & 2 & 3 & 4 & 5 & 6 & 7 \\
8 & 9 & 10 
\end{smallmatrix}}&=&
\renewcommand{\arraystretch}{0.9}\arraycolsep=3pt
\begin{array}{ccccc} 1 &  2 &  3 &  4 &  7 \\ \hline 8 &  9 & 10 &  5 & 6\end{array}+
\begin{array}{ccccc} 1 &  2 &  3 &  5 &  7 \\ \hline 8 &  9 & 10 &  4 & 6\end{array}+
\begin{array}{ccccc} 1 &  2 &  3 &  6 & 7 \\ \hline 8 &  9 & 10 &  4 & 5\end{array}+\dots\\ 
&&
\renewcommand{\arraystretch}{0.9}\arraycolsep=3pt
\phantom{}-\begin{array}{ccccc} 8 &  9 & 10 &  4 &  7 \\ \hline 1 &  2 & 3 &  5 & 6\end{array}-
\begin{array}{ccccc} 8 &  9 & 10 &  5 &  7 \\ \hline 1 &  2 & 3 &  4 & 6\end{array}-
\begin{array}{ccccc} 8 &  9 & 10 &  6 & 7 \\ \hline 1 &  2 & 3 &  4 & 5\end{array}
\end{eqnarray*}
has $|\Cn_T|\cdot|\Sing(X_C)|=2^3\cdot \frac12\binom42=24$ summands and the last entry of $C$, $c_4=7$, always appears in the first row.

\end{example}

\begin{lemma}\label{lemma!points-inv} Let $T$ be a Young tableau of shape $\lambda=(n-p,p)$ with $p$ odd. The $\vs_T$ satisfy the following properties:
\begin{enumerate}
\item[i)] Let $\tau\in \mathfrak C_T$, then $\tau(\vs_T)=\vs_{\tau(T)}= \sign(\tau)\vs_T$.
\item[ii)] If $\gamma \in \mathfrak S_C$ permutes the $C$-coordinates, then
$\gamma(\vs_T)=\vs_{\gamma(T)}= \vs_T$. 
\item[iii)] For any $\rho\in\Sn_n$, we have $\rho(\vs_T)=\vs_{\rho(T)}$.
\item[iv)] The set of $\vs_T$ for all standard tableaux $T$ forms a basis of the representation $V_\lambda$. 
\end{enumerate}
\end{lemma}
\begin{proof}
The proof is rather similar to that for polynomials (cf. Lemma \ref{lemma!inv}). Thus we only prove $ii)$, because there the distinguished affine chart is important.

$ii)$ Suppose $\gamma$ is a transposition of elements in the C-part. Let $T':=\gamma(T)$. Since $T$ and $T'$ differ only in their C-parts, we have $\Cn_T=\Cn_{T'}$. Also, $X_C=X_{C'}$ so the only difficulty is the distinguished affine chart.

If $\gamma$ does not involve the last entry $c_{2l}$, then the affine chart is the same and since the column stabilisers and set of nodes are equal, we see that $\gamma(\vs_T)=\vs_T=\vs_{\gamma(T)}$.

Now consider the transposition $\gamma=(c_x,c_{2l})$. Let $P=\frac{P^+}{P^-}$ be a node of $X_C$.
If $c_x$ is in $P^+$ then $P=\gamma(P)$ and so $P$ lies in both distinguished affine charts, hence any summands of $\vs_T$ involving $P$ also appear in $\vs_{T'}$ and $\gamma(\vs_T)$. 
Now suppose $c_x$ is in $P^-$. Then $\frac{P^-}{P^+}$ lies in the $c_x$ distinguished affine chart  of $X_{C'}$. Thus any summand of $\vs_T$ involving $P$ uniquely determines a complementary summand of $\vs_{T'}$ appearing with opposite sign:
\[\sign(\tau)\frac{\tau(A)\cup P^+}{\tau(B)\cup P^-}\longleftrightarrow
-\sign(\tau)\frac{\tau(B)\cup P^-}{\tau(A)\cup P^+}\]
The coefficient is $-1$ because we have applied the order two permutation $\prod_{i=1}^p(a_i,b_i)$ in $\Cn_{T'}$ which has sign $-1$ because $p$ is odd. The two sides of the above displayed formula are then equal as elements in $H^0(\mathcal{T})$, using the $\CC^*$-action on $\PP^{n-1}$. Hence $\vs_T=\vs_{\gamma(T)}$ and an analogous argument shows that $\gamma(\vs_T)=\vs_T$.
\end{proof}

We are now able to prove the following:
\begin{prop}\label{RepPointsS}
The $\mathfrak{S}_n$-representation  $H^0(\mathcal T)$  decomposes as
\[\bigoplus_{P\in\Sing {X_S}}\CC_P\cong V_{(n-1,1)}\oplus V_{(n-3,3)}\oplus\dots\oplus 
\begin{cases}
V_{(k,k)} & \text{ if }n\equiv2\mod 4\\
V_{(k+1,k-1)} & \text{ if }n\equiv0\mod 4.
\end{cases}\]
\end{prop}
\begin{proof}By Frobenius' Formula (\cite[\S 4.1]{FultonHarris}) the dimension of $V_{(n-p,p)}$ is
$\binom{n}{p} - \binom{n}{p-1}$: the coefficient of    $x^{n-p+1}y^p$ in $(x-y)\cdot(x+y)^n$.

Let $\hat k$ be the odd number $k$ or $k-1$ according to the residue class of $2k=n \mod 4$.
Using the identity $\sum_{j=0}^s (-1)^{j}\binom n j= (-1)^s \binom{n-1}s$  we get
\[
\sum_{\substack{p=1,\\ p \text{ odd}}}^{\hat k} \dim V_{(n-p,p)} =\sum_{j=0}^{\hat k} (-1)^{j+1} \binom{n}{j}
  =\binom{n-1}{k-1} = \dim H^0(\mathcal T)\,.\]
\end{proof}

\subsubsection{Representations of GK points}\

Let $X_{GK}\subset\mathbb P^{n}$ be the Goryunov--Kalker cubic and let 
 $H^0(\mathcal{T})=\bigoplus_{P\in\Sing X_{GK}} \CC_P$ be the vector space of global sections of the skyscraper sheaf $\sT$.

We denote the  basis vector $1_P$ in $\CC_P$ by the representative of $P$ having $x_{n+1}=1$ (see Remark \ref{rem!GK-points}). In analogy to the Segre-case, we associate to this a 2-rows labelled array:
\[\renewcommand{\arraystretch}{1.0}
\begin{array}{cccccc}
p_1  & \dots & p_{k-1} \\
\hline
p_{k} & \dots & p_n
\end{array}
\begin{array}{cc}
\leftarrow\text{positive} \\
\leftarrow\text{negative} 
\end{array}
\]
where $\{p_1,\dots,p_n\}=\{1,\dots,n\}$. Here we do not write the last coordinate, since it is fixed to be $1$.
The first row  corresponds to the remaining $k-1$  coordinates equal to $1$, and the second row  to the $k+1$ coordinates equal to $-1$. 

For the remainder of this section we consider Young diagrams of shape $\lambda=(n-p,p)$ 
\[\renewcommand{\arraystretch}{0.9}\arraycolsep=3pt
\begin{matrix}
a_1  & \dots & a_{p} & c_1 & \dots & c_{2l}\\
b_1 & \dots & b_{p}
\end{matrix}
\]
and we use the same notation as in the Segre-case. Note that now $p$ can be even.

To a fixed Young tableau $T$, we associate a decomposition of the Goryunov--Kalker cubic $X_{GK}$ as follows:
\begin{multline*}
T\rightsquigarrow X_{GK}\colon\left(\sum_{i=1}^p (x_{a_i}+x_{b_i})+\sum_{j=1}^{2l}x_{c_j}
+2x_{n+1}=\right.\\
\left.\sum_{i=1}^p (x_{a_i}+x_{b_i})q_{a_i,b_i}+\sum_{j=1}^{2l}x_{c_j}^3 +2x_{n+1}^3=0\right)
\end{multline*}
where $q_{a_i,b_i}:=x_{a_i}^2-x_{a_i}x_{b_i}+x_{b_i}^2$. Note that $X_C:=X_{GK}\cap\PP^{2l}_{c_1,\dots,c_{2l}, x_{n+1}}$ is the GK-cubic of dimension $2l-2$. Thus $X_C$ has an $\Sn_{2l}$-orbit of nodes generated by  $(1^{l-1}:-1^{l+1}:1)$.
\begin{defin}
In the GK-case, we define the \textit{Specht generator} $\ws_T\in H^0(\mathcal{T})$ to be
\[\vs_T:=
\sum_{\tau\in \Cn_T}\sum_{P\in \Sing(X_C)}\sign(\tau)
{\frac{\tau(A) \cup P^+}{\tau(B) \cup P^-}} 
\]
where $P^+$ denotes the set of positive coordinates of $P$, and $P^-$ the negative. 
Note that also here we are forgetting the coordinate $x_{n+1}$ which is always $1$.
\end{defin}

\begin{example}
Let $T:=\begin{smallmatrix}
1  & 2 & 3 & 4 & 5   \\
6  
\end{smallmatrix}$. 

Then $X_C$ is a cubic surface with with four nodes represented by $\frac{2\phantom{11}}{345}$, $\frac{3\phantom{11}}{245}$, $\frac{4\phantom{11}}{235}$, $\frac{5\phantom{11}}{234}$ in $\PP^4_{x_2,x_3,x_4,x_5,x_7}$, so that
\renewcommand{\arraystretch}{0.9}\arraycolsep=3pt
\begin{eqnarray*}
\ws_T&=&
\begin{array}{ccccc} 1 &  2 &        \\ \hline 6 & 3 & 4 & 5\end{array}+
\begin{array}{ccccc} 1 &  3 &      \\ \hline  6& 2 & 4 & 5\end{array}+
\begin{array}{ccccc} 1 &  4 &       \\ \hline 6  &2 &  3 &5\end{array}+
\begin{array}{ccccc} 1 &  5 &       \\ \hline 6 & 2&3 &  4\end{array}- \\
&&-
\begin{array}{ccccc} 6 &  2 &        \\ \hline 1 & 3 & 4 & 5\end{array}-
\begin{array}{ccccc} 6 &  3 &      \\ \hline  1& 2 & 4 & 5\end{array}-
\begin{array}{ccccc} 6 &  4 &       \\ \hline 1  &2 &  3 &5\end{array}-
\begin{array}{ccccc} 6 &  5 &       \\ \hline 1 & 2&3 &  4\end{array}
\end{eqnarray*}
has $|\Cn_T|\cdot|\Sing(X_C)|=2\cdot  \binom41=8$ summands.
\end{example}

\begin{lemma}\label{lemma!pointsGK-inv} Let $T$ be a Young tableau of shape $\lambda=(n-p,p)$. The $\ws_T$ satisfy the following properties:
\begin{enumerate}
\item[i)] Let $\tau\in \mathfrak C_T$, then $\tau(\ws_T)=\ws_{\tau(T)}= \sign(\tau)\ws_T$.
\item[ii)] If $\gamma \in \mathfrak S_C$ permutes the $C$-coordinates, then
$\gamma(\ws_T)=\ws_{\gamma(T)}= \ws_T$. 
\item[iii)] For any $\rho\in\Sn_n$, we have $\rho(\ws_T)=\ws_{\rho(T)}$.
\item[iv)] The set of $\ws_T$ for all standard tableaux $T$ forms a basis of the representation $V_\lambda$. 
\end{enumerate}
\end{lemma}

\begin{proof}
The proof is analogous to that for the Segre-case (cf. Lemma \ref{lemma!points-inv} ), or even simpler, since we do not have to consider a distinguished affine chart in view of the special role of $x_{n+1}$.
\end{proof}

\begin{prop}\label{RepPointsGK}
The vector space $H^0(\mathcal T) \cong \bigoplus_{P\in\Sing X}\CC_P $  decomposes as
\[\bigoplus_{P\in\Sing X_{GK}}\CC_P\cong V_{(n)} \oplus V_{(n-1,1)}\oplus V_{(n-2,2)}\oplus\dots\oplus 
V_{(k +1,k -1)} \,.\]
\end{prop}

\begin{proof}
 As in   Proposition \ref{RepPointsS} we have $\dim V_{(n-p,p)}= \binom{n}{p} - \binom{n}{p-1}$, hence 
we get the telescopic sum
\[
\sum_{p=0}^{k-1} \dim V_{(n-p,p)} =\sum_{p=0}^{ k-1} \binom{n}{p} - \binom{n}{p-1} 
 =\binom{n}{k-1} = \dim H^0(\mathcal T)\,.\]
\end{proof}

\chapter{Nodal Quintics}
\chapterauthor{Fabrizio Catanese, Alessandro Verra}

\section{Codes of nodal quintic surfaces}

In this section we shall describe the codes of nodal quintics, in particular
we shall give new proofs of the code theoretical part of Beauville's theorem $\mu(5) = 31$.

Recall that  (\cite{angers}, see also section 3.7 of \cite{cascat}) on quintic surfaces even sets of nodes have cardinality $16$ or $20$ (these are the weights of $\sK$),
and they correspond to representations of the surface as the determinant of a symmetric matrix with
(homogeneous) polynomial entries.

If all the weights are $16$, then $\sK$ is a shortening of the biduality code (code of a Togliatti quintic with $31$ nodes),
and these codes all occur because there is an unobstructed  quintic with $31$ nodes  (see theorems \ref{unobstructed} 
and \ref{quintic-codes}).

The major interest is then to investigate the codes where the weight $20$ occurs. 

This is the first basic result in this direction.

\begin{theo}\label{web}
Assume that $Y$ is a nodal quintic surface, and  that $\sN$ is an even set of nodes 
of cardinality $20$.

Then the number $\nu$ of nodes of $Y$ is at most $28$.

\end{theo}

 We refer the reader to the list of Appendix B for a classification of all the codes with weights $16$ or $20$,
with a sublist of those satisfying the B-inequality.

From this list it appears that not all the above codes which satisfy  the B-inequality do indeed occur, 
since we shall show that if the weight $20$ occurs, then $ \nu \leq 28$.

With the inequality $ \nu \leq 28$
there are only two interesting codes where the weight $20$ occurs, namely we have:

\begin{lemma}\label{20}
Assume that $\sK \subset \FF_2^{\nu} $ has only weights $16$ or $20$, that the weight $20$ occurs and that
$\nu \leq 28$. 

Then $ k : = dim (\sK) \leq 2$, and if $k=2$ there are two unique codes, one with weights $(20,20,16)$,
the other with weights $(20,16,16)$. In the first case the length $n= 28$, in the second the length $n= 26$.
\end{lemma}

\begin{proof}
Let $v \in \sK \subset \FF_2^{\sS}$ be a vector with weight $20$, let $A : = {\rm \supp} (v)$, $ B : = \sS \setminus A$.

If there is another vector $u$ different from $v, 0$, then set $$ a : = | \supp (u) \cap A|, b : = | \supp (u) \cap B| .$$

Then $ b \leq 8$ and  if $a + b = 20$, then $ 20 - a + b = 16, 20 \Rightarrow a = b+4, b.$

The second case is excluded since then $ a=b=10$, a contradiction. In the first case $a=12, b=8$.

If instead $a + b = 16$, and  $ 20 - a + b = 16$, then $a= 10, b=6$.

To finish the proof it suffices to show that $ k <3$, i.e., that $k=3$ is not possible.

Assume that there are two vectors $v,u$ of respective weights $20,20$, hence $ a=12, b=8$,
or   of respective weights $20,16$, hence $ a=0, b=6$.

Assume that there is another vector $v' \in \sK$, linearly independent of $v,u$. 

Denote by $A_1 : = A \cap \supp(u)$, $|A_1|=12, 10$, set $A_2 : = A \setminus A_1$, so  $|A_2|=8, 10$.

Then we define  $a_1, a_2, b'$ through the cardinalities of the respective  intersections of $\supp(v')$ 
with $A_1, A_2, B$. In particular, $ b' = 6$ or $8$, correspondingly $a_1 + a_2 = 10$ or $12$.

Since also  the following numbers belong to the set $\{16,20\}$:
$$ \ 20 - a_1 + a_2 - b', \ 16 + a_1 - a_2 - b',$$
and they equal in the first case 
$$ 14 - 2 a_1 + 10 = 24 - 2 a_1,{\rm or } \ \ 2 a_1 \Rightarrow  24 = 24 - 2 a_1 + 2 a_1 \geq 32,$$
a contradiction.
In the second case we get 
$$ 12 - 2 a_1 + 12 = 24 - 2 a_1,{\rm or } \ \ 2 a_1 - 4  \Rightarrow  20 = 24 - 2 a_1 + 2 a_1- 4 \geq 32,$$
again a contradiction.
\end{proof}

\begin{proof}{\em (of theorem \ref{web}).}
There is a symmetrix matrix of linear forms
$$ A (x) = (a_{i,j} (x)) , \ i,j= 1,\dots 5,$$
such that $Y$ is equal to $\{ (x) | det (A(x)) = 0\}$, while 
$$\sN = \{ (x) | \corank (A(x)) = 2\}.$$

$Y$ is the projection in $\PP^3$ of a complete intersection $\hat{Y}$
inside $\PP^3 \times \PP^4$ of 5 hypersurfaces of bidegree $(1,1)$.
$$ \hat{Y} : = \{ (x,y) | \sum_{j=1}^5  a_{i,j} (x) y_j = 0 \}.$$
In particular, $K_{\hat{Y}} \cong \hol_{\hat{Y}}(1,0)$ and
$ p : \hat{Y} \ra Y$ is a partial  resolution of $Y$,  of the nodes in $\sN$.

In fact, at the points $x$ where $ \corank (A(x)) = 2$ the fibre of $p$ is a $\PP^1$, 
and it is easy to verify that $\hat{Y}$ is smooth at the points of this exceptional curve.

At the other points $P$ of $Y$, where $ \corank (A(P)) = 1$, the matrix is equivalent
(as a quadratic form with coefficients in the local ring $\hol_{\PP^3, P}$)
to a diagonal matrix $ Diag (1,1,1,1,f(x))$ where then $f(x)=0$ is the local equation of $Y$ at $P$.

Then the equations of $\hat{Y}$ in the fibre over $P$ are:
$$ y_1 = y_2 = y_3= y_4 = 0, f(x) = 0, $$
hence we have a local isomorphism $\hat{Y} \cong Y$.

In this spirit we analyse the singularities of $\hat{Y}$. To simplify our computations,
assume that $P$ is the point $e_0$,  that $ \corank (A(P)) = 1$,
and that $A(e_0) =  Diag (1,1,1,1,0)$. Hence
$$ A(x) = x_0 I_4 + x_1 A(1) + x_2 A(2)+ x_3 A(3),$$
and over $e_0$ lies the point $e'_5$.

Taking in $\PP^3 \times \PP^4$ affine local coordinates
$$(1,x_1, x_2, x_3, x_4), ( y_1, y_2, y_3, y_4 , 1),$$
the equations of $\hat{Y}$ write as:
$$ y_1 + \sum_1^3 x_i ( \sum_1^4 A(i)_{1,j} y_j + A(i)_{1,5}) = 0,$$
$$ y_2 + \sum_1^3 x_i ( \sum_1^4 A(i)_{2,j} y_j + A(i)_{2,5}) = 0,$$
$$ y_3 + \sum_1^3 x_i ( \sum_1^4 A(i)_{3,j} y_j + A(i)_{3,5}) = 0,$$
$$ y_4 + \sum_1^3 x_i ( \sum_1^4 A(i)_{4,j} y_j + A(i)_{4,5}) = 0,$$
$$  \sum_1^3 x_i ( \sum_1^4 A(i)_{5,j} y_j + A(i)_{5,5}) = 0,$$
so that we can eliminate $y_1, y_2, y_3, y_4$ and we have a singular point if and only if
$$ A(1)_{5,5} = A(2)_{5,5} = A(3)_{5,5} = 0.  $$

The above conditions amount to requiring that $e'_5$ is a base point for the system of quadrics.
 
Therefore  $(e_0, e'_5)$ is a singular point of $\hat{Y}$ if and only if $e'_5$ is a base point for the system of quadrics,
and a singular point for the quadric corresponding to $A(e_0)$, which has rank 4.

This implies that $e'_5$ is a base point for the 4 quadrics, of multiplicity at least $2$.
Conversely, if $e'_5$ is a base point for the 4 quadrics, of multiplicity at least $2$, then the four gradients are not linearly independent,
hence there is a quadric which is singular at $e'_5$ .

Let $\sB$ be the base locus of the web of quadrics ($\PP^3$ of quadrics, spanned by the 4 quadrics $A(e_i)$).

If there were a component $\Ga$ in $\sB$  of dimension at least $1$, then at the general point $e' \in \Ga$
the gradients would span a space of dimension 3, hence there would be a singular quadric
containing $e'$ and  $\hat{Y}$ (hence also $Y$) would have infinitely many singular points, a contradiction.

Therefore $\sB$ is finite and, by Bezout's theorem, it consists of $16$ points counted with multiplicity.
Therefore there can be at most $8$ points of multiplicity $\geq 2$.

The conclusion is that $\hat{Y}$  has at most $8$ singular points, hence $\nu (Y) \leq 20 + 8 = 28$.
\end{proof}

\begin{remark}\label{CMsextic}
The same proof works for sextics having a half-even set of cardinality $35$, corresponding to a symmetric matrix of linear forms:
then $\nu \leq 35 + 16 = 51.$ 
\end{remark}

\begin{cor}\label{Beau}
If the number of nodes of a quintic is  at least $  29$, then the  codewords in $\sK$ have all  weight $16$. 

In particular  Beauville's theorem: $\mu(5) \leq 31$, with equality if and only if $\sK$ is the simplex (biduality code),
follows  immediately.
\end{cor}

\begin{proof}

The first assertion was already shown.

The second follows from the first two McWilliams' identities, and the B-inequality.

As we mentioned already, the  McWilliams' identity reads out as a series of identities, the first two are:
$$ \sum_{i >0} a_i = 2^k -1, \ \sum_{i >0} i a_i = 2^{k -1} n.$$

If  we have a code $\sK$ where all the weights are just $2^m$,  let $k=dim(\sK)$, $n $ its effective length.

In this case we get $$ 2^m (2^k -1) = 2^{k -1} n \Rightarrow n = (2^k -1) 2^ {(m - k +1)}  \Rightarrow k \leq m+1.$$
Here $m=4$, moreover,  by the B-inequality $ k \geq \nu - 26$, hence $\nu\leq 31$, equality if and only if $k=5$, $n=31$,
and the code is seen to be the simplex code. Because an easy counting argument
shows that a  basis of $\sK$ yields a bijection of the nodal set $\sS$
with $\FF_2^5 \setminus \{0\}$.
\end{proof}

The next question is whether the two codes with $k=2$ of lemma \ref{20} are realized. 

For the  code with length $28$,
the answer is positive by virtue of the following theorem, whose proof shall be given
in  the following  section two (see especially corollary \ref{20/28}).
\
\begin{theo}\label{n=28} 

The subset of the Severi variety  $\sF(5,28)$ such that $\sK$ admits the weight  $20$
is an irreducible component, and the length of $\sK$ equals $28$.

\end{theo}

We have  then the following main result for the codes of nodal quintics:

\begin{theo}\label{quintic-codes}
The   codes of  nodal quintic surfaces are all the shortenings of the biduality code (simplex) of dimension $5$,
all the shortenings of the code $\sK$ of dimension $2$, effective length $28$, and weights $(20,20,16)$, and
possibly also the code $\sK$ of dimension $2$, effective length $26$, and weights $(20,16,16)$.

\end{theo}

\begin{proof}
That all the shortenings occur follows from the unobstructedness of the Goryunov cubic, 
by theorem \ref{quintic-codes}, and by the fact that a general symmetroid (determinant of a $5\times 5$ matrix of linear forms)
has $20$ nodes forming an even set. 

We conclude then by lemma \ref{20}.
\end{proof}

\begin{remark}
It is not yet clear to us whether the code $\sK$ of dimension $2$ and weights $(20,16,16)$ does actually occur.
\end{remark}

\begin{remark}
A more refined and very interesting question is for instance the one of determining the fundamental groups 
$\pi_1 (Y^*)$ of nodal maximizing surfaces of degree $d$ ($\nu (Y) = \mu(d)$).
We saw the answer for $ d \leq 4$, but for $d=5, 6$ this is an open problem.

The case $d=5$ appears rather difficult, at least for $d=6$ there is a strategy, by viewing the Barth sextic
as a Galois covering of the plane with group $(\ZZ/2)^3$.

\end{remark}

\section[Prym varieties and   nodal quintics with  $28$ nodes]{Prym varieties and   nodal quintics with  $28$ nodes, and code  admitting the weight $20$}

The following construction  is based on the results of \cite{verra}:

\begin{theo}
Let $C$ be a general canonical curve of genus $5$,
the complete intersection of three quadrics $Q_0,  Q_1 , Q_2$ in $\PP^4$, 
$$ C = \{ Q_0 = Q_1 = Q_2 = 0 \},$$ 
 and let $\eta \in \Pic(C)_2 \setminus \{0\}$ be a non trivial 2-torsion divisor class.

Let $D$ be a general divisor in the linear series $K_C + \eta$.

Let $Q_3$ be a quadric in $\PP^4$ such that $ div_C (Q_3) = 2 D$ on $C$.

Then, if we let $Y$  be the discriminant of the web of quadrics in $\PP^4$ generated by $Q_0,  Q_1 , Q_2, Q_3 $,
 $Y$ is a nodal quintic surface with 
 $28$ nodes and such that the length of the code $\sK (Y) $ (which  admits the weight 20) equals $28$.
\end{theo}

\begin{proof}
The linear system $|K_C + \eta|$ is base point free, since if $P$ were a base point, we would have
that $h^1(\hol_C (K_C + \eta - P)) \neq 0$, hence the linear system $|P + \eta|$ contains an effective divisor, and there
is a point $P'$ such that $ 2 P \equiv 2 P'$; since $ P \neq P'$, the curve $C$ is hyperelliptic, 
contradicting that it is canonically embedded in $\PP^5$.

{\bf Claim 1 }: the linear system $|K_C + \eta|$ is birational onto its image for a general choice of the curve $C$.

{\em Proof of the claim}: otherwise the morphism  $\phi$ associated to $|K_C + \eta|$, $\phi : C \ra \PP^3$
has image of degree strictly smaller than $8$, and greater or equal to $3$, hence its image is a curve $\Ga$ of degree $4$;
and $\phi$ factors through a double covering $ p : C \ra C'$, where $C'$ cannot have genus $0$ since $C$
is not hyperelliptic.

Hence $C'$ has genus $1$, and is isomorphic to $\Ga$ which is a complete intersection of two quadrics.
Since the covering is branched in $8$ points, such curves $C$ depend on $8$ moduli, while curves of genus $5$ have $12$ moduli.

{\bf Consequence of claim 1:} the monodromy group of the hyperplane sections of $C$ (the divisors $D$) is a  transitive subgroup.

Choose now a general divisor $ D \in |K_C + \eta|$, and let $Y$ be the discriminant of the web of quadrics 
 in $\PP^4$ generated by $Q_0,  Q_1 , Q_2, Q_3 $, where  $ div_C (Q_3) = 2 D$.
 In particular, $D$ consists of $8$ distinct points $P_1, \dots, P_8$ and all the quadrics in the web intersect $C$ 
 in the divisor $2D$.  And we have 
 
 {\bf Claim 2 :} for each of the points $P_i$, there is a unique quadric $Q(i)$ in the web
 such that $P_i \in Ker (Q(i))$. 
 
 {\em Proof of claim 2}: there are local coordinates $x,y,z,t$ at $P$ such that $Q_0= x, Q_1 = y, Q_2 = z, 
 Q_3 = t^2 + a x + b y + cz + g$, where $a,b,c, \in \CC$ and  $g$ is in the ideal $(x,y,z)$ and vanishes of order at least $2$ at $P$.
 
 Hence the desired quadric is $Q_3 - a Q_0 - b Q_1 - cQ_2$. 
 
 We observe now that the $8$ quadrics $Q(i)$ are all distinct, provided that the following claim holds:
 
  {\bf Claim 3 :} each of the quadrics $Q(i)$ has rank at least $4$, for $D$ general.
  
  To prove claim 3, by the consequence of claim 1, it suffices to show that these quadrics cannot all have rank $ \geq 3$.
 
 {\bf Claim 4 :} for $D$ general, there are exactly $20$ distinct quadrics of rank $3$ in the web.
 
 Observe that claims 3 and 4, by virtue of the argument given in the proof of theorem \ref{web}, show the desired assertion:
  that $Y$ possesses $20$ nodes where
 $ \corank (A(x))= 2$, plus other $8$ nodes, corresponding to the base points of the web through which passes a quadric which is singular
 at the point.
 
 \bigskip
 
 We begin now the main argument by observing first that the locus $\Sigma $ of rank $\leq 3$ quadrics in the web is finite: because the net generated by $Q_0, Q_1, Q_2$
 contains no such quadric, for general choice of the three quadrics (the quadrics of corank $\geq 2$ form a subvariety of codimension $3$
 in the space of quadrics in $\PP^4$).
 
 We consider now the variety 
 $$ \sS : = \{ (D, Q) | D  \in |K_C + \eta | , div_C(Q) = 2D , \rank (Q) \leq 3 \}.               $$

 $\sS$ fibres onto the linear system $ \in |K_C + \eta|$ with finite fibres $\Sigma$, and we want to show that $\sS$ is reduced,
 so that the general fibre consists of $20$ distinct points, thus proving claim 4.

 Indeed, from the work of \cite{verra},
 it follows that $\sS$ splits into two irreducible components, which map with respective degrees $8, 12$.
 
 Then we shall prove claim 3 simply by showing that the component mapping of degree 8 does not consist of pairs
 where the quadric $Q$ has vertex $ker(Q)$ containing a point of $C$.
 
 \bigskip

  A first remark is that, for general $D$, no quadric $Q$ with $div_C(Q) = 2D$ can have rank equal to $2$. 
 Because otherwise there are two hyperplanes $H_1, H_2$ such that $2 D = div_C(Q) = div_C (H_1) + div_C(H_2)$.
 
 Writing $div_C (H_i) = 2 D'_i +B$, we obtain, assuming that for each $D$ there were such a rank 2 quadric,
 by monodromy invariance of such decomposition, and by transitivity of the monodromy group, that $ B=0$;
 then $ D= D'_1 + D'_2$, and $D'_1, D'_2$ would be effective thetacharacteristics. Since $C$ is not hyperelliptic,
 by Clifford's theorem, follows that $h^0 (C, \hol (D'_i)) \leq 2$, in particular the general divisor $D \in |K + \eta| \cong \PP^3$
 cannot be written as the sum of two such effective divisors $D'_1, D'_2$.
 
 We can therefore restrict ourselves to the consideration of rank 3 quadrics. 
 
 \bigskip

 {\bf Prym interpretation of rank 3 quadrics.}
 
 Assume now that $ Q$ is a rank 3 quadric, such that $Q \cdot C = 2 D$, $D  \in |K_C + \eta | $
 (so in particular, $ C$ is not contained in $Q$).
 
 Projection from the vertex $\sV : = ker(Q)$ is a rational map to $\PP^2$ with image a conic $q$ of equation $ x_1 x_2 - x_3 ^2 =0$,
 so that $\sV  = \{ x_1= x_2 = x_3 = 0 \}$: we recall that the rational map is given by setting 
  $\la  : = \frac{x_1}{x_3} = \frac{x_3}{x_2} \in \PP^1$. We also denote by $ H_{\la} $ the tangent hyperplane to $Q$ 
  corresponding to  the point $\la$
  of the conic $q$.
 
 We consider now the double cover $ \bar{C} \ra C$ defined as follows:
 $$  \bar{C} \subset  C \times \PP^1, \   \bar{C} : = \{ (x , \la) | x \in H_{\la} \}.$$
 
 For each point $ x \in C \setminus Q$, $ x \mapsto (x_1,x_2,x_3) \in \PP^2 \setminus q$, and we have two distinct points
 $y_1, y_2$ of $  \bar{C} $ lying over $x$. 
 
 Hence $  \bar{C} \ra  C$ is certainly unramified outside $Q$.
 If instead $ x \in Q$, observe that 
 
 $$  \bar{C}=   C \times _{\PP^2} Z, \  Z : = \{ (x_1,x_2,x_3,z) \in \PP^3 | z^2 = -( x_1 x_2 - x_3 ^2) \}.$$
 
Since $ div_C (Q) = 2D$, the normalization $p : \tilde{C} \ra \bar{C}$ is an \'etale covering of $C$, corresponding to 
$\eta  \in \Pic(C)_2 \setminus \{0\}$: because $H \equiv K_C, D \equiv K_C + \eta$.
 And $\tilde{C}$ is connected since  $\eta$ is nontrivial.
 
 Since $Z \cong \PP^1 \times \PP^1$ we get two  maps $ \psi_{+},  \psi_{-}: \tilde{C} \ra \PP^1$,
 exchanged by the covering involution $\iota :  z \mapsto -z$, since $-z^2 + x_3 ^2 =   x_1 x_2  $
 yields $$ \psi_{\pm} (x,z) : = \frac{ x_3 \pm z}{x_2} = \frac{x_1}{x_3 - \pm z}.$$
 
 Denote by $\tilde{H}_{+}, \tilde{H}_{-}$ the divisor classes such that $ \psi_{+},  \psi_{-}$ yield two pencils
 respectively inside $|\tilde{H}_{+}|, |\tilde{H}_{-}|$. We clearly have:
 $$ \tilde{H}_{+} + \tilde{H}_{-} \equiv  p^* (K_C) = K_{  \tilde{C} }, \  \iota^* (\tilde{H}_{+}) = \tilde{H}_{-},$$ 
and moreover 
$$ Norm (\tilde{H}_{-}) =  Norm (\tilde{H}_{+})= K_C,$$
since the images of the lines of a ruling of $Z$ map to lines in $\PP^2$.
 
 Recall the definition of the Prym variety (see \cite{prym}, \cite{verra}):
 
 \begin{defin}
 Let $p : \tilde{C} \ra \bar{C}$ be an \'etale covering of $C$, corresponding to 
$\eta  \in \Pic(C)_2 \setminus \{0\}$ (hence such that, if $g$ is the genus of $C$, then
$\tilde{C} $ has genus $2g-1$).

In this situation $\{ \tilde{L} | Norm (\tilde{L} ) = K_C\} \subset Pic^{2g-2}(\tilde{C} )$ splits 
into two connected components, distinguished by the parity of $h^0(\tilde{L} )= h^1(\tilde{L} )=h^0(K_{  \tilde{C} } - \tilde{L} )=
h^0(\iota^* \tilde{L} )$.

$P (C, \eta) : = Prym ( \tilde{C}, C)$ is defined then as the connected component where the parity is even,
and it is isomorphic to the Abelian subvariety of $Pic^0 ( \tilde{C})$ which is the image of $ 1 - \iota^*$. 

And the Prym-theta divisor $\Xi$ of $P (C, \eta)$ consists of the effective divisors inside $P (C, \eta)$.
 \end{defin}
 
  Conversely, giving two pencils inside the respective  linear systems $|\tilde{L} |$ and $\iota^* |\tilde{L} |$,
  exchanged by the covering involution $\iota$,
 for $ \tilde{L}$ such that $Norm (\tilde{L} ) = K_C$, yields a subvariety $Z \cong \PP^1 \times \PP^1$
 inside the linear system $p^* | K_C|$.  It is explained in \cite{verra} p. 228 (2.23) that  $\iota^*$ acts on $Z$, and on   $\PP^3$
 with eigenvalues $(1,1,1,-1)$, so that intersecting with the $+1$ eigenspace we obtain a rank $3$ conic $q \subset \PP^2$,
 whose points parametrize divisors of $C$ of the form $2D$, where $ D \in |K_C + \eta |$.
 
 Now, a point $y \in \tilde{C}$ lies in the union of the base loci of the respective pencils if and only if 
 all divisors coming from $Z$ vanish at $ y + \iota (y)$, or, equivalently, 
 the point $ x : = p(y)$ of the divisor $D$  lies in the vertex of the quadric $Q$.
 
Recall now (see \cite{verra} theorem 2.5)  that we have a finite covering 
 
 $\sS \ra | K_C + \eta|$, which has two irreducible components, which we denote by  
 $W$ and  $U^{even}$, corresponding to the respective cases $h^0(\tilde{L}) \equiv 1,0 \ mod (2)$. 
 
Here, $W$ parametrizes the quotient of the variety of pairs $(\tilde{L}, V)$ where $h^0(\tilde{L}) \geq 3$ is odd, and $ V \subset 
H^0(\tilde{L})$ s a two-dimensional subspace, by the action of $\iota$. And similarly for $U^{even}$.

As a matter of fact, for  $C$ general, for any such $\tilde{L} \in |\Xi|$ we have 
 $h^0(\tilde{L}) =2 $, hence  it turns out that $U^{even}/ \iota \cong \Xi /\pm 1$.

We establish now the needed property.

{\bf Claim 5} For a general $C$, for both cases $h^0(\tilde{L}) \geq 2$ even or odd,
 there exists a pair $(\tilde{L}, V)$ where  $ V \subset 
H^0(\tilde{L})$ yields a two-dimensional base point free subsystem.
 
 {\em Proof of the claim.}

Take a special $C$ such that  $ C = Q_1 \cap Q_2 \cap Q_3$, where $Q_1, Q_2, Q_3$ are quadrics in $\PP^4$,  $Q_3$ has rank $=3$, 
and  $ \Sing (Q_3) \cap C = \emptyset.$ This is possible, just starting from $C_3$ and then taking two general quadrics $Q_1, Q_2$.

Now, the projection with centre the vertex $\Sing(Q_3)$ of $Q_3$ yields a morphism 
$\Psi : C\ra \PP^1$ corresponding to a thetacharacteristic $\sG$ on $C$.

Since $C$ is not hyperelliptic, and the degree of $\sG$ is $4$, it follows by the theorem
of Clifford that $h^0 (\sG) \leq 2$, hence equality holds: $h^0 (\sG) =2$.

Now, we consider 
$$\{ \sG \otimes \eta | \eta  \in \Pic(C)_2 \setminus \{0\} \} .$$

This set contains even and odd thetacharacteristics, in particular, since  as we have shown $h^0 (\sG \otimes \eta) \leq 2$,
we can choose an $\eta$ such that $h^0 (\sG \otimes \eta) =1$,
and also we can choose an $\eta$ such that $h^0 (\sG \otimes \eta) =0$
(since $\Sing (\Theta_C)$ cannot contain $2^{g-1} (2^g +1)$ points, see Lemma 4 , pag. 233 of \cite{igusa},
showing that at least one theta constant is non zero).
 
 We set now $\tilde{L} : = p^* (\sG)$, where $ p : \tilde{C} \ra C$ is the \'etale double covering associated to $\eta$.
 
 By the projection formula
 $$ H^0 (\tilde{L} ) = H^0(\sG) \oplus H^0(\sG \otimes \eta).$$
 
 So, the dimension $ h^0 (\tilde{L} ) = 2  + h^0(\sG \otimes \eta)= 3, {\rm or } \ 2,$
 and in the second case we get a base point free pencil, while in the first we get
 a base point free pencil for general choice of $V$, hence for general choice of the divisor $D$.
 
 By a monodromy argument, we get that for general choice of $C$ and $D$ there are no quadrics of rank $3$ 
 in the web satisfying the property $ Sing(Q) \cap C \neq \emptyset$.
\end{proof}

\begin{cor}\label{20/28}
In the Nodal Severi variety $\sF(5,28)$ the condition that $\sK$ contains a codeword of weight 20 determines an irreducible component.

\end{cor}
\begin{proof}
$Y$ must be the discriminant of a web of quadrics with $8$ distinct base points $P_i$, for which there is exactly one quadric
in the web which is singular at $P_i$.  Therefore, for general choice of
$Q_0, Q_1, Q_2$ in the web, $ C : = Q_0 \cap Q_1 \cap Q_2$ is a smooth curve, and $ div_C(Q_3) = 2 D$,
where $ D $ is not contained in a hyperplane, hence $ D \equiv K_C + \eta$, with $\eta$ a nontrivial 2-torsion divisor.

Therefore our family is dominated by the family of genus $5$ canonical curves, given together with a Prym-canonical divisor $D$ as above.
\end{proof}

\appendixOn\section{Codes $\sK$ with only $16, 20$ as weights}\appendixOff
\vskip+10pt\chapterauthor{Michael Kiermaier}
\label{append_quintic_codes}

For quintics, the maximum number $\nu$ of nodes is $31$.
This has been shown by Beauville in  \cite{angers}, where he proved the following 
 important  algebro geometric results:

i) B-inequality: the even sets of nodes of a quintic yield a linear binary $[\nu,k]$ code $C$ with $k \geq \nu-26$ , and with

ii) non-zero weights in $\{16,20\}$.

The main purpose of this appendix is to classify the codes satisfying properties i), ii), respectively only property ii).

This is interesting, because it shows, compared with the results of Part III, that not all of the codes satisying i)
and ii) do geometrically occur as codes of a nodal quintic surface (because of Theorem \ref{web}).

In passing we shall also  give a third proof of the code theoretic part of  this result (the second proof being the content of Corollary \ref{Beau})
as a consequence of the MacWilliams identities for the weight enumerator of linear codes, of course  using i) and ii).

We shall call  \emph{admissible} all the codes which satisfy i) and ii), and we shall here classify them
  up to isomorphism.

\subsection{The MacWilliams identities}

For a code $C$ of length $n$, let $a_i$ be the number of codewords of Hamming weight $i$.
Then $(a_i)_i$ is the weight enumerator of $C$, which may also be written in polynomial form as $w_C(x) = \sum_i a_i x^i\in\ZZ[x]$ or as a homogeneous polynomial $w_C(x,y) = \sum_i a_i x^i y^{n-i}\in\ZZ[x,y]$.
Clearly, $a_0 = 1$.

For a $\F_q$-linear code $C$ of length $n$, the \emph{dual code} is defined as the orthogonal subspace of $\F_q^n$ with respect to the standard bilinear form $\langle \vek{x}, \vek{y}\rangle = \sum_{i = 1}^n x_i y_i$.
We have $\dim(C) + \dim(C^\perp) = n$.
We denote the weight distribution of $C^\perp$ by $(a^\perp_i)_i$.
Then $a^\perp_0 = 1$.
For $n\geq 1$ we have $a^\perp_1 = 0$ if and only if $C$ is \emph{full-length}, i.e., if $C$ does not have an all-zero coordinate.
For $n\geq 2$ we have $a^\perp_2 = a^\perp_1 = 0$ if and only if $C$ is projective.

The MacWilliams identities connect the weight distribution of $C$ and $C^\perp$.
The usage of homogeneous weight enumerators allows the compact formula
\[
	w_{C^\perp}(x,y) = \frac{1}{\#C} w_C(x-y,x+(q-1)y)\text{.}
\]
For the weight enumerators in univariate polynomial form, we get
\[
	w_{C^\perp}(x) = \frac{1}{\#C} (x + (q-1))^n w_C\left(\frac{x-1}{x+(q-1)}\right)\text{.}	
\]
We remind that the code size $\#C$ may also be written in terms of the (homogeneous) weight enumerator as $\#C = w_C(1) = w_C(1,1)$.

Writing down the equations for the individual coefficients, after a few transformations one arrives at the \emph{Krawtchouk form} of the $(i+1)$-th MacWilliams identity
\[
	a^\perp_i = \frac{1}{\#C} \sum_{j=0}^n K_i(j;n,q) a_j\text{,}
\]
where
\[
	K_m(x;n,q) = \sum_{j=0}^m (-1)^j \binom{x}{j} \binom{n-x}{m-j} (q-1)^{m-j}
\]
is the $m$-th $q$-ary \emph{Krawtchouk polynomial} of order $n$.

\subsection{The statement of Beauville}

Let $C$ be a linear binary $[n,k]$ code $C$ with $k \geq n-26$ and non-zero weights in $\{16,20\}$.
Let $(a_i)$ be the weight enumerator of $C$ and $(a^\perp_i)$ the dual weight enumerator.
A priori, we don't know if necessarily $a^\perp_1 = 0$ (i.e. $C$ is full-length).
However by puncturing the zero-coordinates, all admissible codes can be reduced to an admissible full-length code of the same dimension and the same weight distribution.
So in the following we will assume that $a^\perp_1 = 0$, keeping in mind that for an admissible code, all extensions by zero coordinates up to the length $n = k+26$ are admissible, too.
Also, we do not know for sure if $a^\perp_2 = 0$ (i.e. $C$ is projective).

With $a_0 = a_0^\perp = 1$, the first two MacWilliams identities give
\begin{align*}
	1 + a_{16} + a_{20} & = 2^k \\
	n + (n-32)a_{16} + (n-40)a_{20} & = 0
\end{align*}
Therefore
\[
	\begin{pmatrix}
		1 & 1 \\
		n-32 & n-40
	\end{pmatrix}
	\begin{pmatrix}a_{16} \\ a_{20}\end{pmatrix}
	= \begin{pmatrix}2^k - 1 \\ -n\end{pmatrix}
\]
Left-multiplication with the inverse matrix
\[
	-\frac{1}{8}\begin{pmatrix}n-40 & -1 \\ -n+32 & 1\end{pmatrix}
\]
yields
\[
	a_{16} = -2^{k-3}n + 5(2^k-1)
	\qquad\text{and}\qquad
	a_{20} = 2^{k-3}n - 4(2^k-1)\text{.}
\]
From the integrality of these values we get $4\mid n$ for $k=1$ and $2\mid n$ for $k=2$.
From $a_{16} \geq 0$ and $a_{20} \geq 0$ we get $4(2^k-1) \leq 2^{k-3}n \leq 5(2^k-1)$ or equivalently
\begin{equation}
\label{eq:length_bound}
	32 (1 - 2^{-k}) \leq n \leq 40 (1 - 2^{-k})\text{.}
\end{equation}

The third MacWilliams identity gives
\[
	2^k a_2^\perp = (\frac{1}{2} n^2 - \frac{1}{2} n) + (\frac{1}{2} n^2 - \frac{65}{2}n + 512) a_{16} + (\frac{1}{2} n^2 - \frac{81}{2} n + 800) a_{20}\text{.}
\]
Plugging in the expressions for $a_{16}$ and $a_{20}$, we get
\begin{align*}
	a_2^\perp
	& = -\frac{1}{2}n^2 + \frac{71}{2}n - 640 \left(1-\frac{1}{2^k}\right) \\
	& = -\frac{1}{2}\left(n-\frac{71}{2}\right)^2 - \frac{79}{8} + \frac{640}{2^k}
\end{align*}
The condition $a_2^\perp \geq 0$ implies $\frac{640}{2^k} - \frac{79}{8} \geq 0$ or equivalently $2^k \leq \frac{8\cdot 640}{79} \approx 64.8$.
So $k \leq 6$.
In these cases, $a_2^\perp \geq 0$ is equivalent to restriction
\[
	\abs{n - \frac{71}{2}} \leq \sqrt{\frac{1280}{2^k} - \frac{79}{4}}
\]
on $n$.
For $k \leq 5$, this is weaker than Inequality~\ref{eq:length_bound}.
For $k = 6$, Inequality~\ref{eq:length_bound} is sharpened to $n\in\{35,36\}$.
Therefore, the condition $k \geq n-26$ is violated and we have $k \leq 5$, implying that $n \leq k+26 \leq 31$, which is the result of Beauville.

Let $\mathcal{P}$ be the set of points of the projective  space $\PP^{k-1}_{\FF_{q}}$.
 Up to isomorphism, a $\F_q$-linear code $C$ of dimension $k$ and  full length $n$  is determined by a multiset of $n$ 
 points in $\sP$,
given   by the  $n$  rows  of a generator matrix  $G$ (a multiset is  a map $\mathcal{C} : \mathcal{P} \to \ZZ_{\geq 0}$). Two multisets yield the same code if and only if they are projectively
equivalent (this amounts to a change of basis for $C$). Moreover, since  $C$ has dimension $k$, the multiset, 
that is, the support of the function
$\sC$, must span $\sP$.

For $A \subseteq \mathcal{P}$ we consider  the multiplicity function $\mathcal{C}(A) := \sum_{x \in A} \mathcal{C}(x)$,
and observe that 
to a codeword  $\vek{c} = G \vek{x}$ corresponds  the hyperplane $H_{\vek{x}} : = \vek{x}^\perp$, and the weight  $w(\vek{c}) $
equals $  \mathcal{C}(\mathcal{P} \setminus H_{\vek{x}})  = n - \sC(H)$.
Two codewords lead to the same hyperplane if and only if they  yield the same point in $\sP$. 

By Proposition \ref{weights} it follows that the weights of the code vectors determine the function $\sC(x)$,
hence the code $C$.

Another way to argue    is that the incidence matrix of the points vs. hyperplanes of    is   invertible
(see \cite{kung} section 1.2, page 35),
hence, if the multiplicities of the hyperplanes (or their complements) are fixed, so are the multiplicities  $\sC(x)$ of the points.

\subsection{Classification of the admissible codes}

Table~\ref{tbl:codes_quintic} shows all remaining tuples $(n,k,a_{16},a_{20},a_2^\perp)$ with $k \geq n - 26$.
We see that, indeed, there always exists a code with the given parameters.

\begin{table}
\centering 
	$\begin{array}{llllllll}
	\text{\textnumero} & n & k & a_{16} & a_{20} & a_2^\perp &\text{Codes} & \text{shorten.} \\
	\hline
	1 & 31 & 5 & 31 & 0 & 0 & \operatorname{Sim}_2(5) & 2 \\
	2 & 30 & 4 & 15 & 0 & 15 & 2\operatorname{Sim}_2(4) & 4 \\
	3 & 29 & 3 & 6 & 1 & 49 & \begin{tabular}[t]{@{}l@{}} in pr. plane:  \\[-0.5em]
																\quad line of multipl. 3,\\[-0.5em] \quad otherw. multipl. 5 \end{tabular} & 6, 7 \\
	4 & 28 & 3 & 7 & 0 & 42 & 4\operatorname{Sim}_2(3) & 7 \\
	5 & 28 & 2 & 1 & 2 & 122 & \begin{tabular}[t]{@{}l@{}} points on pr. line\\[-0.3em] \quad of multipl. $(8^2 12^1)$ \end{tabular} & 8, 9 \\
	6 & 26 & 2 & 2 & 1 & 105 & \begin{tabular}[t]{@{}l@{}} points on pr. line\\[-0.3em] \quad of multipl. $(6^1 10^2)$ \end{tabular} & 8, 9 \\
	7 & 24 & 2 & 3 & 0 & 84 & \begin{tabular}[t]{@{}l@{}} $8\operatorname{Sim}_2(2) \equiv$ \\[-0.5em]
																\quad points on proj. line \\[-0.5em]\quad of multipl.$(8^3)$ \end{tabular} & 9 \\
	8 & 20 & 1 & 0 & 1 & 190 & \begin{tabular}[t]{@{}l@{}} $20\operatorname{Sim}_2(1) \equiv$ \\[-0.5em]
																\quad $20$--fold repetition code \end{tabular} & 10 \\
	9 & 16 & 1 & 1 & 0 & 120 & \begin{tabular}[t]{@{}l@{}} $16\operatorname{Sim}_2(1) \equiv$ \\[-0.5em]
																\quad $16$--fold repetition code \end{tabular} & 10 \\
	10 & 0 & 0 & 0 & 0 & 0 & \text{zero code}
	\end{array}$
	\caption{Binary linear $[n,k]$-codes with weights in $\{16,20\}$ and $n \leq k+26$}\label{tbl:codes_quintic}
\end{table}

All codes are unique up to isomorphism by the following argument.
In all cases, it is obvious  from the list that up to isomorphism there is only a single choice to assign the weights $16$ or $20$ to the hyperplanes.

Therefore the Radon transform of the multiset $\sC(x)$ is unique, hence, by Radon duality, $\sC(x)$ is unique,
hence the code -- if one exists -- is unique.

So we have:

\begin{theorem}
If there is a quintic with $\nu$ nodes, its even sets of nodes form a binary linear $[\nu,k]$ code with $k \geq \nu-26$.
After removing the zero coordinates, the code is isomorphic to one of the codes in Table~\ref{tbl:codes_quintic}, where each line represents a single isomorphism type of a linear code.
\end{theorem}

\subsection{Remark on the classification without the dimension restriction}
Without the restriction $k \geq \nu-26$, further hypothetical parameters of codes show up, see Table~\ref{tbl:codes_quintic_2}.

The two codes with $k=6$ are projective two-weight codes, which have been  classified in \cite{tonchev}.
For $n=35$ there are $7$ isomorphism types and for $n=36$ (via the complement with $n=27$) there are $5$ isomorphism types.
The most symmetric ones have a nice geometric description: for n=36, take the complement in $\PP G(5,2)$ of the points on a full rank elliptic quadric.
For $n=35$, take the points on a full rank hyperbolic quadric, which is the Klein quadric.

For code \textnumero{32}, take the three points on a line with multiplicities $2$, $6$ and $6$ and the remaining $4$ points with multiplicity $4$. The code is uniquely determined by the above argument.

For code \textnumero{31}, the three lines corresponding to the codewords of weight $20$ are either a line pencil or a line triangle.
In the first case, the code is uniquely determined by a point of multiplicity $1$ and the other six points of multiplicity $5$.
In the second case, the code is uniquely determined by a point triangle of multiplicity $3$, the other three points on the secants with multiplicity $5$ and the remaining point with multiplicity $7$.

For code \textnumero{30}, the three lines corresponding to the codewords of weight $16$ are either a line pencil or a line triangle.
In the first case, the code is uniquely determined as a point of multiplicity $8$ and the other six points of multiplicity $4$.
In the second case, take a point triangle with multiplicity $6$, the other three points on the secants with multiplicity $4$ and the remaining point with multiplicity $2$.

For code \textnumero{29}, take the three points on a line with multiplicities $3$, $3$ and $7$ and the remaining $4$ points with multiplicity $5$.

For code \textnumero{28}, take three collinear points of multiplicity $6$ and the remaining four points with multiplicity $4$.

For code \textnumero{26}, select a line $L$ and a plane $E$ through $L$ in $\PP G(3,2)$. Take the three points on $L$ with multiplicity $1$, the four points on $E\setminus L$ with multiplicity $3$ and the remaining eight points with multiplicity $2$.

\begin{table}
\centering 
	$\begin{array}{llllllll}
	\text{\textnumero} & n & k & a_{16} & a_{20} & a_2^\perp &\text{Codes} & \#\text{types} \\
	\hline
	11 & 36 & 6 & 27 & 36 & 0 & \text{cpl. of full rank elliptic quadric} & 5 \\
	12 & 35 & 6 & 35 & 28 & 0 & \text{Klein quadric} & 7 \\
	13 & 38 & 5 & 3 & 28 & 7 & \operatorname{Sim}_2(5) + \operatorname{Sim}_2(3) & ? \\
	14 & 37 & 5 & 7 & 24 & 9 & & ? \\
	15 & 36 & 5 & 11 & 20 & 10 & & ? \\
	16 & 35 & 5 & 15 & 16 & 10 & & ? \\
	17 & 34 & 5 & 19 & 12 & 9 & & ? \\
	18 & 33 & 5 & 23 & 8 & 7 & & ? \\
	19 & 32 & 5 & 27 & 4 & 4 & & ? \\
	20 & 37 & 4 & 1 & 14 & 29 & 2\operatorname{Sim}_2(4) + \operatorname{Sim}_2(3) & 1\\
	21 & 36 & 4 & 3 & 12 & 30 & 2\operatorname{Sim}_2(4) + 2\operatorname{Sim}_2(3) & 1\\ 
	22 & 35 & 4 & 5 & 10 & 30 & & ? \\
	23 & 34 & 4 & 7 & 8 & 29 & & ? \\
	24 & 33 & 4 & 9 & 6 & 27 & & ? \\
	25 & 32 & 4 & 11 & 4 & 24 & & ? \\
	26 & 31 & 4 & 13 & 2 & 20 & \text{see text} & 1 \\
	27 & 35 & 3 & 0 & 7 & 70 & 5\operatorname{Sim}_2(3) & 1 \\
	28 & 34 & 3 & 1 & 6 & 69 & 4\operatorname{Sim}_2(3) + 2\operatorname{Sim}_2(2) & 1 \\
	29 & 33 & 3 & 2 & 5 & 67 & \text{see text} & 1 \\
	30 & 32 & 3 & 3 & 4 & 64 & \text{see text} & 2 \\
	31 & 31 & 3 & 4 & 3 & 60 & \text{see text} & 2 \\
	32 & 30 & 3 & 5 & 2 & 55 & \text{see text} & 1 \\
	33 & 30 & 2 & 0 & 3 & 135 & 10\operatorname{Sim}_2(2) & 1 \\
	\end{array}$
	\caption{Binary linear $[n,k]$-codes with weights in $\{16,20\}$ and $n > k-26$}\label{tbl:codes_quintic_2}
\end{table}

\renewcommand{\thesection}{\thechapter.\arabic{section}}
\chapter{Nodal Sextics}
\chapterauthor{Fabrizio Catanese, Yonghwa Cho, Michael Kiermaier}

\section{Half-even sets on nodal sextic surfaces}\label{half-even-sets-sextic}

 This section is important for our purposes, even if we cannot yet completely classify the half-even sets of nodes on nodal sextic surfaces.  Indeed, using homological algebra and earlier results, we get some information about   the possible numbers attained by the cardinality $t$ of a half-even set of nodes on a sextic.
 
It satisfies $t \equiv -1 \ (mod  \ 4)$ and $t \geq 15$. Moreover, it  can  be $15, 27, 31, 35$ but cannot be $19, 23$;
 using the description of the code of the Barth sextic we see that also the cases $t=39,43$ occur. We cannot determine in general whether the cases with $t \geq 47$ occur, but we see that this occurrence is not allowed if the number of nodes $ \nu = 65$.

Building from  this we achieve here  our original aim, to give a sufficient condition in order that  a nodal sextic surface $Y$ be the discriminant of
a cubic hypersurface. The condition is that $Y$ contains a half-even set $\sN$  of  cardinality  31, and such that 
a certain linear system, $|2H - L|$, has dimension zero (observe 
that, if $t=31$, then necessarily $ dim |2H - L| \geq 0$: we do not yet know whether the first condition alone is sufficient).

We set up again the notation: $Y$ is a nodal sextic surface, $S : = \tilde{Y}$ is its minimal resolution,
and $\sN$ is an half-even set of nodes. 

We denote by $E_1, \dots, E_{\nu}$ the exceptional $(-2)$-curves,
therefore there exists a divisor $L$ on $S$ such that
$$(*)  \sum_{i \in \sN} E_i + H \equiv 2 L. $$

We denote by $t : = | \sN|$, and without loss of generality we assume that $\De: =  \sum_{i \in \sN} E_i = \sum_1^t E_i$.

In this situation we attach to $(*)$ a double covering $ f : Z \ra S$ ramified on $\De + H$, where $H$ is 
here any smooth hyperplane section, and such that
$$ f_* (\hol_Z) = \hol_S \oplus \hol_S (-L). $$

The inverse images of the  $(-2)$-curves $E_1, \dots, E_t $ are  $(-1)$ curves, and can be contracted 
to smooth points, yielding $p : Z \ra Z'$ and a finite double covering $ f' : Z' \ra Y$.

We have $$K_Z = f^* (K_{\tilde{Y}} + L) = f^* (2H  + L)  $$
and, since by Kodaira-Ramanujam  vanishing $H^1(Z, \hol_Z( K_Z + n f^* H))= 0$ for $n \geq 1$,
we get
$$0 = H^1(Z, \hol_Z(   f^* (n + 2) H + L)) =   H^1(S, \hol_S(    (n + 2) H + L) \oplus   H^1(S, \hol_S(    (n + 2) H ).$$

We obtain $H^1(S, \hol_S(   m  H + L) ) = 0$ for $ m \geq 3$, and using

(i) Serre duality on $S$: 
$$h^1(S, \hol_S(   m  H + L) ) =  h^1(S, \hol_S(  (2- m)  H - L) )  =  h^1(S, \hol_S(  (1- m)  H + L - \De)  $$

(ii) the exact cohomology sequence associated to the exact sequence
$$ 0 \ra  \hol_S(  a H + L - \De) \ra  \hol_S(  a H + L )  \ra \hol_{\De} (L) \ra 0,$$
and the vanishing of all cohomology groups of $\hol_{\De} (L) $ since $h^j(\hol_{\PP^1} (-1) = 0$,
yielding
$$(**) \ h^i( \hol_S(  (a +1)  H - L ) ) = h^i( \hol_S(  a H + L - \De) ) = h^i ( \hol_S(  a H + L ) ) \ \ \forall i,$$ 

we conclude that 

$$h^1(S, \hol_S(   m  H + L) ) =  h^1(S, \hol_S(  (2- m)  H - L) )  =  h^1(S, \hol_S(  (1- m)  H + L )  ,$$
hence we have the following vanishing result

\begin{prop}\label{cm}
The first cohomology groups $H^1(S, \hol_S(   m  H + L) ) $ and $H^1(S, \hol_S(   m  H - L) ) $
are zero for all $m$ if and only if $H^1(S, \hol_S(   m  H + L) ) = 0$ for $ m=1,2 $ or for $m = 0, -1$.

\end{prop}

Recall that \cite{cascat} on $Y$ we have the so called quadratic sheaf 
$$\sF : = \pi_*  \hol_S (-L), \ {\rm such \ that } \  f'_* (\hol_{Z'})  = \hol_Y \oplus \sF.$$

\begin{cor}\label{CM}

$h^1(S, \hol_S(   m  H + L) ) = 0$ for $ m=1,2 $ if and only if the following equivalent conditions hold:

(1) the sheaf $\sF$ is arithmetically Cohen-Macaulay,

(2) $H^1 (\sF) = H^1 (\sF (1))   = 0$,

(3) $\sN$ is a symmetric set of nodes, that is, 
 there is a symmetric matrix of homogeneous polynomials $A_{i,j}(x)$ such that
$$Y = \{ x | \ \det (A(x) ) = 0 \}, \ \sN =  \{ x | \ \corank  (A(x)) = 2 \}.$$

\end{cor}

We use now the Riemann-Roch theorem to calculate Euler-Poincar\'e characteristics.

\begin{prop}

\begin{align*}
\chi (a H + L):  &= \chi (\hol_S ((a H + L)) \\ &=  \chi (\hol_S (((a+1)  H - L )) \\ &=   11 + \frac{1}{4} [ 3 (2a+1)(2a-3) - t].
\end{align*}

In particular, 

\begin{align*}
	\chi (\sF) = \chi (\sF(3)) = 11 +  \frac{1}{4} (15-t),
\end{align*}
and accordingly we have
\begin{gather}
	\chi (\sF(3)) \geq 0 \Leftrightarrow t \leq 59, \nonumber\\
	\chi (\sF) > 0 \Leftrightarrow t \leq 55. \nonumber
\end{gather}
\end{prop} 

\begin{proof}

It was already shown that $ \chi (\hol_S ((a H + L)) =  \chi (\hol_S (((a+1)  H - L ))$.
Then 
$$\chi (a H + L):  = \chi (\hol_S ((a H + L)) =  \chi (\hol_S) + \frac{1}{2} (a H + L)(a H + L - K_S)= $$
$$ = 11 + \frac{1}{8} ((2a+1) H + \De) ((2a-3) H + \De) =  11 + \frac{1}{4} [ 3 (2a+1)(2a-3) - t].$$

We finish observing that, because of rational singularities, 

$ \chi (\sF (m) )=  \chi (m H - L)) = \chi ((m-1) H + L)$.
\end{proof}

\begin{lemma}\label{degrees}
The cardinalities of the symmetric sets of nodes 
depend on the degrees $d_1 \leq d_2 \leq \dots \leq d_p$ of the diagonal entries of the matrix $A(x)$,
which must be odd numbers.

\begin{itemize}
\item
degrees (1,5) : $t = 15$, in this case $L$ is effective ($|L|$ has dimension $0$);
\item
degrees (3,3) : $t = 27$, in this case $L$ is not effective, but $L+H$ is effective ($|L+H|$ has dimension $1$);
\item
degrees (1,1,1, 3) : $t = 31$, in this case $L$ is not effective, but $L+H$ is effective ($|L+H|$ has dimension $0$);

\item
degrees (1,1,1,1,1,1) : $t = 35$, in this case $L+ H $ is not effective, but $L+ 2 H$ is effective.

\end{itemize}
\end{lemma}

\begin{proof}
We use the inequalities of \cite{babbage}, and of \cite{cascat}, page 249, especially:
$ d_i + d_{p-1-i} > 0$.

This inequality implies, since $\sum_1^6 d_i = 6$, and since $d_i$ is odd, that $p$ is even, hence
$$  \sum_{2i \leq p} (d_i + d_{p-i} ) = 6.$$ 
Since $ d_i + d_{p-i} > 0 \geq  d_i + d_{p-1-i} > 0$, $d_i + d_{p-i} \geq 2 $, therefore $ p $ equals one of the three
numbers $2,4,6$.

From $ d_i \leq d + \de -2$ follows that $d_i \leq 5$, hence $d_p \in \{1,3,5\}$. If $d_1 \geq 1$, then all the $d_i$'s are positive and we
get the claimed four cases.

Since $d_1 + d_{p-2} > 0$, follows that $ d_1 \geq -3$, and if $d_1 = -3$, then $d_{p-2},  d_{p-1} , d_{p} =5$; 
hence $ p=6$ and, since the sum yields $6$, we must have $(-3, -3, -3, 5,5,5)$, which contradicts however $d_2 + d_3 > 0$.

So, the case $d_1 = -3$ is excluded, and similarly the case $d_1 = -1$ is excluded. Since then $d_{p-2},  d_{p-1} , d_{p} \geq 3$,
hence we must have $(-1, -1, -1, 3,3,3)$, which again contradicts  $d_2 + d_3 > 0$.

The other assertions follow since $L+ mH$ is effective, then also $L+ mH - \De$ is effective, 
since $(L + mH) \cdot E_i = -1$. Hence $L+ mH$ is effective if and only if $L+ mH - \De = (m+1) H - L$
is effective, equivalently $ H^0 (\sF(m+1)) \neq 0$.  

Then we use the exact sequence
$$ 0 \ra \sum_1^p \hol_{\PP^3} (\frac{-7 - d_j}{2}) \ra \sum_1^p \hol_{\PP^3} (\frac{-7 + d_i}{2}) \ra \sF  \ra 0,$$
which is exact on global sections.
\end{proof}

The same argument shows that, in general,  the restriction $\sF \otimes \hol_H$ 
of $\sF$ to a smooth hyperplane section $H$ has the following resolution:

$$ 0 \ra \sum_1^p \hol_{\PP^2} (\frac{-7 - d_j}{2}) \ra \sum_1^p \hol_{\PP^2} (\frac{-7 + d_i}{2}) \ra \sF \otimes \hol_H \ra 0.$$

The associated global section sequence is also exact, and we conclude

\begin{rem}\label{restriction}
The cohomology groups of the restriction $\sF |H : = \sF \otimes \hol_H$ 
of $\sF$ to a smooth hyperplane section $H$ have respective dimensions
 (according to the degrees $(1 , 5), (3 , 3) ,(1^3, 3), (1^6)$):
\begin{itemize}
\item
$h^0 ( \sF |H ) = 0,$
\item
$h^0 ( \sF(1) |H ) = 1,0,0,0,$
\item
$h^0 ( \sF(2) |H ) = 3,2,1,0, $
\item
$h^1 ( \sF |H ) = 12,$
\item
$h^1 ( \sF(1) |H ) = 7,6,6,6,$
\item
$h^1 ( \sF(2) |H ) = 3,2,1,0. $
\end{itemize}

\end{rem}

It follows that $h^0 ( \sF  ) = 0,$ and $h^0 ( \sF(1)  ) = 0$ except possibly in the first case. This case is however clear, since we have

\begin{lemma}\label{Leff}
$h^0 ( \sF(1)  ) \neq  0$, equivalently, $L$ is effective, if and only $\sN$ is symmetric and  we are in the first case of lemma \ref{degrees},
and also if and only if  $t = 15$.
\end{lemma} 

\begin{proof}
If $L$ is effective, there is a plane $H$ cutting on $\tilde{Y}$ the divisor $ 2 D + \De$,
hence we have a plane $z=0$ such that the intersection with $Y$ is twice a curve $\Ga$, and moreover
the plane passes through the nodes of $\sN$.

Hence the equation $F = 0$ of $Y$ equals the square of a cubic polynomial  $g (x_0, x_1, x_2)$, modulo $z$,
hence
$$ F (x_0, x_1, x_2, z) = g (x_0, x_1, x_2)^2 + z P(x_0, x_1, x_2, z) ,$$
for some degree $5$  polynomial $P$ and $F$ is the  determinant of the symmetrical matrix with entries $P,g,g, z$.

At the nodes of $\sN$, $g=0$ must be a smooth curve (else the quadratic part of $F$ would not have rank equal to $3$),
and the singular points of $Y$ are then clearly the $15$ points where $z=P=g=0$. These are the points where the corank of the symmetrical matrix equals $2$.

The last equivalence follows from Theorem 1.10, part (ii) of Endrass, \cite{endrass1}, showing that $ t \geq 15$, always,
equality holding if and only if  $h^0 ( \sF(1)  ) \neq  0$.
\end{proof} 

\begin{lemma}\label{linsymm}
Assume $h^0 ( \sF(1)  ) = 0$ and $ t \leq 31$: then  $h^0 ( \sF(2)  ) \neq 0$, equivalently, $H+ L$ is effective.
If $h^0 ( \sF(2)  ) =  0$, then $h^1( \hol_S (H+L) ) = \frac{1}{4} ( t - 35)$.

In particular, if $h^0 ( \sF(2)  ) =  0$ and $ t = 35$, then the set $\sN$ is  symmetric and we are in the fourth case of
Lemma \ref{degrees}.
\end{lemma} 

\begin{proof}
We apply the Riemann-Roch inequality to the divisor $H+L$:
$$  h^0( \hol_S (H+L) ) =     h^0( \hol_S (H+L) ) +  h^0( \hol_S (H- L) )= h^0( \hol_S (H+L) ) +h^2( \hol_S (H+L) ) \geq $$
$$\geq  \chi ( \hol_S (H+L) )  =  11 + \frac{1}{4} (-9 - t) . $$
Hence $  h^0( \hol_S (H+L) )  \geq 1$ as soon as $10 + \frac{1}{8} (-18 - 2t) \geq 0 \Leftrightarrow 2 t \leq 62.$

The second assertion follows right away from the previous computation, since then $ \chi ( \hol_S (H+L) ) =  - h^1 ( \hol_S (H+L) )$.

For the third  assertion, we know that from $h^0 ( \sF(2)  ) =  0$ and $ t = 35$ follows that $h^i( \hol_S (H+L) )= 0$
for all $i$, equivalently, $h^i (\sF(1) = 0$ for all $i$.

The exact cohomology sequence associated to
$$ 0 \ra  \sF(1) \ra \sF(2) \ra \sF(2)| H \ra 0$$
implies then that $H^0 ( \sF(1)| H) = 0$, whence we are in the last case for the hyperplane section
and we also infer that $h^i (\sF(2) = 0$ for all $i$.

The exact cohomology sequence associated to
$$ 0 \ra  \sF(2) \ra \sF(3) \ra \sF(3)| H \ra 0$$
shows then that $h^0 ( \sF(3)) = h^0(  \sF(3)| H) = 6,$
and that these global sections  generate $ \sF(3)$ on the smooth locus, hence everywhere
by the reflexivity of $\sF$. 

We can  now just apply theorem 2.23 of \cite{babbage}, or give a direct argument as follows:
we can  show that there is an exact sequence 
$$ 0 \ra \hol_{\PP^3} (-1)^6 \ra \hol_{\PP^3} ^6 \ra \sF(3) \ra 0$$
by the analogous result for the restriction to the hyperplane section. 
\end{proof}

\begin{rem}
  \label{remark_half_even_sets}
(i) The analysis of the code of the Barth sextic, together with remark \ref{CMsextic} show that there
are half-even sets of nodes with $t=35$ and not symmetric, hence with $h^0 ( \sF(2)  ) \neq  0$.

It shows moreover that the cases $ t=39, t=43$ do indeed occur ($t=39$ had been previously constructed by Barth \cite{barth1}).

(ii) The same argument as in lemma \ref{linsymm} 
shows that $h^0 ( \sF(3)  ) \neq 0$, equivalently, $2 H+ L$ is effective as soon as $ t <  59$.

 The cases where   $t \geq 47$ can be ruled out for $\nu = 65$ by coding theory, see appendix C.

(iii) Endrass \cite{endrass1} proved that $ t = 19, 23$ do not occur.

\end{rem}
In another vein, we can give an upper bound for $t$:

\begin{prop}
Assume that $ H + L$ is effective. Then  $ t \leq 43$.

\end{prop}
\begin{proof}
If $ H + L$ is effective, there is an effective  divisor $D$ such that $ 2 D \equiv 2 H + 2L = 3  H +  \De$,
and since $ D \cdot \De = -1$, $D = C + \De$, with $C$ effective, hence $ 2 C \equiv  3 H -  \De$.

Consider the polar system, spanned by the partial derivatives of $F$: it is a subsystem of $ 5 H - \De$, 
and it has no base points, whence $ (5 H - \De) \cdot C \geq 0 $.

We get  $$ 0 \leq (5 H - \De) \cdot ( 3   H - \De) =  90 - 2 t  \,\,\Leftrightarrow\, t \leq 45,$$
and we conclude since $t +1$ is divisible by $4$.
\end{proof}

\begin{prop}\label{fixedpart}
Assume that $h^0 (\sF(1)) = 0$ and that $h^0 (\sF(2) )\geq 2$, observing that
$$h^0 (\sF(2)) \geq 2 \Leftrightarrow \ dim  |H + L| \geq 1 \Leftrightarrow \ dim  |2H - L| \geq 1.$$ 

 Then for a smooth hyperplane section $H$ we have
 $$ h^0 (\sF|_H(2)) = 3 ,$$
unless $t=27$, $h^0 (\sF(2) )=  2$ and $ | 2H - L| $ is a base point free pencil.
\end{prop}
\begin{proof}
By the exact sequence $ 0 \ra \sF(1) \ra \sF(2) \ra \sF|_H(2) \ra 0$ we obtain that 
$2 \leq  h^0 (\sF|_H(2)) \leq 3 ,$ hence we are just in one of the first two cases of remark \ref{restriction}).

In the second case we have $ h^0 (\sF|_H(2)) = 2 $ and the linear system has no base point on $H$. 

Writing therefore $$ | 2H - L| = | H + L - \De| = F + |M|.$$
where $F$ is the fixed part, we have $F H = 0$. 

Therefore the divisor $F$ is supported on the exceptional divisors, and we can write:
$$ F = \sum_i a_i E_i.$$

Write $G : = F + M$, so that $ 2 G \equiv 3 H - \De$. We have $ G E_i = 1$, $G H = GM = 9$
$$G^2	= \frac{1}{4} ( 3 H - \De)^2=  \frac{1}{4} ( 54 - 2t).$$
On the other hand first of all
$$ 1 = G E_i = M E_i + F E_i  \Rightarrow M E_i = 1 + 2 a_i,$$
secondly
$$G^2 = (M+F)^2= M^2 + 2 MF + F^2 = M^2 + 2 MF + \sum_i ( - 2 a_i^2) = $$
$$ =  M^2 +    \sum_i ( 2 a_i ( 2 a_i + 1) - 2 a_i^2)=
M^2 +   2  \sum_i (   a_i^2 + a_i).$$

Since $ 3 H \equiv 2 G + \De \equiv 2 M + \sum_i ( 2 a_i + 1) E_i$, and a cubic cannot have a point of multiplicity bigger than $4$,
it follows that $2 a_i + 1 \leq 3 \Leftrightarrow a_i \leq 1$. 

By our assumption follows that $ t \geq 27$.

Since $M^2 \geq 0$, we obtain a contradiction if $ t \geq 31$. 

If instead $ t = 27$, it follows that $a_i = 0$ for all $i$, hence $F$=0, and $M^2=0$.
\end{proof}

\begin{prop}
Assume that $t=31$ and that $h^0 (\sF(2) )= 1$. Then we are in the symmetric case with degrees $(1,1,1,3)$.
\end{prop}

\begin{proof}
We have $h^0 (\sF(1) )= 0$ by lemma \ref{Leff}, so by duality  $h^2 (\sF(2) )= 0$. Since $h^0 (\sF(2) )= h^2 (\sF(1) )= 1$
and $\chi  (\sF(2) )= \chi  (\sF(2) ) = 1$, follows the vanishing $h^1 (\sF(2) )= h^1 (\sF(1) )= 0$. 

By the exact sequence $ 0 \ra \sF(1) \ra \sF(2) \ra \sF|_H(2) \ra 0$ follows that we are in the third case for $ \sF|_H$, hence $h^0 (\sF|_H(1)) =0$.

By the exact cohomology sequence associated to the sequence $ 0 \ra \sF \ra \sF(1) \ra \sF|_H(1) \ra 0$ follows
that $h^1 (\sF) \leq  h^1 (\sF(1) )= 0$, hence we can apply corollary \ref{CM} and obtain that $\sN$ is symmetric
(the sheaf $\sF$ is Cohen-Macaulay), and of the desired form by lemma \ref{degrees}.
\end{proof}

\begin{rem}
It is an interesting question to see whether if $t=27$ or $t=31$ then necessarily the half-even set $\sN$ is symmetric.

\end{rem}

In order to construct symmetric half-even sets of nodes with $t=31$ we consider three distinct half-even sets of nodes 
$\sN_1, \sN_2, \sN_3$ with $t=15$.

We first observe that the sum $\sN_1+ \sN_2$ of $\sN_1$ and $\sN_2$, that is, the symmetric difference  
$(\sN_1\cup\sN_2) \setminus (\sN_1\cap\sN_2) $, is a strictly even set of weight (= cardinality) $t=24$.

In fact, its weight is at most $30$, and lies in the set $\{24, 32, 40, 56\}$ by \cite{cat-ton}.

Hence $|\sN_1\cap\sN_2| = 3$. 

Since half-even sets of nodes 
$\sN_1, \sN_2, \sN_3$ with $t=15$ correspond to nodes lying in a plane section $H_i$ of $Y$,
the above observation shows that   $\sN_1\cap\sN_2$ consists of three collinear nodes (in the line $H_1 \cap H_2$).

Therefore a priori $\sN_1\cap \sN_2 \cap \sN_3$ has cardinality $0,1,3$. Accordingly 
the half-even set $\sN_1 +  \sN_2 + \sN_3$ would have cardinality $t =  27, 31, 39$.

Hence we get $t=31$ if and only if the three planes $H_1, H_2, H_3$ are linearly independent and meet in
a node of $Y$. In this case we shall show that $\sN_1 +  \sN_2 + \sN_3$ is symmetric.

\begin{prop}\label{discriminant}
Assume  that we have a nodal sextic $Y$ with three distinct half-even sets of nodes 
$\sN_1, \sN_2, \sN_3$ with $t=15$, such that $\sN : = \sN_1 +  \sN_2 + \sN_3$ has  cardinality $t =  31$.

Denote by $L_i$ the corresponding divisor classes on $S = \tilde{Y}$ such that 
$$ 2 L_i \equiv H + \sum_{j \in \sN_i} E_j,$$
and $L$ such that
$$ 2 L \equiv H + \sum_{j \in \sN} E_j.$$

Then $ | 2H - L|$ has projective dimension $0$, that is, $\sN$ is symmetric.

\end{prop}
\begin{proof}
We know that $t=15$ implies that $H- L_i$ is effective, $|H- L_i| = \{ C_i\}$,
where $C_i$ maps to a plane cubic in $Y$, where the plane $H_i$ is everywhere tangent to $Y$.
Since $ 2 C_i +  \sum_{j \in \sN_i} E_j  \equiv H $, it follows that $C_i$  does not contain any exceptional curve
and that $C_i$ is reduced.

In fact, let $x_0$ be a linear form vanishing on the plane $H_i$: then the equation of $Y$ can be writen as
$$ F(x) = g_3(x_1,x_2,x_3)^2 + x_0 P_5(x_1,x_2,x_3)+ x_0^2 B(x) = 0.$$
The singular points on $Y$ on the plane $ x_0 = 0$ are just the points $x_0 = g_3 = P_5=0$.
These are $15$ points if we have a transversal intersection, which is only possible if $g_3$ has no multiple factors, and then
the full transform of $x_0=0$ contains each node precisely with multiplicity $1$.

Moreover, for $i \neq j$, $C_i$ and $C_j$ have no common component: since $C_i$ spans a plane $H_i \neq H_j$,
the only possibility were that $C_i$ and $C_j$ have a common line. But if $g_3$ has a linear factor $x_1$,
then the line $ x_0 = x_1 = 0$ passes through $5$ nodes of $Y$: this is a contradiction to $|\sN_i \cap \sN_j| = 3.$

We have  that 
$$  C_1 + C_2 + C_3 + \De'   \in | 2 H - L| , \ \De' = E'  + E_{1,2} + E_{2,1} + E_{1,3}+ E_{3,1}+ E_{2,3}+ E_{3,2},$$
where $E'  + E_{i,j} + E_{j,i}$ are the exceptional curves corresponding to the three nodes in $\sN_i \cap \sN_j$.

Clearly then $C_i \cdot E' = 1, C_i \cdot E_{j,k} = 1,0$ according to $ i \in \{j,k\}, \  i \notin \{j,k\}$.

Since $C_i \cdot  \De' = 5$, $C_i \cdot H = 3$, $C_i L = \frac{1}{2} (3 + 11) =  7$ follows that $C_i \cdot (2H - L) = -1$.

If $C_i$ is irreducible, it follows immediately that $C_i$ is in the fixed part of the linear system $ | 2H - L|$.

 If instead $C_i$ is reducible, write $ C_i = R + Q$, where $R$ is a line; then $C_i \cdot (2H - L) = -1$
 implies that an irreducible component of $C_i$ is in the fixed part.
 
 If the line $R$ is in the fixed part, let us calculate:
 $$ ( 2H - L - R ) \cdot Q = ( 2H - L  ) \cdot Q - R \cdot Q = $$
 $$ = ( 2H - L  ) \cdot C_i - ( 2H - L  ) \cdot R - 2 = ( 2H - L  ) \cdot R  -3.
  $$
  We have $$( 2H - L  ) \cdot R = \frac{1}{2} ( 3H - \sum_{j \in \sN} E_j  ) \cdot R \leq \frac{3}{2},$$ 
  hence $ ( 2H - L - R ) \cdot Q < 0$ and $Q$ is in the fixed part if $Q$ if irreducible, or at least it contains a
  line in the fixed part.
  
  Hence we may assume without loss of generality (changing the decomposition of $C_I$)
   that $Q$ is in the fixed part, and we calculate:
  $$   ( 2H - L - Q ) \cdot R = ( 2H - L ) \cdot R - 2 \leq \frac{1}{2} \Rightarrow ( 2H - L - Q ) \cdot R < 0,$$ 
hence we have shown that $C_i$ is in the fixed part of $ | 2H - L|$.

Therefore $$ | 2H - L| = C_1 + C_2 + C_3 + |\De' | = \{ C_1 + C_2 + C_3 + \De' \},$$
and our assertion is proven.
\end{proof}

\section{Discriminant nodal sextic surfaces}

This  section is devoted to  the study of discriminant sextics, obtained via the projection to $\PP^3$, 
with centre a plane $L \subset X$,  
  of  a cubic hypersurface $X$ of $\PP^6$.
  
  \subsection{Preliminaries on discriminant nodal sextics surfaces}

We proceed now to some considerations concerning the case of discriminantal nodal sextics, in view of (4) of Theorem 
\ref{smooth centre}.

The finite set $ \Sing(X) \cap L$ is the base locus in $L \cong \PP^2$ of the system  of quadrics spanned by $Q_0, Q_1, Q_2, Q_3$,
which can only be either a net or a web, and not a pencil: otherwise there would be  a line $\{x_1=x_2=0\}$ where $A'$
vanishes, hence this line would be  singular for $Y$.

\begin{theo}\label{at most 3}
$ \Sing(X) \cap L$ has cardinality at most $3$. If the cardinality equals $3$, then we may choose coordinates
such that the matrix 
\begin{equation}\label{triplanes}
A': =\left(\begin{matrix}0 &x_1&x_2\cr  x_1& 0&x_3\cr x_2&x_3&0
\end{matrix}\right),
\end{equation} 
in particular in this case  the cubic $ M : = \{ x | \det (A'(x)=0\}$ consists of three general planes.

If the cardinality of $ \Sing(X) \cap L$ equals $2$, then there are $3$ possibilities for the cubic $M$:
either it  consists of a quadric cone (rank $3$ quadric)  and a 
plane passing through the vertex, or of a rank $4$ quadric and a tangent plane, or of three general planes.

Finally, if $M$ is reducible, then $ \Sing(X) \cap L$ can have
  cardinality only $0,2,3$, while cardinality $1$ occurs for $M$ irreducible.

There is exactly one  case   with   $M$  reducible   where $ \Sing(X) \cap L$ is empty:
$M$ consists of a quadric cone and a plane intersecting it transversally in a smooth conic.
\end{theo}
\begin{proof}

We have seen that for each such base point (that is, a point of $ \Sing(X) \cap L$) we can take coordinates such that the point is $e_0$,
and the two points $x_1=x_2= x_0 x_3=0$ in $\PP^3$ yield nodal singularities of the conic
bundle $\sC$. 

It follows then that the cubic $M$ (everywhere tangent to $Y$) is singular along the line $ x_1=x_2=0$.

There are two cases, the first is when we have a web of conics. Here we can use the list of symmetric cubics
which are discriminant of a web of conics, from 
\cite{ann arbor}, page 33,  and deduce that, since $M$ is singular along a line, we are in cases i),  or iv), or vi). 

However the enumeration in loc. cit. is incomplete, since the pencil of conics $ \la x_0^2 + \mu x_1^2=0$
was not considered, dually this leads to the web with matrix

\begin{equation}\label{missing}
A': =\left(\begin{matrix}0   &x_0 &x_1\cr  x_0& 0&x_2 \cr x_1&x_2  &x_3
\end{matrix}\right).
\end{equation} 

In this case, the base locus consists of $2$ points, and $M$ consists of the plane $x_0 =0$ and a 
smooth quadric, to which the plane is tangent.

The first of case i) is the matrix
\begin{equation}
A': =\left(\begin{matrix}0 &0&x_1\cr  0& x_0&x_2\cr x_1&x_2&x_3
\end{matrix}\right),
\end{equation} 
which must be excluded since then $\det(A) \in (x_1, q_0)^2$, a contradiction.

The second of case i) is the matrix
\begin{equation}
A': =\left(\begin{matrix}x_0  &0&x_3\cr  0& x_1&0 \cr x_3&0 &x_2
\end{matrix}\right),
\end{equation} 
yielding an empty base locus and $M = \{ x_1 (x_0 x_2 - x_3^2)=0 \}$,
a quadric cone and a plane intersecting it in a smooth conic.

While case iv) yields a point as base locus and cannot be excluded: in this case $M$ is irreducible:
\begin{equation}
A': =\left(\begin{matrix}x_0  &0&x_2\cr  0& x_1&x_3 \cr x_2&x_3  &0
\end{matrix}\right),
\end{equation}

Similarly case vi), 
\begin{equation}
A': =\left(\begin{matrix}0   &x_0 &x_1\cr  x_0& x_1&x_2 \cr x_1&x_2  &x_3
\end{matrix}\right),
\end{equation} 
yields a  point as base locus and  $M$ is irreducible.

Next, we consider the second case,  where we have a net (hence $M$ is a cone over a plane curve). We may assume  that 
\begin{equation}
A': =\left(\begin{matrix}0   &x_1 &x_2\cr  x_1& b (x) & c (x)  \cr x_2 &c (x)   &d(x)
\end{matrix}\right).
\end{equation} 
If the base locus contains three non collinear points, then we may assume $b=d=0$, and 
since we have a net $c(x)$ is linearly independent from $x_1, x_2$,
so that we may set $c(x) = x_3$; hence we obtain the desired normal form \ref{triplanes},
for which the base locus consists of exactly $3$ points, and $M$ consists of $3$ general planes.

The base locus cannot contain three  collinear points, else it would contain a line, and
be infinite, a contradiction.

If the base locus contains exactly  two points, we may assume $b=0$, $d \neq 0$, 
and by rows and columns operations that $c(x) = c(x_2, x_3)$. 

Then the equation of $M$ is  
$$ - d(x) x_1^2 + 2 x_1x_2 c(x) = x_1 ( 2 x_2 c(x_2, x_3) - d(x) x_1) =0.$$
Hence $M$ is the cone over a cubic union of a line and a  conic.

If $c \equiv  0$ we may set $d = : x_3$ and  we would get that the base locus consists of a simple base point and a double one.
But this case is excluded since then $ \det (A) \in (x_1, q_1)^2$, and $Y$ does not have isolated singularities.

We may assume then that either   $c(x) = x_3$, or $d(x) = x_3$, $c(x) = \la x_2$, $\la \neq 0$.
In the latter case  the conic is  $ 2 \la x_2^2 - x_1 x_3 = 0$, it is smooth, and tangent to the line $x_1=0$.
In the former case we get  the conic $2 x_2 x_3 - x_1 d (x)=0$, intersecting the line $x_1=0$ in two distinct points.
The conic is  smooth unless $d = - \frac{1}{2}  \la \mu x_1 + \la x_2+  \mu x_3$, in which  
 case  with rows and columns operations and with new variables we obtain the normal form with $b=0$,
$ d = \la x_1$, where $\la \neq 0$ (since otherwise  we would have three base points).

For the  assertion concerning the hypothesis that  $M$ is reducible, we simply run the list again in the case of a web, 
and observe that $M$ contains a plane only in the second case of i),
where there are no base points, or in the `missing case'  \ref{missing} where $M$ is the union of a plane and a smooth quadric
and there are two base points.

If the net has two base points, then we saw that $M$ is reducible.

We consider next the case of a net where there are no base points, and  show that then $M$ is irreducible.

Otherwise, we can assume that $x_3$ divides the equation of $M$. 

In this case the net yields a morphism  $ \phi : \PP^2 \ra \PP^2$ of degree $4$,  and $x_1, x_2, x_3$ are coordinates of the dual plane of the target $\PP^2$. The line $x_3=0$ is dual to a point of the target $\PP^2$, and if we take a general other point, the pencil of lines through it
pulls back to a general pencil
in the net which has
$4$ distinct base points, hence it has normal form $x_1 z_1 z_2 + x_2 z_3 (z_1 + z_2 + z_3)= 0$. 

Then the net has the form 
$$2 x_1 z_1 z_2 + 2 x_2 z_3 (z_1 + z_2 + z_3) + x_3 Q_3 (z) = 0,$$
where $Q_3(z)$ does not vanish at the $4$ points $e_1, e_2, e_2 - e_3, e_1 - e_3$.

The condition that $x_3$ divides the discriminant means that setting $ x_3=0$, we get a discriminant which is identically zero
as a polynomial in $x_1, x_2$. But indeed we get $ 2 x_1 x_2 (x_2 - x_1) \ mod (x_3)$, a contradiction.

\bigskip

We show now  that  the case of a net where  there is exactly one base point, and for which $M$ is reducible,
does not exist.

Case II) : assume that the net contains $2$ double lines, say $z_1^2=0$ and $z_2^2=0$.

Then the matrix

\begin{equation}\label{12}
A': =\left(\begin{matrix}0   &\la x_3 & \mu x_3\cr  \la x_3& x_1+ b_3 x_3  & c_3 x_3  \cr \mu x_3 &c_3 x_3   & x_2 + d_3 x_3
\end{matrix}\right)
\end{equation} 
and we get that $ det (A) \in (x_3, q_0)^2$, a contradiction.

Case I) : assume that the net contains exactly $1$ double line, say $z_1^2=0$, and assume as usual that the base point is $e_0$.

Then the matrix 
\begin{equation}\label{13}
A': =\left(\begin{matrix}0   &x_1 &x_2\cr  x_1& x_3 & c(x_1, x_2)  \cr x_2 &c(x_1, x_2)   &d(x_1, x_2)
\end{matrix}\right).
\end{equation} 
With row and column operations we obtain first that $ c = c(x_2) =  \la x_2$, and then 
we reduce to the case where  $\la=0$, if we take $ x_3 - 2 \la x_1$ as new third variable.

Hence we may assume $ c =0$, so that 
$$ \det (A') = - d x_1^2  -  x_2^2 x_3,$$ which, being of degree $1$ in $x_3$,  is irreducible unless 
$d = \la x_2$.

But in this case the base locus is given by 
$$ z_0 z_1 = z_1^2 = 2 z_0 z_2 + \la z_2^2 = 0 \Leftrightarrow z_1 = z_2 (2 z_0 + \la z_2) = 0,$$
and it consists of  two points.

Case 0) : there is no rank $1$ conic in the net. Hence $M $ is the cone over a reducible plane curve $C$,
 given  together with a nontrivial  invertible sheaf $\sF$ of
2-torsion and with a basis of $H^0(\sF(1))$ ($\sF$ is the cokernel of $ A' : \hol_{\PP^2}( -2)^3 \ra  \hol_{\PP^2}( -1)^3$). Hence,
since $Pic^0 (C) $ has  nontrivial 2-torsion,  $C$ consists either of three general lines, or a line and a smooth conic meeting in two points.

In both cases there is just one normal form since $Pic^0 (C) = \CC^*$ (see the theorem of  \cite{barth0}  about ineffective theta characteristics, 
 and also  \cite{babbage}, Proposition 2.28
and Remark 2.29).

In the former case we have the normal form \ref{triplanes}, since in this case there is no point in the net where
the rank of $A'$ drops down to $1$; this case however must be excluded since   then we have three base points of the net.

  In the latter we may take the normal form (again the rank 
of $A'$ is at each point $\geq 2$)
\begin{equation}\label{14}
A': =\left(\begin{matrix}0   &x_1 &x_2\cr  x_1& x_3 & x_3  \cr x_2 &x_3   &x_2
\end{matrix}\right).
\end{equation}  

The curve $C$ is then
$$ x_2  \ ( -x_1^2 + 2 x_1 x_3 - x_2 x_3 ) = 0,$$
but we have  $2$ base points instead of  $1$, the solutions of 
$$ z_0 z_1 =  2 z_0 z_2 +  z_2^2 = 2 z_1 z_2 +  z_1^2 = 0 \Leftrightarrow z_1 = z_2 (2 z_0 +  z_2) = 0.$$
\end{proof}

\begin{theo}\label{discr}
Let $Y$ be a nodal sextic surface in $\PP^3$ with $\nu(Y)$ nodes and a symmetric half-even set 
$\sN$ of cardinality $31$:
then $Y$ is the discriminant of a nodal cubic $X\subset \PP^6$ for the projection with centre a plane $L \subset X$.

  $X$ has exactly $\ga (X) = \nu(Y) - 31 - r$ nodes, where $ r \in \{0,1,2,3\} $ is the cardinality
 of  $L \cap \Sing(X)$.

\end{theo}
 
 \begin{proof}
 
 Since the set $\sN$ is symmetric, there is a matrix of homogeneous polynomials $A(x)$,
 where $ \deg A_{i,j} = \frac{1}{2} (d_i + d_j)$, and $d_1,d_2, d_3 , d_4 = 1,1,1,3$,
 and such that $\sN$ is the set of points where the matrix $A(x)$ has rank $=2$.
 
 Hence $Y$ is the discriminant of a cubic hypersurface, and
 the  assertion follows directly from Theorem \ref{smooth centre} and from theorem \ref{at most 3}.
 \end{proof}
 
 In the next subsection we shall show that, in the case $\nu(Y)=65$, $r$ can only be $2,3$, that is, the cubic has then 
 $31$ or $32$ nodes.

\subsection{$65$ nodal sextics as cubic discriminants}

By theorem~\ref{theo_code_sextic_65_unique}, the (extended) code of Barth's sextic is isomorphic to the (extended) code of an arbitrary sextic with 65 nodes.

 In particular, any such sextic has 26 half-even sets of cardinality 15. Let $C_1,\ldots,C_{26}$ be the corresponding reduced plane sections.
 
  Let $\pi \colon \tilde Y \to Y$ be the minimal resolution. For a subset $\mathcal N \subset \Sing Y$ of nodes,  we denote 
  by $E_\mathcal N$ the sum of the corresponding $(-2)$-curves in $\tilde Y$.

	We fix an even set $\mathcal N$ of cardinality $31$ such that there are  three half-even sets $\mathcal N_\alpha$, $\mathcal N_\beta$, $\mathcal N_\gamma$ of cardinality $15$ whose (code-theoretic) sum is  $\mathcal N$. The proof of  Proposition \ref{discriminant} applies in the same way: let $\tilde H \in \lvert \pi^*\mathcal O_Y(1) \rvert$, and $L \equiv \frac 12(\tilde H + E_{\mathcal N})$. Then
	\[
		\lvert 2\tilde H - L \rvert = \Big\lvert \frac12 ( 3\tilde H - E_\mathcal N) \Big\rvert= \{ \tilde C_\alpha + \tilde C_\beta + \tilde C_\gamma + E_ {\mathcal N'} \},
	\]
	where $\tilde C_\alpha$ is the proper transform of $C_\alpha$ and $\mathcal N'$ is the set of nodes that appear precisely twice in $\mathcal N_\alpha + \mathcal N_\beta + \mathcal N_\gamma$ ($\mathop{\#}{\mathcal N'} = 6$).

	We claim that there are $v_0,v_1,v_2 \in H^0(3 \tilde H-L)$ and $v_3 \in H^0(2 \tilde H-L)$ such that $\{ v_0, v_1, v_2, v_3 \otimes H^0(H) \}$ spans $H^0(3 \tilde H-L)$, and that there are linear forms $\{\ell_i : i=0,1,2\}$ with the properties
	\begin{enumerate}
		\item $B_{ii} := v_i^2$ and $B_{i3} := v_iv_3$ are divisible by $\ell_i$,
		\item $B_{33}$ is divisible by $\ell_0\ell_1\ell_2$.
	\end{enumerate}
	Then the $(i,i)$-minor of the matrix $\bigl( B_{ij} \bigr )$ is divisible by $\ell_j\ell_k \times (\text{defining equation of }Y)^2$, but has degree $5+5+3=13$, thus $a_{ii}=0$.

	\begin{lemma}\label{lem: basisLemma}\ 
		\begin{enumerate}
			\item $\lvert 3\tilde H - L - \tilde C_\alpha - \tilde C_\beta \rvert = \lvert \tilde H \rvert + \tilde C_\gamma + E_{\mathcal N'}$. 
			
			In particular, $H^0(3\tilde H - L - \tilde C_\alpha - \tilde C_\beta)$ can be identified with the image of $H^0( \tilde H ) \otimes H^0( 2\tilde H - L) \to H^0(3\tilde H - L)$.
			\item There is a decomposition of $\mathcal N$ as the sum of five half-even sets of cardinality 15, among which is   $\mathcal N_\alpha$\,(resp. for $\mathcal N_\beta,\ \mathcal N_\gamma$).
		\end{enumerate}
	\end{lemma}
	\begin{proof}\ 
		\begin{enumerate}
			\item Let $P_i \in \mathcal N'$. Then,
			\[
				(E_i \mathbin . \tilde C_\gamma) = \left\{
					\begin{array}{ll}
						1 & \text{if }P_i \in \mathcal N_\gamma \\
						0 & \text{otherwise}.
					\end{array}
				\right.
			\]
			Hence, $( E_i \mathbin. \tilde H + \tilde C_\gamma + E_{\mathcal N'} ) = (E_i \mathbin. \tilde C_\gamma) -2 < 0$. In this way, we prove that $E_{\mathcal N'}$ is in the fixed part. Let $R \subset \tilde C_\gamma$ be an irreducible component of degree $d = (R \mathbin . \tilde H )$. Then, $R$ contains $5d$ nodes of $\mathcal N_\gamma$([Endra\ss, Prop.~2.4]), hence
			\[
				(R\mathbin. \tilde H + \tilde C_\gamma) = d + \frac 12(d -5d) = -d < 0,
			\]
			showing that $R$ is in the fixed part of $\lvert \tilde H + \tilde C_\gamma \rvert$.
			\item By theorem~\ref{theo_code_sextic_65_unique}, it suffices to check the statement for Barth's sextic. It can be checked by computer\footnote{see the script \cite[Mathematica 17]{Scripts}.}.
			\qedhere
		\end{enumerate}
	\end{proof}
	The previous proposition implies that any decomposition of $\mathcal N$ into five $15$-half-even sets cannot contain more than one of $\mathcal N_\alpha, \mathcal N_\beta, \mathcal N_\gamma$.
	\begin{prop}
		Let $v_0$ (resp. $v_1,\ v_2$) be the section corresponding to a $5$-decomposition of $\mathcal N$ containing $\mathcal N_\alpha$ (resp. $\mathcal N_\beta,\ \mathcal N_\gamma$), and let $v_3 \in H^0(2\tilde H - L) \setminus \{0\}$. Then, $v_0,v_1,v_2$ and $W:= v_3 \otimes H^0(\tilde H) $ spans $H^0(3\tilde H-L)$.
	\end{prop}
	\begin{proof}
		The subspace spanned by $v_0$ and $W$ is $5$-dimensional and the corresponding linear system contains $\tilde C_\alpha$ in its base locus. Since $(v_1=0) \not\supset \tilde C_\alpha$, $v_0,v_1$ and $W$ spans a $6$-dimensional space, say $K$.
	
		To prove $v_2 \not\in K$, we first consider the following: let $\mathcal N_0 = \mathcal N_1 \cap \mathcal N_2 \cap \mathcal N_3$, then
		\[
			\bigl | 3\tilde H-L-E_{\mathcal N' \setminus \mathcal N_0} \bigr| = \bigl| \tilde H + \tilde C_\alpha + \tilde C_\beta + \tilde C_\gamma + E_{\mathcal N_0} \bigr|
		\]
		For simplicity, we put $D := \tilde H + \tilde C_\alpha + \tilde C_\beta + \tilde C_\gamma + E_{\mathcal N_0}$.
		Suppose $\tilde C_\alpha$ is irreducible, then it is in the fixed part since
		\[
			( \tilde C_\alpha \mathbin. D ) = 3 - 6 + 0 + 0 + 1 < 0.
		\]
		If $\tilde C_\alpha$ is reducible, let $R \subset \tilde C_\alpha$ be an irreducible component that does not contain $\mathcal N_0$ ($R$ always exists by \cite[Lemma~2.3]{endrass1}). Let $d = (R \mathbin. \tilde H)$, then $( R \mathbin. E_{\mathcal N_\alpha} ) = 5d$ and $( R \mathbin. E_{\mathcal N_i} ) = d$ for $i=\beta,\gamma$. Hence,
		\[
			( R \mathbin. D \bigr) = d + \Bigl(\frac d2 - \frac{5d}2\Bigr) + \Bigl( \frac d2 - \frac d2 \Bigr) + \Bigl( \frac d2 - \frac d2 \Bigr) + 0 = -d < 0
		\]
		thus $R$ is in the fixed part. If $R' \subset \tilde C_\alpha$ is the component containing $\mathcal N_0$, then $(R' \mathbin. D - R ) = -d+1 - (R'\mathbin.R) < 0$ (irreducible components of $C_\alpha$ do not intersect at nodes, thus $(R'\mathbin.R)>0$). This shows that $\tilde C_\alpha$ is in the fixed part. It follows that 
		\[
			\lvert D \rvert = \lvert \tilde H \rvert + \tilde C_\alpha + \tilde C_\beta + \tilde C_\gamma + E_{\mathcal N_0}.
		\]
		Next we claim that $(v_2=0)$ do not contain $E_{\mathcal N' \setminus \mathcal N_0}$. Suppose not, then $(v_2=0)$ contains $\tilde C_\gamma + E_{\mathcal N'\setminus \mathcal N_0}$, thus it comes from $\lvert 3 \tilde H - L  - \tilde C_\gamma - E_{\mathcal N'\setminus \mathcal N_0} \rvert = |\tilde H|+ \tilde C_\alpha + \tilde C_\beta + E_{\mathcal N_0}$. By Lemma~\ref{lem: basisLemma}, $v_2 \in W$, a contradiction.
		
		Pick $P_k \in \mathcal N' \setminus \mathcal N_0$ such that $(v_2 =0) \not\supset E_k$. Then, $(3\tilde H - L - \tilde C_\gamma \mathbin . E_k) \geq 0$. This implies $(\tilde C_\alpha \mathbin . E_k) = (\tilde C_\beta \mathbin. E_k) = 1$ since
		\[
			3\tilde H-L- \tilde C_\gamma \sim \tilde H + \tilde C_\alpha + \tilde C_\beta + E_{\mathcal N'}
		\]
		and $(E_{\mathcal N'} \mathbin. E_k) = -2$. In particular, $(\tilde C_\gamma \mathbin  . E_k ) =0$ and $(3\tilde H - L \mathbin. E_k)=0$. The divisor $(v_0=0)$ contains $\tilde C_\alpha$, which positively intersects with $E_k$, but $(3 \tilde H - L \mathbin. E_k)=0$, thus $(v_0=0)$ contains $E_k$. Similarly, $(v_1=0)$ contains $E_k$. This shows that $\lvert K \rvert$ contains $E_k$ in its fixed part. By choice of $k$, we have $(v_2=0) \not\supset E_k$, thus $K$ and $v_2$ are linearly independent. 
	\end{proof}
	
	Let $v_0$ (and $v_1,v_2$ similarly) to be the section defined by $\bigl( \tilde C_\alpha + \text{ sum of four $\tilde C_i$}'s\bigr )$, and $v_3$ by $\tilde C_\alpha + \tilde C_\beta + \tilde C_\gamma + E_{\mathcal N'}$. The set $\{v_0,v_1,v_2,v_3\}$ satisfies the described properties, hence $a_{11}=a_{22}=a_{33}=0$.

\begin{cor}\label{cor: 31=15+15+15}
 The Barth  sextic and every other sextic surface with 65 nodes  is the discriminant of the projection $p_L$ of a cubic hypersurface $X \subset \PP^6$ with 31 nodes, 
 with centre a plane $L$ intersecting  $\Sing(X)$ in three points, 
 and not contained in a three dimensional linear subspace $\Lam \subset X$.

\end{cor}

\bigskip

In a nodal sextic with 65 nodes, the number of half-even sets of cardinality 31 is 1690. Amongst them, 1300 admit the decomposition 
as the sum of three half-even sets of cardinality 15. 
Each of those 1300 half-even sets yields a pair $(X,L)$ as  in corollary~\ref{cor: 31=15+15+15}. 
We observe here  that every half-even set of cardinality  31 is  either the sum of three half-even sets of cardinality 15, 
or the sum of a half-even set of cardinality 15 with an even set of cardinality 24.

We consider the remaining 390 cases in which $\mathcal N$ is a 31-half-even set that admits a  $\mathcal N_\alpha + \mathcal Q $ decomposition where  $\mathcal N_\alpha$ is an 15-half-even set, and $\mathcal Q$ is a $24$-even set cut out by a quadric $V \subset \PP^3$. Consider the degree 6 curve $Q$ such that  $2Q = Y\cdot V$, and let $\tilde Q \subset \tilde Y$ be the proper transform
of $Q$. We remark that $V$ is an irreducible quadric; otherwise, $V$ is a union of two planes, thus $\mathcal Q$ is a sum of two $15$-half-even sets.
		
	To prove that  $\mathcal N$ is symmetric, we need to study numerical invariants such as $\tilde Q^2$. To do this, we first regard $Q$ as a divisor in $V$, and compare intersection theories between $V$ and $Y$.
	\begin{lemma}\label{lem: intersection comparison}
		Let $X$ be a smooth projective variety, and let $Q_1,Q_2 \subset X$ be reduced closed curves. Let $S, S' \subset X$ be smooth projective surfaces containing $Q_1,Q_2$. Suppose $Q_1$ and $Q_2$ share no common irreducible components. Then,
		\[
			(Q_1 \mathbin. Q_2)_S = (Q_1 \mathbin. Q_2)_{S'},
		\]
		where the subscript indicates the ambient surface where the intersection takes place.
	\end{lemma}
	\begin{proof}
		Let $Q = Q_1 \cup Q_2$ endowed with the reduced structure. 
		
This follows from the exact sequence
$$ 0 \ra \hol_Q \ra ( 	\hol_{Q_1} \oplus \hol_{Q_2}) \ra \hol_{Q_1 \cap Q_2} \ra 0$$
defining the intersection multiplicity, so that  
$$ (Q_1 \cdot Q_2)_{S} = h^0(\hol_{Q_1 \cap Q_2}) = (Q_1 \cdot Q_2)_{S'}.$$
	\end{proof}

	\begin{prop}
		$\lvert 2 \tilde H - L \rvert = \{ \tilde C_\alpha + \tilde Q + E_{\mathcal N'} \} $, where $\mathcal N' = \mathcal N_\alpha \cap \mathcal Q$.
	\end{prop}
	\begin{proof}
		In the first part of the proof, we claim that $\tilde C_\alpha$ is in the fixed part. If $\tilde C_\alpha$ is irreducible, it is trivial since $( \tilde C_\alpha \mathbin. 2 \tilde H - L ) = -1$. Assume $\tilde C_\alpha$ is reducible, let $\{R_i\}$ be the set of irreducible components, and let $d_i = (R_i \mathbin. \tilde H )$. Then,
		\begin{align*}
			\bigl (R_i \mathbin. \tilde C_\alpha + \tilde Q + E_{\mathcal N'} \bigr) 
			&= \Bigl( \frac{d_i}2 - \frac{5d_i}2 \Bigr) + \Bigl( d_i - \frac12 (R_i \mathbin. E_\mathcal Q) \Bigr) + (R_i \mathbin. E_{\mathcal N'} ) \\
			&= -d_i + \frac 12 (R_i \mathbin. E_{\mathcal Q}).
		\end{align*}
		Hence, $(R_i \mathbin. E_{\mathcal Q})$ is even and the sum over $i$ equals $(\tilde C_\alpha \mathbin. E_\mathcal Q ) = 4$.
		\par {\it Case 1. } Assume $(R_1 \mathbin. E_{\mathcal Q}) = 4$. Then, all $R_j$ with $j \neq 1$ are in the fixed part, and
		\[
			(R_1 \mathbin. 2 \tilde H - L - \sum_{j \neq 1} R_j ) = -d_1 + 2 - \sum_{j \neq 1} ( R_1 \mathbin . R_j).
		\]
		By Proposition~\cite[Lemma~2.3]{endrass1}, $R_i$ are proper transform of the planes curves that do not intersect at the nodes of $Y$. Thus, $(R_1\mathbin. R_j) = d_1d_j$, so the latter is always negative.
		\par {\it Case 2. } Assume $(R_i \mathbin. E_{\mathcal Q}) = 2$ for $i=1,2$. Since $\sum d_i=3$, there are two subcases: $d_1=d_2=1$ and there is a third component $R_3$, or (w.l.o.g) $d_1=2$. In the first subcase, $R_3$ is in the fixed part, and $(R_1 \mathbin. 2\tilde H - L - R_3 ) = -(R_1 \mathbin. R_3) < 0$ thus $R_1$ is in the fixed part. Similarly, $R_2$ should be in the fixed part. In the remaining subcase, $R_1$ is in the fixed part, and then $(R_2 \mathbin. 2 \tilde H - L - R_1) < 0$, so $R_2$ is in the fixed part. We conclude that $\tilde C_\alpha$ is in the fixed part, and so is $E_{\mathcal N'}$.
		
		It remains to prove $\lvert \tilde Q \rvert = \{ \tilde Q \}$. Let $Q_1 \subset Q$ be an irreducible component. Again, we split into two cases: whether $V$ is singular or not.
		\par {\it Case 1.} Assume $V$ is a quadric cone. In this case $V \simeq \PP(1,1,2)$. Let $h_V \in \lvert \mathcal O_V(1)\rvert$, and let $e$ be an integer determined by $Q_1 \sim e\cdot h_V$. For $H_{\PP^3} \in \lvert \mathcal O_{\PP^3}(1) \rvert$, $H_{\PP^3}\big\vert_V = 2h_V$, and $h_V^2 = 1/2$. Thus,
		\[
			(\tilde Q_1 \mathbin. \tilde H) = (Q_1 \mathbin. H_{\PP^3})_{\PP^3} =  (Q_1 \mathbin. 2h_V)_V = e.
		\]
		By adjunction formula one may write $\tilde Q_1^2$ in terms of $e$ and $p_{\rm a}(\tilde Q_1)$. To estimate the arithmetic genus, we look at the proper transform of $Q_1$ along the resolution $\F_2 \to V$.
		
		By \cite[Proposition~2.4]{endrass1}, $Q_1$ cuts out 
		at least $4e$ nodes of $Y$. If we vary $Q_1$ and do the same procedure for all irreducible components of $Q$, we find that the number of nodes cut out by $Q_1$ is exactly $4e$ as the sum of all $e$ equals to $(Q \mathbin. \tilde H) =6$ and $Q$ cuts out $4 \times 6 = 24$ nodes. This also shows that the vertex $v$ of $V$ cannot be a smooth point of $Y$; otherwise \cite[Proposition~2.4]{endrass1} would imply that $Q_1$ cuts out $4e+1$ nodes of $Y$.
	
		Let $Q_{1,\F_2} \subset \F_2$ be the proper transform of $Q_1$. We claim that $Q_{1,\F_2}$ and $\tilde Q_1$ may be identified via the proper transform along the blowing up $\tilde \PP^3 \to \PP^3$ at $\{P_1,\ldots,P_{65}\}$. By \cite[Lemma~2.3]{endrass1}, $Q$ is smooth at the nodes of $Y$ except possibly at $v$. Hence, blowing up at $\{P_1,\ldots, P_{65}\} \setminus \{v\}$ does not affect both $Q$ and $V$. So, in the case $v \not\in Y$ the claim is clear. If $v \in Y$, then $v$ is a node of $Y$, so the proper transform of $V$ along $\tilde \PP^3 \to \PP^3$ is nothing but the resolution $\F_2 \to V$.

		Let us estimate $p_{\rm a}(Q_{1,\F_2})$. We have
		\[
			(Q_{1,\F_2})_{\F_2}^2 \leq (Q_1^2)_V = \frac{1}2 e^2
		\]
		and
		\[
			(K_{\F_2}\mathbin .Q_{1,\F_2})_{\F_2} = (K_V \mathbin. Q_1)_V = -2e,
		\]
		hence $2 \mathop{p_{\rm a}}(Q_{1,\F_2}) -2 \leq \frac12 {e^2} - 2e$. The adjunction formula for $\tilde Q_1 \subset \tilde Y$ reads
		\[
			\tilde Q_1^2 + (2 \tilde H \mathbin. \tilde Q_1) = 2 \mathop{p_{\rm a}}( \tilde Q_1 ) - 2 \leq \frac 12 e^2 - 2e,
		\]
		thus $\tilde Q_1^2 \leq \frac12 e^2 - 4e$. Let $Q' = Q - Q_1$. In $V$, we have $Q' \sim (6-e) h_V$, hence
		\[
			(Q_1 \mathbin. Q')_V = \frac 12 e(6-e).
		\]
		By Lemma~\ref{lem: intersection comparison} and the identification $\tilde Q = Q_{\F_2}$ in $\tilde \PP^3$, $( \tilde Q_1 \mathbin . \tilde Q')_{\tilde Y} = ( Q_{1,\F_2} \mathbin . Q'_{\F_2})_{\F_2}$. Thus,
		\[
			(\tilde Q_1\mathbin. \tilde Q) = (Q_{1,\F_2} \mathbin. Q'_{\F_2})_{\F_2} + \tilde Q_1^2 \leq \frac 12 e(6-e) + \frac 12 e^2 - 4e = -e < 0,
		\]
		thus $\tilde Q_1$ is in the fixed part of $\lvert \tilde Q \rvert$. This proves the case when $V$ is a quadric cone.
		\par {\it Case 2.} Now we consider the case if $V$ is nonsingular. Let $Q_1 \in \lvert \mathcal O_V(a,b) \rvert$. Since $Q \in \lvert \mathcal O_V(3,3) \rvert$, $0 \leq a,b \leq 3$. The adjunction formula:
		\[
			(Q_1^2)_V + (K_V \mathbin. Q_1)_V = 2ab - 2(a+b) = 2\mathop{p_{\rm a}}(Q_1) - 2.
		\]
		Since $Q_1$ is smooth at the nodes of $Y$, $p_{\rm a}(\tilde Q_1) = p_{\rm a}(Q_1)$. So,
		\[
			\tilde Q_1^2 = 2ab - 2(a+b) - (2 \tilde H \mathbin. \tilde Q_1) = 2ab - 4(a+b).
		\]
		For $Q' = Q - Q_1 \in \lvert \mathcal O_V(3-a,3-b) \rvert$, by Lemma~\ref{lem: intersection comparison},
		\[
			(\tilde Q_1 \mathbin. \tilde Q') = (Q_1 \mathbin. Q')_V = a(3-b) + b(3-a),
		\]
		thus
		\[
			(\tilde Q_1 \mathbin. \tilde Q) = a(3-b) + b(3-a) + 2ab - 4(a+b) = -(a+b) < 0,
		\]
		thus $\tilde Q_1$ is in the fixed part of $\lvert \tilde Q \rvert$. This implies $\lvert \tilde Q \rvert = \{ \tilde Q \}$ as desired.
	\end{proof}

	We want to choose the generators $v_0, v_1,v_2 \in H^0(3 \tilde H - L)$ and $v_3 \in H^0( 2 \tilde H - L)$, having following property: after suitable reordering of indices in $\{L_i\}$,
	\begin{equation}\label{eq: 31=15+24 basis form}
		\begin{array}{r@{}l}
		(v_0=0) &{}= L_0 + L_1 + L_2 + L_3 + L_4 + \text{(some $E_i$ terms)}\\
		(v_1=0) &{}= L_4 + L_5 + L_6 + L_7 + L_8 + \text{(some $E_i$ terms)}\\
		(v_2=0) &{}= L_8 + L_9 + L_{10} + L_{11} + L_{12} + \text{(some $E_i$ terms)} \\
		(v_3=0) &{}= L_0 + Q + E_{\mathcal N'} 
		\end{array}
	\end{equation}
	Then the some entries of the matrix $(B_{ij})$ have linear factors described as follows:
	\[
		\begin{array}{c|c|c|c}
			\ell_0 \ell_4 & \ell_4 & & \ell_0  \\ \hline
			\ell_4 & \ell_4\ell_8 & \ell_8 & \\ \hline 
				 & \ell_8 & \ell_8 &  \\ \hline
			\ell_0 &  & & \ell_0
		\end{array}
	\]
	The  second principal (3,3)-minor of this matrix  has determinant  of the form
	\[
		\ell_0 \ell_4 \left\vert
			\begin{array}{cc}
				\ell_4\ell_8 & \\
				& \ell_0
			\end{array}
		\right\vert + \ell_4 \left\vert
			\begin{array}{cc}
				\ell_4 & \\
				\ell_0 & \ell_0
			\end{array}		
		\right\vert + \ell_0 \left\vert
			\begin{array}{cc}
				\ell_4 & \ell_4\ell_8 \\
				\ell_0 & 
			\end{array}		
		\right\vert,
	\]
	so it is divisible by $\ell_0\ell_4 \cdot (\text{defining equation of Y})^2$. Due to degree reason, $a_{33}=0$. 
	
	This  excludes (7) of  theorem~\ref{at most 3}, hence we conclude by  theorem~\ref{at most 3}
	that  there are two base points.
	\par Now, we directly check the $\mathfrak A_5$-orbits of $31 = 15+24$ half-even sets, and found the basis \ref{eq: 31=15+24 basis form} directly. The following example is an explicit demonstration; the other orbits can be done by the same procedure.
	\begin{example}
		We consider the following $\mathcal N = \mathcal N_7 + \mathcal Q$:
		\begin{align*}
			\mathcal N &= {} \scalebox{0.8}{$1\ 11001111111110011011110110000000010000100000000111111000011000011$} \\
			\mathcal N_7 &= \scalebox{0.8}{$1\ 11000000000010000{\red 1}0000{\red 1}1100000000100001000000{\red 11}000000000011000011$} \\
			\mathcal Q &=	\scalebox{0.8}{$	0\ 00001111111100011{\red 1}1111{\red 1}0000000000000000000000{\red 11}111111000000000000$}
		\end{align*}
		We have $\mathcal N' = \mathcal N_7 \cap \mathcal Q = \{ P_{18}, P_{23}, P_{46}, P_{47} \}$. The sections $v_i$ are given by the following decompositions of $\mathcal N$:
		\[
			\begin{array}{l@{\quad }|@{\quad}l@{\quad}|@{\quad}l}
				v_0 &  \mathcal N_1 + \mathcal N_7 + \mathcal N_9 + \mathcal N_{10} + \mathcal N_{25} & (2,2,2,2) \\
				v_1  & \mathcal N_1 + \mathcal N_2 + \mathcal N_{12} + \mathcal N_{13} + \mathcal N_{17} & (0,0,2,2) \\
				v_2  & \mathcal N_{12} + \mathcal N_{14} + \mathcal N_{19} + \mathcal N_{23} + \mathcal N_{24} & (0,0,0,0) \\
				v_3 &  \mathcal N_7 + \mathcal Q & (2,2,2,2)
			\end{array}
		\]
		The quadruple in the right hand side is the vanishing order of the corresponding curves at $\mathcal N'$; for instance, the quadruple in $(0,0,2,2)$ indicates that the corresponding divisor
		\[
			C_1 + C_2 + C_{12} + C_{13}+ C_{17}
		\]
		has zeros at $(P_{18}, P_{23}, P_{46}, P_{47})$ with orders $(0,0,2,2)$ respectively. From this we may read that $v_0,v_1,v_2,v_3$ is the desired set of generators. Indeed, all sections in $W := \bigl\langle v_0, v_3 \otimes H^0(\tilde H) \bigr\rangle $ have vanishing orders at least $(2,2,2,2)$, so $v_1 \not\in W$. Then, all sections in $K := \bigl\langle v_1, W\bigr\rangle$ have vanishing orders at least $(0,0,2,2)$, hence $v_2\not\in K$.
\end{example}

	\begin{cor}\label{32}
		Every nodal sextic $Y$ with $65$ nodes admits a pair $(X,L)$ where $X \subset \PP^6$ is a nodal cubic with $32$ nodes and $L\subset X$ is a plane intersecting $\Sing X$ at two points, such that $Y$ is the discriminant of the projection $p_L \colon X \dashrightarrow \PP^3$ with centre  $L$.
		
		 There is no such pair $(X,L)$ with $ | X\cap L | \leq 1$ such that $Y$ is the discriminant of the projection $p_L \colon X \dashrightarrow \PP^3$ with centre  $L$.
	\end{cor}

\begin{prop}\label{2nodes}
Let $X \subset \PP^6$ be a nodal cubic and let $P_1, P_2$ be nodes of $X$. Then the line $\Lam : = P_1 * P_2$
is contained in $X$ and there is a plane $L$ containing $P_1, P_2$ such that $ L \subset X$.
\end{prop}
\begin{proof}
$\Lam \cdot X \geq 4$, so $\Lam \subset X$. Take coordinates 
$$(u,x) : = (u_0,u_1, x_0, \dots, x_4)$$
such that $\Lam = \{ x=0\}$ and such that the nodes are $P_1= e_0, P_2 = e_1$.

Then the equation of $X$ writes as
$$ F(u,x) = u_0^2 \la_{0,0}(x) + 2 u_0 u_1  \la_{0,1}(x) + u_1^2 \la_{1,1}(x) + u_0 Q_0 (x) + u_1 Q_1 (x) + G(x).$$
Since $e_0, e_1$ are nodes the linear parts of the Taylor development vanish at them, hence $\la_{0,0}(x) \equiv 0, \la_{1,1}(x) \equiv 0$,
hence 
$$ F(u,x) =  2 u_0 u_1  \la_{0,1}(x)  + u_0 Q_0 (x) + u_1 Q_1 (x) + G(x).$$
Let $L_x : = \Lam * \{(0,x)\}$: then $L_x \subset X$ if and only if
$$ \la_{0,1}(x)  = Q_0 (x) = Q_1 (x) = G(x) =0.$$

These are 4 equations in $\PP^4$, so they always have a solution (and in general 12 solutions).
\end{proof}

\begin{rem}\label{3nodes}
Let $X \subset \PP^6$ be a nodal cubic and let $P_1, P_2, P_3$ be (linearly independent) nodes of $X$.
Then if $L$ is the plane containing $P_1, P_2, P_3$, then we can choose coordinates 
$$(v,x) : = (v_0, v_1,v_2,  x_0, \dots, x_3)$$
such that $L  = \{ x=0\}$ and such that the nodes are $e_0, e_1, e_2$.

Then the equation of $X$ writes as
$$F(v,x) = \la v_0 v_1 v_2 + \sum_{i,j} v_i v_j  \la_{i,j}(x) +  \sum_{i} v_i Q_i (x) + G(x).$$
And $L \subset X$ if and only if $\la=0$.

In particular if we have an equisingular family $\sX$ of nodal cubics $X_t \subset \PP^6$,
and we let $L_t $ be the plane spanned by 3 nodes $P_1(t), P_2(t), P_3(t)$ of $X_t$, then $L_t \subset X_t$ if and only if $\la(t)=0$,
hence we get a subfamily of codimension at most $1$.

\end{rem}

We can now summarize the previous Corollaries \ref{cor: 31=15+15+15} and \ref{32} in the following

 \begin{theo}\label{thm: sextic as discrim}
 The Barth's sextic as well as every other sextic surface with 65 nodes  

 \begin{enumerate}
 \item
 is the discriminant of the projection $p_L$ of a cubic hypersurface $X \subset \PP^6$ with  31 nodes, and 
 with centre a plane $L$ intersecting  $\Sing(X)$ in three points   and not contained in a three dimensional linear subspace $\Lam \subset X$,
 \item
is  the discriminant of the projection $p_L$ of a cubic hypersurface $X \subset \PP^6$ with  32 nodes, and 
 with centre a plane $L$ intersecting  $\Sing(X)$ in two points   and not contained in a three dimensional linear subspace $\Lam \subset X$,
 \item
 cannot be the discriminant of the projection $p_L$ of a cubic hypersurface $X \subset \PP^6$ with  33 or more nodes.
 \end{enumerate}

$\sF( 6, 65)$ is smooth of dimension $18$ at the point corresponding to the Barth sextic, and at any other point corresponding to an unobstructed surface.

   Taking  
 cubic discriminants of cubic hypersurfaces as in i) or ii) one obtains 
 irreducible components  of dimension $\geq 18$ of  $\sF( 6, 65)$.
 \end{theo} 
 
 \begin{proof}
 (i) : by proposition \ref{discriminant} it suffices to show the existence of three distinct half-even sets of nodes $\sN_i$, $i=1,2,3$ 
 of cardinality $15$ such that $\sN : = \sN_1 +  \sN_2 + \sN_3$ has  cardinality $t =  31$.
 
 We use the theorem proven in appendix B, that all sextic surfaces with 65 nodes have extended code isomorphic to
 the one of the Barth sextic. Then it suffices to show that the assertion holds for the Barth sextic. 
 
  While assertions  (ii) and (iii) follow from Theorem \ref{at most 3} and Corollary \ref{32}.
 
  That the Barth sextic is unobstructed  is shown in Subsection \ref{barthunobstructed}, while  an explicit determinantal equation for the Barth sextic is contained in 
 Section \ref{append_Barth_determinant}: in this case
  $L \cap \Sing(X)$ consists of three points  since the cubic $M= \{\det (A') = 0\}$
  clearly consists of three planes, and we can then invoke Theorem \ref{at most 3} and Theorem \ref{smooth centre}.

 For the  assertion about the dimension, this follows easily from the fact that the dimension  of the space 
 of sextics with $65$ nodes
  is at least $ 83 -65  = 18$. The smoothness follows from assertion  vii) of theorem \ref{unobstructed} that the Barth sextic is unobstructed,
 which shows indeed that at this point the dimension of $\sF( 6, 65)$  is exactly $18$.

   We can  use  the  unobstructedness of the Segre cubic $X_0$ (see  v) of theorem \ref{unobstructed})
 in order to  obtain an irreducible component   $ \sX_{32} \subset \sF_5(3, 32)$ with number of moduli $m=3$
 and an irreducible component   $\sX_{31} \subset \sF_5(3, 31)$ with number of moduli $m=4$.
 
 We can now take two nodes $P_1(t), P_2(t)$ of $X_t \in \sX_{32}$ and use proposition \ref{2nodes} to infer the existence
 of  a 
 plane $L_t \subset X_t$ passing through $P_1(t), P_2(t)$; in this way we construct a family 
 of dimension at least $18$ of  sextics which,  by Theorem \ref{smooth centre},  are 65-nodal if their singularities 
 are just nodes. 
 
 Similarly, choosing three non collinear nodes of $X'_t \in \sX_{32}$, we obtain, on a subfamily of $ \sX_{32}$
 of codimension at most $1$, by remark \ref{3nodes}, that the  plane $L'$ spanned by the 3 nodes
 is  contained in  the cubic hypersurface $X'_t$.
 
In this way we also construct a family 
of dimension at least $18$ of  sextics which,  by Theorem \ref{smooth centre},  are 65-nodal if their singularities 
 are just nodes. 
\end{proof}
 
 \begin{rem}
 The previous theorem raises the question about the irreducibility of $\sF( 6, 65)$, in possible analogy with the
 irreducibility of $\sF( 3, 4)$, $\sF( 4, 16)$, $\sF( 5, 31)$.
 
It also raises the question: given a general 32-nodal cubic $X$ and two nodes,  is there only a finite number of planes contained in $X$ and passing through the 2 nodes? Then the above construction would yield  a family of dimension exactly 18.
 
 Observe that up to now there are four different constructions of a $65$-nodal sextic: the Barth equation,
 the construction by Pettersen \cite{pettersen1998nodal} via nodal determinantal quartic hypersurfaces in $\PP^4$, and the above
 two constructions via cubic discriminants.

 \end{rem}

 \bigskip
 \subsection{Nodal sextics as quartic discriminants}
 Let $W \subset \PP^4$ be a nodal quartic hypersurface, and choose coordinates $(w,x) : = (w, x_0, \dots, x_3)$
 such that $ (1,0)$ is a node $P_0$ of $W = \{ f(w,x)=0\}$.
 Then the equation of $W$ writes as
 $$ f(w,x) = w^2 Q (x) + w B(x) + g(x).$$
 Projecting  to $\PP^3$ with centre the node $P_0$, we get the discriminant surface 
 $$ Y = \{ x | F(x) : =  4 Q(x) g(x) - B^2 (x) = 0\},$$
 which is the branch locus of the rational double covering $W \dasharrow \PP^3$.
 
 Comparing $ \Sing(W)$ with $ \Sing(Y)$ we see that
 $$ \Sing (W) = \{ (w,x) | 2w Q + B = 0, w^2 \nabla Q  + w \nabla B + \nabla g=0\},$$
 $$ \Sing (Y) = \{ x  | 4 g  \nabla Q  + 4 Q  \nabla g - 2 B  \nabla B=0\}.$$
 
 One  observes that if $P$ is a singular point of $W$, then the line $P_0 * P$ is either contained in $W$, 
 or it intersects $W$ only in $P_0, P$, hence $P$ maps to a point of $Y$. The former situation occurs precisely
 when $Q(x)= g(x) = B(x)=0$.
 
 If $Q(x) \neq 0$, then  we have at most one singular point $P$ of $W$ on the line $ P_0 * (0,x)$,
 which satisfies  $w = - \frac{B}{2Q} $ and then we see that  the equation of  $W$ at $P$ is the suspension
 of the one of $Y$ at $x$,  since
 $$ 4 Q f = (2 w Q)^2 + 2  (2 w Q) B + 4 Qg = (2wQ + B)^2 + ( 4 Qg - B^2) = (2wQ + B)^2 + F(x).$$
 We observe that $w' : = (2wQ + B)$ completes the affine $x$ coordinates to a system of local coordinates.
 
 Hence over the open set $ Q(x) \neq 0$ we have a bijection 
  between singular points of $W$ and singular points of $Y$,
  such  that the equation of $W$ is of the form $ (w')^2 + F(x)=0$.
  
  In particular, we have a node $x \in Y$  downstairs if and only if we have a node 
  of $W$, $ P = ( - \frac{B}{2Q}, x) \in W$ upstairs.
  
  \medskip
  
 If instead $Q(x)=0$, and $P$ is a singular point, then $B(x)=0$ for $P$, and  also $g(x)=0$ for $P$.
 
 Similarly, if $Q(x)=0$ and $ x \in Y$, then $B(x)=0$. Over the points for which $Q(x)= B(x)=0$,
 the equations for the singular points are then simplified to
 $$ w^2 \nabla Q  + w \nabla B + \nabla g=0,$$ respectively 
 $$g \nabla Q=0.$$
 
 At this point it is crucial to use the hypothesis that $P_0$ is a node,
 ensuring that the quadric $Q$ is a smooth quadric, so that $Q= \nabla Q=0$
 has no solution except $P_0$.
 
 Hence the singular points of $Y$ with $Q=0$ are only the points $x$ where
 $Q=B=g=0$, and over them we do not have any singular point of $W$
 if the three gradients are linearly independent, a condition which is equivalent to
 the condition that these points $x$ are  nodes of $Y$.
 Follows:
 
 \begin{prop}\label{4icdiscr}
 If the sextic surface $Y$ is the discriminant of the projection of a quartic hypersurface from a node $P_0$,
 then outside the closed set $ Q(x)= B(x)=0$ we have a bijection between the singular points $P = (w,x)$ of $W$ and 
 the singular points $ x \in \Sing(Y)$, such that the singularity at $P$ is the suspension of the one of $Y$ at $x$.
  
 In particular,  if $Y$ is nodal, then $W$ is also nodal, and there is a bijection between the singular points of
 $W$ different from $P_0$ and the singular points of $Y$ different from the 24 points $\{ Q= B = g =0\}$.
 
 In particular, if $Y$ has $\nu$ nodes, then $W$ has $\de =  \nu - 23$ nodes.
 
 Conversely, if $W$ is nodal, again  there is a bijection between the nodes of
 $W$ different from $P_0$ and the singular points of $Y$ which are not in the set $\{ Q= B = g =0\}$,
 which are also nodes. While the  points in the set $\{ Q= B = g =0\}$ yield nodes of $Y$ if and only if
 the number of lines $\Lam$ passing through $P_0$ and contained in $W$  is 24.

 \end{prop}
 \begin{proof}
 We have already proven the assertion concerning the points in the open set $ Q(x)\neq 0$.
 
 For the others, we use that since $P_0$ is a node, the quadric $Q$ is smooth. Then we may restrict to consider the points
 lying in  the closed set $\{ Q= B = g =0\}$. If $Y$ is nodal these are 24 nodes, moreover these are nodes if and only if
 this set consists of 24 distinct points. This is however equivalent to requiring that on $W$ these equations define 24 different lines through $P_0$.
 \end{proof}
 
 Pettersen \cite{pettersen1998nodal} considered quartics which are determinants of $ 4 \times 4$ matrices with entries
 linear forms, which have in general $20$ singular points (Proposition 1.3.7 ibidem).
 
 With very hard (and also computer supported) work he was able to prove the following result (Theorem 5.5.1 of  \cite{pettersen1998nodal}):
 
 \begin{theo}\label{pettersen}
 ({\bf Pettersen}) 
 There exists a unique irreducible 3-dimensional family of 65-nodal sextic surfaces which are the discriminant of a 
  nodal determinantal quartic hypersurface $W$ with $42$
 nodes for the projection from one node. This family contains the Barth sextic.
  \end{theo}
  
  This theorem  (taken cum grano salis replacing the word `dimension' with `number of moduli')
   produces then an irreducible family of dimension 18  of  65-nodal sextics.
  
   Conversely, we have:
  
  \begin{theo}
  Each  65-nodal sextic $Y$ admits an even set $\sN$ of cardinality 24. $\sN$ is symmetric hence determines a realization of $Y$ 
  as the discriminant    of a quartic hypersurface $W$ for the projection with centre a double point $P_0$.  
  
  $W$  has $42$ nodes if and only if    the unique quadric $Q$ cutting on $\tilde{Y}$ the linear system $| L + H|$ is smooth
  ($L$ is here the divisor such that $ 2L \equiv \sum_{x \in \sN} E_x$). 
  
  There exist choices of $\sN$ such  that either the rank of $Q$ is either 3, or $Q$ is smooth.
  
  In the case where $Q$ has  rank 3,  $P_0$ is not a node, and $W$ has as singular points $P_0$ and  40 nodes,
   plus possibly another singular point $P'$.

  The open set $\mathfrak U$  inside  the family of all 65-nodal sextics, consisting of the sextics 
  admitting   an even set $\sN$ 
  of cardinality 24 with $Q$ smooth,  and for which the corresponding 42-nodal quartic hypersurface $W$
  is  determinantal, 
  is irreducible of dimension 18 and contains the Barth sextic. 
  \end{theo}
  
  \begin{proof}
  Since we have seen that all 65-nodal sextics have the same code (Appendix C), there is an even set $\sN$ of cardinality 24.
  
  We have already seen in this section that we may assume that $\sN$ is not the sum of two half-even sets (of cardinality 15). 
  
  $\sN$ is symmetric by Theorem 3.8.2 of \cite{cascat}, hence is associated to a representation
  $$ Y = \{ x | 4 Q(x) g(x) - B^2(x) = 0\},$$ 
  where the quadric $\{ Q(x)=0\}$ is uniquely defined. We may write $Q = Q_{\sN}$ to recall that the quadric depends upon the choice of $\sN$.
  
  Observe  that $Q= Q_{\sN}$ is irreducible, else it is the union of two planes, and $\sN$ is  the sum of two half-even sets
  of cardinality 15.
  
  Since $\sN = \{ Q= B = g =0\}$ we conclude that, if $\{ Q(x)=0\}$ is smooth, then, by Proposition \ref{4icdiscr},  $Y$ is the discriminant of a quartic $W$ with 42 nodes.

  The condition that $\{ Q(x)=0\}$ is smooth is $disc (Q) \neq 0$, hence defines an open set  
  $\mathfrak U' : = \{ Y | \exists \  \sN \  {\rm s.t.} \  disc (Q_{\sN}) \neq 0\}$  inside  the family of all 65-nodal sextics.
  
   Moreover, the condition that $W$ is determinantal is an open condition by virtue of the following Lemma \ref{det}. 
  
  The rest follows then from Pettersen's Theorem \ref{pettersen}.
  
  Finally, if $\{ Q(x)=0\}$ is a rank 3 quadric, then, since we assume $Y$ nodal, a point $x'$ where $ Q=B= \nabla Q =0$
  would be  a double point of $Y$, hence a node of $Y$ where $g \neq 0$.
  
  Then on the line $P_0 * (0,x')$ the singular points of $W$ different from $P_0$ are determined by the equation
  $ w \nabla B + \nabla g =0$. There is then at most one solution $P'$,  since  $ \nabla g =0$ implies
  $g=0$, a contradiction.   
  
  If there is no such singular point $x'$, then $W$ shall have  41 nodes plus the singular point $P_0$.
  
  If such a point $x'$ exists, then   $W$ shall have  40 nodes plus the singular point $P_0$ plus possibly another singular point $P'$ on the line $P_0 * (0,x')$.
  \end{proof}
 
 The following result is indeed valid in a higher generality, but we just state in a form suitable to our purposes here.
 \begin{lemma}\label{det}
 Let $W \subset \PP^4$ be a linear determinantal quartic $ W = \{ x | det (a_{ij}(x) = 0 \}$, which has only $\de$ nodes as singularities.
 Then  there exists an open set of the nodal Severi variety $\sF_3( 4,\de)$ containing $W$ and consisting of
 nodal linear determinantal quartics.
 \end{lemma}
 \begin{proof}
 Consider the complete intersection variety 
 $$ Z \subset \PP^4 \times \PP^3, \ Z : = \{ (x,v) | \sum_j a_{ij}(x) v_j = 0, \ i=1,2,3,4\}.$$
 Among the nodes of $W$ there are the 20 nodes (see for instance  \cite{pettersen1998nodal}, Prop. 1.3.7)
 where the corank of the the matrix $A(x) : = (a_{ij})_{i,j=1, \dots 4}$
 equals 2. Let $\sN_r$ be this set of nodes.
 
 Then $\pi : Z \ra W$ is the small resolution of the set of nodes $\sN_r$, and the direct image
 $$ \sF : = \pi_* ( \hol_Z(0,1)) , $$
 of the invertible sheaf  of bidegrees $(0,1)$
 is a torsion free sheaf admitting the resolution
 $$ 0 \ra \hol_{\PP^4} (-2)^4 \ra \hol_{\PP^4} (-1)^4 \ra \sF \ra 0,$$
 where the middle homomorphism is given by the matrix $A(x)$.
 
 Indeed, $H^0(\sF(1))$ has dimension 4 , and since $\sF$ is Cohen Macaulay, 
 and we have those 4 relations $\sum_j a_{ij}(x) v_j = 0$, then the sequence is exact
 (alternatively we can use the direct image of the Koszul complex on $P : = \PP^4 \times \PP^3$).
 
 Consider now a neighbourhood of $W$ inside $\sF_3( 4,\de)$: then we can take the corresponding
 small resolutions $Z_t$  of the 20 nodes of $W_t$ which are deformation of the nodes of $W$ lying in $\sN_r$.
 
 We use now that $H^i( \hol_Z)=0$ for $i=1,2$, because of the Koszul complex on $ V : = \hol_P(-1, -1)^4$
 associated to the four sections $$\sum_j a_{ij}(x) v_j \in H^0 ( \hol_P(1, 1)):$$
 $$ 0 \ra    \Lam^4 (V) \ra \Lam^3 (V) \ra \Lam^2 (V) \ra  V \ra  \hol_P  \ra \hol_Z \ra 0.$$
 By semicontinuity also $H^i( \hol_{Z_t})=0$ for $i=1,2$. Hence $\Pic(Z_t)$ is locally constant,
 and the line bundle $\hol_Z(0,1)$ extends for $Z_t$ in a neghbourhood. Hence  also $Z_t$ admits
 such a line bundle,  which  yields an  embedding    in $P =  \PP^4 \times \PP^3$ as such a complete intersection.
 
 We conclude therefore that also $W_t$, being the projection of $Z_t$,  is linear determinantal.
 \end{proof}

 \subsection{Segre type representation of 65-nodal sextics}\label{subsec: Segre realization}
 We shall now show that if $Y$ has 65 nodes, then it admits a Segre realization
 $$ Y = \{ \la_1(x) \dots \la_6(x) - B(x)^2 = 0\}.$$
 
 This follows immediately from Theorem \ref{Segre} once we prove that the Barth sextic (hence, any other 65 nodal sextic)
 admits a subcode of $\sK''$ which is the Segre code.
 
 This in turn follows once we find a Segre realization of the Barth sextic: but this is precisely given by the defining equation 
$$(\tau^2 x^2 - y^2) (\tau^2 y^2 - z^2)  (\tau^2 z^2 - x^2) - \frac{1}{4} ( 2 \tau +1) w^2(x^2 + y^2 + z^2 - w^2)^2,$$ 
where  $\tau : = \frac{1}{2} ( 1 + \sqrt{5})$ is the inverse of the golden ratio. 

Here the six linear forms $\la_i(x)$  are not in general position, since they have a common zero set, but the code is indeed the Segre code since the 15 lines intersection of pairs of planes $\{ \la_i(x)=\la_j(x)=0\}$ pass through 3 nodes each.

 We may ask whether   there is a realization like for the Barth sextic: 
$$ Y = \{x |  \la_1(x) \dots \la_6(x) - \la_0(x)^2 Q(x)^2=0 \}.$$

Since $Y$ has the same code $\sK'$ as the Barth sextic, we saw that there are codewords $v_1, \dots, v_6$, such that they correspond to  half-even sets of cardinality 15
produced by linear forms $ \la_1(x), \dots , \la_6(x)$. 

The question boils down to the question whether  the divisor of $B$
contains the divisor of a linear form $\la_0$,
equivalently, whether  $B(x) = \la_i(x)=0$ defines a plane cubic that splits as a conic and a line $\Lam_i$,
and  the six lines $\Lam_1, \dots, \Lam_6$ are coplanar. However it seems that the above realization might be a special
occurrence, by virtue of the  following observation.

It can be computationally checked\footnote{see the script \cite[Mathematica 20]{Scripts}.} 
that the extended code of Barth sextic contains precisely $65$ Segre subcodes. This leads to further $64$ Segre realizations. For instance, one of the realizations consists of six planes
\[
	(\tau w - \tau^2 x+ z)(-\tau w -\tau^2 x + z)(-w+x-y+z)(w+x-y+z)(-w+x+y+z)(w+x+y+z),
\]
where $B(x)$ equals
\[
	\tau^2 x^3 - \tau^2 w^2 x  - \bar\tau xy^2 - w^2 z + \tau^3 x^2z - \tau^3 y^2 z + \bar\tau xz^2 + z^3 .\footnotemark
\]\footnotetext{see the script \cite[Mathematica 21]{Scripts}.}%
In this case $B(x)$ is an irreducible cubic, unlike the Segre realization introduced in the beginning. We also remark that in the latter realization the six planes linearly span $\PP^4$. This happens for every other Segre realization except for the first one.

  It follows from this computational result that the group of projective automorphisms of the Barth surface fixes the
point $\{ x=y=z=0\}$, leaves invariant the plane $ \{ w=0\}$ and the sphere $\Sigma : = \{ x^2 + y^2 + z^2 - w^2=0\}$,  and permutes the six planes 
$$ \{ \tau x  \pm  y =0\}, \{ \tau y  \pm  z =0\}, \{ \tau z  \pm  x =0\} .$$

We use this remark in order to give a computational proof of the following result

\begin{theo}\label{barthauto}
The group of automorphisms of the Barth sextic $Y$ has order 120, hence it is the Icosahedral group $\ZZ/2 \times \mathfrak A_5$.

\end{theo}\

We first observe that all automorphisms are projective automorphisms since they preserve the canonical divisor,
which is equal to $2H$. Since the surface $Y$ has no two torsion in second homology, these automorphisms preserve $H$,
hence they are induced from projectivities.

The following argument is due to Stephen Coughlan.

We start from the defining equation $f(w,x,y,z)=0$ of the Barth sextic $Y 
\subset  \PP^3$, where $F(w,x,y,z):=$
$$(\tau^2 x^2 - y^2) (\tau^2 y^2 - z^2)  (\tau^2 z^2 - x^2) - \frac{1}{4} ( 2 
\tau +1) w^2(x^2 + y^2 + z^2 - w^2)^2$$ 
and $\tau : = \frac{1}{2} ( 1 + \sqrt{5})$ is the inverse of the golden ratio. 

Using this observation, we show that the projective automorphism group of 
$Y$ has order 120, with the help of a computer.

First change the coordinates to
\[a:=\tau x-y,\ b:=\tau y-z,\ c:=\tau z-x,\ w:=w,\]
and define expressions
\[d:=\tau(a+b)+c,\ e:=\tau(b+c)+a,\ f:=\tau(c+a)+b.\]
In these 
coordinates, the Barth sextic has the nice form
\[abcdef
-1/4\tau w^2(ad+be+cf-\tau w^2)^2,\]
the six permuted planes are
\begin{equation}\label{eq!permuted-planes}
\{a=0\},\{b=0\},\{c=0\},\{d=0\},\{e=0\},\{f=0\} 
\end{equation}
and the invariant sphere is $\{ad+be+cf-\tau w^2\}$.

In order to make the proof of Theorem \ref{barthauto} more transparent,
we make its statement more explicit.

\begin{prop}
The projective automorphism group of the Barth sextic has 120 elements. 
Moreover, in the coordinates $a,b,c,w$, a projective automorphism 
preserving the Barth sextic can be written in the form
\[w\mapsto \pm w,\ a\mapsto p_1\sigma_1,\ b\mapsto p_2\sigma_2,\ c\mapsto 
p_3\sigma_3\]
where $(\sigma_1,\sigma_2,\sigma_3)$ are certain triples of three distinct 
elements from $\{a,b,c,d,e,f\}$, and $(p_1,p_2,p_3)$ are certain triples in 
$\{\pm1\}^3$.
\end{prop}
\begin{proof}
Let $\Phi$ be a projective automorphism of the Barth sextic.
There exist scalars 
$\alpha,\beta,\gamma,\delta,\epsilon,\varphi$ such that
\[
\Phi(a)=\alpha\sigma_1,\Phi(b)=\beta\sigma_2,\Phi(c)=\gamma\sigma_3,
\Phi(d)=\delta\sigma_4,\Phi(e)=\epsilon\sigma_5,\Phi(f)=\varphi\sigma_6.
\]
where $\sigma=(\sigma_1,\dots,\sigma_6)$ is an ordered permutation of 
$a,\dots,f$.

Since $\Phi$ is a homomorphism, we have $\Phi(d)=\Phi(\tau(a+b)+c)$ and 
similar for $\Phi(e)$, $\Phi(f)$. Substituting the above general form for
$\Phi$ into these equalities gives three equations
\[\delta \sigma_4=\tau(\alpha \sigma_1+\beta \sigma_2)+\gamma \sigma_3,\
 \epsilon \sigma_5=\tau(\beta \sigma_2+\gamma \sigma_3)+\alpha \sigma_1,\
\varphi \sigma_6=\tau(\gamma \sigma_3+\alpha \sigma_1)+\beta \sigma_2.
\]
Comparing the coefficients of $a,b,c$ in each of the three equalities gives a 
homogeneous system of nine linear equations in six unknowns $\alpha,\dots,\varphi$.

Using a computer, we consider each 6-tuple $\sigma$, construct its system of linear
equations, and determine the space $V_\sigma$ of solutions. We find that 
there are 60 permutations $\sigma$ for which $V_\sigma$ is spanned by a single vector
in $\{\pm1\}^6$; for the other permutations, $V_\sigma$ is the zero space and therefore
does not correspond to an automorphism.

We examine the cases where $V_\sigma$ is nonempty and write $p=(p_1,\dots,p_6)$
for the basis vector in $\{\pm1\}$. We have a 
1-parameter family of possible automorphisms corresponding to the space $V_\sigma$:
$$
\Phi(a)=p_1\alpha\sigma_1,\Phi(b)=p_2\alpha\sigma_2,\Phi(c)=p_3\alpha\sigma_3,$$
$$
\Phi(d)=p_4\alpha\sigma_4,\Phi(e)=p_5\alpha\sigma_5,\Phi(f)=p_6\alpha\sigma_6.
$$

Next observe that the sphere $ad+be+cf-\tau w^2$ is invariant under 
$\Phi$, and $\Phi(w)=\lambda w$. Substituting this 1-parameter family and 
then scaling $\lambda=\pm\alpha$, we get that 
$\Phi(ad+be+cf-\tau w^2)=\lambda^2(ad+be+cf-\tau w^2)$.
Thus there are 60 permutations $\sigma$ with $V_\sigma$ non-zero, and for 
each such $V_\sigma$ there are two possibilities for $\lambda$, this gives 120 
automorphisms in total. It is easy to see from the normal form, that in each 
case, the Barth sextic is also invariant under $\Phi$.
\end{proof}

The proof uses a computer to work out the solutions of the system of linear equations for each of the 720 permutations.
 
 \bigskip
 
\section{The code of the Barth sextic}
\label{ap_code_barth_sextic}

We start again from the defining equation $f(w,x,y,z)=0$ of the Barth sextic $Y \subset  \PP^3$, where $f(w,x,y,z):=$
$$(\tau^2 x^2 - y^2) (\tau^2 y^2 - z^2)  (\tau^2 z^2 - x^2) - \frac{1}{4} ( 2 \tau +1) w^2(x^2 + y^2 + z^2 - w^2)^2$$ 
and $\tau : = \frac{1}{2} ( 1 + \sqrt{5})$ is the inverse of the golden ratio. 
\smallskip

	In \cite{barth}, Barth described the icosahedral symmetry of $Y$. Consider the icosahedron with vertices
	\[
	 	(1, 0, \pm \tau, \pm 1),\ (1, \pm1, 0, \pm \tau), \  (1, \pm\tau, \pm1, 0).
	\]
	Each face of the icosahedron can be labelled by an ordered  pair of distinct integers $ij$, $1 \leq i,j \leq 5$. Figure~\ref{fig: Dodecahedron} shows the labelling using the dual dodecahedron.

	\begin{figure}[h]
	\scalebox{0.66}{
		\tdplotsetmaincoords{30+90*3}{90+90*3}
		\begin{tikzpicture}[tdplot_main_coords,scale=5.45]
		\pgfmathsetmacro{\pgftau}{1.618}
		\pgfmathsetmacro{\ta}{{1/3*(\pgftau+1)}}
		\pgfmathsetmacro{\tb}{{1/3*\pgftau}}
		\pgfmathsetmacro{\tc}{{1/3*(2*\pgftau+1)}}
		\tikzset{inner sep = 5pt, shape=circle}
		\definecolor{coordinatecolor}{rgb}{0,0,1}
		\draw(\ta,\ta,\ta) node[anchor=center] (12){$12$};
		\node[color=coordinatecolor] at ([xshift=5pt,yshift=2pt] 12) {$\scriptstyle \frac13 (\tau^2,\tau^2,\tau^2) $};
		\draw(-\ta,-\ta,-\ta) node[anchor=center] (21){$ 21$};
		\node[color=coordinatecolor] at ([xshift=-6pt,yshift=-2pt] 21) {$\scriptstyle \frac13 (-\tau^2,-\tau^2,-\tau^2) $};
		\draw(\tb,\tc,0) node[anchor=center] (34){$ 34$};
		\node[color=coordinatecolor] at ([xshift=5pt,yshift=0pt] 34) {$\scriptstyle \frac13 (\tau,\tau^3,0) $};
		\draw(-\tb,-\tc,0) node[anchor=center] (43){$ 43$};
		\node[color=coordinatecolor] at ([xshift=-6pt,yshift=0pt] 43) {$\scriptstyle \frac13 (-\tau,-\tau^3,0) $};
		
		\draw(-\tb,\tc,0) node[anchor=center] (25){$ 25$};
		\node[color=coordinatecolor] at ([xshift=-6pt,yshift=0pt] 25) {$\scriptstyle \frac13 (-\tau,\tau^3,0) $};
		\draw(\tb,-\tc,0) node[anchor=center] (52){$ 52$};
		\node[color=coordinatecolor] at ([xshift=6pt,yshift=0pt] 52) {$\scriptstyle \frac13 (\tau,-\tau^3,0) $};
		\draw(-\ta,\ta,\ta) node[anchor=center] (41){$ 41$};
		\node[color=coordinatecolor] at ([xshift=-6pt,yshift=1.25pt] 41) {$\scriptstyle \frac13 (-\tau^2,\tau^2,\tau^2) $};
		\draw(\ta,-\ta,-\ta) node[anchor=center] (14){$ 14$};
		\node[color=coordinatecolor] at ([xshift=6.5pt,yshift=-1.5pt] 14) {$\scriptstyle \frac13 (\tau^2,-\tau^2,-\tau^2) $};
		\draw(0,\tb,\tc) node[anchor=center] (53){$ 53$};
		\node[color=coordinatecolor] at ([xshift=-5pt,yshift=0pt] 53) {$\scriptstyle \frac13 (0,\tau,\tau^3) $};
		\draw(0,-\tb,-\tc) node[anchor=center] (35){$ 35$};
		\node[color=coordinatecolor] at ([xshift=6pt,yshift=0pt] 35) {$\scriptstyle \frac13 (0,-\tau,-\tau^3) $};
		\draw(-\ta,\ta,-\ta) node[anchor=center] (13){$ 13$};
		\node[color=coordinatecolor] at ([xshift=-5.5pt,yshift=1.75pt] 13) {$\scriptstyle \frac13 (-\tau^2,\tau^2,-\tau^2) $};
		\draw(\ta,-\ta,\ta) node[anchor=center] (31){$ 31$};
		\node[color=coordinatecolor] at ([xshift=5.5pt,yshift=-2pt] 31) {$\scriptstyle \frac13 (\tau^2,-\tau^2,\tau^2) $};
		\draw(-\tc,0,\tb) node[anchor=center] (32){$ 32$};
		\node[color=coordinatecolor] at ([xshift=6pt,yshift=-0.5pt] 32) {$\scriptstyle \frac13 (-\tau^3,0,\tau) $};		
		\draw(\tc,0,-\tb) node[anchor=center] (23){$ 23$};
		\node[color=coordinatecolor] at ([xshift=-6pt,yshift=0.5pt] 23) {$\scriptstyle \frac13 (\tau^3,0,-\tau) $};		
		\draw(0,-\tb,\tc) node[anchor=center] (24){$ 24$};
		\node[color=coordinatecolor] at ([xshift=5.5pt,yshift=-1pt] 24) {$\scriptstyle \frac13 (0,-\tau,\tau^3) $};		
		\draw(0,\tb,-\tc) node[anchor=center] (42){$ 42$};
		\node[color=coordinatecolor] at ([xshift=5.5pt,yshift=1pt] 42) {$\scriptstyle \frac13 (0,\tau,-\tau^3) $};		
		\draw(\tc,0,\tb) node[anchor=center] (45){$ 45$};
		\node[color=coordinatecolor] at ([xshift=4.5pt,yshift=-1.25pt] 45) {$\scriptstyle \frac13 (\tau^3,0,\tau) $};
		\draw(-\tc,0,-\tb) node[anchor=center] (54){$ 54$};
		\node[color=coordinatecolor] at ([xshift=6.5pt,yshift=1pt] 54) {$\scriptstyle \frac13 (-\tau^3,0,-\tau) $};
		\draw(\ta,\ta,-\ta) node[anchor=center] (51){$ 51$};
		\node[color=coordinatecolor] at ([xshift=5.5pt,yshift=1.75pt] 51) {$\scriptstyle \frac13 (\tau^2,\tau^2,-\tau^2) $};
		\draw(-\ta,-\ta,\ta) node[anchor=center] (15){$ 15$};
		\node[color=coordinatecolor] at ([xshift=-6pt,yshift=-1.75pt] 15) {$\scriptstyle \frac13 (-\tau^2,-\tau^2,\tau^2) $};
		
		\draw[thick] (12) -- (34) -- (25) -- (41) -- (53) -- (12);
		\draw[thick] (25) -- (13) -- (54) -- (32) -- (41);
		\draw[thick] (32) -- (15) -- (24) -- (53);
		\draw[thick] (24) -- (31) -- (45) -- (12);
		\draw[thick] (45) -- (23) -- (51) -- (34);
		\draw[thick] (51) -- (42) -- (13);
		\draw[dotted,thick] (54) -- (21) -- (43) -- (15);
		\draw[dotted,thick] (43) -- (52) -- (31);
		\draw[dotted,thick] (52) -- (14) -- (23);
		\draw[dotted,thick] (14) -- (35) -- (42);
		\draw[dotted,thick] (21) -- (35);
	\end{tikzpicture}
	}\vskip-\baselineskip
		\caption{Labelling vertices of the dual dodecahedron}\label{fig: Dodecahedron}
	\end{figure}
	The action of $\mathfrak A_5$ via $\sigma \cdot ij \mapsto \sigma(i)\sigma(j)$ realizes $\mathfrak A_5$ as the group of rotational icosahedral symmetries.

Consider a line which connects two opposite centres of the faces of the icosahedron. Such a line is called a \emph{centre line}. In Figure~\ref{fig: Dodecahedron}, centre lines are the lines connecting opposite vertices, so they can be labelled by 
the unordered pair $(ij)$. The parentheses mean that we do not distinguish $(ij)$ and $(ji)$. The \emph{mid lines} are the lines which connect mid points of two opposite edges of the icosahedron. Equivalently, they are the lines connecting mid points of the opposite edges of the dual dodecahedron. They correspond to double transpositions, hence they can be labelled by $(ij)(kl)$, where $i,j,k,l$ are all distinct.

Each mid line contains $3$ nodes, and each centre line contains $2$ nodes. Since there are $15$ mid lines and $10$ centre lines, all the $65 = 3 \times 15 + 2 \times 10$ nodes belong to these lines.
\begin{enumerate}
	\item[(A,B)] Each mid line $(ij)(kl)$  intersects  the plane $(w=0)$ at a node, which will be denoted by ${\rm A}_{(ij)(kl)}$. In the affine space $(w\neq 0)$, the node closer to $ij$ will be denoted by $B_{ij(kl)}$.
	\item[(C)] For each centre line $(ij)$, the node close to $ij$ will be denoted by $C_{ij}$.
\end{enumerate}

The nodes of type A form  an orbit of $\mathfrak A_5$, they correspond to double pairs (stabilised by a Klein subgroup $(\ZZ/2)^2)$,
that is, double transpositions.

Similarly those of type C, form the orbit of ordered pairs, and are stabilised by a 3-cycle. These correspond to  cyclic subgroups
of order 3 (self normalizing).

Those of type B form  the orbit stabilised by a double transposition:  one takes 
 an ordered pair and a pair (in this way we get an orbit of 60 elements on which $\mathfrak A_5$ acts  transitively),
 but then identifying $ab (cd) $ with $ba(dc) $  the stabiliser becomes  
$(ab)(cd) $, hence we get the cosets for the subgroup generated by this double transposition.

\bigskip

Clearly, these are three orbits of the $\mathfrak A_5$-action on the set of nodes.

	A direct computation (which  can be performed by hand) leads to the following table, presenting explicit coordinates of the nodes\,(see Table~\ref{table: Nodes in Barth}).
	
	\medskip

	\begin{table}[h!]
	\scalebox{0.8}{$
		\begin{array}{c|c|c|c|c|c|c|c}
			{\rm A}_{(12)(34)} & {\rm A}_{(41)(25)} & {\rm A}_{(31)(52)} & {\rm A}_{(51)(34)} & {\rm A}_{(12)(53)} & {\rm A}_{(41)(53)} & {\rm A}_{(31)(24)} & {\rm A}_{(51)(42)} \\ \hline
			\scriptstyle (0,1,\tau,-\bar\tau) & \scriptstyle(0,-1,\tau,-\bar\tau) & \scriptstyle(0,1,-\tau,-\bar\tau) & \scriptstyle(0,1,\tau,\bar\tau) & \scriptstyle(0,-\bar\tau, 1, \tau) & \scriptstyle(0,\bar\tau,1,\tau) & \scriptstyle(0,-\bar\tau,-1,\tau) & \scriptstyle(0,-\bar\tau,1,-\tau) \\ \multicolumn{8}{l}{} \\[-0.5em]
			{\rm A}_{(12)(45)} & {\rm A}_{(41)(32)} & {\rm A}_{(31)(45)} & {\rm A}_{(51)(23)} & {\rm A}_{(23)(45)} & {\rm A}_{(25)(34)} & {\rm A}_{(24)(53)} \\ \hline
			\scriptstyle(0,\tau,-\bar\tau,1) & \scriptstyle(0,-\tau,-\bar\tau,1)  & \scriptstyle(0,\tau,\bar\tau,1) & \scriptstyle(0,\tau,-\bar\tau,-1) & \scriptstyle(0,1,0,0) & \scriptstyle(0,0,1,0) & \scriptstyle(0,0,0,1) \\ \multicolumn{8}{l}{} \\[1em]
			{\rm B}_{12(34)} & {\rm B}_{41(25)} & {\rm B}_{31(52)} & {\rm B}_{51(34)} & {\rm B}_{21(43)} & {\rm B}_{14(52)} & {\rm B}_{13(25)} & {\rm B}_{15(43)} \\ \hline
			\scriptstyle(2,1,\tau,-\bar\tau) & \scriptstyle(2,-1,\tau,-\bar\tau) & \scriptstyle(2,1,-\tau,-\bar\tau) & \scriptstyle(2,1,\tau,\bar\tau) & \scriptstyle(2,-1,-\tau,\bar\tau) & \scriptstyle(2,1,-\tau,\bar\tau) & \scriptstyle(2,-1,\tau,\bar\tau) & \scriptstyle(2,-1,-\tau,-\bar\tau) \\ \multicolumn{8}{l}{} \\[-0.5em]
			{\rm B}_{12(53)} & {\rm B}_{41(53)} & {\rm B}_{31(24)} & {\rm B}_{51(42)} & {\rm B}_{21(35)} & {\rm B}_{14(35)} & {\rm B}_{13(42)} & {\rm B}_{15(24)} \\ \hline
			\scriptstyle(2,-\bar\tau,1,\tau) & \scriptstyle(2,\bar\tau,1,\tau) & \scriptstyle(2,-\bar\tau,-1,\tau) & \scriptstyle(2,-\bar\tau,1,-\tau) & \scriptstyle(2,\bar\tau,-1,-\tau) & \scriptstyle(2,-\bar\tau,-1,-\tau) & \scriptstyle(2,\bar\tau,1,-\tau) & \scriptstyle(2,\bar\tau,-1,\tau) \\ \multicolumn{8}{l}{} \\[-0.5em]
			{\rm B}_{12(45)} & {\rm B}_{41(32)} & {\rm B}_{31(45)} & {\rm B}_{51(23)} & {\rm B}_{21(54)} & {\rm B}_{14(23)} & {\rm B}_{13(54)} & {\rm B}_{15(32)} \\ \hline
			\scriptstyle(2,\tau,-\bar\tau,1) & \scriptstyle(2,-\tau,-\bar\tau,1) & \scriptstyle(2,\tau,\bar\tau,1) & \scriptstyle(2,\tau,-\bar\tau,-1) & \scriptstyle(2,-\tau,\bar\tau,-1)  & \scriptstyle(2,\tau,\bar\tau,-1) & \scriptstyle(2,-\tau,-\bar\tau,-1) & \scriptstyle(2,-\tau,\bar\tau,1) \\ \multicolumn{8}{l}{} \\[-0.5em]
			{\rm B}_{23(45)} & {\rm B}_{32(54)} & {\rm B}_{25(34)} & {\rm B}_{52(43)} & {\rm B}_{24(53)} & {\rm B}_{42(35)} & \\ \hline
			\scriptstyle(1,1,0,0) & \scriptstyle(1,-1,0,0) & \scriptstyle(1,0,1,0) & \scriptstyle(1,0,-1,0) & \scriptstyle(1,0,0,1) & \scriptstyle(1,0,0,-1) &  \\ \multicolumn{8}{l}{} \\[1em]
			
			{\rm C}_{12} & {\rm C}_{41} & {\rm C}_{31} & {\rm C}_{51} & {\rm C}_{21} & {\rm C}_{14} & {\rm C}_{13} & {\rm C}_{15} \\ \hline
			\scriptstyle(1,1,1,1) & \scriptstyle(1,-1,1,1) & \scriptstyle(1,1,-1,1) & \scriptstyle(1,1,1,-1) & \scriptstyle(1,-1,-1,-1) & \scriptstyle(1,1,-1,-1) & \scriptstyle(1,-1,1,-1) & \scriptstyle(1,-1,-1,1) \\ \multicolumn{8}{l}{} \\[-0.5em]
			{\rm C}_{53} & {\rm C}_{24} & {\rm C}_{35} & {\rm C}_{42} & {\rm C}_{45} & {\rm C}_{32} & {\rm C}_{54} & {\rm C}_{23} \\ \hline
			\scriptstyle(1,0,-\bar\tau,\tau) &  \scriptstyle(1,0,\bar\tau,\tau) & \scriptstyle(1,0,\bar\tau,-\tau) & \scriptstyle(1,0,-\bar\tau,-\tau) & \scriptstyle(1,\tau,0,-\bar\tau) & \scriptstyle(1,-\tau,0,-\bar\tau) & \scriptstyle(1,-\tau,0,\bar\tau) & \scriptstyle(1,\tau,0,\bar\tau) \\ \multicolumn{8}{l}{} \\[-0.5em]
			{\rm C}_{34} & {\rm C}_{25} & {\rm C}_{43} & {\rm C}_{52} & & & \\ \hline
			\scriptstyle(1,-\bar\tau,\tau,0) & \scriptstyle(1,\bar\tau,\tau,0) & \scriptstyle(1,\bar\tau,-\tau,0) & \scriptstyle(1,-\bar\tau,-\tau,0) & & & 
		\end{array}
	$}\medskip
	\caption{Nodes in Barth sextic}\label{table: Nodes in Barth}
	\end{table}

	Using Table~\ref{table: Nodes in Barth}, we may explicitly describe the representation $\mathfrak A_5 \to \GL(\mathcal K)$ via the permutation representations $\rho_A \colon \mathfrak A_5 \to \mathfrak S_{15}$, $\rho_B \colon \mathfrak A_5 \to \mathfrak S_{30}$, and $\rho_C \colon \mathfrak A_5 \to \mathfrak S_{20}$.
	
	\begin{prop}
		Let $\sigma = (1\ 2\ 3)$ and $\tau = (1\ 4)(2\ 5)$ be generators of $\mathfrak A_5$. We have
		\[
			\begin{array}{r@{}l}
				\rho_A(\sigma) ={}& (1\ 10\ 7)( 2\ 15\ 4)(3\ 5\ 12)( 6\ 8 \ 14)(9\ 13\ 11)\\
				\rho_A(\tau) ={}& (1\ 11)(3\ 14)( 4\ 7)(5\ 13)(6\ 10)(12\ 15) \\[5pt]
				\rho_B(\sigma) ={}& (1\ 22\ 11)(2\ 30\ 8)(3\ 9\ 20)(4\ 6\ 29)(5\ 18\ 15)(7\ 13\ 24) \\
								& (10\ 12\ 28)(14\ 16\ 27)(17\ 25\ 19)(21\ 26\ 23) \\
				\rho_B(\tau) ={}& (1\ 19)(2\ 6)(3\ 27)(4\ 11)(5\ 23)(7\ 28)(8\ 15)(9\ 25)\\
								& (10\ 22)(12\ 16)(13\ 26)(14\ 18)(20\ 29)(24\ 30) \\[5pt]
				\rho_C(\sigma) ={}& (1\ 16\ 3)(2\ 12\ 19)(4\ 20\ 9)(5\ 14\ 7)(6\ 10\ 17)(8\ 18\ 11) \\
				\rho_C(\tau) ={}& (1\ 13)(2\ 6)(3\ 17)(4\ 10)(5\ 15)(7\ 19)(8\ 12)(9\ 16)(11\ 14)(18\ 20).
			\end{array}
		\]
	\end{prop}

Due to the relation
$$
  f(x+y+z,x,y,z)=-\tau \cdot h(x,y,z)^2,
$$
where
$$
  h(x,y,z)=\tau^2\left(x^2y+y^2z+z^2x\right)+ \left(xy^2+yz^2+zx^2\right)+ \left(\tau^2+1\right)xyz,
$$
the singularities contained in the hyperplane $H$ perpendicular to $(1,-1,-1,-1)$ form a half-even set $Sing(Y)\cap H=$
\[
	\left\{
	\begin{array}{c}
		{\rm A}_{(31)(52)},\  {\rm A}_{(51)(42)},\  {\rm A}_{(41)(32)},\ \\ {\rm B}_{51(34)},\  {\rm B}_{41(53)},\  {\rm B}_{31(45)},\  {\rm B}_{23(45)},\  {\rm B}_{25(34)},\  {\rm B}_{42(35)},\  \\
		{\rm C}_{41},\  {\rm C}_{31},\  {\rm C}_{51},\  {\rm C}_{24},\  {\rm C}_{23},\  {\rm C}_{25}
	\end{array}\right\}
\]
	The corresponding codeword of $\sK'\backslash \sK$ is given by 
$$
	\scalebox{0.8}{$
	\bigl(
    1\,|\,0010\, 0001\, 0100\, 000\,|\,00010000\, 01000000\, 00100000\, 101010\,|\,0111\, 0000\, 0100\, 0001\, 0100
	\bigr)
	$}.
$$
The $\mathfrak A_5$-orbit consists \footnote{For this and the next computation see the script \cite[Mathematica 3]{Scripts}.} of $20$ half-even sets having cardinality $15$:
\[
	\scalebox{0.8}{$
	\left(
	\begin{array}{c @{\,|\,} c @{\,} c @{\,} c @{\,} c @{\,|\,} c @{\,} c @{\,} c@{\,} c @{\,|\,} c@{\,} c@{\,} c@{\,} c}
		1&0000&0000&0110&001&01000010&10010000&10010000&000000&01000010&1001&1001&0000\\
		1&0000&0000&1001&001&00001001&00100100&00100100&000000&00001001&0110&1001&0000\\
		1&0000&0000&1001&001&10010000&01000010&01000010&000000&10010000&1001&0110&0000\\
		1&0000&0000&0110&001&00100100&00001001&00001001&000000&00100100&0110&0110&0000\\
		1&0000&1010&0000&010&01000001&01000001&00001010&000000&00001010&1100&0000&0110\\
		1&0000&0101&0000&010&10100000&10100000&00010100&000000&00010100&1100&0000&1001\\
		1&0000&1010&0000&010&00010100&00010100&10100000&000000&10100000&0011&0000&1001\\
		1&0000&0101&0000&010&00001010&00001010&01000001&000000&01000001&0011&0000&0110\\
		1&1100&0000&0000&100&00100001&11000000&00100001&000000&11000000&0000&1100&0011\\
		1&0011&0000&0000&100&11000000&00100001&11000000&000000&00100001&0000&1100&1100\\
		1&1100&0000&0000&100&00010010&00001100&00010010&000000&00001100&0000&0011&1100\\
		1&0011&0000&0000&100&00001100&00010010&00001100&000000&00010010&0000&0011&0011\\
		1&0100&1000&0010&000&10000000&00000010&00000100&101001&10000110&0010&1000&0100\\
		1&1000&0100&0001&000&00000100&00000001&10000000&100110&10000101&1000&0001&0010\\
		1&0001&0010&1000&000&00000010&10000000&00000001&011010&10000011&0100&0010&1000\\
		1&0010&0001&0100&000&00010000&01000000&00100000&101010&01110000&0100&0001&0100\\
		1&0100&1000&0010&000&00001000&00100000&01000000&010110&01101000&1000&0010&0001\\
		1&0001&0010&1000&000&00100000&00001000&00010000&100101&00111000&0001&1000&0010\\
		1&0010&0001&0100&000&00000001&00000100&00000010&010101&00000111&0001&0100&0001\\
		1&1000&0100&0001&000&01000000&00010000&00001000&011001&01011000&0010&0100&1000
  \end{array}
  \right)
  $}
\]
	and the $\F_2$-basis is as follows.
	\[
		\scalebox{0.8}{$
		\left(
		\begin{array}{c @{\,|\,} c @{\,} c @{\,} c @{\,} c @{\,|\,} c @{\,} c @{\,} c@{\,} c @{\,|\,} c@{\,} c@{\,} c@{\,} c}
			1 & 0000 & 0000&0110&001&00000110&10000001&00011000&111111&10011111&0011&1100&1010 \\
			0 & 1000 & 0001&0001&010&00001110&00001011&11000001&100110&11000100&1011&0001&0100 \\
			0 & 0100 & 0001&0111&111&00001011&11000010&11101001&100110&00100000&1101&1011&0111 \\
			0 & 0010 & 0001&0010&001&00000111&10000101&00011010&101010&10011000&0010&1000&1011 \\
			0 & 0001 & 0001&0100&100&00001101&00010110&00001110&010101&00010101&0001&0111&0010 \\
			0 & 0000 & 1001&0011&110&00000101&01100011&10110001&001111&11010111&0110&0101&1100 \\
			0 & 0000 & 0101&0110&011&00001100&10001011&01011001&111111&11011110&0000&1100&1100 \\
			0 & 0000 & 0011&0101&101&00000110&10110010&00101011&001111&10011111&0011&1100&1010 \\
			0 & 0000 & 0000&1111&000&00001111&10100101&00111100&111111&10010110&0101&0101&1010 \\
			0 & 0000 & 0000&0000&000&10001000&00100010&01000100&111111&11101110&1010&1010&0101 \\
			0 & 0000 & 0000&0000&000&01000100&00010001&10001000&111111&11011101&1010&0101&1010 \\
			0 & 0000 & 0000&0000&000&00100010&10001000&00010001&111111&10111011&0101&1010&1010 \\
			0 & 0000 & 0000&0000&000&00010001&01000100&00100010&111111&01110111&0101&0101&0101 \\
		\end{array}
		\right)
		$}
	\]

Since we know from e.g.\ \cite{j-r} that $\dim(\sK)\le 12$, the above matrix is a generator matrix of $\sK'$ and 
we obtain a generator matrix of $\sK$ by removing the first row and the first column\footnote{Another argument for the fact that 
the above matrix $M$ is a generator matrix of $\sK'$, without using the result of \cite{j-r}, goes as follows. By construction, 
removing the first row and column generates a subcode of $\sK$. It can be easily checked that this subcode is projective, so that we can conclude 
from Proposition \ref{prop_classification} that no extension to a $13$-dimensional code exists. 
}.

\medskip

The weight enumerator \footnote{see the script \cite[Mathematica 4]{Scripts}.} of $\sK$ is given by 
$$
  W(x) = 1x^{0} + 390x^{24} + 3055x^{32} + 650x^{40}
$$
and the weight enumerator of $\sK'\backslash \sK$ is given by
$$
  W(x)=26x^{16} + 650x^{28} + 1690x^{32} + 1300x^{36} + 300x^{40} + 130x^{44}. 
$$
Both codes $\sK$ and $\sK'$ have an automorphism group of order $15\,600=2^4 \cdot 3\cdot 5^2\cdot 13$ that is isomorphic to 
$\operatorname{P\Sigma L}(2,25)=\operatorname{PSL}(2,25)\rtimes \langle F\rangle$, where $F\colon x\mapsto x^5$ is the Frobenius 
automorphism of $\mathbb{F}_{25}$, as can be  checked by a MAGMA computation \footnote{
see the script \cite[MAGMA 22]{Scripts}.}
 and for reasons that we shall now explain. 

In fact, we can associate to $Y$ a graph $\sG$ with the 
same automorphism group. To this end let $P_1,\dots,P_{65}$ be the vertices of $\sG$. Two vertices $P_i$ and $P_j$ are 
joined by an edge if and only if
$$
  g_{i,j}(t):=f\!\left(t\cdot P_i+(1-t)\cdot P_j\right) =ct^2(1-t)^2(t-x_1)(t-x_2),
$$   
where $c,x_1,x_2\in\CC$ and $\left|\left\{0,1,x_1,x_2\right\}\right|=4$, i.e., the line $t\cdot P_i+(1-t)\cdot P_j$ intersects 
the surface $Y$ in the two nodes $P_i$, $P_j$ and two different additional points. 

Every linear automorphism $\rho$ preserving the equation $f$ corresponds to an automorphism $\tilde{\rho}$ of $\sK$, i.e., $\rho$ induces a permutations on $Sing(Y)$, 
which can be written as a permutation $\tilde{\rho}$ of the columns of a generator matrix of $\sK$. The converse is not true in general.

 In  subsection \ref{subsec: Aut(Barth)}, we shall prove that the group of linear automorphisms of the Barth surface  is indeed $\bigl\langle \mathfrak A_5, (\ZZ/2)^3 \bigr\rangle$, where $(\ZZ/2)^3$ is the group of the coordinate sign changes, and has order $120$. 

In particular the Barth surface does not admit a linear automorphism of order $13$, while the automorphism group of $\sK$ contains elements of order $13$.

The graph $\sG$ is nowadays called 
\emph{Doro-Hall graph} and is one of the three \emph{locally Petersen graphs}, i.e., graphs where the neighbourhood of every vertex is isomorphic 
to the Petersen graph (quotients of the dodecahedral graph by the antipodal involution), see \cite{hall}.

\newcommand*{\PGL}{\operatorname{PGL}}
\subsection{The automorphism group of the Barth sextic}\label{subsec: Aut(Barth)}

We know that the set $N$ of the $65$ nodes is invariant under a group of linear collineations isomorphic to the icosahedral group $\mathfrak{A}_5 \times \mathbb{Z}/2\mathbb{Z}$.

The aim of this subsection is to given a non computational proof of Theorem \ref{barthauto}, that 
 the icosahedral group $\mathfrak{A}_5 \times \mathbb{Z}/2\mathbb{Z}$ is the full group of automorphisms of the Barth sextic.

The action of $\mathfrak{A}_5$ partitions $N$ into three orbits $N_A$, $N_B$ and $N_C$, which are the sets of nodes of type A, B and C of size $15$, $30$ and $20$, respectively.

We call a line $L$ a  \emph{node-secant} if it contains at least $2$ nodes.
By abuse of notation we shall say that $L$ is a secant.
If $L$ contains exactly $k$ nodes, $L$ is called a $k$-secant.

\subsubsection{Secants through $N_A$}\label{subsubsec: secants in N_A}
The nodes of type A span the plane $(w=0) =: E$.
Throughout \ref{subsubsec: secants in N_A}, we shall only consider the nodes and secants inside $E$. As $w = 0$ for all points in $E$, the $w$-coordinate is skipped in computations.

\begin{lemma}
	\label{lem:secants_A}
	Let $P\in N_A$.
	\begin{enumerate}[label=(\alph{enumi})]
		\item\label{lem:secants_A:secants} The node $P$ is contained in precisely two $2$-secants, two $3$-secants and two $5$-secants.
		\item\label{lem:secants_A:triads} Let $Q, Q'\in N_A$ be the second nodes on the two $2$-secants through $P$.
		Then $\overline{QQ'}$ is a $2$-secant.
	\end{enumerate}
\end{lemma}

\begin{proof}
	As the nodes of type $A$ form a single orbit under the $\mathfrak{A}_5$-action, it is enough to consider the single node $P = (1,0,0)$.
	For part~\ref{lem:secants_A:secants}, we compute the list of secants through $P$ shown in Table~\ref{tbl:secants_A}.
	For part~\ref{lem:secants_A:triads}, we see that $\{Q,Q'\} = \{(0,1,0),(0,0,1)\}$, so $\overline{QQ'} = (x=0)$, which contains no other node of type A.
\end{proof}

\begin{table}
\caption{Secants through $P = (1,0,0) = A_{(23)(45)}$}\label{tbl:secants_A}
\[
	\begin{array}{llll}
	\multicolumn{2}{c}{\text{secant }L\text{ through }P} & \multicolumn{2}{c}{\text{further nodes on }L} \\
	\hline
	\text{2-secant} & (z=0) & A_{(25)(34)} = (0,1,0) \\
	\text{2-secant} & (y=0) & A_{(24)(35)} = (0,0,1) \\
	\text{3-secant} & (y+\tau^2 z=0) & A_{(31)(52)} = (1,-\tau,-\bar{\tau}) & A_{(51)(34)} = (1,\tau,\bar{\tau}) \\
	\text{3-secant} & (y-\tau^2 z=0) & A_{(12)(34)} = (1,\tau,-\bar{\tau}) & A_{(41)(52)} = (-1,\tau,-\bar{\tau}) \\
	\text{5-secant} & (\tau y-z=0) & A_{(12)(53)} = (-\bar{\tau},1,\tau) & A_{(41)(53)} = (\bar{\tau},1,\tau) \\ & &  A_{(12)(45)} = (\tau,-\bar{\tau},1) & A_{(41)(32)} = (-\tau,-\bar{\tau},1) \\
	\text{5-secant} & (\tau y+z=0) & A_{(31)(24)} = (-\bar{\tau},-1,\tau) & A_{(51)(42)} = (\bar{\tau},1,-\tau) \\ & & A_{(31)(45)} = (\tau,\bar{\tau},1) & A_{(51)(23)} = (\tau,-\bar{\tau},-1) \\
	\end{array}
\]
\end{table}

In the next steps, we will avoid doing direct coordinate computations.
Instead, the conclusions will solely be drawn from the incidence properties of the secants provided by Lemma~\ref{lem:secants_A}.
The exception will be the final Lemma~\ref{lem:G_A_cong_A_5}, where the incidence properties are not enough to show the statement, see Remark~\ref{rem:proof_of_G_A_cong_A_5}.

\begin{lemma}\label{lem: counting secants}
	The secants of $ N_A$ fall into fifteen $2$-secants, ten $3$-secants and six $5$-secants.
\end{lemma}

\begin{proof}
	Use Lemma~\ref{lem:secants_A} and count the flags $(P,L)$ with $P\in N_A$, $L$ a secant and $P\in L$ in two ways.
\end{proof}

By Lemma~\ref{lem:secants_A}\ref{lem:secants_A:triads}, the $2$-secants partition $N_A$ into five parts of size $3$, which we will call \emph{triads}.
Let $\mathcal{T}$ be the set of the five triads.
For $P\in N_A$, let $T_P$ be the triad containing $P$.

\begin{remark}\label{rem:triads_labels}
	In the proof of Lemma~\ref{lem:secants_A}\ref{lem:secants_A:triads}, we have seen that the three nodes $A_{(23)(45)}$, $A_{(25)(34)}$ and $A_{(24)(53)}$ form a triad, which are exactly the nodes of type A where the number $1$ does not appear in the label.
	Since by construction, the $\mathfrak{A}_5$-action is compatible with the labelling, we get that the five triads $T_i$ with $\{1,\ldots,5\}$ are given by the five nodes of type A where the number $i$ does not appear in the label.
\end{remark}

\begin{lemma}
	\label{lem:secants_A_triads}
	\begin{enumerate}[label=(\alph{enumi})]
		\item\label{lem:secants_A_triads:5} Each $5$-secant contains exactly one node of every triad.
		\item\label{lem:secants_A_triads:3} For each $P\in N_A$, the two $3$-secants passing through $P$ together contain exactly one node from each triad different from $T_P$.
	\end{enumerate}
\end{lemma}

\begin{proof}
	Any pair of nodes in the same triad spans a $2$-secant, so no $3$- or $5$-secant can contain two nodes from the same triad.
	This implies part~\ref{lem:secants_A_triads:5}.

	For part~\ref{lem:secants_A_triads:3}, let $P\in N_A$ and let $T$ be one of the four triads different from $T_P$.
	For each $Q\in T$, the line $\overline{PQ}$ is a $3$-secant or a $5$-secant.
	By part~\ref{lem:secants_A_triads:5} and since there are two $5$-secants passing through $P$, for two of the three nodes $Q\in T$ the line $\overline{PQ}$ is a $5$-secant.
	Therefore, for the third node $Q\in T$ the line $\overline{PQ}$ is a $3$-secant.
\end{proof}

\begin{lemma}
	\label{lma:triad_triple}
	Let $T_1, T_2, T_3$ be three different triads.
	Then there is a unique $3$-secant containing a node from each triad $T_1, T_2, T_3$.
\end{lemma}

\begin{proof}
	We denote the other two triads by $T_4$ and $T_5$.
	Assume that there are two $3$-secants $L_1, L_2 $ containing a node from each triad $T_1, T_2, T_3$.
	By Lemma~\ref{lem:secants_A_triads}\ref{lem:secants_A_triads:3}, $L_1 \cap L_2$ is not a node, so $L_1$ and $L_2$ together cover $6$ distinct nodes.
	Denote them by $P_{ij}$ with $P_{ij}\in L_j $, $P_{ij}\in T_i$ where $i\in\{1,2,3\}$ and $j\in\{1,2\}$.
	Each node $P_{ij}$ is contained in a second $3$-secant $L_{ij}$.
	By Lemma~\ref{lem:secants_A_triads}\ref{lem:secants_A_triads:3}, the other two nodes on $L_{ij}$ are contained in $T_4$ and $T_5$.
	In particular, the six $3$-secants $L_{ij}$ are pairwise distinct.

	Assume that there is a node $Q\in T_4\cup T_5$ contained in two $3$-secants $L_{ij}\neq L_{i'j'}$ with $i,i'\in\{1,2,3\}$ and $j,j'\in\{1,2\}$.
	Then $L_{ij}$ and $L_{i'j'}$ are two distinct $3$-secants through $Q$ both containing a node of $T_4$ and $T_5$.
	This not possible by Lemma~\ref{lem:secants_A_triads}\ref{lem:secants_A_triads:3}, so the six $3$-secants $L_{ij}$ cover $12$ distinct nodes in $T_4\cup T_5$.
	This is a contradiction to $\#(T_4\cup T_5) = 6$.

	So there is at most one $3$-secant containing a node of $T_1$, $T_2$ and $T_3$.
	Now the claim follows as the number of $3$-secants equals the number $\binom{5}{3} = 10$ of $3$-sets of triads.
\end{proof}

We will call a $2$-set of two disjoint $2$-sets of triads a \emph{tetrad} and we will denote the set of tetrads by $X$.
Note that $\#X = \binom{5}{2} \binom{3}{2} / 2 = 15$.
A node $P\in N_A$ defines the set $\{L_1,L_2\}$ of the two $3$-secants passing through $P$.
For $i\in\{1,2\}$, the $3$-secant $L_i$ determines the set $\{T_{i1}, T_{i2}\}$ of the two triads different from $T_P$ which contain a point on $L_i$.
By Lemma~\ref{lem:secants_A_triads}\ref{lem:secants_A_triads:3}, the set
\[
	t(P) := \{\{T_{11},T_{12}\},\{T_{21},T_{22}\}\}
\]
is a tetrad.

\begin{example}
	\label{ex:t_of_100}
	Let $\mathcal{T} = \{T_1,T_2,T_3,T_4,T_5\}$ with $T_i$ defined as in Remark~\ref{rem:triads_labels} and let $P = A_{(23)(45)} = (1,0,0)$.
	From Table~\ref{tbl:secants_A}, the two $3$-secants passing through $P$ are $L_1 = (y+\tau^2z=0)$ containing the nodes $(1,-\tau,-\bar{\tau}) = A_{(31)(52)} \in T_4$ and $(1,\tau,\bar{\tau}) = A_{(51)(34)}\in T_2$ and $L_2 = (y-\tau^2z=0)$ containing the nodes $(1,\tau,-\bar{\tau}) = A_{(12)(34)} \in T_5$ and $(-1,\tau,-\bar{\tau}) = A_{(41)(52)} \in T_3$. 
	Therefore, $t(A_{(23)(45)}) = \{\{T_4,T_2\},\{T_5,T_3\}\}$.

	This example shows that to our surprise, the map $t$ does not match the labelling of the nodes of type A introduced in the construction.
\end{example}

\begin{lemma}
	\label{lem:t_bijection}
	The map $t : N_P \to X$ is a bijection.
\end{lemma}

\begin{proof}
	By $\#X = 15 = \#N_A$, it suffices to show that $t$ is injective.
	Let $P\in N_A$ and $t(P) = \{\{T_1,T_2\}, \{T_3,T_4\}\}\in X$.
	We denote the fifth triad by $T_5$.
	By Lemma~\ref{lma:triad_triple}, there is a unique $3$-secant $L$ passing through nodes on $T_1$, $T_2$ and $T_5$, so $P$ is determined as the node $L \cap T_5$.
\end{proof}

Now let $G_A$ be the stabilizer of the point set $N_A$ in $\Aut(E)$.
The map $t$ induces a map $t_G : G_A \to \mathfrak{S}(X)$.

\begin{lemma}
	\label{lem:t_G_injective}
	The map $t_G : G_A \to \mathfrak{S}(X)$ is an injective homomorphism of groups.
\end{lemma}

\begin{proof}
	As $t$ is defined based on incidence properties of lines spanned by the nodes in $N_A$, the map $t_G$ is a homomorphism of groups.
	Let $g\in G_A$ and let $\varphi = t_G(g)$.
	For $P\in N_A$, $g(P)$ is uniquely determined as $t^{-1}(\varphi(t(P)))$ by Lemma~\ref{lem:t_bijection}.
	Therefore $g$ is uniquely determined by $\varphi = t_G(g)$ and thus $t_G$ is an injection.
\end{proof}

\begin{remark}
	We can use Lemma~\ref{lem:t_G_injective} to efficiently complete Example~\ref{ex:t_of_100} to a list of all $t$-values of the nodes of type A.
	From Example~\ref{ex:t_of_100} we know that
	\[
		t(1,0,0) = t(A_{(23)(45)}) = \{\{T_2,T_4\},\{T_3,T_5\}\}\text{.}
	\]
	Since $t$ is a group homomorphism and the $\mathfrak{A}_5$-action is compatible with the labelling of the nodes of type A, for $\sigma = (234)\in\mathfrak{A}_5$ we get that
	\begin{multline*}
		t(0,1,0) = t(A_{(34)(25)}) = t(\sigma \cdot A_{(23)(45)}) = \sigma(t(A_{(23)(45)})) \\
		= \sigma(\{\{T_2,T_4\},\{T_3,T_5\}\}) = \{\{T_2,T_3\},\{T_4,T_5\}\}\text{.}
	\end{multline*}
	Using other elements $\sigma\in\mathfrak{A}_5$, we get the list in Table~\ref{tbl:t_values}.
	For completeness, we also included the list of the triads of the nodes.
\end{remark}

\begin{table}
	\caption{Triads and tetrads of the nodes of type A}\label{tbl:t_values}
	\[
		\begin{array}{lll}
			P\in N_A & T_P &  t(P) \\
			\hline
			A_{(12)(34)} = (1,\tau,-\bar\tau)   & T_5 & \{\{T_1,T_3\},\{T_2,T_4\}\} \\
			A_{(41)(25)} = (-1,\tau,-\bar\tau)  & T_3 & \{\{T_1,T_5\},\{T_2,T_4\}\} \\
			A_{(31)(52)} = (1,-\tau,-\bar\tau)  & T_4 & \{\{T_1,T_2\},\{T_3,T_5\}\} \\
			A_{(51)(34)} = (1,\tau,\bar\tau)    & T_2 & \{\{T_1,T_4\},\{T_3,T_5\}\} \\
			A_{(12)(53)} = (-\bar\tau, 1, \tau) & T_4 & \{\{T_1,T_5\},\{T_2,T_3\}\} \\
			A_{(41)(53)} = (\bar\tau,1,\tau)    & T_2 & \{\{T_1,T_3\},\{T_4,T_5\}\} \\
			A_{(31)(24)} = (-\bar\tau,-1,\tau)  & T_5 & \{\{T_1,T_4\},\{T_2,T_3\}\} \\
			A_{(51)(42)} = (-\bar\tau,1,-\tau)  & T_3 & \{\{T_1,T_2\},\{T_4,T_5\}\} \\
			A_{(12)(45)} = (\tau,-\bar\tau,1)   & T_3 & \{\{T_1,T_4\},\{T_2,T_5\}\} \\
			A_{(41)(32)} = (-\tau,-\bar\tau,1)  & T_5 & \{\{T_1,T_2\},\{T_3,T_4\}\} \\
			A_{(31)(45)} = (\tau,\bar\tau,1)    & T_2 & \{\{T_1,T_5\},\{T_3,T_4\}\} \\
			A_{(51)(23)} = (\tau,-\bar\tau,-1)  & T_4 & \{\{T_1,T_3\},\{T_2,T_5\}\} \\
			A_{(23)(45)} = (1,0,0)              & T_1 & \{\{T_2,T_4\},\{T_3,T_5\}\} \\
			A_{(25)(34)} = (0,1,0)              & T_1 & \{\{T_2,T_3\},\{T_4,T_5\}\} \\
			A_{(24)(35)} = (0,0,1)              & T_1 & \{\{T_2,T_5\},\{T_3,T_4\}\}
		\end{array}
	\]
\end{table}

\begin{lemma}
	\label{lem:G_A_cong_A_5}
	$G_A \cong \mathfrak{A}_5$.
\end{lemma}

\begin{proof}
	Let $\mathcal{T} = \{T_1,T_2,T_3,T_4,T_5\}$ with $T_i$ defined as in Remark~\ref{rem:triads_labels}.
	Assume that there is a $\rho\in G_A$ such that $t_G(\rho)$ is the transposition $T_3 \leftrightarrow T_4$.
	From Table~\ref{tbl:t_values} we get that
	\[
		\rho(1,0,0) = (0,1,0)\text{,}\quad
		\rho(0,1,0) = (1,0,0)\quad\text{and}\quad
		\rho(0,0,1) = (0,0,1)\text{.}
	\]
	Therefore, $\rho$ can be represented by a matrix
	\[
		S = \begin{pmatrix}
			0 & \alpha & 0 \\
			\beta & 0 & 0 \\
			0 & 0 & 1
		\end{pmatrix}
		\quad\text{with}\quad\alpha,\beta \neq 0\text{.}
	\]
	Again Table~\ref{tbl:t_values} shows that the node $(\tau,\bar\tau,1)$ is fixed by $\rho$, and that the node $(1,\tau,\bar\tau)$ is mapped to $(\bar\tau,1,\tau)$.
	The first condition yields $\alpha = -(\tau + 1)$ and $\beta = -(\bar\tau + 1)$.
	Now one checks that $S$ does not map $(1,\tau,\bar\tau)$ to a scalar multiple of $(\bar\tau,1,\tau)$, which is a contradiction to the second condition.

	So using Lemma~\ref{lem:t_G_injective}, $G_A \cong t(G_A)$ is a proper subgroup of $\mathfrak{S}(\mathcal{T}) \cong \mathfrak{S}_5$.
	By the known $\mathfrak{A}_5$-invariance of $N_A$, the only remaining possibility is $G_A\cong\mathfrak{A}_5$.
\end{proof}

\begin{remark}\label{rem:proof_of_G_A_cong_A_5}
	It can be checked that the set of permutations of $N_A$ preserving the incidences among all secants is indeed isomorphic to the full symmetric group $\mathfrak{S}_5$.
	For that reason, a proof of Lemma~\ref{lem:G_A_cong_A_5} must involve some additional property, like in the above proof the concrete coordinate representation of the nodes of type A.
\end{remark}

\subsubsection{The automorphisms of the Barth sextic}

Let $G$ be the stabilizer of $N$ in $\PGL(\CC^4)$.

\begin{lemma}\label{lem: 5-sec in N_A}
	The $5$-secants in Lemma~\ref{lem: counting secants} are the only secants that pass through more than four nodes. In particular, $G$ preserves the nodes of type $A$.
\end{lemma}
\begin{proof}
	Since nodes of type B are in the conic $w^2=x^2+y^2+z^2$, no line contains more than two of them. The same is true for nodes of type C since they are contained in $3w^2 = x^2 + y^2 + z^2$. Thus, if a 5-secant $L$ is not contained in $E$, then it contains precisely one node of type A, two nodes of type B, and two nodes of type C. Using $\mathfrak A_5$-symmetry, we may assume that this node is $A_{(23)(45)}=(0,1,0,0)$. Then, $L$ can be parametrized by $t \mapsto (1,t,y_0,z_0)$, but it can be directly checked from Table~\ref{table: Nodes in Barth} that nodes of type B and C are distinguishable by looking at $yz$-coordinates. Thus, $L$ cannot join nodes of type B and of type C. Now, it is obvious that a $k$-secant ($k \geq 5$) contains at least two nodes of type $A$, hence $k = 5$.
	
	The last statement follows from the fact that the $5$-secants in $E$ covers $N_A$.
\end{proof}

\begin{lemma}\label{lem:centre}
	The point $O:= (1,0,0,0)$ is fixed by all elements of $G$.
\end{lemma}

\begin{proof}
	By Lemma~\ref{lem: 5-sec in N_A}, $G$ preserves the plane $E$. 
	Let $T$ be the union $\bigcup_{P \in N_A} \overline{OP}$. Then, for any element $\phi \in G$, $\phi T = \bigcup_{P \in N_A} \overline {(\phi O) P}$. We recall that the lines $\{\overline{OP}\}$ are the intersections of two planes among the six planes $(\tau^2x -y^2)(\tau^2y^2-z^2)(\tau^2z^2-x^2)$, hence are $3$-secants.
	
	We claim that if $Q \not\in N$ and $\overline {QP}$ is a 3-secant for every $P \in N_A$, then $Q = O$. To do this, first we consider $P = A_{(23)(45)} = (0,1,0,0)$. Consulting Table~\ref{table: Nodes in Barth}, we find that the $3$-secants containing $P$ but not contained in $E$ are determined by the pairs
	\[
	\begin{array}{c}
		(2,\,\pm 1,\, \epsilon_1 \tau ,\, \epsilon_2 \bar\tau),\quad (2,\,\pm \tau,\, \epsilon_1 \bar\tau ,\, \epsilon_2 ),\quad (2,\,\pm\bar\tau,\, \epsilon_1,\, \epsilon_2 \tau), \\
		(1,\, \pm 1 ,\, \epsilon_1,\, \epsilon_2),\quad (1,\, \pm1 ,\, 0 ,\, 0),\quad (1,\,\pm \tau,\,  0,\, \epsilon_2\bar\tau),\quad (1,\, \pm \bar\tau,\, \epsilon_1\tau,\, 0),
	\end{array}
	\]
	where $\epsilon_1,\epsilon_2 \in \{ +1, -1\}$ are fixed signs. The $3$-secants through $(0,0,1,0)$ and $(0,0,0,1)$ are obtained by cyclic rotation of $xyz$-coordinates.
	
	Now, put $Q = (1,x_0,y_0,z_0)$. The line $\overline {Q P}$ is parametrized by $t \mapsto (1,t,y_0,z_0)$, and is one of the $3$-secants listed above. Hence, up to $\pm$ signs, the pair $(y_0,z_0)$ is one of 
	\begin{equation}\label{eq: common 3-secant candidates}
		\tfrac 12(\tau,\bar\tau),\  \tfrac 12 (\bar\tau, 1),\ \tfrac12 (1,\tau),\ (1,1),\ (0,0),\ (0,\bar\tau),\ (\tau,0).
	\end{equation}
	Due to cyclic $xyz$-symmetry, the pairs $(z_0,x_0)$ and $(x_0,y_0)$ (up to $\pm$ signs) also belong to \ref{eq: common 3-secant candidates}. The possibilities for $(x_0,y_0,z_0)$ are reduced to $(0,0,0)$, $(\pm1 ,\pm1, \pm1)$, $\frac 12 (\pm \tau, \pm\bar\tau, \pm1)$\ ($+$cyclic rotations). Among these, only $Q=O$ satisfies $Q \not\in N$.
\end{proof}

By Lemma~\ref{lem:centre} we will call the point $O$ the \emph{centre} of the Barth sextic.

\begin{lemma}\label{lem:secants_through_O}
	The secants through the centre $O$ consist of fifteen $3$-secants of type pattern ABB%
	\footnote{Type pattern ABB means that such a $3$-secant passes through a single node of type A and two nodes of type B.}
	and ten $2$-secants of type pattern CC.
\end{lemma}

\begin{proof}	
	It can be immediately seen that the line joining $O$ and $A_{(ij)(kl)}$ contains $B_{ij(kl)}$ and $B_{ji(lk)}$. A line joining $O$ and $C_{ij}$ also contains $C_{ji}$. In both cases, the line cannot contain further nodes as there is no $5$-secants of type pattern ABBCC\,(cf. Lemma~\ref{lem: 5-sec in N_A} and its proof).
\end{proof}

\begin{remark}
	Let $U \subset \PP^3$ be the union of non-secant lines which intersect $N$. In the proof of Lemma~\ref{lem:secants_through_O}, we observed that $O \not\in U$. In fact, a computer-based approach\footnote{see the script \cite[MAGMA 22]{Scripts}.} verifies that $U = \PP^3 \setminus\{O\}$. To verify, we compute the intersections of all the pairs of secants, and for each intersection point, find the subset of $N$ that are covered by secants passing though this point. The centre is the unique point through which the secants cover the entire $N$.

	This is another interesting property that characterizes the centre.
\end{remark}

\begin{lemma}\label{lem:G_preserves_type}
	The elements of $G$ preserve the type of the nodes.
\end{lemma}

\begin{proof}
	By Lemma~\ref{lem: 5-sec in N_A}, $G$ preserves the nodes of type A.
	Also, by Lemmas~\ref{lem:centre} and \ref{lem:secants_through_O}, $G$ preserves the set $\{\overline{C_{ij}C_{ji}}\}$. This shows that $G$ preserves the nodes of type C, and then also the nodes of type B.
\end{proof}

\begin{prop}
	The group of linear automorphisms of the set of nodes of the Barth sextic is isomorphic to the icosahedral group $\mathfrak{A}_5 \times \mathbb{Z}/2\mathbb{Z}$.
\end{prop}

\begin{proof}
	By Lemma~\ref{lem:G_preserves_type}, $N_A$ (and hence the plane $E$ spanned by $N_A$) is invariant under the elements of $G$.
	Therefore, the group homomorphism $\Phi : G \to G_E$, $\phi \mapsto \phi|_E$ is well-defined.
	By Lemma~\ref{lem:G_A_cong_A_5} we have $\im\Phi \leq G_E \cong \mathfrak{A}_5$, so $\#\im\Phi \leq 60$.

	Let $\phi\in \ker\Phi$.
	Then $\phi$ fixes all the points in $E$.
	Moreover, the centre $O\notin E$ is fixed by Lemma~\ref{lem:centre}.
	Now as an element of $\PGL$, the collineation $\phi$ is determined by the image of a single further point $P\notin (E\cup\{O\})$.
	We choose the node $P = (1,1,0,0)$ of type B.
	The line $L = \overline{OP}$ is a secant of type pattern ABB (see Lemma~\ref{lem:secants_through_O}).
	As the single node $Q$ of type A is fixed by $\phi\in\ker\Phi$, we get that $\phi(L) = \phi(\overline{OQ}) = \overline{OQ} = L$.
	Therefore, the node $P$ of type B is mapped to one of the two nodes of type B on $L$, implying that $\#\ker\Phi \leq 2$.

	Now by the fundamental theorem on homomorphisms,
	\[
		\#G = \#\ker\Phi \cdot \#\im \Phi \leq 2 \cdot 60 = 120\text{.}
	\]
	This finishes the proof as we already know that $G$ contains a subgroup isomorphic to the icosahedral group of order $120$.
\end{proof}

\subsection{The Doro-Hall graph}

\subsubsection{Petersen graph}
As we already mentioned,  the Petersen graph is constructed by the vertices and edges of the projective dodecahedron, that is, a dodecahedron where opposite points, lines and faces are identified with each other.

As we have also seen, the  Petersen graph can  also be constructed as the particular  Kneser graph with $n=5$ and $k=2$.

Recall that the Kneser graph with indices $n,k$ is formed from a 
 set $X$ with $n$ elements, taking as  vertex set $V$  the set of subsets of cardinality $k$ of $X$,
 and joining   two vertices  with an edge   if and only if the corresponding sets are disjoint.
 We denote by $E$ the set of edges.
 
 \bigskip
 
The number of vertices of the Petersen graph is then $10$ and the number of edges is $15$.

The $45$ (unordered) pairs of vertices fall into $15$ at distance $1$ (edges) and $30$ at distance $2$, 
 in particular the Petersen graph has diameter $2$.

The Petersen graph is the unique strongly regular graph with the parameters $v = 10$, $k = 3$, $\lambda = 0$ and $\mu = 1$.
This means that  the Petersen graph has $v = 10$ vertices, each vertex has $k = 3$ neighbours.
For every neighbouring pair of vertices, there are $\lambda = 0$ common neighbours, and for any non-neighbouring pair or vertices, there is $\mu = 1$ common neighbour.

Moreover, the Petersen graph has diameter $2$, and since its  automorphism group is $\mathfrak S_5$,
one sees easily that it is distance-transitive (that is, it acts transitively on the ordered pairs of vertices having a fixed distance).

The girth (the minimal length of a cycle) of the Petersen graph is $5$.
Each edge is contained in four $5$-cycles, and the total number of $5$-cycles is $12$.

The independence number (maximum number of a set of vertices not connected by an edge) of the Petersen graph is $4$.
It has $5$ maximal independent sets, which in the Kneser representation are given by the sets of $4$ unordered pairs containing a common element $x\in X$.
Let us denote by $I_x$ the independent set  corresponding to  $x\in X$; hence   the set of all $5$ independent sets  $\mathcal{I}$
is in bijection with $X$.
 The stabilizer of any maximal independent set is therefore  $\mathfrak S_4$ (the stabilizer of $x$ in $\mathfrak S_5$).
 
Each vertex $a = \{x,y\}$ is contained in two maximal independent sets, namely $I_x$ and $I_y$.
In a maximal independent set, any pair of vertices is at distance $2$.
Each pair of vertices ${a,b}$ of the Petersen graph with $d(a,b) = 2$ is contained in a unique maximal independent set.
\footnote{$d(a,b) = 2$ implies $\#(a\cap b) = 1$, so $a\cap b = \{x\}$ with $x\in X$ and the unique maximal independent set containing $\{a,b\}$ is $I_x$.}

\subsubsection{Doro-Hall graph}
In \cite{hall}, the connected local Petersen graphs (that is, the graphs such that, for any vertex, the  subgraph having as vertex set the set of its neighbours is a Petersen graph) have been classified.

There are only three such graphs, having respectively  $21$, $63$ and $65$ vertices.
We shall refer to the largest of these graphs  as the \emph{Doro-Hall graph} (it is often simply called the Hall-graph).

For an overview on the Doro-Hall graph, see \cite[Prop.~12.2.2]{brouwer-cohen-neumaier} and the subsequent paragraph.

In \cite[Thm.~1.2]{hall}, it has been shown that there is a unique distance-regular graph on $65$ vertices with intersection array $\{10,6,4;1,2,5\}$.
In \cite{hall}, the original discovery of this graph is attributed to a manuscript ''Two new distance-transitive graphs'' by Stephen Doro, which apparently has never been published.

We  call it the \emph{Doro-Hall graph} $\mathcal{\Gamma}$ to acknowledge credit to the original discoverer.
The graph has been rediscovered independently in \cite{gordon-levinston}.

By \cite[Prop.~6]{gordon-levinston}, combined with the uniqueness result of  \cite[Thm.~1.2]{hall}, the automorphism group of $\Gamma$ is isomorphic to $\PSigmaL(2,25)$.
The graph $\Gamma$ is distance-transitive, where already the normal index-$2$ subgroup $\PSL(2,25)$ is acting transitively on all 
ordered pairs  of vertices of fixed distance.

Moreover, $\Gamma$ is a commuting involutions graph, meaning that, for any $x\in V$, there is a unique nontrivial automorphism $\tau_x$ (which necessarily is an involution) of $\Gamma$ pointwise stabilizing $x$ and all neighbours of $x$, and moreover any two distinct vertices of $x,y\in V$ are neighbours if and only if $\tau_x$ and $\tau_y$ commute.

The following constructions of $\Gamma$ are known.
\begin{itemize}
	\item
	Let $F : x \mapsto x^5$ be the Frobenius automorphism of $\F_{25}$.
	
	Then $G = \PSigmaL(2,25) = \PSL(2,25) \rtimes \langle F\rangle$ is a group of order $15600 = 2^4 \cdot 3 \cdot 5^2 \cdot 13$.
	
	The set $V$ of vertices of the Hall graph is the conjugacy class $F^G$ of $F$ in $G$.
	It is of size $65$.
	The set $E$ of edges is given by the $325$ unordered pairs of distinct commuting elements in $F^G$. \cite{hall}
	\item
	Let $q \in \F_5[x_1,x_2,x_3,x_4]$ be a non-degenerate elliptic quadratic form%
	\footnote{We may take $q(x) = x_1 x_2 + x_3^2 + 2 x_4^2$.}
	and $Q = \{\langle x\rangle\in \PP^3_{\FF_5}  \mid q(x) = 0\}$ the corresponding quadric.
	
	The vertex set $V$ is defined as the set of all points $\langle x \rangle\in \PP^3_{\FF_5}$ with $x\in\F_5^4$ such that $q(x) \in \{\pm 1\}$.
	
	The edge set $E$ is defined as the set of pairs $\{x,y\}\subset V$ such that the line $L = x * y$ is a secant of $Q$, i.e. $L$ intersects $Q$ in $2$ distinct points (that is, these points are in  $Q \subset \PP^3_{\FF_5}$). \cite[Sec.~12.2]{brouwer-cohen-neumaier}
	\item 
	An inversive plane of order $5$ is a combinatorial $3$-$(26,6,1)$ design.
	The set of $130$ blocks is invariant under the action of $\PP SL(2,25)$, partitioning the design in two orbits of size $65$.
	Now let $V$ be one of the orbits and define two blocks in $V$ to be neighbours if they are disjoint. \cite{gordon-levinston}
\end{itemize}

For each vertex, the distance distribution of the other $64$ vertices is $(1^{10} 2^{30} 3^{24})$.

\subsubsection{From the Doro-Hall graph to the code $\sK$ of the Barth sextic}
The code $\sK$ of the Barth sextic is indeed isomorphic to  the code associated to the Doro-Hall graph as in the following Theorem \ref{barth}:
for each vertex, the characteristic function of the set of $24$ vertices at distance $3$ yields a codeword in $\F_2^{65}$ of Hamming weight $24$.
The linear hull of these $65$ codewords is the code $\sK$ consisting of the (strictly) even set of nodes of the Barth sextic.
By construction, this code is invariant under the action of the $\PSigmaL(2,25)$ in the Doro construction of the Doro-Hall graph.

Theorem \ref{barth} shows also how $\sK'$ is obtained from the Doro-Hall graph; moreover in Appendix C we give another
description of how  $\sK'$ can be derived from $\sK$ by coding theoretic means.
In particular, $\sK'$ is invariant under the action of $\PSigmaL(2,25)$ too, where the extra coordinate is always left constant.
In fact, $\PSigmaL(2,25)$ is the full automorphism group of $\sK$ as well as of $\sK'$.

\subsubsection{From the Barth sextic to the Doro-Hall graph}

Indeed, the Doro-Hall graph structure can also be seen from the geometry of the Barth sextic.

The following has been found computationally \footnote{see the script \cite[MAGMA 23]{Scripts}.}

By the $\mathfrak A_5$-action, the $65$ nodes of the Barth surface $Y_B$ are partitioned into $3$ orbits of respective lengths $15$ (Type A), $20$ (Type B) and $30$ (Type C).
For the line $L$ connecting an ordered pair of  nodes, $x$ and $y$, of the Barth sextic, one of the following possibilities occurs:
\begin{itemize}
	\item Type X: $L$ intersects $Y_B$ in two further points.
	\item Type Y: $L$ intersects $Y_B$ in a further point with multiplicity $2$.
	\item Type Z: $x$ has multiplicity $4$ and $y$ has multiplicity $2$ on $L$ (Type Z$^+$) or the other way round (Type Z$^-$).
	\item Type 0: $L$ is contained in $Y_B$.
\end{itemize}

The frequencies of the types of the lines passing through a node $x$ depend on the type of $x$.
We get the following table:
\[
	\begin{array}{cccccc}
		\text{Type of }x & \text{Type X} & \text{Type Y} & \text{Type Z$^+$} & \text{Type Z$^-$} & \text{Type 0} \\
		\hline
		\text{A} & 10 & 46 & 0 & 0 & 8 \\
		\text{B} & 10 & 42 & 6 & 6 & 0 \\
		\text{C} & 10 & 46 & 4 & 4 & 0
	\end{array}
\]

Let us call  two nodes neighbours if their connecting line is of type X: then we obtain a graph, which  is  indeed 
isomorphic to the  Doro-Hall graph \footnote{see the script \cite[Mathematica 10]{Scripts}.}.

We remark that there are exactly $6$  lines connecting a pair of nodes which are  completely contained in $Y_B$.
Every such line contains $5$ nodes, all of type~A.
Each node of type A is contained in two such lines.

We have the following synthetic description of the codes of the Barth sextic in terms of the Doro-Hall graph:

\begin{theo}\label{barth}
The projection $\sK''$ of the extended code  $\sK'$ of the Barth sextic is isomorphic to the code generated
by the characteristic functions of  independent sets of maximal cardinality of the Doro-Hall graph (a set  $\sN$ 
is independent if no edge joins 
two vertices in $\sN$).

The code $\sK$ is generated by the sums of pairs of  such characteristic functions (which have weight $15$):
these  sums  yield $325$ vectors of  weight $24$.

Another generating set for $\sK$ is given by the other vectors of weight $24$: these 
are the $65$ characteristic functions of the spheres of radius $3$
in the Doro-Hall graph $\sG$ (set of vertices at distance $3$ from a given vertex $v$).
\end{theo}
\begin{proof}

We know that the extended code $\sK'$ is generated by the codewords corresponding to the half-even sets
of cardinality $15$.

Observe that in a given half-even set of cardinality 15, the line connecting any two nodes is not an edge of $\sG$.

This is because, setting $x=0$ the equation of the corresponding plane,  the equation of $f$
writes as $ f = x \phi + g_3^2$, so that the plane section with $\{x=0\}$ is  the plane   cubic curve $\{x= g= 0\}$
counted with multiplicity $2$, and the $15$ nodes are the intersections of the cubic
curve with the quintic surface $\phi$; hence each such line intersects in three points.

So the corresponding vertices of $\sG$  form an independent set (a subset of vertices whose induced subgraph is edgeless)
of cardinality $15$. 

It can be checked that these sets are  independent sets of maximal  cardinality, and indeed every independent set of $15$ vertices 
corresponds to a half-even set of cardinality $15$, as shown by 
the  Mathematica built-in library about the (Doro-) Hall graph, including the independence polynomial \footnote{see the script \cite[Mathematica 10]{Scripts}.}:
$$1 + 65 x + 1755 x^2 + 25805 x^3 + 227890 x^4 + 1259583 x^5 +  4414930 x^6 +$$
$$+  9762935 x^7 +  13342485 x^8 + 10860980 x^9 +  4966663 x^{10 }+$$
$$+ 1164540 x^{11} + 120380 x^{12} + 4980 x^{13}
 + 390 x^{14} +  26 x^{15},$$

where the  coefficient of $x^k$  is the number of independent sets consisting of $k$  vertices. In particular, $15$ is the maximum cardinality
of an independent set, and there are only $26$ of them, namely, the $26$ half-even sets of cardinality $15$.
\footnote{The coefficient of $x^{14}$ coincides with the number of even sets of cardinality 24: pure coincidence?} 

We have seen that if $\sN_1, \sN_2$ are half-even sets of cardinality $15$, then their intersection
has cardinality $3$ (the line intersection of the two planes meets the surface in $3$  points), 
hence the symmetric difference $(\sN_1\cup \sN_2) \setminus ( \sN_1\cap \sN_2)$
is a strictly even set of cardinality $24$. These are $\frac{1}{2} 26 \cdot 25 = 325$ vectors,  because such sets of cardinality $24$
determine uniquely a couple of planes containing $15$ nodes (in fact,  a line cannot contain $6$ nodes of the sextic, otherwise each plane containing the line would cut the line with mutiplicity $2$, hence
the line would be contained in the singular set of $Y$);
these $325$ vectors  generate the code $\sK$.

Our knowledge of the weight enumerator for our code (see Appendix C)   shows that there are exactly $390$ vectors of weight $24$,
and it can be verified via computer calculations \footnote{see the script \cite[Mathematica 10]{Scripts}.} that the remaining $65$ are the characteristic functions of the spheres of radius $3$
and  that they generate the code $\sK$.   
Observe in fact that,  given a vertex $v$ of $\sG$, there are exactly $10$ vertices at distance
$1$ from $v$, $30$ at distance $2$, and $24$ at distance $3$ (so $3$ is the diameter of $\sG$).  
\end{proof}

\begin{rem}
I) The code $\sK$ of the Barth sextic is an irreducible representation of its symmetry group $\operatorname{P\Sigma L}(2,25)$.

 In fact, let $\sC$ be a subcode of $\sK$: if $\sC$ contains a vector of weight $24$, then either it contains one of the $65$ 
 characteristic functions of a sphere of radius $3$, or it contains one of the $325$  sums of pairs of characteristic functions of independent sets of maximal cardinality. In both cases $\sC$  would contain the whole orbit if $\sC$ is
 $\operatorname{P\Sigma L}(2,25)$-invariant, hence we would have $\sC = \sK$. 
 
 There remains to exclude that $\sC$ can only have the weights $32,40$.
 
  $\operatorname{P\Sigma L}(2,25)$-invariance implies that the effective length of $\sC$ is $65$, 
  hence by Bonisoli's theorem it cannot be a one weight code (because none of $65, 13, 5$ is a power of $2$ diminished by $1$).

Therefore $\sC$ would have to admit both weights $32,40$. If we use that the vectors in $\sK$ of weight $40$ form
a single orbit of cardinality $650$  (this is again verified by computer calculations \footnote{see the script \cite[MAGMA 22]{Scripts}.}), then the dimension of $\sC$ is at least $10$
and we can use the Griesmer bound to compute $n \geq 67$, a contradiction.

 Another possibility, using   Part II) of this remark, would be  to exclude that $\sC$ cannot be one of the $\mathfrak A_5$-submodules
$U_{43}$, $U_{44}$.

\bigskip

II) The inclusion $\sK \subset V_A \oplus V_B \oplus V_C$ produces six projections, and the  kernels are $\mathfrak A_5$-submodules.

MAGMA calculations \footnote{see the script \cite[MAGMA 22]{Scripts}.} show the following:
denoting by 
$$\sK_A = ker (  \sK \ra V_A)  = : U_{42}, \ \  \sK_C = ker (  \sK \ra V_C) = : U_{41},$$
 these are irreducible $\mathfrak A_5$-submodules
of dimension $4$,  while $\sK_B = 0$. 

These calculations also  show  that the only  nontrivial $\mathfrak A_5$-submodules are four $4$-dimensional irreducible $\mathfrak A_5$-submodules, 
 $U_{41} ,  U_{42},  U_{43},  U_{44}$, and four $8$-dimensional submodules,
 $$U_{81} = \langle U_{42},  U_{43},  U_{44}\rangle , U_{8j} =  \langle U_{41},  U_{4j}\rangle , \ \ j=2,3,4.$$
 The only inclusions of a $4$-dimensional submodule $U_{4i}$ into an $8$-dimensional submodule $U_{8j}$
 are those which are derived from the definitions of the $U_{8j}$'s.

$\sK$ is the direct sum $U_{41} \oplus   U_{42} \oplus   U_{43}$ of the first  three irreducible $\mathfrak A_5$-submodules.

Moreover the only submodules which do not admit the weight $24$ are exactly $U_{43} , U_{44}$. More precisely the weight enumerators of these submodules are

\begin{align*}
W_{U_{41}}(x) & = 1 + 15 x^{24} \\
W_{U_{42}}(x) & = 1 + 10 x^{24} + 5 x^{32} \\
W_{U_{43}}(x) & = 1 + 10 x^{32} + 5 x^{40} \\
W_{U_{44}}(x) & = 1 + 10 x^{32} + 5 x^{40} \\
W_{U_{81}}(x) & = 1 + 30 x^{24} + 175 x^{32} + 50 x^{40} \\
W_{U_{82}}(x) & = 1 + 40 x^{24} + 155 x^{32} + 60 x^{40} \\
W_{U_{83}}(x) & = 1 + 45 x^{24} + 145 x^{32} + 65 x^{40} \\
W_{U_{84}}(x) & = 1 + 45 x^{24} + 145 x^{32} + 65 x^{40}
\end{align*}

\end{rem}

\section{Determinantal equations of the Barth sextic.}
\label{append_Barth_determinant}

In this section we present an explicit description of Barth's sextic $Y \subset \PP^3$ as a determinant of the matrix $A$ whose diagonal degrees are $(1,1,1,3)$. In view of corollary \ref{CM} and lemma \ref{degrees}, the first step is to find a symmetric half-even set of cardinality 31. To do this, we consider the $\mathfrak A_5$-orbit of the linear form $\ell_1 := w-x-y-z$:
\[
	\begin{array}{l@{\qquad\qquad}l}
		\ell_1 = w-x-y-z & \ell_2 = w+x-y+z \\ \ell_3 = w-x+y+z & \ell_4 = w+x+y-z \\
		\ell_5 = w+x+y+z & \ell_6 = \tau w - \tau^2 x + z \\ \ell_7 = \bar\tau w - \bar\tau^2 y + z & \ell_8 = \bar\tau w - \bar\tau^2 x + y \\
		\ell_9 = w-x+y-z & \ell_{10} = \tau w + \tau^2 x - z \\ \ell_{11} = \bar\tau + \bar\tau^2 x + y & \ell_{12} = \bar\tau w -\bar\tau^2 y - z \\
		\ell_{13} = w+x-y-z & \ell_{14} = \bar\tau w + \bar\tau^2 y - z \\ \ell_{15} = \bar\tau w - \bar\tau^2 x - y & \ell_{16} = \tau w - \tau^2 x - z \\
		\ell_{17} = w-x-y+z & \ell_{18} = \bar\tau w + \bar\tau^2 x - y \\ \ell_{19} = \bar\tau w + \bar\tau^2 y + z & \ell_{20} = \tau w + \tau^2 x + z.
	\end{array}
\]
	In appendix \ref{ap_code_barth_sextic}, we have seen that the plane section $H_1 := (\ell_1=0) \cap Y$ is the double curve, say $2C_1$, passing through a half-even set of nodes $\mathcal N_1$ of cardinality 15. The ideal $I_{C_1}$ of $C_1$, regarded as an ideal of $\CC[w,x,y,z]$, is generated by $\ell_1$ and
	\[
		h_1 = \tau^2(x^2y+y^2z+z^2x) + (xy^2 + yz^2 + zx^2) + (\tau^2+1)xyz
	\]
	We may also consider $H_i$, $C_i$, $\mathcal N_i$, $h_i$ for $i=2,\ldots,20$. It can be directly checked that $\mathcal N := \mathcal N_1 + \mathcal N_9 + \mathcal N_{17}$ has cardinality $31$. By proposition \ref{discriminant}, $\mathcal N$ is symmetric. 
	
	Let $S \to Y$ be the minimal resolution, let $L_i \subset S$ be the proper transform of $C_i$, and let $E_{*}$ be the sum of $(-2)$-curves over a given set of nodes $* \subset Y$. We have
	\[
		2H - L \equiv L_1 + L_9 + L_{17} + E_{\mathcal M}
	\]
	where $\mathcal M$ is the set of nodes that appear more than once in $\{\mathcal N_1, \mathcal N_9, \mathcal N_{17}\}$. On the other hand, we have three additional expression of $\mathcal N$ as the sum of half even sets
	\begin{equation}
		\begin{array}{c}
			\{ \mathcal N_1,\, \mathcal N_2,\, \mathcal N_4,\, \mathcal N_{12},\, \mathcal N_{19} \}, \\
			\{ \mathcal N_9,\, \mathcal N_4,\, \mathcal N_5,\, \mathcal N_8,\, \mathcal N_{18} \}, \\
			\{ \mathcal N_{17},\, \mathcal N_2,\, \mathcal N_5,\, \mathcal N_6,\, \mathcal N_{10} \}. \\
		\end{array}
	\end{equation}
	For each expression, we associate a section in $H^0(3H-L)$. For instance,
	\[
		3H-L \equiv L_1 + L_2 + L_4 + L_{12} + L_{19} + E_{\mathcal M_1} + E_{\mathcal M_1'}
	\]
	where $\mathcal M_1$ is the set of nodes that appear 2--3 times and $\mathcal M_1'$ is the set of nodes that appear 2--5 times in $\{ \mathcal N_1,\, \mathcal N_2,\, \mathcal N_4,\, \mathcal N_{12},\, \mathcal N_{19} \}$. Since the right hand side is effective, it corresponds to a section $v_1 \in H^0(3H-L)$. Similarly, we may take $v_2$ (resp. $v_3$) corresponding to $\{ \mathcal N_9,\, \mathcal N_4,\, \mathcal N_5,\, \mathcal N_8,\, \mathcal N_{18} \}$ (resp. $\{ \mathcal N_{17},\, \mathcal N_2,\, \mathcal N_5,\, \mathcal N_6,\, \mathcal N_{10} \})$. Also, let $v_4 \in H^0(2H-L)$ be a nonzero section.
	
	We consider the double cover $f \colon Z \to S$ branched over $H + E_\mathcal N \equiv 2L$ and the graded algebra
	\[
		R := \bigoplus_{m \geq 0} H^0(S, f_*\mathcal O_Z(m)),
	\]
	which admits eigenspace decomposition into $R^+ := \bigoplus_{m \geq 0} H^0(S, \mathcal O_S(mH))$ and $R^- := \bigoplus_{m \geq 0} H^0(S, \mathcal O_S(mH-L))$. Then, $\{v_1,\ldots,v_4\}$ is a generating set of $R^-$, and by \cite[Section 2]{babbage}, $B_{ij} := v_iv_j$ form an adjoint matrix of $A(x)$. The divisor form of $B_{ij}$ can be explicitly carried out; for instance,
	\begin{align*}
		\Div B_{14} &= \Div v_1 + \Div v_4  \\
		&= (L_1 + L_9 + L_{17}) + (L_1 + L_2 + L_4 + L_{12} + L_{19}) + \sum_{\mathcal M} E_{\mathcal M}
	\end{align*}
	where the last term indicates a sum of ($-2$)-curves. It can be decomposed into $2L_1$ and $(L_9 + L_{17} + L_2 + L_4 + L_{12} + L_{19})$, which are respectively a hyperplane section and a cubic section. We may check \footnote{see the script \cite[Macaulay2 15]{Scripts}.} that the ideal $I_{C_9} \cap I_{C_{17}} \cap I_{C_2} \cap I_{C_4} \cap I_{C_{12}} \cap I_{C_{19}}$ contains a unique cubic equation
	\[
		w^2y + \tau^3 x^2y - y^3 - \tau^2 w^2z - \bar\tau x^2 z + \bar\tau y^2 z - \tau^3 yz^2 + \tau^2 z^3,
	\]
	showing that $B_{14}$ is the product of $\ell_1$ and this cubic (up to constant multiple). In similar way, we may determine other $B_{ij}$ as follows.
	\newlength{\tempindent}\setlength{\tempindent}{90pt}
	\begin{enumerate}
	\item $B_{i4}$, $i<4$: 
	\begin{align*}
		 B_{14} &= \lambda_{14} (w-x-y-z)(w^2y + \tau^3 x^2y - y^3 - \tau^2 w^2z \\
		 & \hspace{\tempindent} - \bar\tau x^2 z + \bar\tau y^2 z - \tau^3 yz^2 + \tau^2 z^3) \\
		 B_{24} &= \lambda_{24} (w-x+y-z)(w^2x - x^3 + \tau^2 w^2y - \bar\tau x^2 y \\
		& \hspace{\tempindent} - \tau^3 xy^2 - \tau^2 y^3 + \tau^3 xz^2 + \bar\tau yz^2) \\
		 B_{34} &= \lambda_{34} (w-x-y+z)(w^2x - x^3 + \bar\tau^3 xy^2 + \bar\tau^2 w^2z\\
		& \hspace{\tempindent} - \tau x^2 z + \tau y^2 z - \bar\tau^3 xz^2 - \bar\tau^2 z^3 ).
	\end{align*}
		\item $B_{ij}$, $i<j< 4$:
		\begin{align*}
			 B_{12} & = \lambda_{12} (w+x+y-z)( w^4 - w^2x^2 + \tau w^2xy - \tau x^3y \\
			& \hspace{\tempindent}  - 2w^2y^2 - \sqrt{5} x^2y^2 - \bar\tau^2 xy^3 + y^4  \\
			& \hspace{\tempindent} - w^2xz + \bar\tau^3 x^3z - \bar\tau w^2yz - \tau^2x^2yz  \\
			& \hspace{\tempindent} - 3xy^2z + \tau^2 y^3z - w^2z^2 + 2x^2z^2 \\
			& \hspace{\tempindent} + \bar\tau^2 xyz^2 + \sqrt{5} y^2z^2 + \tau^3 xz^3 + \bar\tau yz^3) \\
			 B_{13} &= \lambda_{13} (w+x-y+z)( w^4 - w^2 x^2 - w^2xy + \tau^3 x^3y \\
			& \hspace{\tempindent} - w^2y^2 + 2x^2y^2 + \bar\tau^3 xy^3 + \bar\tau w^2xz \\
			& \hspace{\tempindent} -\bar\tau x^3z - \tau w^2yz - \bar\tau^2 x^2yz \\
			& \hspace{\tempindent} + \tau^2 xy^2z + \tau y^3z - 2w^2z^2 + \sqrt{5} x^2 z^2 \\
			& \hspace{\tempindent} - 3xyz^2 - \sqrt{5} y^2z^2 - \tau^2 xz^3 + \bar\tau^2 yz^3 + z^4) \\
			 B_{23} &= \lambda_{23} (w+x+y+z)( w^4 - 2w^2x^2 + x^4 + \bar\tau w^2xy \\
			& \hspace{\tempindent} - \tau^2 x^3y - w^2y^2 + \sqrt{5} x^2y^2 - \bar\tau xy^3 \\
			& \hspace{\tempindent} + \tau w^2xz - \bar\tau^2 x^3z + w^2yz + 3x^2yz \\
			& \hspace{\tempindent} + \bar\tau^2 xy^2 z - \tau^3 y^3z - w^2z^2 - \sqrt{5} x^2z^2 \\
			& \hspace{\tempindent} + \tau^2 xyz^2 + 2y^2z^2 - \tau xz^3 - \bar\tau'^3 yz^3)
		\end{align*}
		\item $B_{ii}$:
		\begin{align*}
			 B_{11} &= \lambda_{11} \ell_1 \ell_2 \ell_4 \ell_{12} \ell_{19}\\
			 B_{22} &= \lambda_{22} \ell_9 \ell_4 \ell_5 \ell_8 \ell_{18} \\
			 B_{33} &= \lambda_{33} \ell_{17} \ell_2 \ell_5 \ell_6 \ell_{10} \\
			 B_{44} &= \lambda_{44} \ell_1\ell_9\ell_{17}.
		\end{align*}
	\end{enumerate}

	To determine $\lambda_{ij}$, we impose the condition $B_{ij}(P)=1$ for a fixed $P \in Y$. For instance, we may pick $P=[2:1:1:1] \in Y$.

	Now, it can be shown that $\det( B_{ij})= (\text{constant})\cdot f^3$, where $f$ is the Barth's sextic equation, and $\operatorname{adj}( B_{ij} ) = (\text{constant}) \cdot f^2  A$, where $A$ is the matrix
	\[
		A = \left( 
			\begin{array}{cccc}
				0 & a_{12} & a_{13} & q_1 \\
				a_{12} & 0 & a_{23} & q_2 \\
				a_{13} & a_{23} & 0 & q_3 \\
				q_1 & q_2 & q_3 & G
			\end{array}
		\right),
	\]
	with
	\begin{align*}
		a_{12}(w,x,y,z) &= 3\sqrt{5}(w-x-y+z) \\
		a_{13}(w,x,y,z) &= -3\sqrt{5}(w-x+y-z) \\
		a_{23}(w,x,y,z) &= 15(w-x-y-z) \\
		q_1(w,x,y,z) &= (-x+ \tau y + \bar\tau z)(w+x+y+z) \\
		q_2(w,x,y,z) &= \sqrt{5}(\tau x - \bar\tau y - z)(w+x-y+z) \\
		q_3(w,x,y,z) &= \sqrt{5}(-\bar\tau x + y + \tau z)(w+x+y-z) \\
		G(w,x,y,z) &= -\frac{2}{3} (w+x+y-z)(w+x-y+z)(w+x+y+z).
	\end{align*}

 A final observation is that one could try to construct surfaces with many nodes also in higher degree
 by taking discriminants of cubic hypersurfaces with sufficiently many nodes.
 
 However, while as we saw for degree $d= 4,5$ we get exactly nodal maximizing surfaces by choosing a cubic hypersurface with the
 maximal number of nodes ( $16= 10 + 6$, $31 = 15 + 16$), we have shown here that we cannot take for degree $d=6$
 such a cubic, since $ 35 + 31 = 66 > 65$.
 
 For $d=7$ we would get $ 56 + 52 = 108 > 104$, which is an upper bound for $\mu(7)$  (see for instance \cite{oliver2005hypersurfaces}),
 so we should take cubics with at most $48$ nodes.
 
  The case  $d=8$, since, as shown in  lemma \ref{Sing-L}, for a  cubic hypersurface in $\PP^8$ the $4$-dimensional linear subspaces
 $L \subset X$ pass through some singular points of $X$, shows the necessity of taking (as it happens already for sextics with 65 nodes)
 subspaces  passing  through several  nodes of $X$.
 
  \bigskip

\appendixOn\section{Codes of nodal sextics with many nodes}\appendixOff
\vskip+10pt\chapterauthor{Sascha Kurz}
\label{ap_codes_nodal_sextics}

The aim of this section is to determine the possible strict codes $\sK$ for nodal sextics in $\PP^3$ with $\nu$ nodes, when $\nu$ is \textit{large}, i.e., $\nu\in\{64,65\}$.  

Clearly for the effective length $n$ we have $n\le\nu$ and the B-inequality (remark~\ref{rem_low_degree}) gives 
 $k \geq \nu - 53$, i.e., $k \geq 12$ for $\nu=65$ and $k\ge 11$ for $\nu=64$. 
 
The set of possible non-zero weights for the 
strict code $\sK$ of a nodal sextic in $\PP^3$ is $\{24,32,40,56\}$, see \cite{cat-ton}.

\begin{prop}
  \label{prop_classification}
  If $\sC$ is a binary linear code with  effective length at most $65$, dimension $k\ge 12$, and non-zero weights in $\{24,32,40,56\}$, then $k=12$ and 
  $\sK$ is isomorphic to one of the following three cases:
  \begin{enumerate}
    \item[(1)]
    $\left(\begin{smallmatrix} 
    00110000111000000111110100001111111001001010010000110000000000000\\
    10100111111100000011011101000010011010001101100000001000000000000\\
    00010011101110001111011100100001000011000011011010000100000000000\\
    01000111111111001100110000100110010001000110100000100010000000000\\
    11000111000001011100111101100001110010001100001010000001000000000\\
    00000001100011011110001101001110111001000101111000000000100000000\\
    01001111000111110101000011010010001110111011111111100000010000000\\
    00100011011110110000111111000000000110011000011110000000001000000\\
    00011111000110001100000000110001111110000110000111100000000100000\\
    00000000111110000011111111110000000001111110000001100000000010000\\
    00000000000001111111111111110000000000000001111111100000000001000\\
    00000000000000000000000000001111111111111111111111100000000000100\\
    \end{smallmatrix}\right)$\\[1mm]
    $w_{\sC}(x,y)=x^{65}y^{0}+630x^{41}y^{24}+3087x^{33}y^{32}+378x^{25}y^{40}$\\
    \item[(2)]
    $\left(\begin{smallmatrix}
    10000110110100110000010101100110001010011000001000010000011001100\\
    00011001111101100000000001000000011000000011000111010011001110010\\
    00001010111101000010000101010010000100011110010000110000001100110\\
    00001111101000010000101000110010011100011000010000001001010101100\\
    00010011111110110110101100110010000101011101110000000110010110010\\
    00010010110110110000100101000000000001100110000100000000111111110\\
    01010101010110100100010100100011101001011101011010000000000000000\\
    00001100110011001000111101110101011010011110010000000000000000000\\
    00001100001111000111000001110100100101111110001110000000000000000\\
    00001111111111110110110010010011000110001000001000000000000000000\\
    00000011110011000110001100001110100110001111110110000000000000000\\
    00111111111111110110111101100000011001100111100110000000000000000\\
    \end{smallmatrix}\right)$\\[1mm]
    $w_{\sC}(x,y)=x^{65}y^{0}+502x^{41}y^{24}+3087x^{33}y^{32}+506x^{25}y^{40}$\\
    \item[(3)]
    $\left(\begin{smallmatrix}  
    10000100000000110110010001110100111101010001011110010100000000000\\
    10100100011000001001000110100110111111001000001100011010000000000\\
    01000010011100011000000100110100110000011111011110001001000000000\\
    11110100001110110100000011010110100001011100000100001000100000000\\
    01101011000001100011010001000011001010001111000010111000010000000\\
    00101001110111101011000001011000000110111001001000100000001000000\\
    00011000111111100000111110001000100010001010101001100000000100000\\
    00000111001011100111110001010100000001111001100000011000000010000\\
    00011111000111100000001111001101111110000111100111111000000001000\\
    00000000111111100000000000111100011110000000011111111000000000100\\
    00000000000000011111111111111100000001111111111111111000000000010\\
    00000000000000000000000000000011111111111111111111111000000000001\\
    \end{smallmatrix}\right)$\\[1mm]
    $w_{\sC}(x,y)=x^{65}y^{0}+390x^{41}y^{24}+3055x^{33}y^{32}+650x^{25}y^{40}$\\
  \end{enumerate}
\end{prop}
\begin{proof}
  We have applied the software package \texttt{LinCode}, see \cite{kurz2019lincode}, to exhaustively enumerate all codes with the specified properties. 
  The computations were performed on a linux cluster of the University of Bayreuth and took less than a CPU year in total. 
\end{proof}

The codes of proposition~\ref{prop_classification} are the only possible candidates for the strict code $\sK$ of a nodal sextic in $\PP^3$ with $\nu=65$ nodes, 
where the code in case (3) is isomorphic to the Barth code $\sK_B$. To 
further thin out this list we use known properties of the extended code $\sK'$, i.e., the minimum distance is at least $16$ and the weights are divisible 
by four, see e.g.\ \cite[Theorem 1.10]{endrass1}. Moreover, $\sK'\backslash \mathcal{K}$ does not contain codewords of weight $20$ or $24$, see 
\cite[Corollary 1.11]{endrass1}.

\begin{lemma}
  \label{lemma_ILP}
  Let $\sC$ be a linear code of effective length $n$ and $\sC'$ be spanned by the codewords of $\sC$ and an additional codeword $c'$ such that 
  all non-zero weights in $\sC'\backslash\sC$ are either $16$ or at least $28$. If $\sC'$ has effective length $n+\delta$, all codewords have a weight 
  of at most $\Lambda$, and $w(c')=\gamma$, then there exist $y_c\in \{0,1\}$ for all $c\in\sC$ with $w(c)\neq 0$ and $x_i\in\{0,1\}$ for all $1\le i\le n$ 
  such that 
  \begin{equation}
    \sum_{i=1}^n x_i=\gamma-\delta
  \end{equation}
  and 
  \begin{equation}
    \label{ie_for_codeword}
    \frac{\gamma+w(c)}{2}-8-\left(\tfrac{\Lambda}{2}-8\right)y_c \,\le\,\sum_{1\le i\le n\,:\, c_i=1}\!\!\!\!\! x_i \,\le\, \frac{\gamma+w(c)}{2}-8-6y_c,
  \end{equation}  
  for all $c\in C$ with $w(c)\neq 0$.
\end{lemma}
\begin{proof}
  W.l.o.g.\ we assume that $\sC'$ is spanning and that $c_i=0$ for all $n+1\le i\le n+\delta$ and all codewords $c\in \sC$. Since $\sC'$ has effective 
  length $n+\delta$, we have $c'_i=1$ for all $n+1\le i\le n+\delta$. Now let us choose $x_i=c_i'$ for $1\le i\le n$. Since $c'$ has weight $\gamma$ we 
  have $\sum_{i=1}^n x_i=\gamma-\delta$. For each $c\in \sC$ with $w(c)\neq 0$ we set $y_c=0$ if $w(c'+c)=16$ and $y_c=1$ otherwise. Note that
  $$
    w(c'+c)=\gamma+w(c)-2 \sum_{1\le i\le n\,:\, c_i=1} x_i,
  $$ 
  which is equivalent to
  $$
    \sum_{1\le i\le n\,:\, c_i=1} x_i=\frac{\gamma+w(c)}{2}-\frac{w(c'+c)}{2}.  
  $$
  If $y_c=0$ then inequality~\ref{ie_for_codeword} is given by $$\frac{\gamma+w(c)}{2}-8\le \sum_{1\le i\le n\,:\, c_i=1}x_i\le \frac{\gamma+w(c)}{2}-8,$$ which 
  is equivalent to $w(c'+c)=16$. For $y_c=1$ inequality~\ref{ie_for_codeword} is given by $$\frac{\gamma+w(c)}{2}-\Lambda/2\le \sum_{1\le i\le n\,:\, c_i=1} x_i
  \le \frac{\gamma+w(c)}{2}-14,$$ which is equivalent to $28\le w(c'+c)\le \Lambda$.
\end{proof}

Given a code $\sC$ and the parameters $\gamma$, $\delta$, and $\Lambda$ we can computationally check whether the system of lemma~\ref{lemma_ILP} has a solution 
using an integer linear programming solver like e.g.\ \texttt{Cplex}. As an example we spell out all conditions for $\gamma=\Lambda=44$, $\delta=1$ and 
$\sC=\sK_B$ (case (3) of proposition~\ref{prop_classification}):  
\begin{eqnarray}
\sum_{i=1}^{65}x_i=43 &&\nonumber\\
6y_c+\sum_{1\le i\le 65\,:\, c_i=1} x_i \le 34 && \forall c\in\sC:w(c)=40,\nonumber\\
14y_c+\sum_{1\le i\le 65\,:\, c_i=1} x_i \ge 34 && \forall c\in\sC:w(c)=40,\nonumber\\
6y_c+\sum_{1\le i\le 65\,:\, c_i=1} x_i \le 30 && \forall c\in\sC:w(c)=32,\nonumber\\
14y_c+\sum_{1\le i\le 65\,:\, c_i=1} x_i \ge 30 && \forall c\in\sC:w(c)=32,\nonumber\\
6y_c+\sum_{1\le i\le 65\,:\, c_i=1} x_i \le 26 && \forall c\in\sC:w(c)=24,\nonumber\\
14y_c+\sum_{1\le i\le 65\,:\, c_i=1} x_i \ge 26 && \forall c\in\sC:w(c)=24,\nonumber\\
x_i\in\{0,1\}&&\forall 1\le i\le 65,\nonumber\\ 
y_c\in\{0,1\}&&\forall c\in\sC:w(c)\neq 0.\nonumber
\end{eqnarray}

\begin{theo}
  \label{theo_code_sextic_65_unique}
  For nodal sextics in $\PP^3$ with $\nu=65$ nodes, the codes  $\sK$ and $\sK'$ are isomorphic to the codes $\sK_B$ and $\sK'_B$
  of the Barth sextic. $\sK_B$ is code (3) of Proposition \ref{prop_classification}, respectively, see section~\ref{ap_code_barth_sextic}
  for their description.   
\end{theo}
\begin{proof}
  Due to remark~\ref{rem_low_degree} the dimension of $\sK$ is at least $12$ and from \cite{cat-ton} we conclude that the non-zero weights in $\sK$ are 
  contained in $\{24,32,40,56\}$. Thus we can assume that $\sK$ is isomorphic to one of the three cases listed in proposition~\ref{prop_classification}. Note 
  that the effective length is $62+i$ in case ($i$), where $1\le i\le 3$, and that the dimension is $12$ in all cases.  
  
  From proposition~\ref{ineq} we conclude that the dimension of $\sK'$ is $13$. The effective length of $\sK'$ lies between $63+i$ and $66$. As shown in 
  \cite{endrass1} the weights of the codewords in $\sK'\backslash \sK$ are contained in 
  $$
    \{16,28,32,36,40,44,48,52,56,60,64\}=:W.
  $$
  
  So, if $\sK$ is isomorphic to the code in case ($i$) in proposition~\ref{prop_classification}, then the existence of the extended code $\sK'$ implies that 
  the system of lemma~\ref{lemma_ILP} has a feasible solution for at least one of the choices $1\le\delta\le 4-i$ and $\Lambda=\gamma\in W$. Using \texttt{Cplex} 
  we computationally check that there are no feasible solutions except when $i=3$, $\delta=1$, and $\Lambda=\gamma=44$. Thus, $\sK$ is isomorphic to 
  $\sK_B$. For $i=3$, $\delta=1$, and $\Lambda=\gamma=44$ we let \texttt{Cplex} enumerate all solutions. There are exactly $130$ of them, which all yield 
  the same code $\sC'$ with 
  \begin{eqnarray*}
    w_{\sC'}(x,y) -w_{\sK_B}(x,y) &=&  26x^{50}y^{16}+650x^{38}y^{28}+1690x^{34}y^{32}\\ 
    &&+1300x^{30}y^{36}+300x^{26}y^{40}+130x^{22}y^{44},
  \end{eqnarray*} 
  so that $\sK'$ is isomorphic to $\sK'_B$.
\end{proof}


Also for nodal sextics in $\PP^3$ with $\nu=64$ nodes we can apply the same techniques to obtain a list of candidates for the strict codes.
\begin{theo}
  \label{theo_code_candidates_sextic_64_nodes}
  If $\sK$ is the strict code of a nodal sextic in $\PP^3$ with $\nu=64$ nodes, then $\sK$ is isomorphic to one of the codes of the following seven cases:
  \begin{itemize}
    \item[(a)] $\left(\begin{smallmatrix}
1000011001100000111111001011100000011011001100110000100000000000\\
1111000011100010111010111000000000010111000011010010010000000000\\
0111001000100010100010010001101001010010001110111101001000000000\\
0101001001100000000110001110100011001101000111010111000100000000\\
0011011010000001001100110100001011000111000011111100000010000000\\
0010111110010000011101110010111001000000100001001011000001000000\\
0010000110001101000011100001100110111100111111000000000000100000\\
0001111110000100111111101111100001111100011101000111000000010000\\
0000000001111100000000011111100000000011111100111111000000001000\\
0000000000000011111111111111100000000000000011111111000000000100\\
0000000000000000000000000000011111111111111111111111000000000010\\
\end{smallmatrix}\right)$\\[1mm] 
$w_{\sK}(x,y)=x^{64}y^0+310x^{40}y^{24}+1551x^{32}y^{32}+186x^{24}y^{40}$\\
    \item[(b)] $\left(\begin{smallmatrix}
0001010110111111100001010001100001010011010000011010010000000000\\
0100100000011101100100010001100001001110011011101001101000000000\\
1001100100010000101001010110000100010111011010111001000100000000\\
0101000000001110001100101110000100001110101000111011100010000000\\
1100011000100110011000000110000011011100000101010111100001000000\\
0001001001100010000100101101010000110101001101001111100000100000\\
0011000100011110000011101110111110001000100011000000100000010000\\
0011000011111110000000011101111110111011100000111111100000001000\\
0000111111111110000000000011110001111000011111111111100000000100\\
0000000000000001111111111111110000000111111111111111100000000010\\
0000000000000000000000000000001111111111111111111111100000000001\\
\end{smallmatrix}\right)$\\[1mm]
$w_{\sK}(x,y)=x^{64}y^0+246x^{40}y^{24}+1551x^{32}y^{32}+250x^{24}y^{40}$\\
    \item[(c)] $\left(\begin{smallmatrix}
0000011010000001000110111110001110100101111000110010010000000000\\
0111001100100010000011101000001011100100101101101010001000000000\\
0000101110101011000110100111100101110100100101000000000100000000\\
0111101110001010001100101011001010000101010011010000000010000000\\
1001110100000110100000011010000101011011000100111001100001000000\\
1101010010011110010001000110111100000011001100001000100000100000\\
0101010001111110001011000001110100001000000011000111100000010000\\
0011001111111110000111000000010010111000111110111111100000001000\\
0000111111111110000000111111110001111000000001111111100000000100\\
0000000000000001111111111111110000000111111111111111100000000010\\
0000000000000000000000000000001111111111111111111111100000000001\\
\end{smallmatrix}\right)$\\[1mm]
$w_{\sK}(x,y)=x^{64}y^0+246x^{40}y^{24}+1551x^{32}y^{32}+250x^{24}y^{40}$\\
    \item[(d)] $\left(\begin{smallmatrix}
0001010101110111100001101010100010001101001101000101011001111110\\
0100010001010001011000010010001110000110010010010001010011001011\\
0000000100010000111010110110101011101110110000011011011111100110\\
0101000001010101101011110001100011111110101010101000111110100000\\
0101010100000101111110010000011011100111111010010100011011110000\\
1100000000110011001100000000000000000011000011001111111111000011\\
0011000000110011110011110000000000000011001100001111000000110011\\
0000110000110000001100110000000000110000111111000000110011001111\\
0000001100110000001111000000000000110011000000111100110000111111\\
0000000011110000111100000000000000111100001111111100111100000000\\
0000000000001111111111110000000000111111111111000000000000000000\\
\end{smallmatrix}\right)$\\[1mm]
$w_{\sK}(x,y)=x^{64}y^0+246x^{40}y^{24}+1551x^{32}y^{32}+250x^{24}y^{40}$\\
    \item[(e)] $\left(\begin{smallmatrix}
0001100001001110011000111010111101111110001001110010010010011001\\
0000000001010110000000000110100101101001100110011100001100110101\\
0000111101010110010101010000111100111100011001100011001100101011\\
0000111101010110001111001001011000000000000000001010010101111000\\
1001100101011001000011111010101001011010010101010000000000000000\\
0101101000000000010101010101010100111100011010011001100100000000\\
0011110001011001010110100000000001100110110000111001100100000000\\
0000000011000011000011110000111100111100001111000011110000000000\\
0000000000111100000011111111000000111100001111001100001100000000\\
0000000000000000111111111111111100000000111111111111111100000000\\
0000000000000000000000000000000011111111111111111111111100000000\\
\end{smallmatrix}\right)$\\[1mm] 
$w_{sK}(x,y)=x^{64}y^0+243x^{40}y^{24}+1559x^{32}y^{32}+244x^{24}y^{40}+x^{8}y^{56}$\\
\item[(f)]
    $\left(\begin{smallmatrix} 
    0011000011100000011111010000111111100100101001000011000000000000\\
    1010011111110000001101110100001001101000110110000000100000000000\\
    0001001110111000111101110010000100001100001101101000010000000000\\
    0100011111111100110011000010011001000100011010000010001000000000\\
    1100011100000101110011110110000111001000110000101000000100000000\\
    0000000110001101111000110100111011100100010111100000000010000000\\
    0100111100011111010100001101001000111011101111111110000001000000\\
    0010001101111011000011111100000000011001100001111000000000100000\\
    0001111100011000110000000011000111111000011000011110000000010000\\
    0000000011111000001111111111000000000111111000000110000000001000\\
    0000000000000111111111111111000000000000000111111110000000000100\\
    0000000000000000000000000000111111111111111111111110000000000010\\
    \end{smallmatrix}\right)$\\[1mm]
    $w_{\sK}(x,y)=x^{65}y^{0}+630x^{41}y^{24}+3087x^{33}y^{32}+378x^{25}y^{40}$\\
    \item[(g)]
    $\left(\begin{smallmatrix}
    1000011011010011000001010110011000101001100000100001000001100110\\
    0001100111110110000000000100000001100000001100011101001100111001\\
    0000101011110100001000010101001000010001111001000011000000110011\\
    0000111110100001000010100011001001110001100001000000100101010110\\
    0001001111111011011010110011001000010101110111000000011001011001\\
    0001001011011011000010010100000000000110011000010000000011111111\\
    0101010101011010010001010010001110100101110101101000000000000000\\
    0000110011001100100011110111010101101001111001000000000000000000\\
    0000110000111100011100000111010010010111111000111000000000000000\\
    0000111111111111011011001001001100011000100000100000000000000000\\
    0000001111001100011000110000111010011000111111011000000000000000\\
    0011111111111111011011110110000001100110011110011000000000000000\\
    \end{smallmatrix}\right)$\\[1mm]
    $w_{\sK}(x,y)=x^{55}y^{0}+502x^{41}y^{24}+3087x^{33}y^{32}+506x^{25}y^{40}$\\
  \end{itemize}
\end{theo}
\begin{proof}
  Due to remark~\ref{rem_low_degree} the dimension $k$ of $\sK$ is at least $11$ and from \cite{cat-ton} we conclude that the non-zero weights in $\sK$ are 
  contained in $\{24,32,40,56\}$. If $k\ge 12$, then we can read off the possibilities for $\sK$ from proposition~\ref{prop_classification}. The two 
  cases with effective length at most $64$ are listed as cases (f) and (g). 
  
  For the remaining cases we can assume that $k=11$. Using the software package \texttt{LinCode} we enumerate all $357$ codes with dimension $11$, length $64$, 
  and non-zero weights in $\{24,32,40,56\}$. To sift out this list we consider the extended code $\sK'$. From proposition~\ref{ineq} we conclude that the dimension 
  of $\sK'$ is $12$. As shown in \cite{endrass1} the weights of the codewords in $\sK'\backslash \sK$ are contained in 
  $$
    \{16,28,32,36,40,44,48,52,56,60,64\}=:W.
  $$
  
  So, if a candidate for $\sK$ has effective length $n$, then the existence of the extended code $\sK'$ implies that 
  the system of lemma~\ref{lemma_ILP} has a feasible solution for at least one of the choices $1\le\delta\le 65-n$ and $\Lambda=\gamma\in W$. Using \texttt{Cplex} 
  we computationally check that there are no feasible solutions except for the five codes listed in cases (a)-(e).
\end{proof}  
We can check by computer that the Barth code $\sK_B$ has a unique shortening to length $64$ which is isomorphic to case (d). So, by 
theorem~\ref{thm_d_realized} this code is realized by a nodal sextic surface. For the other six cases it remains an open problem 
whether they are realized by a nodal sextic surface.  

\medskip

We remark that the approach of the proofs of theorem~\ref{theo_code_sextic_65_unique} and theorem~\ref{theo_code_candidates_sextic_64_nodes} is, in principle, also 
feasible for nodal sextics in $\PP^3$ with $\nu\le 63$ nodes. The candidates for $\sK$ can be easily enumerated using \texttt{LinCode}. However, 
for $\nu=63$ there are more than 70\,000 candidates, so that further coding theoretic criteria obtained from algebraic geometry methods are desirable.

\end{document}